\documentclass[10pt,reqno]{amsart}
\usepackage{amsfonts,amsrefs,latexsym,amsmath, amssymb, mathrsfs, verbatim,cancel}
\usepackage{url,color}
\usepackage{pifont}
\usepackage{upgreek}
\usepackage{fancyhdr}
\usepackage{hyperref}
\usepackage{calligra}
\usepackage{marvosym}
\usepackage[percent]{overpic}
\usepackage{pict2e}
\usepackage[hypcap=false]{caption}
\usepackage{xcolor}
\usepackage{wasysym}
\usepackage{accents}
\usepackage{enumitem}
\usepackage{tabularx}
\setlist[enumerate]{label*=\arabic*.}

\topmargin - .23 in

\newtheorem{theorem}{Theorem}[section]
\newtheorem{proposition}[theorem]{Proposition}
\newtheorem{lemma}[theorem]{Lemma}
\newtheorem{corollary}[theorem]{Corollary}

\theoremstyle{definition}
\newtheorem{definition}[theorem]{Definition}
\newtheorem{remark}[theorem]{Remark}
\newtheorem{convention}[theorem]{Convention}
\newtheorem{example}[theorem]{Example}

\DeclareMathAlphabet{\mathcalligra}{T1}{calligra}{m}{n}
\DeclareFontShape{T1}{calligra}{m}{n}{<->s*[2.2]callig15}{}

\newcommand{\Ent}{s}
\newcommand{\GradEnt}{S}
\newcommand{\LeftoverGradEnt}{U}

\newcommand{\LogDensity}{\uprho}

\newcommand{\vortrenormalized}{\Omega}
\newcommand{\Speed}{c}

\newcommand{\VortVort}{\mathcal{C}}
\newcommand{\DivGradEnt}{\mathcal{D}}

\newcommand{\dive}{\mbox{\upshape div}}
\newcommand{\curl}{\mbox{\upshape curl}}

\newcommand{\Transport}{\mathbf{B}}

\newcommand{\Timefunction}{\uptau}

\newcommand{\gensmoothfunction}{\mathrm{f}}

\newcommand{\enmomem}{\mathbf{Q}}

\newcommand{\SigmatTan}{V}

\newcommand{\spherenormal}{Z}

\newcommand{\gen}{\underline{H}}
\newcommand{\modgen}{\breve{\underline{H}}}
 
\newcommand{\lapsemodgen}{\underline{\iota}} 
\newcommand{\utang}{\Theta}
\newcommand{\tang}{\check{\Theta}}
\newcommand{\uspecialgen}{\underline{E}}
\newcommand{\specialgen}{E}
\newcommand{\keydetvectorfield}{Y}
\newcommand{\sidehypnorm}{\underline{N}}
\newcommand{\tophypnorm}{N}
\newcommand{\modtophypnorm}{Q}

\newcommand{\urescalednewgenminushypnorm}{P}
\newcommand{\rescalednewgenminushypnorm}{\check{P}}

\newcommand{\projectedtransport}{K}

\newcommand{\utandecompvectorfielddownarg}[1]{W_{(#1)}}
\newcommand{\utandecompvectorfieldmixedarg}[2]{W_{(#1)}^{#2}}

\newcommand{\tandecompvectorfielddownarg}[1]{\check{W}_{(#1)}}

\newcommand{\lengthofgen}{\underline{\upeta}}
\newcommand{\lengthofmodgen}{\underline{\ell}}

\newcommand{\lengthoftophypnorm}{\upnu}
\newcommand{\lengthofmodtophypnorm}{q}
\newcommand{\lengthofsidehypnorm}{\underline{\upnu}}

\newcommand{\uposinnerproduct}{\underline{z}}
\newcommand{\seconduposinnerproduct}{\underline{h}}
\newcommand{\posinnerproduct}{z}
\newcommand{\secondposinnerproduct}{h}

\newcommand{\Jenarg}[2]{{^{(#1)} \mkern-3mu \mathbf{J}^{#2}}}
\newcommand{\Jenwithlowerarg}[2]{{^{(#1)} \mkern-3mu \widetilde{\mathbf{J}}^{#2}}}

\newcommand{\deformarg}[3]{{^{(#1)} \mkern-1mu \pmb{\pi}_{#2 #3}}}

\newcommand{\deformmixedarg}[3]{{^{(#1)} \mkern-1mu \pmb{\pi}_{\ #2}^{#3}}}

\newcommand{\gfour}{\mathbf{g}}

\newcommand{\gsphere}{g \mkern-8.5mu / }

\newcommand{\euc}{\mathbf{e}}
\newcommand{\euct}{e}

\newcommand{\topfirstfund}{\widetilde{g}}
\newcommand{\sidefirstfund}{\underline{g}}

\newcommand{\Chfour}{\pmb{\Gamma}}

\newtheorem{notation}{Notation}[section]

\newcommand{\Flatdiv}{\mbox{\upshape div}\mkern 1mu}
\newcommand{\Flatcurl}{\mbox{\upshape curl}\mkern 1mu}

\newcommand{\sideproject}{\underline{\Pi}}
\newcommand{\sphereproject}{{\Pi \mkern-12mu / } \, }
\newcommand{\Sigmatproject}{\Pi}
\newcommand{\topproject}{\widetilde{\Pi}}

\newcommand{\Lunit}{L}
\newcommand{\uLunit}{\underline{L}}
\newcommand{\Lgeo}{L_{(Geo)}}
\newcommand{\uLgeo}{\underline{L}_{(Geo)}}
\newcommand{\newL}{\breve{L}}
\newcommand{\newuL}{\breve{\underline{L}}}

\newcommand{\ReciprocalLunitAppliedtoTimeFunction}{\iota}
\newcommand{\ReciprocaluLunitAppliedtoTimeFunction}{\underline{\iota}}
\newcommand{\MagnitueofinnerproductofnewLandnewuL}{\uplambda}

\newcommand{\Dfour}{\mathbf{D}}

\newcommand{\angD}{ {\nabla \mkern-14mu / \,} }

\newcommand{\angdiv}{\mbox{\upshape{div} $\mkern-17mu /$\,}}

\newcommand{\angpartial}{ {\partial \mkern-10mu / \,} }
\newcommand{\angpartialarg}[1]{{\angpartial_{\mkern-3mu #1}}}
\newcommand{\angpartialuparg}[1]{{\angpartial^{\mkern-3mu #1}}}
\newcommand{\toppartial}{\widetilde{\partial}}
\newcommand{\toppartialarg}[1]{\widetilde{\partial}_{#1}}
\newcommand{\toppartialuparg}[1]{\widetilde{\partial}^{#1}}
\newcommand{\sidepartial}{\underline{\partial}}
\newcommand{\sidepartialarg}[1]{\underline{\partial}_{#1}}
\newcommand{\sidepartialuparg}[1]{\underline{\partial}^{#1}}

\newcommand{\angvortrenormalized}{ { {\vortrenormalized \mkern-12mu /} \, } }
\newcommand{\angGradEnt}{ { {\GradEnt \mkern-11mu /} \, } }

\newcommand{\angV}{ { {V \mkern-14mu /} \, } }
\newcommand{\angVarg}[1]{ {{V \mkern-14mu /}^{\mkern7mu #1} \, } }

\newcommand{\weight}{\mathscr{W}}

\newcommand{\Lie}{\mathcal{L}}

\newcommand{\toten}{\mathbb{Q}}
\newcommand{\spacetimeen}{\mathbb{K}}

\begin{document}
\title{Remarkable localized integral identities for $3D$ compressible Euler flow
				and the double-null framework
}
\author[LA]{Leonardo Abbrescia$^{* \dagger}$}
\author[JS]{Jared Speck$^{** \dagger \dagger}$}
	
\thanks{$^{\dagger}$LA gratefully acknowledges support from NSF GRF award \# DGE-1424871.}

\thanks{$^{\dagger \dagger}$JS gratefully acknowledges support from NSF grant \# 1162211,
from NSF CAREER grant \# 1454419,
and from a Sloan Research Fellowship provided by the Alfred P. Sloan foundation.}

\thanks{$^{*}$Michigan State University, East Lansing, MI, USA
\texttt{abbresci@msu.edu}}

\thanks{$^{**}$Vanderbilt University, Nashville, TN, USA
\texttt{jared.speck@vanderbilt.edu}}

\begin{abstract}
We derive new, localized geometric integral identities for solutions to the $3D$ compressible Euler equations 
under an arbitrary equation of state when the sound speed is positive.
The identities are coercive in the first derivatives of the specific vorticity (defined to be vorticity divided by density)
and the second derivatives of the entropy,
and the error terms exhibit remarkable regularity and null structures.
Our framework allows one to simultaneously unleash the full power of the geometric vectorfield method for both the wave- and transport- parts of the flow
on compact regions, and our approach reveals fundamental new coordinate invariant structural features of the flow.
In particular, the integral identities yield localized control over one additional derivative 
of the vorticity and entropy compared to standard results,
assuming that the initial data enjoy the same gain. 
Similar results hold for the solution's higher derivatives.
We derive the identities in detail for two classes of spacetime regions that frequently arise in PDE applications:
\textbf{i)} compact spacetime regions that are globally hyperbolic with respect to the acoustical metric,
	where the top and bottom boundaries are acoustically spacelike -- but not necessarily
	equal to portions of constant Cartesian-time hypersurfaces;
and
\textbf{ii)} compact regions covered by double-acoustically null foliations.
As we describe in the paper,
the results have implications for the geometry and regularity of solutions, 
the formation of shocks,
the structure of the maximal classical development of the data,
and for controlling solutions whose state along a pair of
intersecting characteristic hypersurfaces is known. 
Our analysis relies on a recent new formulation 
of the compressible Euler equations that splits the flow into
a geometric wave-part coupled to a div-curl-transport part.
The main new contribution of the present article 
is our analysis of the positive co-dimension, spacelike boundary integrals that arise in the
div-curl identities. By exploiting interplay between the elliptic and hyperbolic parts of the new formulation
and using careful geometric decompositions,
we observe several crucial cancellations,
which in total show that after a further integration with respect to an acoustical time function,
the boundary integrals 
have a good sign,
up to error terms that can be controlled due to their good null structure and regularity properties.

\bigskip

\bigskip


\end{abstract}

\maketitle

\centerline{\today}

\tableofcontents

\newpage

\section{Introduction}
\label{S:INTRO}
The compressible Euler equations are the fundamental equations of compressible fluid mechanics.
They are arguably on par with Einstein's equations of general relativity in terms of their 
mathematical richness, physical relevance,
and variety of subtle nonlinear behaviors that solutions can exhibit.
The equations remain a fertile source of outstanding mathematical challenges,
even though they are among the earliest PDE systems written down
(they were formulated by Euler in 1757 \cite{lE1757}).
Despite the complexity of the equations, 
in recent years, the rigorous mathematical theory of solutions has enjoyed dramatic progress, 
driven by insights and identities that have their origins in Lorentzian geometry, 
general relativity,
the theory of geometric wave equations,
and, in some key cases,
the remarkable structures exhibited by a new formulation of the equations \cites{jLjS2016a,jS2019c}
as geometric wave equations coupled to div-curl-transport equations.
As examples of progress, 
we note Christodoulou's work \cite{dC2007} on shock formation in $3D$ without symmetry assumptions 
for irrotational relativistic Euler solutions, 
his joint extension \cite{dCsM2014} of this result to the case of the non-relativistic compressible Euler equations in $3D$,
his subsequent resolution of the restricted shock development problem \cite{dC2019} in both the relativistic and non-relativistic cases, 
the second author's joint extension \cite{jLjS2018} of Christodoulou's shock formation result to allow for
the presence of vorticity in $2D$ in the barotropic\footnote{Barotropic equations of state are such that 
the pressure can be expressed a function of the density alone.} case, 
the second author's recent
joint work \cite{mDcLgMjS2019} on low regularity solutions with vorticity and entropy in $3D$,
Wang's extension of this work \cite{qW2019} to further lower the regularity of the vorticity in the barotropic case,
and the existence of initially $C^{\infty}$ solutions that form an infinite-density singularity in
finite time \cites{fMpRiRjS2019b,fMpRiRjS2019c}.

In this paper, we study
the $3D$ compressible Euler equations 
under an arbitrary physical\footnote{By a ``physical equation of state,'' we mean that the speed of sound, defined in \eqref{E:SOUNDSPEED}, 
is assumed to be positive, at least for an open set of positive density values (which would be a set of density values for which our results hold).}
equation of state in which the pressure $p$ is a given function of the density $\varrho$ and entropy $\Ent$. 
Our main results augment the geometric insights 
and tools developed in \cites{dC2007,jLjS2016a,jLjS2018,jS2019c}
by allowing for a sharp localization of the key structures
that are needed to control the ``div-curl-transport-part'' of the flow, that is, 
the vorticity and entropy.
We now informally and tersely summarize our main results;
see Theorem~\ref{T:MAINIDSCHEMATICSTATEMENT} 
for a more precise -- but still schematic -- statement of the results,
and
Theorems~\ref{T:STRUCTUREOFERRORTERMS}, \ref{T:MAINREMARKABLESPACETIMEINTEGRALIDENTITY},
\ref{T:LOCALIZEDAPRIORIESTIMATES}, and \ref{T:DOUBLENULLMAINTHEOREM} for precise statements.
See also Subsect.\,\ref{SS:REMARKSONTHEPROOF} for a summary of the key ideas in the proofs.
\begin{quote}
	We derive a new family of coercive, localized (i.e., on compact spacetime regions) integral identities 
	that yield spacetime $L^2$-type control over the vorticity and entropy at one derivative level higher 
	compared to standard estimates. By ``one derivative level,'' 
	we mean that the gain in regularity for the vorticity and entropy holds
	for \emph{all} of their Cartesian partial derivatives, not just 
	in the direction of the material derivative vectorfield
	(see \eqref{E:MATERIALVECTORVIELDRELATIVECTORECTANGULAR}),
	for which the improved regularity is a standard result.\footnote{The additional
	regularity for the specific vorticity and entropy gradient
	in the direction of the material derivative vectorfield $\Transport$ is straightforward to derive by using
	the transport equations \eqref{E:RENORMALIZEDVORTICTITYTRANSPORTEQUATION} and \eqref{E:GRADENTROPYTRANSPORT}
	for algebraic substitution and showing that the terms on the right-hand sides of
	\eqref{E:RENORMALIZEDVORTICTITYTRANSPORTEQUATION} and \eqref{E:GRADENTROPYTRANSPORT} have the desired regularity.
	In fact, one can derive the additional regularity of the material derivative of the 
	specific vorticity and entropy gradient
	with respect to norms of type $L^{\infty}(Time)L^2(Space)$, 
	which is a stronger result (at least locally in time) compared to achieving control in a spacetime $L^2$-type norm;
	see Footnote~\ref{FN:LINFINITYTIMEL2SPACECONTROLOFMATERIALDERIVATIVETERMS}.}
	The error terms in the integral identities, \emph{especially the positive co-dimension boundary integrals},
	exhibit remarkable quasilinear null structures and regularity properties.
	The identities hold on a large family of compact acoustically globally hyperbolic spacetime regions
	of the type that typically arise in PDE applications; 
	see Fig.\,\ref{F:SPACETIMEDOMAIN} for an example of a region
	and Remark~\ref{R:ACOUSTICALLYGLOBALLYHYPERBOLIC} for a discussion of why the domains are acoustically globally hyperbolic.
	In particular, our results allow us to implement the framework of \emph{double-(acoustically)-null} foliations
	in compressible fluid mechanics; see Fig.\,\ref{F:DOUBLENULL}. 
	Our approach does not rely on special coordinate systems such as Lagrangian coordinates,
	but rather reveals fundamental new coordinate invariant structural features of the flow.
	We believe that the integral identities and the remarkable structure of the error terms
	are key new ingredients for studying important aspects of solutions
	that were previously inaccessible. We discuss some of these aspects in Subsect.\,\ref{SS:APPLICATIONS}.
\end{quote}

While analogous integral identities are entirely standard in the context of wave equations 
and are easily derivable by the vectorfield multiplier method
(see Subsect.\,\ref{SS:VECTORFIELDMULTIPLIERMETHOD}),
the compressible equations are not wave equations. 
That is, it is well-known that solutions exhibit two kinds of propagation phenomena: 
the propagation of sound waves,
which is present even in the simplified setting of irrotational and isentropic solutions,
and the transporting of $\Ent$ and the vorticity $\upomega := \Flatcurl v$ (where $v$ is the velocity), 
which is present in general solutions and which
occurs at a ``slower'' speed\footnote{Sound waves can propagate along acoustically null hypersurfaces,
while \eqref{E:TRANSPORTISUNITLENGTHANDTIMELIKE} and the equations of Theorem~\ref{T:GEOMETRICWAVETRANSPORTSYSTEM}
imply that specific vorticity and entropy are transported along acoustically timelike curves.} 
compared to sound waves.
More precisely, as we alluded to above,
it is more accurate, at least at the top derivative level, 
to describe the evolution of vorticity and entropy as being driven by ``div-curl-transport'' equations. 
That is, \cites{jLjS2016a,jS2019c} 
showed that the compressible Euler equations
can be formulated as geometric wave equations coupled to div-curl-transport equations
for the specific vorticity $\vortrenormalized$ (defined below to be $\upomega$ divided by a dimensionless density)
and the entropy.\footnote{More precisely, the div-curl-transport system 
\eqref{E:TRANSPORTFLATDIVGRADENT}-\eqref{E:CURLGRADENTVANISHES}
for the entropy is expressed in terms of the entropy gradient vectorfield,
which we denote by $\GradEnt$ and define below in \eqref{E:ENTROPYGRADIENT}.}
We now further explain -- still informally -- the significance of our main results.

\begin{itemize}
	\item (\textbf{Extending the geometric vectorfield method to the div-curl-transport part on spatially compact regions}) 
		It is precisely the div-curl-transport part of the flow that lies outside of the scope
		of the traditional geometric vectorfield method (developed for wave and wave-like equations), 
		and our integral identities allow us to handle this part of the flow,
		\emph{notably the difficult boundary integrals} 
		that arise when we derive elliptic Hodge-type identities on spatially compact regions.
		Handling these (positive co-dimension) boundary integrals requires a combination of elliptic, hyperbolic, and geometric techniques
		that rely on the special structures of the formulation of compressible Euler flow derived in
		\cites{jLjS2016a,jS2019c}
		as well as the precise structure of equations\footnote{For example, in the proof of Lemma~\ref{L:TRANSPORTLOGDENSITYISTANGENTIALEXCEPTINNULLCASE},
		we use equations \eqref{E:TRANSPORTDENSRENORMALIZEDRELATIVECTORECTANGULAR}-\eqref{E:TRANSPORTVELOCITYRELATIVECTORECTANGULAR}.}  
		\eqref{E:TRANSPORTDENSRENORMALIZEDRELATIVECTORECTANGULAR}-\eqref{E:ENTROPYTRANSPORT},
		which are a standard first-order formulation of compressible Euler flow;
		see also Remark~\ref{R:EXPLOITINGSPECIALSTRUCTUREOFCOMPRESSIBLEEULER}.
	\item (\textbf{Regularity}). Specifically, our results yield a large family of energy identities
	for the \emph{first} derivatives of the specific vorticity $\vortrenormalized$ 
	and the \emph{second} derivatives of the entropy $\Ent$
	on the (compact) spacetime regions under consideration, where the error terms
	involve the up-to-\emph{first} order derivatives of velocity $v$, the density\footnote{In practice, rather than
	working with the density, we prefer to work with the logarithmic density $\LogDensity$, which we define in \eqref{E:LOGDENSITY};
	since we consider only solutions with strictly positive density, these two variables are essentially equivalent for the purposes of this article.
	Moreover, rather than working with the second derivatives of $\Ent$, 
	we prefer to work with $\pmb{\partial} \GradEnt$, 
	where $\GradEnt$ is the entropy gradient vectorfield defined in \eqref{E:ENTROPYGRADIENT}.} 
	$\varrho$, and $\Ent$.
	That is, with $\pmb{\partial}$ denoting the Cartesian coordinate spacetime gradient and
	$\partial$ denoting the Cartesian coordinate spatial gradient,
	our results yield $L^2$ control of
	$\pmb{\partial} \vortrenormalized$ and $\pmb{\partial} \partial \Ent$ in terms of $L^2$ norms of
	$\pmb{\partial}^{\leq 1} \varrho$, 
	$\pmb{\partial}^{\leq 1} v$, 
	$\pmb{\partial}^{\leq 1} \Ent$,
	$\VortVort$,
	and 
	$\DivGradEnt$,
	where the latter two quantities,
	discovered in \cites{jLjS2016a,jS2019c} and recalled below in Def.\,\ref{D:RENORMALIZEDCURLOFSPECIFICVORTICITY},
	are special combinations of fluid solution variables that solve transport equations 
	(see \eqref{E:EVOLUTIONEQUATIONFLATCURLRENORMALIZEDVORTICITY} and \eqref{E:TRANSPORTFLATDIVGRADENT})
	with source terms exhibiting unexpectedly good regularity and null properties.
	In turn, 
	under suitable assumptions on the initial data,
	the quantities
	$\pmb{\partial}^{\leq 1} \varrho$, 
	$\pmb{\partial}^{\leq 1} v$, 
	$\pmb{\partial}^{\leq 1} \Ent$,
	$\VortVort$,
	and 
	$\DivGradEnt$,
	can be shown to have sufficient $L^2$ regularity
	by virtue of standard energy estimates for the wave equations \eqref{E:VELOCITYWAVEEQUATION}-\eqref{E:ENTROPYWAVEEQUATION}
	and the transport equations \eqref{E:EVOLUTIONEQUATIONFLATCURLRENORMALIZEDVORTICITY} and \eqref{E:TRANSPORTFLATDIVGRADENT}.
	In particular, our results allow 
	one to locally propagate a gain of one derivative 
	worth of Sobolev regularity for the
	vorticity and entropy compared to standard estimates,
	which is important for the study of the maximal development in the context of shock formation; see Subsect.\,\ref{SS:APPLICATIONS}.
	\item (\textbf{Null structure}). The ``error terms'' in the integral identities for $\pmb{\partial} \vortrenormalized$ and $\pmb{\partial} \partial \Ent$
	have remarkable quasilinear null structures and regularity properties 
	that, as we explain in Subsect.\,\ref{SS:APPLICATIONS}, are crucial for applications. 
	For example, for regions that have acoustically null hypersurface boundaries,
	the identities feature error integrals along the null hypersurfaces,
	but these error terms involve only \emph{tangential} derivatives of quantities
	that have \emph{sufficient regularity}. By ``sufficient regularity,'' we mean in particular
	that the error terms can be treated with transport equation estimates or
	wave equation energy estimates, 
	where the latter, though well-known to be \emph{degenerate} along null hypersurfaces, 
	yield control over tangential derivatives.
	\item (\textbf{Flexibility of the approach}).
	In total, the integral identities allow one to simultaneously implement the 
	full power of the geometric vectorfield method
	for both the wave- and div-curl-transport-parts of the solution on domains that are important for PDE applications,
	in a manner such that the error terms exhibit good regularity and quasilinear null structures.
	Moreover, our framework is well-suited to handle the kinds of ``custom modifications''
	that are typically needed in applications.
	In particular, one could commute the equations with appropriate geometric vectorfields to obtain
	similar integral identities for the solution's higher derivatives, 
	one could incorporate weights\footnote{In fact, in our main results, we allow for the presence of an arbitrary weight function, denoted by
	$\weight$.} into the identities, etc.
	That is, our framework affords flexibility in the ways it can be implemented.
\end{itemize}

Central to the approach of the present paper
are the div-curl-transport systems 
for the specific vorticity and entropy gradient
derived in \cites{jLjS2016a,jS2019c};
see equations
\eqref{E:FLATDIVOFRENORMALIZEDVORTICITY}-\eqref{E:EVOLUTIONEQUATIONFLATCURLRENORMALIZEDVORTICITY}
and
\eqref{E:TRANSPORTFLATDIVGRADENT}-\eqref{E:CURLGRADENTVANISHES}.
These systems are adapted to constant Cartesian time hypersurfaces.
From the perspective of analysis, the main new contributions of the present work are as follows.
\begin{enumerate}
	\item 
	Even though the
	div-curl systems from \cites{jLjS2016a,jS2019c}
	(which we recall as 
	\eqref{E:FLATDIVOFRENORMALIZEDVORTICITY}-\eqref{E:EVOLUTIONEQUATIONFLATCURLRENORMALIZEDVORTICITY}
	and
	\eqref{E:TRANSPORTFLATDIVGRADENT}-\eqref{E:CURLGRADENTVANISHES})
	are PDEs relative to flat hypersurfaces of constant Cartesian time,
	our results yield coercive integral identities on regions foliated by \emph{arbitrary acoustically spacelike hypersurfaces}.
\item The work \cite{jS2019c} outlined how to derive the integral identities on much simpler spacetime regions of the form
$[0,T] \times \Sigma$, where the $3D$ ``spatial manifold'' $\Sigma$
had an \emph{empty boundary} (for example, $\Sigma = \mathbb{R}^3$ or $\Sigma = \mathbb{R} \times \mathbb{T}^2$);
see Subsubsect.\,\ref{SSS:SIMPLERDOMAINS} for the main ideas.
In the present article, we handle the \emph{spatial boundary terms} that arise in various
div-curl identities on \emph{spatially compact} domains.
In particular, we show that the boundary integrals have a compatible amount of regularity and, at the same time,
exhibit the remarkable null structures that have proven to be important in applications. 
As we will explain, our localized results \emph{do not hold 
for typical div-curl-transport systems}, 
especially on the full class of spacetime regions 
that we treat in this paper; our results are possible only 
because we have exploited some newly identified structures in the compressible Euler equations.
\end{enumerate}

We close this introduction by highlighting the significance of allowing the equation of state to depend on entropy:
\begin{itemize}
	\item Entropy is an unavoidable ingredient in fluid models that incorporate thermodynamic effects.
		Moreover, non-constant entropy is generated past the formation of a shock \cite{dC2007}, 
		even if before the shock the solution is smooth, irrotational, and isentropic.
	\item In the integral identities, the most difficult error terms 
		-- by far -- involve the second derivatives of the entropy.
		We devote Subsects.\,\ref{SS:MOSTSUBTLEENTROPYTERM}
		and \ref{SS:REMARKABLESTRUCTUREOFMOSTSUBTLETERM}
		to proving that these terms exhibit the remarkable structures that are
		crucial for our main results. This requires a host of insights about
		hidden ``quasilinear'' geometric and analytic structures in the equations.
\end{itemize}

\subsection{Basic notation and a standard first-order formation of the equations}
\label{SS:FIRSTORDERFORMULATION}
Before further discussing our main results, we set up our study of compressible Euler flow
by introducing a standard first-order formulation of the equations,
specifically the system \eqref{E:TRANSPORTDENSRENORMALIZEDRELATIVECTORECTANGULAR}-\eqref{E:ENTROPYTRANSPORT}.
We again stress that \emph{the first-order formulation is not the one we use to derive our main results};
rather, we use the equations of Theorem~\ref{T:GEOMETRICWAVETRANSPORTSYSTEM},
which are consequences of the first-order formulation and which were derived in \cite{jS2019c}.

\subsubsection{Basic notation and conventions}
\label{SSS:BASICNOTATION}
Throughout,
$\lbrace x^{\alpha} \rbrace_{\alpha=0,1,2,3}$ 
denotes the standard Cartesian coordinates 
on $\mathbb{R}^{1+3}$,
where $x^0 :=t$ denotes time (we also refer to $t$ as the ``Cartesian time function'')
and $\lbrace x^a \rbrace_{a=1,2,3}$ are the Cartesian spatial coordinates.
We denote the standard partial derivative vectorfields with respect to the Cartesian coordinates by
$\partial_{\alpha} := \frac{\partial}{\partial x^{\alpha}}$,
and we often use the alternate notation $\partial_t := \partial_0$.
In addition, $\Sigma_t$ denotes the standard flat three-dimensional hypersurface of constant Cartesian time $t$.
Moreover, lowercase Latin ``spatial'' indices such as $a$ vary over $1,2,3$,
lowercase Greek ``spacetime'' indices such as $\alpha$ vary over $0,1,2,3$, where $0$ is the ``time'' index,
and we use Einstein's summation convention in that repeated indices are summed over their ranges.
If $f$ is a scalar function and $\mathbf{X}$ is a vectorfield, then $\mathbf{X} f := \mathbf{X}^{\alpha} \partial_{\alpha} f$
denotes the derivative of $f$ in the direction $\mathbf{X}$.
If $\SigmatTan$ is a $\Sigma_t$-tangent vectorfield with Cartesian spatial components $\lbrace \SigmatTan^i \rbrace_{i=1,2,3}$, 
then  $\curl \SigmatTan := \upepsilon_{ijk} \updelta^{jl} \partial_l \SigmatTan^k$ denotes its standard Euclidean curl,
where $\upepsilon_{ijk}$ denotes the fully antisymmetric symbol normalized by
$\upepsilon_{123} = 1$ and $\updelta^{ij}$ denotes the Kronecker delta.
In addition, $\dive \SigmatTan := \partial_a \SigmatTan^a$ denotes its Euclidean divergence.
Finally, $\upepsilon_{\alpha \beta \gamma \delta}$ denotes the fully antisymmetric symbol normalized by
$\upepsilon_{0123} = 1$. See Subsubsect.\,\ref{SSS:ACOUSTICALMETRIC} and Convention~\ref{C:LOWERANDRAISE} 
for our conventions for lowering and raising indices.

\subsubsection{Setup and definitions of the basic fluid variables}
\label{SSS:SETUPANDDEFSOFBASICFLUIDVARIABLES}
The compressible equations can be formulated as evolution equations for the density $\varrho : \mathbb{R}^{1+3} \rightarrow [0,\infty)$,
the Cartesian velocity components $v^i : \mathbb{R}^{1+3} \rightarrow \mathbb{R}$, ($i=1,2,3$),
and the entropy $\Ent : \mathbb{R}^{1+3} \rightarrow \mathbb{R}$.
When setting up the equations, authors frequently also use the pressure 
$p : \mathbb{R}^{1+3} \rightarrow [0,\infty)$.
The resulting PDE system is under-determined unless one supplies an additional equation. Here we close the system
in the standard fashion by assuming an equation of state $p = p(\varrho,\Ent)$, that is, a function yielding
the pressure in terms of the density and entropy.
In this article, we consider only equations of state such that the speed of sound $\Speed$ is positive when the density is positive:\footnote{Of course, if $\Speed$ is positive only on an open set of positive density values, then our results hold for solutions whose density is contained in that open set.}
\begin{align} \label{E:SOUNDSPEED}
	\Speed & := \sqrt{\frac{\partial p}{\partial \varrho}\left|\right._{\Ent}} > 0,
	& & \mbox{\upshape when } \varrho > 0.
\end{align}
In \eqref{E:SOUNDSPEED},
$\frac{\partial p}{\partial \varrho}\left|\right._{\Ent}$ denotes the derivative of 
the pressure with
respect to the density at fixed entropy.
The positivity of $\Speed$ is fundamental for the hyperbolicity of the PDEs.

In what follows, we will refer to the vorticity $\upomega$, 
which is defined to be the $\Sigma_t$-tangent vectorfield
with the following Cartesian spatial components, ($i=1,2,3$):
\begin{align} \label{E:VORTICITYDEFINITION}
	\upomega^i
	& := (\curl v)^i
	= \upepsilon_{ijk} \updelta^{jl} \partial_l v^k.
\end{align}

In this article, we study only solutions with strictly positive density. Thus, rather than studying the density and vorticity,
we can fix an (arbitrary) ``background density'' $\bar{\varrho} > 0$ and instead
study the logarithmic density $\LogDensity : \mathbb{R} \times \mathbb{R}^3 \rightarrow \mathbb{R}$
and the specific vorticity $\vortrenormalized$. The advantage of working with $\LogDensity$ is that
some our equations take a simplified form when expressed in terms of this variable. 
Our analysis also crucially relies on a $\Sigma_t$-tangent vectorfield
$\GradEnt$ equal to the spatial gradient of $\Ent$. We now precisely define these quantities. 

\begin{definition}[\textbf{Logarithmic density, specific vorticity, and entropy gradient}]
\label{D:ADDITIONALFLUIDVARIABLES}
Relative to the Cartesian coordinates, 
we define the logarithmic density $\LogDensity$, which is a scalar function, 
the specific vorticity $\vortrenormalized$, which is a $\Sigma_t-$tangent vectorfield, 
and the entropy gradient $\GradEnt$, which also is a $\Sigma_t-$tangent vectorfield,
as follows, ($i=1,2,3$):
\begin{subequations}
\begin{align} \label{E:LOGDENSITY}
	\LogDensity
	& := \ln \left(\frac{\varrho}{\bar{\varrho}} \right),
		\\
	\vortrenormalized^i
	& := \frac{\upomega^i}{(\varrho/\bar{\varrho})}
	= \frac{\upomega^i}{\exp(\LogDensity)},
		\label{E:SPECIFICVORTICITY} \\
	\GradEnt^i
	& := \updelta^{ia} \partial_a \Ent,
	\label{E:ENTROPYGRADIENT}
\end{align}
\end{subequations}
where $\updelta^{ij}$ denotes the Kronecker delta.
\end{definition}

From now on, we will view $\Speed  = \Speed(\LogDensity,\Ent)$. Similarly, we will view $p$
and its partial derivatives with respect to $\LogDensity$ and $\Ent$ to be functions of
$(\LogDensity,\Ent)$.

\subsubsection{Standard first-order formulation of the equations}
\label{SSS:FIRSTORDERCOMPRESSIBLEEULER}
Relative to the standard Cartesian coordinates $(t,x^1,x^2,x^3)$ on $\mathbb{R}^{1+3}$,
the compressible equations can be expressed in the following standard first-order form, where
$i = 1,2,3$
(see Subsubsect.\,\ref{SSS:BASICNOTATION} regarding our index and summation conventions):
\begin{subequations}
\begin{align} \label{E:TRANSPORTDENSRENORMALIZEDRELATIVECTORECTANGULAR}
	\Transport \LogDensity
	& = - \Flatdiv v,
		\\
	\Transport v^i 
	& = 
	- \Speed^2 \updelta^{ia} \partial_a \LogDensity
	- \exp(-\LogDensity) \frac{p_{;\Ent}}{\bar{\varrho}} \updelta^{ia} \partial_a \Ent,
	\label{E:TRANSPORTVELOCITYRELATIVECTORECTANGULAR}
		\\
	\Transport \Ent
	& = 0.
	\label{E:ENTROPYTRANSPORT}
\end{align}
\end{subequations}
In \eqref{E:TRANSPORTDENSRENORMALIZEDRELATIVECTORECTANGULAR}-\eqref{E:ENTROPYTRANSPORT} and throughout,
$\updelta^{ij}$ denotes the Kronecker delta and
$\Transport$ denotes the \emph{material derivative} vectorfield,
defined relative to the Cartesian coordinates by
\begin{align} \label{E:MATERIALVECTORVIELDRELATIVECTORECTANGULAR}
	\Transport 
	& := 
	\partial_t 
	+ 
	v^a \partial_a.
\end{align}
We refer readers to \cite{dCsM2014} for a discussion of the physical considerations that lead to the system
\eqref{E:TRANSPORTDENSRENORMALIZEDRELATIVECTORECTANGULAR}-\eqref{E:ENTROPYTRANSPORT}.
We clarify that in \cite{dCsM2014}, the compressible Euler equations are stated in terms of the density $\varrho$
rather than the logarithmic density $\LogDensity$; it is straightforward to obtain
\eqref{E:TRANSPORTDENSRENORMALIZEDRELATIVECTORECTANGULAR}-\eqref{E:ENTROPYTRANSPORT} as a consequence of the equations
presented in \cite{dCsM2014}.

\subsection{Informal description of the spacetime regions $\mathcal{M}$}
\label{SS:FIRSTDESCRIPTIONOFREGIONS}
We now give an informal description of the spacetime regions $\mathcal{M}$ on which our integral identities hold.
Our assumptions are geometric and refer to the 
\emph{acoustical metric} $\gfour$ of Def.\,\ref{D:ACOUSTICALMETRIC}, 
a fluid-solution-dependent Lorentzian\footnote{By ``Lorentzian,'' we mean that the $4 \times 4$ matrix $\gfour_{\alpha \beta}$ of Cartesian components
has signature $(-,+,+,+)$.} metric 
on $\mathcal{M}$ that governs the propagation of sound waves.
\begin{quote}
Roughly, one could say that $\mathcal{M}$ is allowed to be any region on which one can derive a coercive
energy identity for solutions to the wave equation $\square_{\gfour} \phi = 0$; see \eqref{E:WAVEOPERATORARBITRARYCOORDINATES}
for the definition of $\square_{\gfour}$.
\end{quote}

We now give a brief description of the first of two classes of regions $\mathcal{M}$ on which our integral identities hold;
see Subsect.\,\ref{SS:ASSUMPTIONSONSPACETIMEREGION} for the precise assumptions
and Fig.\,\ref{F:SPACETIMEDOMAIN} for a schematic depiction of $\mathcal{M}$, where
for reasons explained below, $\mathcal{M}_T$ is alternate notation for $\mathcal{M}$.
Later in the paper, we will define all of the objects in the figure.
For now, we only note that the top boundary $\widetilde{\Sigma}_T$
and the bottom boundary 
$\widetilde{\Sigma}_0$
are allowed to be arbitrary acoustically spacelike portions
and that the lateral boundary $\underline{\mathcal{H}}$ is allowed to be an arbitrary acoustically
spacelike or \emph{acoustically null} hypersurface portion.
Above and throughout, 
the term ``acoustic,'' as well as the
Lorentzian notions of  ``causal,'' ``spacelike,'' ``null,'' etc., 
refer to aforementioned acoustical metric $\gfour$.
On such domains,
our integral identities could be combined with standard energy estimates for wave and transport equations 
to yield a priori estimates for compressible Euler solutions with vorticity and entropy
with initial data given along $\widetilde{\Sigma}_0$; 
see Theorem~\ref{T:LOCALIZEDAPRIORIESTIMATES} for one such result.

\begin{center}
\begin{overpic}[scale=.4,grid=false]{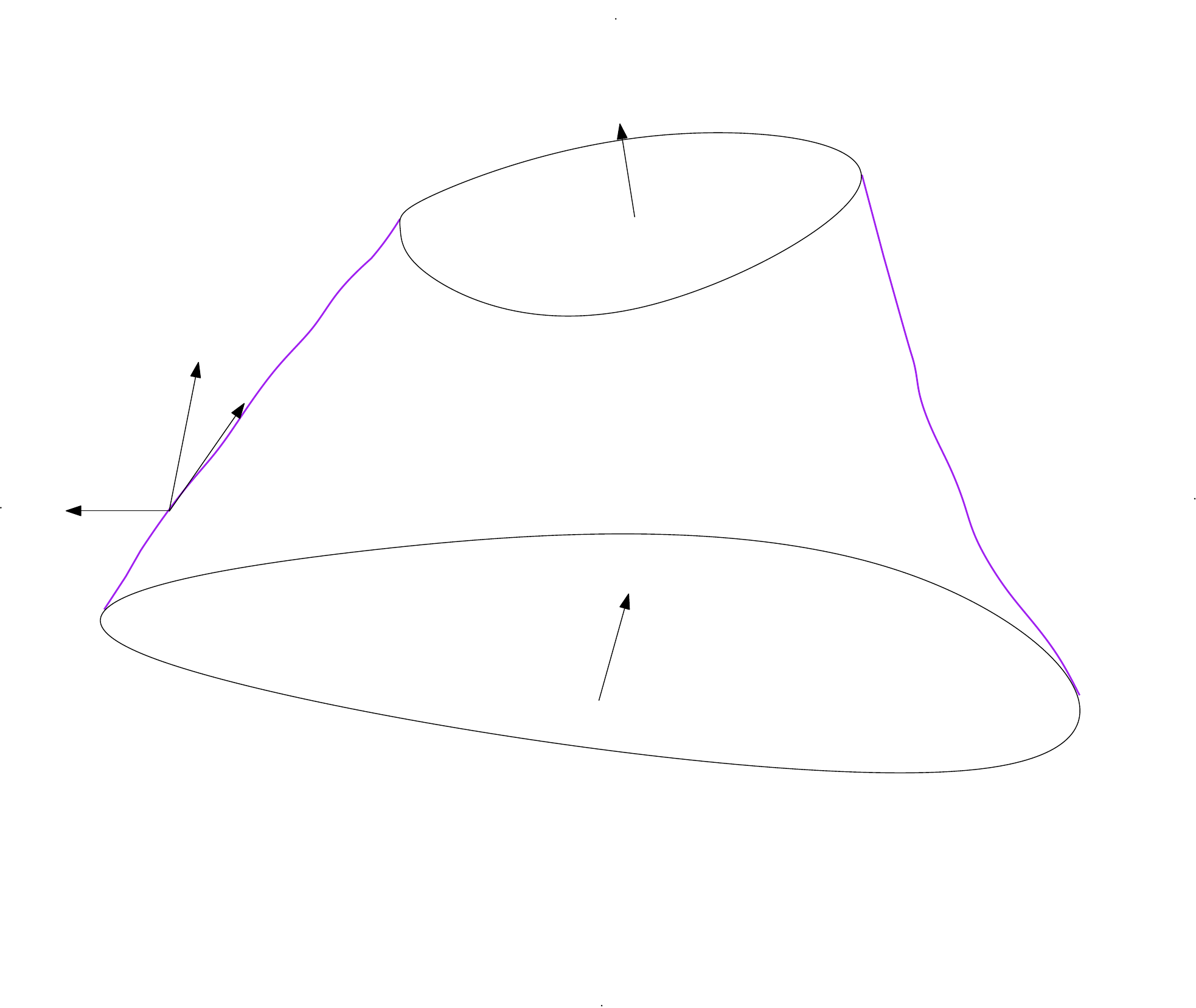} 
\put (0,40) {\large$\displaystyle \spherenormal$}
\put (12.5,56) {\large$\displaystyle \sidehypnorm$}
\put (20,45) {\large$\displaystyle \gen$}
\put (52,26) {\large$\displaystyle \tophypnorm$}
\put (54,65) {\large$\displaystyle \tophypnorm$}
\put (36,30) {\large$\displaystyle \widetilde{\Sigma}_0$}
\put (39,62) {\large$\displaystyle \widetilde{\Sigma}_T$}
\put (81.5,45) {\large$\displaystyle \underline{\mathcal{H}}$}
\put (86.5,16.5) {\large$\displaystyle \mathcal{S}_0$}
\put (68,73) {\large$\displaystyle \mathcal{S}_T$}
\put (45,48) {\large$\displaystyle \mathcal{M} = \mathcal{M}_T$}
\thicklines
\end{overpic}
\captionof{figure}{A spacetime region $\mathcal{M}$ on which the integral identities hold}
\label{F:SPACETIMEDOMAIN}
\end{center}

The second class of regions $\mathcal{M}$ on which our identities hold comprises
regions that are \emph{double-null foliated},
that is, foliated by a pair of acoustic eikonal functions, 
whose level sets are acoustically null;
see Fig.\,\ref{F:DOUBLENULL}.
As in the case of the first class of domains, on double-null-foliated domains,
our integral identities could be combined with standard energy estimates for wave and transport equations 
to yield a priori estimates for compressible Euler solutions with vorticity and entropy.
Here, by ``a priori estimates,'' we mean estimates for the solution on $\mathcal{M}$
in terms of the ``state of the solution'' along a pair of intersecting acoustically null hypersurfaces.
This opens the door for the further study of the (acoustically) characteristic initial value problem;
see Subsect.\,\ref{SS:APPLICATIONS} for further discussion.
To derive the integral identities in the case of double-null foliations,
we have to modify the approach that we use
to treat the domains featured in Fig.\,\ref{F:SPACETIMEDOMAIN}.
This requires additional constructions, which we provide in Sect.\,\ref{S:DOUBLENULL}.

\subsection{A schematic overview of the main results}
\label{SS:SCHEMATICSTATEMENTOFMAINRESULTS}
In this subsection, we provide a schematic overview of the new integral identities on 
spacetime regions $\mathcal{M}$ of the type depicted in Fig.\,\ref{F:SPACETIMEDOMAIN}.
Similar ideas can be used to handle double-null foliated regions of the type featured in
Fig.\,\ref{F:DOUBLENULL}. For this reason, do not provide an overview of the double-null case;
see Sect.\,\ref{S:DOUBLENULL} for the details.

\subsubsection{A slightly more detailed description of $\mathcal{M}$}
\label{SSS:INTRODESCRIPTIONOFREGION}
We now further describe our assumptions on the regions depicted in Fig.\,\ref{F:SPACETIMEDOMAIN};
see Subsect.\,\ref{SS:ASSUMPTIONSONSPACETIMEREGION} for the complete, precise assumptions.
We assume that there is a smooth ``acoustical time function'' $\Timefunction$ with a past-directed, $\gfour$-timelike gradient
such that $\mathcal{M}$ is foliated by the level sets of $\Timefunction$. Here and throughout,
$\gfour$ is the acoustical metric of Def.\,\ref{D:ACOUSTICALMETRIC}.
That is, with $\widetilde{\Sigma}_{\Timefunction'} := \mathcal{M} \cap \lbrace (t,x) \in \mathbb{R}^{1+3} \ | \ \Timefunction(t,x) = \Timefunction' \rbrace$,
we assume that there is a $T > 0$ such that
$\mathcal{M} = \cup_{\Timefunction' \in [0,T]} \widetilde{\Sigma}_{\Timefunction'}$
and such that each $\widetilde{\Sigma}_{\Timefunction'}$ is $\gfour$-spacelike, that is, spacelike with respect to the acoustical metric.
We assume that the $\partial \mathcal{M} = \widetilde{\Sigma}_0 \cup \widetilde{\Sigma}_T \cup \underline{\mathcal{H}}$,
where the lateral boundary $\underline{\mathcal{H}}$ 
is smooth and either $\gfour$-spacelike or $\gfour$-null at each of its points 
(see Def.\,\ref{D:SPACELIKETIMELIKENULL}).
We also assume that for $\Timefunction' \in [0,T]$, $\widetilde{\Sigma}_{\Timefunction'}$ intersects $\underline{\mathcal{H}}$ transversally
in a topological sphere $\mathcal{S}_{\Timefunction'}$, that is, that 
$\mathcal{S}_{\Timefunction'} := \widetilde{\Sigma}_{\Timefunction'} \cap \underline{\mathcal{H}} = \partial \widetilde{\Sigma}_{\Timefunction'}$,
where $\mathcal{S}_{\Timefunction'}$ is diffeomorphic to $\mathbb{S}^2$.
We set $\mathcal{M}_{\Timefunction} := \cup_{\Timefunction' \in [0,\Timefunction]} \widetilde{\Sigma}_{\Timefunction'}$ and
$\underline{\mathcal{H}}_{\Timefunction} = \underline{\mathcal{H}} \cap \mathcal{M}_{\Timefunction} = \cup_{\Timefunction' \in [0,\Timefunction]} \mathcal{S}_{\Timefunction'}$.
Note that $\mathcal{M} = \mathcal{M}_T$ and $\underline{\mathcal{H}} = \underline{\mathcal{H}}_T$.

\subsubsection{Schematic statement of the main results}
\label{SSS:SCHEMATICSTATEMENTOFMAINRESULTS}
For brevity, we will mainly restrict our attention to providing an overview of the integral identities
involving the square integral of the gradient $\pmb{\partial} \vortrenormalized$,
which denotes the gradient of the specific vorticity $\vortrenormalized$ (see \eqref{E:SPECIFICVORTICITY})
with respect to the Cartesian coordinates.
According to equation \eqref{E:RENORMALIZEDVORTICTITYTRANSPORTEQUATION}, 
$\vortrenormalized$ satisfies a transport equation that,
from the point of view regularity, can be caricatured as follows, where
$\partial$ denotes gradient with respect to the Cartesian spatial coordinates
and $\GradEnt$ denotes the entropy gradient vectorfield (see \eqref{E:ENTROPYGRADIENT}):
\begin{align} \label{E:INTROTRANSPORTSPECIFICVORTICITYCARICATURE}
	\Transport \vortrenormalized
	& = \vortrenormalized \cdot \partial v + \pmb{\partial} v \cdot \GradEnt.
\end{align}
Since transport equation solutions are generally no more regular than their source terms,
equation \eqref{E:INTROTRANSPORTSPECIFICVORTICITYCARICATURE} \emph{suggests}, incorrectly, that $\vortrenormalized$ 
can be no more regular than the source term factor $\pmb{\partial} v$. As we will explain, $\vortrenormalized$ is in fact one derivative
more regular, that is, as regular as $\varrho$ and $v$, assuming that the initial data of $\vortrenormalized$ enjoy this property.
The gain in regularity can be revealed by constructing a suitable coercive spacetime identity on $\mathcal{M}$ -- 
the construction of which is in fact one of our main results.
Moreover, as we stressed earlier in the paper, the error terms in the identity have a remarkable structure.
We summarize these results as Theorem~\ref{T:MAINIDSCHEMATICSTATEMENT}, which provides a more precise -- but still schematic -- 
statement of our main results; see Theorems~\ref{T:STRUCTUREOFERRORTERMS}, \ref{T:MAINREMARKABLESPACETIMEINTEGRALIDENTITY},
\ref{T:LOCALIZEDAPRIORIESTIMATES}, and \ref{T:DOUBLENULLMAINTHEOREM} for the precise statements.
In Subsect.\,\ref{SS:REMARKSONTHEPROOF}, we summarize some of the key ideas in the proof.

\begin{theorem}[Remarkable localized integral identities for solutions (Rough summary)]
	\label{T:MAINIDSCHEMATICSTATEMENT}
	Consider a smooth solution (see Remark~\ref{R:SMOOTHNESSNOTNEEDED})
	to the compressible Euler equations \eqref{E:TRANSPORTDENSRENORMALIZEDRELATIVECTORECTANGULAR}-\eqref{E:ENTROPYTRANSPORT} 
	in $3D$ on a spacetime region $\mathcal{M}_T$ foliated by an acoustical time function $\Timefunction \in [0,T]$ with
	$\gfour$-spacelike level set portions $\widetilde{\Sigma}_{\Timefunction}$
	(where $\gfour$ is the acoustical metric of Def.\,\ref{D:ACOUSTICALMETRIC}), 
	as described in Subsubsect.\,\ref{SSS:INTRODESCRIPTIONOFREGION}; see Fig.\,\ref{F:SPACETIMEDOMAIN}.
	For $\Timefunction \in [0,T]$ let $\mathcal{M}_{\Timefunction}$,
	$\underline{\mathcal{H}}_{\Timefunction}$, 
	and $\mathcal{S}_{\Timefunction}$ be the subsets defined in Subsubsect.\,\ref{SSS:INTRODESCRIPTIONOFREGION}.
	Then there exists a solution-adapted, positive definite 
	quadratic form $\mathscr{Q}(\pmb{\partial} \vortrenormalized,\pmb{\partial} \vortrenormalized) \approx |\pmb{\partial} \vortrenormalized|^2$
	(see Def.\,\ref{D:QUADRATICFORMSFORCONTROLLINGFIRSTDERIVATIVESOFSPECIFICVORTICITYANDENTROPYGRADIENT})
	such that for $\Timefunction \in [0,T]$,
	up to $\mathcal{O}(1)$ factors, the following integral identity holds,
	where $\angvortrenormalized$ denotes $\gfour$-orthogonal projection of $\vortrenormalized$ onto 
	the embedded two-dimensional $\gfour$-spacelike manifold $\mathcal{S}_{\Timefunction}$,
	and we have suppressed the volume and area forms
	(see Subsect.\,\ref{SS:VOLUMEFORMSANDINTEGRALS} for their definitions):
	\begin{align} \label{E:SCHEMATICKEYIDENTITY}
			\int_{\mathcal{M}_{\Timefunction}}
				\mathscr{Q}(\pmb{\partial} \vortrenormalized,\pmb{\partial} \vortrenormalized)
			+
			\int_{\mathcal{S}_{\Timefunction}}
				|\angvortrenormalized|^2
			=
			\int_{\mathcal{S}_0}
				|\angvortrenormalized|^2
			+
			\int_{\mathcal{M}_{\Timefunction}}
				\mbox{\upshape Controllable}
				+
				\int_{\underline{\mathcal{H}}_{\Timefunction}}
					\underline{\mbox{\upshape Tangential}}.
	\end{align}
	The term $\int_{\mathcal{S}_0}
	|\angvortrenormalized|^2$
	on RHS~\eqref{E:SCHEMATICKEYIDENTITY} is determined by initial data specified along $\mathcal{S}_0 \subset \widetilde{\Sigma}_0$.
	Moreover, the remaining terms on RHS~\eqref{E:SCHEMATICKEYIDENTITY} enjoy the following properties: 
	\begin{itemize}
	\item The terms ``$\mbox{\upshape Controllable}$'' 
		either \textbf{i)}
		depend on the fluid variables, the unit $\gfour$-normal to the hypersurfaces 
		$\lbrace \widetilde{\Sigma}_{\Timefunction'} \rbrace_{\Timefunction' \in [0,\Timefunction]}$,
		and some first derivatives of these quantities
		or \textbf{ii)} are \underline{linear} in $\pmb{\partial} \vortrenormalized$.
		All of these terms enjoy a compatible\footnote{Using Young's inequality and the positivity of 
		$\mathscr{Q}(\pmb{\partial} \vortrenormalized,\pmb{\partial} \vortrenormalized)$,
		we can absorb the ``$\pmb{\partial} \vortrenormalized$ part'' 
		of the type \textbf{ii)} terms on RHS~\eqref{E:SCHEMATICKEYIDENTITY} back into LHS~\eqref{E:SCHEMATICKEYIDENTITY};
		see, for example, inequality \eqref{E:COMBINEDSPACETIMEINTEGRALALMOSTGRONWALLREADY} and the discussion just below it,
		which leads to the proof of \eqref{E:COMBINEDSPACETIMEINTEGRALGRONWALLREADY}.
		\label{FN:YOUNGSINEQUALITY}}  
		amount of Sobolev regularity, e.g.,
		the type \textbf{i)} terms can be bounded via standard estimates for the wave and transport equations of Theorem~\ref{T:GEOMETRICWAVETRANSPORTSYSTEM}
		(as long as the initial data enjoy compatible regularity);
		see Theorem~\ref{T:LOCALIZEDAPRIORIESTIMATES} for the details in the case that the 
		$\widetilde{\Sigma}_{\Timefunction}$ are standard flat hypersurfaces of constant Cartesian time.
	\item The terms ``$\underline{\mbox{\upshape Tangential}}$'' depend on the fluid variables, 
		on a vectorfield
		frame tangent to the lateral boundary portion $\underline{\mathcal{H}}_{\Timefunction}$, 
		and on the derivatives of some these variables \textbf{in directions tangent to
		$\underline{\mathcal{H}}_{\Timefunction}$}. These terms can be bounded via energy estimates for the wave and transport equations 
		of Theorem~\ref{T:GEOMETRICWAVETRANSPORTSYSTEM},
		\textbf{even in the case that $\underline{\mathcal{H}}_{\Timefunction}$ is acoustically null}.
		The crucial point is that in the acoustically null case,
		the wave energies along $\underline{\mathcal{H}}_{\Timefunction}$ degenerate and \underline{control only 
		$\underline{\mathcal{H}}_{\Timefunction}$-tangential derivatives},
		that is, compared to the $\gfour$-spacelike case, they lose their full positivity and become only positive \textbf{semi}-definite.
		Put differently, all of the terms in ``$\underline{\mbox{\upshape Tangential}}$'' are controllable by the degenerate wave energies
		or transport energies.
		In the case that the $\widetilde{\Sigma}_{\Timefunction}$ are standard flat hypersurfaces of constant Cartesian time,
		the degeneracy of the wave energies 
		is precisely captured by the ``null-flux'' coercivity result \eqref{E:NULLCASEWAVEFLUXESSEMICOERCIVE}.
	\item The integral $\int_{\mathcal{S}_{\Timefunction}} |\angvortrenormalized|^2$ 
		on LHS~\eqref{E:SCHEMATICKEYIDENTITY} is
		\textbf{positive definite} with respect to $\angvortrenormalized$.
		In the precise identity, namely \eqref{E:SPACETIMEREMARKABLEIDENTITYSPECIFICVORTICITY},
		the positivity stems from that of the scalar functions 
		$\lapsemodgen$,
		$\uposinnerproduct$, and $\seconduposinnerproduct$,
		defined respectively in 
		\eqref{E:LAPSEFORMODGENCORRESPONDINGTOTIMEFUNCTION}
		and
		\eqref{E:INGOINGCONDITION}-\eqref{E:SECONDINGOINGCONDITION},
		whose positivity in turn stems from the basic geometric properties of $\mathcal{M}$
		with respect to $\gfour$.
	\end{itemize}
	
	Moreover, the following additional results hold:
	\begin{itemize}
		\item One can incorporate a weight function $\weight$ into the integral identities;
			see Theorem~\ref{T:MAINREMARKABLESPACETIMEINTEGRALIDENTITY}.
		\item  An identity similar to \eqref{E:SCHEMATICKEYIDENTITY} holds with 
			the entropy gradient vectorfield $\GradEnt := \partial \Ent$ in place of $\vortrenormalized$
			(and thus the LHS of the identity controls the integral of $|\pmb{\partial} \GradEnt|^2 
			= |\pmb{\partial} \partial \Ent|^2$); see \eqref{E:SPACETIMEREMARKABLEIDENTITYENTROPYGRADIENT}.
		\item Similar results hold on spacetime regions covered by double-null foliations;
			see Sect.\,\ref{S:DOUBLENULL}, and Fig.\,\ref{F:DOUBLENULL} in particular.
	\end{itemize} 
\end{theorem}

\begin{remark}[The results are most interesting when the lateral boundary is $\gfour$-null]
	\label{R:RESULTSINTERESTINGINNULLCASE}
	When $\underline{\mathcal{H}}_{\Timefunction}$ is $\gfour$-spacelike, the tangential derivative structure of the 
	terms ``$\underline{\mbox{\upshape Tangential}}$'' is perhaps to be expected. 
	This is because when $\underline{\mathcal{H}}_{\Timefunction}$ is $\gfour$-spacelike,
	the compressible Euler equations \eqref{E:TRANSPORTDENSRENORMALIZEDRELATIVECTORECTANGULAR}-\eqref{E:ENTROPYTRANSPORT}
	roughly have the following algebraic content: some $\underline{\mathcal{H}}_{\Timefunction}$-transversal derivative of the solution
	can be re-expressed as $\underline{\mathcal{H}}_{\Timefunction}$-tangential derivatives. Put differently, in the spacelike case,
	one can use the compressible Euler equations to eliminate transversal derivatives in terms of
	tangential derivatives. In contrast, when $\underline{\mathcal{H}}_{\Timefunction}$ is $\gfour$-null,
	this heuristic argument is false, and the tangential-differentiation structure of the 
	terms ``$\underline{\mbox{\upshape Tangential}}$'' becomes much more interesting and difficult to uncover.
\end{remark}

\begin{remark}[We have avoided some geometry to assist future applications]
	\label{R:AVOIDINGCHRISTOFFEL}
	Some of the error terms in the integral identities
	involve first derivatives of various vectorfields in geometric directions.
	We have intentionally chosen to express these error terms relative to the Cartesian coordinates,
	even though they could be rewritten in a much more geometric fashion by using covariant derivatives
	and referring to the second fundamental forms of the foliations. 
	We are motivated by the following consideration:
	covariant derivatives would involve the Cartesian coordinate Christoffel symbols,
	which, in the context of the study of shocks, can be singular \cites{dC2007,jS2016b,jLjS2018}. 
	That is,  
	we have aimed to express error terms in a manner such that
	in the context of shocks, 
	the singular terms will be manifest, which means avoiding certain covariant expressions.
	This is reminiscent of the renormalizations used in the works on impulsive gravitational wave solutions to Einstein's equations \cites{jLiR2015,jLiR2017},
	which showed that the quantities
	with the best analytic structures are generally not the same
	as the quantities with the most geometric interpretation. 
\end{remark}

\begin{remark}[The smoothness assumption could be substantially weakened]
	\label{R:SMOOTHNESSNOTNEEDED}
	In Theorem~\ref{T:MAINIDSCHEMATICSTATEMENT} and throughout, we have assumed that the solution is sufficiently smooth 
	(e.g., assuming $\varrho$, $v$, and $\Ent$ are $C^3$ would suffice).
	We made this assumption only for convenience, as it facilitates our arguments involving integration by parts;
	standard techniques could be used to extend 
	our results to solutions of suitable finite Sobolev regularity.
\end{remark}

\begin{remark}[Similar results for the solution's higher derivatives]
\label{R:HOWTOTREATHIGHERDERIVATIVES}
For convenience, in this paper, we have exhibited the identities only at the lowest derivative levels,
that is, at the level of the undifferentiated equations of Theorem~\ref{T:GEOMETRICWAVETRANSPORTSYSTEM}.
However, as our proofs make clear, similar results hold for the higher-order derivatives of the solution,
assuming that the initial data enjoy compatible regularity.
The reason is that all of the special cancellations that we observe stem from \underline{linear factors}
on the right-hand side of various equations from
Theorem~\ref{T:GEOMETRICWAVETRANSPORTSYSTEM}
and/or the equations \eqref{E:TRANSPORTDENSRENORMALIZEDRELATIVECTORECTANGULAR}-\eqref{E:ENTROPYTRANSPORT}
(which are a standard first-order formulation of compressible Euler flow).
Thus, one could treat the differentiated equations in the same way,\footnote{In applications, 
one would of course have to construct good commutator vectorfields
in order to ensure that the commutator terms also have good structure.} 
where the special cancellations would occur in products such that all the derivatives fall on the key linear factors;
the remaining terms in the differentiated equations are error terms that either have a below-top-order derivative count
(and thus are easy to incorporate into the framework of this paper)
or that could be handled via standard estimates for wave and transport equations.
\end{remark}

\subsection{Applications and potential applications of the integral identities}
\label{SS:APPLICATIONS}
We now highlight some of the main applications/potential applications of the integral identities.
\begin{enumerate}
	\renewcommand{\labelenumi}{\textbf{\Roman{enumi})}}

	\item \textbf{Sharpened picture of the basic regularity theory for compressible Euler solutions}.
		As the statement of Theorem~\ref{T:MAINIDSCHEMATICSTATEMENT} suggests,
		our results can be used to exhibit a gain in regularity compared to standard estimates, 
		that is, that the specific vorticity $\vortrenormalized$ and entropy gradient
		$\GradEnt$ are exactly as differentiable as the velocity $v$ and density $\varrho$.
		We illustrate this in detail in Theorem~\ref{T:LOCALIZEDAPRIORIESTIMATES},
		where as an application,
		we derive a priori Sobolev estimates at the level of the undifferentiated equations of
		\cite{jS2019c}. These estimates provide a new, sharpened picture of the basic regularity theory for solutions. We highlight the following main
		points.
		\begin{quote}
			For classical solutions whose initial data enjoy the gain in regularity for $\vortrenormalized$ and $\GradEnt$,
			the correct\footnote{By ``correct regularity space'' for $\vortrenormalized$ and $\GradEnt$, 
			we mean the function space for which estimates are available
			\emph{and} compatible with the regularity of the other solution variables.} 
			regularity space
			for these fluid variables
			on compact acoustically globally hyperbolic spacetime regions
			is such that at the top derivative level,
			their divergence and curl\footnote{Although $\curl \GradEnt = 0$ (see \eqref{E:CURLGRADENTVANISHES}), 
			the curl of the derivatives of $\GradEnt$ with respect to vectorfields is generally not $0$ due to the commutator 
			of the vectorfield with the operator $\curl$. This issue arises in the study of shock waves without symmetry assumptions,
			where a huge amount of effort is required to construct appropriate vectorfield differential operators
			and to control commutator terms.} belong to 
			$L^{\infty}(Time)L^2(Space)$ (which are the standard regularity spaces for hyperbolic PDE solutions).
			In contrast, \emph{their spatial gradient is generally less regular},\footnote{We note, however,
			that the top-order $L^{\infty}(Time)L^2(Space)$ regularity of $\Transport \vortrenormalized$ and $\Transport \GradEnt$
			can easily be derived by using the transport equations \eqref{E:RENORMALIZEDVORTICTITYTRANSPORTEQUATION} and \eqref{E:GRADENTROPYTRANSPORT}
			for algebraic substitution and showing that the terms on the right-hand sides 
			enjoy the desired $L^{\infty}(Time)L^2(Space)$ regularity. \label{FN:LINFINITYTIMEL2SPACECONTROLOFMATERIALDERIVATIVETERMS}} 
			belonging only to
			$L^2(Time)L^2(Space)$.
			That is, unless one considers special solutions/regions\footnote{If the vorticity and entropy are compactly supported in space,
			then one can deduce a stronger estimate. That is,
			from the standard Hodge estimate $\| \partial \upxi \|_{L^2(\mathbb{R}^3)} 
			\lesssim \| \dive \upxi \|_{L^2(\mathbb{R}^3)} + \| \curl \upxi \|_{L^2(\mathbb{R}^3)}$
			for vectorfields $\upxi$ on $\mathbb{R}^3$ (where $\partial$ denotes gradient with respect to the Cartesian spatial coordinates),
			we see that $\partial \upxi \in L^{\infty}(Time)L^2(Space)$
			would in fact follow from knowing that $\dive \upxi \in L^{\infty}(Time)L^2(Space)$ and $\curl \upxi \in L^{\infty}(Time)L^2(Space)$.} 
			such that the divergence and curl vanish along the lateral boundary,
			the regularity theory at the top spatial derivative level involves \emph{spacetime integrals}
			of $|\pmb{\partial} \dot{\vortrenormalized}|^2$ and $|\pmb{\partial} \dot{\GradEnt}|^2$, where the ``$\cdot$'' schematically denotes
			a top-order derivative that has been commuted through the Euler equations.
			In the context of the priori estimates of Theorem~\ref{T:LOCALIZEDAPRIORIESTIMATES}, this is captured by the
			spacetime integral estimate \eqref{E:COMBINEDSPACETIMEINTEGRALGRONWALLED}.
	\end{quote}
	Relative to Lagrangian coordinates, the gain in differentiability for the vorticity in the barotropic case, 
	achieved via combinations of Hodge estimates
	and transport equation estimates,
	has long been known,
	specifically because it has played a central role in proofs of local well-posedness  
	for the compressible Euler equations for initial data featuring a fluid-vacuum boundary
	satisfying the ``physical vacuum'' boundary condition
	\cites{dChLsS2010,dCsS2011,dCsS2012,jJnM2011,jJnM2009}.
	By the nature of free-boundary problems, these results are spatially localized.
	In the context of estimates across all of space 
	(in particular, without the difficult boundary terms that we handle in this paper),
	the gain in differentiability for the vorticity
	with respect to arbitrary vectorfield differential operators
	was first shown in \cite{jLjS2016a}, while the gain
	in differentiability for the entropy was first shown in \cite{jS2019c}.
	The freedom to gain the derivative relative to general vectorfield differential operators
	is important for the mathematical theory of shock waves without symmetry assumptions.
	The reason is that Lagrangian coordinates seem to be unsuitable for deriving the full structure of the singular set\footnote{Here, 
	in the context of shocks, by ``singular set,''
	we roughly mean the portion of the boundary of the maximal classical development on which the first derivatives of the density and velocity blow up;
	see below for further discussion of the maximal classical development.}
	because they are not adapted to the acoustic characteristics,\footnote{Acoustic characteristics are hypersurfaces that are null with respect 
	to the acoustical metric $\gfour$ defined in of Def.\,\ref{D:ACOUSTICALMETRIC}. Sound cones emanating from points serve as examples.} whose intersection 
	corresponds to the formation of a shock; we refer readers to \cites{jLjS2016a,jS2019c}
	for further discussion of these issues.
	\item \textbf{Localized analysis of solutions with vorticity and entropy that form shocks and the interaction of shocks}.
		In $3D$, 
		integral identities of the form \eqref{E:SCHEMATICKEYIDENTITY}
		are of fundamental importance for the localized analysis
		of stable shock formation in the presence of vorticity and entropy for solutions without symmetry assumptions.
		Roughly, shocks are singularities such that $\varrho$, $v$, and $\Ent$ remain bounded,\footnote{In \cite{jLjS2018},
		it was shown that in the barotropic case in $2D$, the specific vorticity remains Lipschitz up to the shock, 
		at least for the open set of initial data treated there.}
		while some directional derivatives of $\varrho$ and $v$ blow up in a very particular fashion.
		The blowup is tied to the intersection of distinct characteristic hypersurfaces. 
		What blows up are derivatives of the solution in directions transversal to the characteristics;
		the solution and its derivatives in direction tangent to the characteristics remain bounded.
		This is a multiple-spatial-dimension analog of the singularity formation that occurs in the model case of the $1D$ Burgers' equation
		$\partial_t \Psi + \Psi \partial_x \Psi = 0$.
		We refer readers to \cites{dC2007,gHsKjSwW2016,jS2016b,jLjS2016a,jS2018c,mDjS2019,jS2019c} for
		background on shock formation in multiple spatial dimensions without symmetry assumptions. 
		Shocks are singularities of a sufficiently mild nature that one is left with hope that
		it might be possible to uniquely extend the solution in a weak sense past 
		the singularity, subject to suitable admissibility criteria; see Point \textbf{IV} below for further discussion of this issue.
		In the work \cite{jLjS2018}, the second author and J.\,Luk proved a stable shock formation result 
		in $2D$ under barotropic equations of state,
		a setting that is much simpler than the $3D$ case 
		since the absence of vorticity stretching\footnote{For barotropic equations of state in $2D$, 
		the absence of vorticity stretching is equivalent to the fact that RHS~\eqref{E:RENORMALIZEDVORTICTITYTRANSPORTEQUATION} 
		is identically $0$, a well-known fact that holds only in the $2D$ case. 
		Due to this vanishing, one can handle the specific vorticity without using
		the div-curl-transport system
		\eqref{E:FLATDIVOFRENORMALIZEDVORTICITY}-\eqref{E:EVOLUTIONEQUATIONFLATCURLRENORMALIZEDVORTICITY},
		which drastically simplifies the regularity theory of the specific vorticity 
		compared to the general $3D$ case.}
		in $2D$ allows one to avoid elliptic estimates for the vorticity. 
		In particular, integral identities
		of the type that we derive in this paper are not needed to close the estimates.
		In \cites{jLjS2016a,jS2019c}, the second author and J.\,Luk outlined of 
		a proof of stable shock formation in $3D$
		(full details will be presented in a forthcoming paper) in the barotropic case,
		but the approach described there neither required nor yielded localized information
		about the specific vorticity and entropy at the top derivative level.
		The reason is that the elliptic estimates in
		\cites{jLjS2016a,jS2019c} involved integrals taken across of space
		and thus did not involve the difficult boundary terms that we treat here.
		Put differently, the approach described in
		\cites{jLjS2016a,jS2019c} works only for spacetime regions 
		of the form $[0,T] \times \Sigma$, where
		the ``space manifold'' $\Sigma$ has no boundary (e.g., $\Sigma = \mathbb{R}^3$).
		In contrast, integral identities of the form \eqref{E:SCHEMATICKEYIDENTITY}
		allow one to approach the problem of shock formation in the presence of vorticity and entropy 
		via analysis on spatially localized regions. The viability of a local approach  
		is desirable in the sense that solutions exhibit finite speed of propagation,
		and from a physical point of view,
		one would like to be able to describe the shock formation using only ``local information.''
		Such localized information is expected to be important for studying the interaction of shock waves.
	\item \textbf{Sharp information about the boundary of the maximal classical development}.
		Roughly, the maximal classical development is the largest spacetime region on which the solution exists
		classically and is uniquely determined by the initial data; we refer to \cites{jSb2016,wW2013} for further discussion.
		The localized analysis of shock formation described in Point \textbf{II} above is of
		crucial importance for obtaining information about the maximal classical development of the initial data,
		including information about the solution (including the vorticity and entropy) up to the boundary.
		In \cite{dC2007}, Christodoulou provided a sharp picture of the maximal classical development near shock singularities
		for irrotational and isentropic solutions to the $3D$ relativistic Euler equations. 
		He showed that the boundary of the maximal development can have various components enjoying the structure of a smooth manifold
		with respect to an appropriate dynamically constructed coordinate system. For example, the boundary can have singular components,
		along which some Cartesian partial derivative of the solution blows up,
		as well as acoustically null hypersurface portions along which no blowup occurs (roughly, these null portions are Cauchy horizons).
		His proof exploited that in the irrotational and isentropic setting, 
		the equations of motion reduce to a single quasilinear wave equation for a potential function.
		In particular, Christodoulou did not have to derive div-curl-transport estimates for the fluid variables.
		See also \cite{dCsM2014} for a similar sharp description of the maximal classical development 
		of shock-forming solutions to the irrotational and isentropic non-relativistic $3D$ compressible Euler equations.
		We highlight the following key point.
		\begin{quote}
			Integral identities of the form \eqref{E:SCHEMATICKEYIDENTITY} are the main new ingredients needed to extend the framework
			of \cites{dC2007,jLjS2016a,jS2019c} to derive sharp information up to the boundary of the maximal classical development 
			for solutions with vorticity and non-constant entropy. The point is that, roughly, by finite speed of propagation,
			the boundary of the maximal classical development is ``locally determined,''
			which necessitates the use of localized identities/estimates.
			Such a result would elevate our understanding of such solutions
			to the same level achieved by Christodoulou \cite{dC2007} in the irrotational and isentropic case.
		\end{quote}
			
		We also refer to \cites{sA1999a,sA1999b,tBsSvV2019a,tBsSvV2019b} for alternate approaches 
		to proving blowup outside of symmetry. 
		The frameworks used in these works allow one to follow solutions up to the constant-time hypersurface of
		first blowup for an open set of initial data such that 
		the solution's first singularity is non-degenerate in the constant-hypersurface of first blowup.
		Roughly, ``non-degenerate'' means that the first singularity is isolated in the constant-hypersurface of first blowup
		and that the \emph{reciprocal} of the singular directional derivative of the solution behaves 
		quadratically
		(with respect to suitably constructed\footnote{That is, if one dynamically constructs special solution-dependent 
		spatial coordinates $\lbrace y^i \rbrace_{i=1,2,3}$, which
		are degenerate with respect to the Cartesian coordinates and normalized
		so that the first singularity occurs at the origin, then at the time of first blowup, 
		some first-order Cartesian coordinate partial derivative of the 
		solution blows up in space like $\frac{1}{|y|^2}$.} 
		coordinates)
		within the constant-hypersurface of first blowup.
		\item \textbf{Connections to the shock development problem}.
		Although sharp results about the maximal classical development (as described in Point \textbf{III}) are of interest in themselves,
		they are also essential for properly setting up the ``initial'' data for the \emph{shock development problem}.
		This is the problem of locally solving the compressible Euler equations past the shock singularity in a weak sense (uniquely under
		appropriate admissibility conditions tied to jump conditions), and, at the same time, constructing the shock hypersurface
		across which the solution jumps. In \cite{dCaL2016}, Christodoulou--Lisibach solved the problem for the $3D$ relativistic Euler equations
		in spherical symmetry, while in the recent breakthrough monograph \cite{dC2019}, 
		Christodoulou solved the \emph{restricted shock development problem} in $3D$ without
		symmetry assumptions. Here, the term ``restricted'' means that Christodoulou considered only irrotational initial data,
		and he ignored the jump in entropy and vorticity 
		across the shock hypersurface, so that the mathematical problem concerned only irrotational solutions.
		We stress that the ``initial'' data for the restricted shock development problem in \cite{dC2019} 
		are provided by the state of the solution
		on the boundary of its maximal classical development 
		(which was derived by Christodoulou in \cite{dC2007} in the irrotational case, starting from smooth initial conditions along $\lbrace t=0 \rbrace$),
		and moreover, the quantitative estimates near the boundary from \cite{dC2007}
		are crucially used in \cite{dC2019} to implement the iteration scheme that lies at the heart of the solution of the restricted shock development problem.
		We also highlight that in the (yet unsolved ``unrestricted'') shock development problem, 
		vorticity and entropy are generated across the shock hypersurface,
		even if the initial data are irrotational and isentropic. Thus, we expect that
		the results of the present article will be useful for properly setting up/ studying the shock development problem 
		for general solutions (that are smooth along $\lbrace t=0 \rbrace$ but form shocks in finite time),
		in which the vorticity is non-zero and the entropy is dynamic.
	\item \textbf{Characteristic initial value problem}.
		In the theory of hyperbolic PDEs, the characteristic initial value problem is the Cauchy problem with initial data prescribed on
		a pair of intersecting characteristic hypersurfaces (i.e., acoustically null hypersurfaces in the context of compressible Euler flow).
		This kind of Cauchy problem has proven to be of immense value for solving
		important problems in hyperbolic PDE theory; 
		see, for example,
		Christodoulou's celebrated proof \cite{dC2008} of the formation of trapped surfaces in Einstein-vacuum flow.
		Double-null foliations also played a central role in Christodoulou's resolution of the restricted shock development problem \cite{dC2019},
		mentioned above.
		A crucial difference compared to the case of the standard Cauchy problem (in which the data are prescribed on a spacelike hypersurface)
		is that in the characteristic case, some ``components'' of the initial data are ``constrained,'' that is, 
		determined by the ``free components'' by virtue of equations obtained by restricting the PDE to the initial characteristic surfaces.
		In the context of compressible Euler flow, the (acoustically) characteristic initial value problem has not yet been understood, except
		in the case of spherically symmetric barotropic flow \cite{aL2017}.
		The integral identities of the present article, specifically the ones provided by Theorem~\ref{T:DOUBLENULLMAINTHEOREM},
		can be used to derive a priori estimates for solutions to the full $3D$ compressible Euler system in terms of various norms
		along the initial (acoustically) characteristic hypersurfaces.
		To extend the a priori estimates to a full proof of local well-posedness with characteristic initial data, 
		one must in particular understand which components of data are free and which are constrained.
		This will be the subject of a future work.
\end{enumerate}

\subsection{Key ideas in the proof of Theorem~\ref{T:MAINIDSCHEMATICSTATEMENT}}
\label{SS:REMARKSONTHEPROOF}
In this subsection, we provide an outline of the proof of Theorem~\ref{T:MAINIDSCHEMATICSTATEMENT}, highlighting
the key ideas without discussing technical details.
\begin{convention}[The symbol ``$\underline{\mathcal{N}}$'']
	\label{C:INTROALTNOTATION}
	Throughout this subsection,
	we will consider only the most interesting and difficult case in which the lateral boundary $\underline{\mathcal{H}}$ in Fig.\,\ref{F:SPACETIMEDOMAIN}
	is null with respect to the acoustical metric of Def.\,\ref{D:ACOUSTICALMETRIC}. 
	We therefore use the alternate notation 
	``$\underline{\mathcal{N}}$'' in place of ``$\underline{\mathcal{H}}$'' to emphasize its null character.
	In the rest of the paper, we also use this alternate notation in the null case;
	see Convention~\ref{C:NULLCASE}.
\end{convention}

To initiate the proof of Theorem~\ref{T:MAINIDSCHEMATICSTATEMENT} for $\vortrenormalized$,
one must circumvent the difficulties described at the start of Subsubsect.\,\ref{SSS:SCHEMATICSTATEMENTOFMAINRESULTS}, 
which are caused in part by the \emph{apparent} lack of sufficient regularity in the terms on the right-hand side of the 
transport equation \eqref{E:INTROTRANSPORTSPECIFICVORTICITYCARICATURE} for $\vortrenormalized$.
To this end, one considers equations
\eqref{E:FLATDIVOFRENORMALIZEDVORTICITY}-\eqref{E:EVOLUTIONEQUATIONFLATCURLRENORMALIZEDVORTICITY},
which, roughly speaking, show that $\vortrenormalized$ satisfies\footnote{More precisely, equation \eqref{E:EVOLUTIONEQUATIONFLATCURLRENORMALIZEDVORTICITY}
involves a modified version of $\curl \vortrenormalized$,
denoted by $\VortVort$ and defined below in \eqref{E:RENORMALIZEDCURLOFSPECIFICVORTICITY}.
In practice, one must work with $\VortVort$ since it satisfies a transport equation
whose source terms exhibit improved regularity and other good structures
compared to the transport equation satisfied by $\curl \vortrenormalized$.
However, to keep the discussion short, in our schematic overview,
we will ignore the distinction between $\curl \vortrenormalized$ and $\VortVort$. \label{FN:SHOULDUSEMODIFIEDVARIABLES}} 
a div-curl-transport system, where the div and curl operators are the standard Euclidean ones. 
To keep the discussion short, here we caricature the system as:
\begin{align}  \label{E:INTRODIVCURLTRANSPORTONSIGMAT}
	\Flatdiv \vortrenormalized
	& := F = \vortrenormalized \cdot \partial \LogDensity,
	&
	\Transport 
	\Flatcurl \vortrenormalized
	&  := G = \partial v \cdot \partial \vortrenormalized + \cdots,
\end{align}
where $\Flatdiv$ and $\Flatcurl$ denote the standard Euclidean divergence and curl operators on $\Sigma_t$ and
$\cdots$ denotes similar or easier error terms. Here and throughout, $\Sigma_t = \lbrace t \rbrace \times \mathbb{R}^3 \subset \mathbb{R}^{1+3}$ 
denotes the standard flat hypersurface of constant Cartesian time $t$.
We highlight that, as \eqref{E:EVOLUTIONEQUATIONFLATCURLRENORMALIZEDVORTICITY} shows (see also Footnote~\ref{FN:MORENULLFORMS}), 
all derivative-quadratic terms in $G$ are in fact null forms with respect to the acoustical metric $\gfour$.
As we explain in Subsubsect.\,\ref{SSS:STANDARDNULLFORMS}, 
the null form structure is important for applications to shock waves, 
though we downplay the significance of this structure in the present discussion.
We also note that to handle the entropy gradient $\GradEnt$,
one would carry out similar arguments using 
the div-curl-transport system \eqref{E:TRANSPORTFLATDIVGRADENT}-\eqref{E:CURLGRADENTVANISHES}
in place of \eqref{E:FLATDIVOFRENORMALIZEDVORTICITY}-\eqref{E:EVOLUTIONEQUATIONFLATCURLRENORMALIZEDVORTICITY}.

To prove Theorem~\ref{T:MAINIDSCHEMATICSTATEMENT} for $\vortrenormalized$, we must accomplish the following:
\begin{enumerate}
	\item Derive an integral identity whose left-hand side is comparable to 
		$\int_{\mathcal{M}_{\Timefunction}} |\pmb{\partial} \vortrenormalized|^2
			+
			\int_{\underline{\mathcal{S}}_{\Timefunction}} |\angvortrenormalized|^2
		$,
		where $\angvortrenormalized$ is the $\gfour$-orthogonal projection of $\vortrenormalized$ onto $\mathcal{S}_{\Timefunction}$.
		We clarify that in Subsect.\,\ref{SS:REMARKSONTHEPROOF}, we will not explicitly display (or even define) the volume forms in any of the integrals.
	\item We must show that the error terms 
	on the right-hand side of the integral identity
	enjoy a consistent amount of Sobolev regularity,
	assuming that the initial data (which we assume to be prescribed along $\widetilde{\Sigma}_0$) also enjoy it.
	Specifically, the \emph{spacetime} error integrals $\int_{\mathcal{M}_{\Timefunction}} \cdots$ 
	on RHS~\eqref{E:SCHEMATICKEYIDENTITY}
	are allowed to have arbitrary dependence on
		the \emph{un-differentiated} quantities
		$\LogDensity$,
		$v$,
		$\Ent$,
		$\vortrenormalized$, 
		$\GradEnt$,
		and
		up-to-\emph{quadratic} dependence
		on $\pmb{\partial} \LogDensity$,
		$\pmb{\partial} v$,
		$\pmb{\partial} \Ent$,
		and
		$\Flatcurl \vortrenormalized$ (in reality, one needs to use $\VortVort$ in place of 
		$\Flatcurl \vortrenormalized$ -- see Footnote~\ref{FN:SHOULDUSEMODIFIEDVARIABLES}).
		Under standard $C^1$-type bootstrap assumptions enjoyed by classical solutions (see Subsubsect.\,\ref{SSS:ASSUMPTIONSONTHESOLUTION}),
		these spacetime error integrals can be bounded using energy standard estimates for the wave and transport equations featured in
		Theorem~\ref{T:GEOMETRICWAVETRANSPORTSYSTEM}; see e.g.\ the proof of Theorem~\ref{T:LOCALIZEDAPRIORIESTIMATES}.
		Moreover, these spacetime error integrals are allowed to have 
		\emph{linear dependence} on $\pmb{\partial} \vortrenormalized$ and $\pmb{\partial} \GradEnt$
		(see Footnote~\ref{FN:YOUNGSINEQUALITY} regarding the admissibility of such linear dependence).
		Below, we further explain the allowed quadratic dependence on $\Flatcurl \vortrenormalized$.
	\item On the right-hand side of the integral identity \eqref{E:SCHEMATICKEYIDENTITY}, 
		the \emph{null hypersurface} error integrals $\int_{\underline{\mathcal{N}}_{\Timefunction}} \cdots$ 
		are allowed to have arbitrary dependence on
		the \emph{un-differentiated} quantities
		$\LogDensity$,
		$v$,
		$\Ent$,
		$\vortrenormalized$, 
		$\GradEnt$,
		$\Flatcurl \vortrenormalized$,
		and
		up-to-\emph{quadratic} dependence on the first-order $\underline{\mathcal{N}}_{\Timefunction}$-\emph{tangential derivatives}
		of $\LogDensity$
		and
		$v$.
		Under bootstrap assumptions of the type mentioned above,
		these null hypersurface integrals can be also bounded using energy standard estimates for the wave and transport equations 
		featured in Theorem~\ref{T:GEOMETRICWAVETRANSPORTSYSTEM};
		see e.g.\ the proof of Theorem~\ref{T:LOCALIZEDAPRIORIESTIMATES}.
	\item On the right-hand side of the integral identity \eqref{E:SCHEMATICKEYIDENTITY},
		the ``data term'' $\int_{\mathcal{S}_0} |\vortrenormalized|^2$
		is allowed.
	\end{enumerate}

\subsubsection{A warm-up problem: Flat spacetime slabs without spatial boundaries}
\label{SSS:SIMPLERDOMAINS}
We first explain how to prove Theorem~\ref{T:MAINIDSCHEMATICSTATEMENT}
in the drastically simplified setting of spacetime regions of the type 
$\mathcal{M}_T = [0,T] \times \mathbb{R}^3$, 
that is, for spacetime slabs foliated by the constant Cartesian-time hypersurfaces
$\Sigma_t$ (without a boundary). 
The main challenge is to overcome the regularity difficulties described at the beginning of Subsect.\,\ref{SS:REMARKSONTHEPROOF}.
We will proceed by 
using the system \eqref{E:INTRODIVCURLTRANSPORTONSIGMAT} and applying 
the divergence theorem to the  $\Sigma_t$-tangent vectorfield
with the following Cartesian spatial components:
\begin{align} \label{E:INTROSIMPLEELLIPTICCURRENT}
J^i := \vortrenormalized^a \partial_a \vortrenormalized^i
-
\vortrenormalized^i
\Flatdiv \vortrenormalized.
\end{align}
Next, we note that straightforward computations yield the Hodge-type identity
\begin{align} \label{E:INTROVERYSIMPLEHODGEID}
	|\partial \vortrenormalized|^2 
	& = \Flatdiv J + (\Flatdiv \vortrenormalized)^2 + |\Flatcurl \vortrenormalized|^2.
\end{align}
Thus, we can apply the divergence theorem to $J$
on $\Sigma_t$ (relative to the standard Euclidean metric on $\Sigma_t$ and without boundary terms), 
to obtain (suppressing the standard integration measures):
$\int_{\Sigma_t} |\partial \vortrenormalized|^2 
= \int_{\Sigma_t} (\Flatdiv \vortrenormalized)^2 + \int_{\Sigma_t} |\Flatcurl \vortrenormalized|^2$.
Then integrating this elliptic identity with respect to Cartesian time and substituting $\Flatdiv \vortrenormalized$ with $F$ from
\eqref{E:INTRODIVCURLTRANSPORTONSIGMAT}, we deduce:
\begin{align} \label{E:INTROSIMPLEHODGEINTEGRALIDENTITY}
\int_{\mathcal{M}_t} |\partial \vortrenormalized|^2 
& = 
\int_{\mathcal{M}_t} F^2 
+ 
\int_{\mathcal{M}_t} |\Flatcurl \vortrenormalized|^2.
\end{align}
Note that the difficult boundary integrals $\int_{\underline{\mathcal{N}}_t} \cdots$ (see Convention~\ref{C:INTROALTNOTATION}) 
found on RHS~\eqref{E:SCHEMATICKEYIDENTITY} are completely absent in the present simplified setting.

\begin{remark}
	\label{R:WHYINTEGRATEINTIME}
	In the present simplified setting, it is not necessary to integrate the elliptic identity with respect to time
	to obtain the spacetime integral identity \eqref{E:INTROSIMPLEHODGEINTEGRALIDENTITY}.
	We have carried out the integration in time only because later on, 
	when treating regions with spatial boundaries, we will need to integrate in time to see various special structures;
	see the integral identity \eqref{E:INTROMAINSPACETIMEINTEGRALIDENTITY} and the discussion surrounding it.
\end{remark}

To complete our proof sketch of Theorem~\ref{T:MAINIDSCHEMATICSTATEMENT} in the present simplified setting,
we will uncover the regularity properties of the
terms in the integral identity \eqref{E:INTROSIMPLEHODGEINTEGRALIDENTITY}
and explain why they have sufficient regularity for controlling
LHS~\eqref{E:INTROSIMPLEHODGEINTEGRALIDENTITY}, assuming that the initial data enjoy sufficient regularity.
From the point of view of regularity, the most difficult term to handle is the source term $G$
in \eqref{E:INTRODIVCURLTRANSPORTONSIGMAT}, specifically the factor $\partial \vortrenormalized$.
Thus, for illustration, we will assume that the other factor in $G$, namely $\partial v$, 
is uniformly bounded by a constant.
Under this assumption, we will explain how to prove that there exists a $C > 0$
such that the following estimate holds
(note that in the present simplified setting, we have $\Timefunction := t$ and $\mathcal{M}_t = [0,t] \times \mathbb{R}^3$):
\begin{align} \label{E:MODELVORTICITYESTIMATEGRONWALLREADY}
\int_{\Sigma_t} |\Flatcurl \vortrenormalized|^2
+
\int_{\mathcal{M}_t} |\partial \vortrenormalized|^2 
&
\leq 
\frac{1}{2}
\int_{\mathcal{M}_t}  |\partial \vortrenormalized|^2
+
C(1+t^2)
\int_{\mathcal{M}_t} |\Flatcurl \vortrenormalized|^2
	\\
& \ \
+
\int_{\mathcal{M}_t} F^2 
+
t 
\int_{\Sigma_0} |\Flatcurl \vortrenormalized|^2 
+
\cdots.
\notag
\end{align}
We then note that the first term on RHS~\eqref{E:MODELVORTICITYESTIMATEGRONWALLREADY} can be absorbed
back into the left and that the remaining terms on RHS~\eqref{E:MODELVORTICITYESTIMATEGRONWALLREADY}
either 
\textbf{i)} are controlled by the initial data on $\Sigma_0$, 
\textbf{ii)} have a consistent amount of 
Sobolev regularity in the sense described in point 2 below \eqref{E:INTRODIVCURLTRANSPORTONSIGMAT}, 
or \textbf{iii)} can be treated with Gronwall's inequality
(where LHS~\eqref{E:MODELVORTICITYESTIMATEGRONWALLREADY} is the quantity to which one applies Gronwall's inequality).
In total, as a consequence of \eqref{E:MODELVORTICITYESTIMATEGRONWALLREADY},
the assumptions described above, and Gronwall's inequality,
one could derive an estimate of the following form for $t \in [0,T]$:
\begin{align} \label{E:MODELVORTICITYELLIPTICHYPERBOLICESTIMATE}
\int_{\Sigma_t} |\Flatcurl \vortrenormalized|^2
+
\int_{\mathcal{M}_t} |\partial \vortrenormalized|^2 
& \lesssim 
\exp[C(1+t^2)] \cdot \mbox{\upshape data},
\end{align}
where ``$\mbox{\upshape data}$'' denotes an appropriate Sobolev norm of the initial data on $\Sigma_0$;
estimates in the spirit of \eqref{E:MODELVORTICITYELLIPTICHYPERBOLICESTIMATE} are fundamentally important for 
applications of the type described in Subsect.\,\ref{SS:APPLICATIONS}.

It remains for us to explain how to prove \eqref{E:MODELVORTICITYESTIMATEGRONWALLREADY}. 
In view of \eqref{E:INTROSIMPLEHODGEINTEGRALIDENTITY}, we see that
we only have to explain how to control the integral 
$\int_{\Sigma_t} |\Flatcurl \vortrenormalized|^2$ on LHS~\eqref{E:MODELVORTICITYESTIMATEGRONWALLREADY}.
In fact, for this term, control along $\Sigma_t$ can be achieved:
standard transport equation energy identities (see Prop.\,\ref{P:ENERGYFLUXID})
for the second equation in \eqref{E:INTRODIVCURLTRANSPORTONSIGMAT} yield, 
for $t \in [0,T]$,
taking into account our assumed uniform bound on $\partial v$, 
the following energy inequality, which we depict schematically:
\begin{align} \label{E:MODELCURLOFVORTICITYENERGYESTIMATE}
\int_{\Sigma_t} |\Flatcurl \vortrenormalized|^2 
& 
\leq 
\int_{\Sigma_0} |\Flatcurl \vortrenormalized|^2 
+
C
\int_{\mathcal{M}_t} |\Flatcurl \vortrenormalized \cdot \partial \vortrenormalized|
+
\cdots.
\end{align}
Integrating \eqref{E:MODELCURLOFVORTICITYENERGYESTIMATE} 
with respect to time, we obtain
\begin{align} \label{E:MODELCURLOFVORTICITYENERGYESTIMATETIMEINTEGRATED}
\int_{\mathcal{M}_t} |\Flatcurl \vortrenormalized|^2 
& \leq
t 
\int_{\Sigma_0} |\Flatcurl \vortrenormalized|^2 
+
C t \int_{\mathcal{M}_t}  |\Flatcurl \vortrenormalized \cdot \partial \vortrenormalized|
+
\cdots.
\end{align}
Using \eqref{E:MODELCURLOFVORTICITYENERGYESTIMATETIMEINTEGRATED}
to control the last integral on RHS~\eqref{E:INTROSIMPLEHODGEINTEGRALIDENTITY},
adding the resulting inequality to the identity 
\eqref{E:MODELCURLOFVORTICITYENERGYESTIMATE},
and using Young's inequality in the form $ab \lesssim a^2 + b^2$,
we deduce that there is a constant $C > 0$ such that
\eqref{E:MODELVORTICITYESTIMATEGRONWALLREADY} holds.

\subsubsection{An overview of the analysis on the domains $\mathcal{M}$ from our main results}
\label{SSS:DOMAINSOFINTEREST}
We will now explain how to handle all of the additional complications that arise
when deriving the integral identities on domains $\mathcal{M}$ featured in Theorem~\ref{T:MAINIDSCHEMATICSTATEMENT},
specifically domains of the type featured in Fig.\,\ref{F:SPACETIMEDOMAIN}
and under the null lateral boundary assumption stated in Convention~\ref{C:INTROALTNOTATION}.
We will again focus our attention on deriving the integral identities for the specific vorticity vectorfield $\vortrenormalized$.

First, we highlight that $\mathcal{M}$ is assumed to be foliated by level sets of an acoustical time function $\Timefunction$,
which is not generally equal to the Cartesian time function $t$;
see Subsubsect.\,\ref{SSS:INTRODESCRIPTIONOFREGION}.
Since the divergence and curl operators in \eqref{E:INTRODIVCURLTRANSPORTONSIGMAT} 
\emph{are} the standard ``flat'' operators on the level sets of constant Cartesian time $t$,
we must find a way to relate information about the solution on $\Sigma_t$ to information about the solution
on $\widetilde{\Sigma}_{\Timefunction}$, where we recall that $\widetilde{\Sigma}_{\Timefunction}$ denotes
the portion of the level set of $\Timefunction$ in $\mathcal{M}$.
Moreover, to detect the remarkable null structures that are present when the lateral boundary $\underline{\mathcal{N}}_{\Timefunction}$
is null, it is crucial that we adapt our approach to the ``true geometry,'' that is, the geometry corresponding to the acoustical metric $\gfour$
of Def.\,\ref{D:ACOUSTICALMETRIC}.
To this end, we rely on a collection of geometric tensorfields adapted to $\widetilde{\Sigma}_{\Timefunction}$.
Specifically, we let $\topfirstfund$ denote the Riemannian metric induced on $\widetilde{\Sigma}_{\Timefunction}$
by $\gfour$, and we extend $\topfirstfund$ in the standard fashion to a positive semi-definite
quadratic form on spacetime tensors that vanishes along the $\gfour$-normal of  $\widetilde{\Sigma}_{\Timefunction}$.
We also let $\topproject$ denote $\gfour$-orthogonal projection onto $\widetilde{\Sigma}_{\Timefunction}$;
see Subsect.\,\ref{SS:FIRSTFUNDANDPROJECTIONS} for the precise definitions of these tensorfields.
In place of the vectorfield $J$ defined in \eqref{E:INTROSIMPLEELLIPTICCURRENT}, 
we use the following vectorfield, which we stress is
$\widetilde{\Sigma}_{\Timefunction}$-tangent, even though $\vortrenormalized$ is $\Sigma_t$-tangent:
\begin{align} \label{E:INTRONEWELLIPTICHYPERBOLICCURRENT}
	J^{\alpha}
	& := \vortrenormalized^{\beta} \topproject_{\ \beta}^{\delta} \topproject_{\ \gamma}^{\alpha} \partial_{\delta} \vortrenormalized^{\gamma}
		-
	\vortrenormalized^{\gamma} \topproject_{\ \gamma}^{\alpha} 
		\topproject_{\ \lambda}^{\kappa}
		\partial_{\kappa} \vortrenormalized^{\lambda}.
	\end{align}

A key step in the analysis is provided by\footnote{More precisely, in obtaining \eqref{E:INTROSIGMATIDLDEADAPTEDDIVERGENCEIDENTITY},
we have set the weight function $\weight$ in Lemma~\ref{L:BASICDIVERGENCEIDENTTIY} equal to unity.} 
Lemma~\ref{L:BASICDIVERGENCEIDENTTIY}.
This lemma yields the following Hodge-type identity,
which can be viewed of an analog of \eqref{E:INTROVERYSIMPLEHODGEID}
that is adapted to $\widetilde{\Sigma}_{\Timefunction}$:
\begin{align} \label{E:INTROSIGMATIDLDEADAPTEDDIVERGENCEIDENTITY}
|\pmb{\partial} \vortrenormalized|_{\topfirstfund}^2
	-
	(\projectedtransport_{\alpha}
	\projectedtransport \vortrenormalized^{\alpha})^2
&
=
\widetilde{\nabla}_{\alpha} J^{\alpha}
+
\frac{1}{2}
|d (\vortrenormalized_{\flat})|_{\topfirstfund}^2
+
(\Flatdiv \vortrenormalized)^2
	\\
& \ \
		+
			\lengthoftophypnorm^2
			(\projectedtransport_{\alpha} \Transport \vortrenormalized^{\alpha})^2
		+
			2 
			(\projectedtransport_{\alpha}\projectedtransport \vortrenormalized^{\alpha})
			\Flatdiv \vortrenormalized
				\notag \\
& \ \
			-
			2
			\lengthoftophypnorm 
			(\projectedtransport_{\alpha} \Transport \vortrenormalized^{\alpha})
			\projectedtransport_{\alpha} \projectedtransport \vortrenormalized^{\alpha}
			-
			2
			\lengthoftophypnorm 
			(\Flatdiv \vortrenormalized)
			\projectedtransport_{\alpha} 
			\Transport \vortrenormalized^{\alpha}
+
\cdots.
\notag
\end{align}
In \eqref{E:INTROSIGMATIDLDEADAPTEDDIVERGENCEIDENTITY},
$|\cdot|_{\topfirstfund}$ denotes the pointwise seminorm of a tensor with respect to $\topfirstfund$
(it is a seminorm since $\topfirstfund$ is positive definite only on $\widetilde{\Sigma}_{\Timefunction}$-tangent tensors).
Moreover,
$\projectedtransport$ is a $\widetilde{\Sigma}_{\Timefunction}$-tangent
vectorfield satisfying $|\projectedtransport|_{\topfirstfund} < 1$ 
(see \eqref{E:LENGTHOFPROJECTIONOFPONTTOTILDESIGMA} and just above \eqref{E:LENGTHOFTOPHYPNORM}),
which by Cauchy--Schwarz implies that LHS~\eqref{E:INTROSIGMATIDLDEADAPTEDDIVERGENCEIDENTITY}
is \emph{positive definite} in the derivatives of $\vortrenormalized$ in
directions tangent to $\widetilde{\Sigma}_{\Timefunction}$.
Furthermore,
we lower and raise Greek indices with $\gfour$ and its inverse,
$\vortrenormalized_{\flat}$ denotes the $\gfour$-dual of $\vortrenormalized$ 
(i.e., $\vortrenormalized_{\alpha} =  \gfour_{\alpha \beta} \vortrenormalized^{\beta}$),
$d$ denotes exterior derivative,
$\lengthoftophypnorm > 0$ is a function measuring the length of a certain normal to $\widetilde{\Sigma}_{\Timefunction}$,
$\cdots$ denotes error terms that are at most \emph{linear} in $\pmb{\partial} \vortrenormalized$,
and we stress that the terms $\Flatdiv \vortrenormalized$ on RHS~\eqref{E:INTROSIGMATIDLDEADAPTEDDIVERGENCEIDENTITY}
denote the standard Euclidean divergence of $\vortrenormalized$, while $\widetilde{\nabla}_{\alpha} J^{\alpha}$
denotes the covariant divergence of $J$ with respect to the Levi--Civita connection of $\topfirstfund$.
In addition, by using the transport equation \eqref{E:INTROTRANSPORTSPECIFICVORTICITYCARICATURE},
we see that the error terms on RHS~\eqref{E:INTROSIGMATIDLDEADAPTEDDIVERGENCEIDENTITY}
featuring a factor of $\Transport \vortrenormalized$ are also of type ``$\cdots$''.

Next, we apply the divergence theorem\footnote{See Sect.\,\ref{S:VOLUMEFORMSANDINTEGRALS} for the definitions of the geometric volume and area forms
that we use in our integrals.} 
along $\widetilde{\Sigma}_{\Timefunction}$
to the vectorfield $J$ defined in \eqref{E:INTROSIGMATIDLDEADAPTEDDIVERGENCEIDENTITY}
and use \eqref{E:INTRODIVCURLTRANSPORTONSIGMAT} to substitute for the term $(\Flatdiv \vortrenormalized)^2$
on RHS~\eqref{E:INTROSIGMATIDLDEADAPTEDDIVERGENCEIDENTITY},
thus obtaining the following integral identity, where $\spherenormal$  
(see Fig.\,\ref{F:SPACETIMEDOMAIN}) 
is the $\gfour$-unit outer normal to $\mathcal{S}_{\Timefunction}$ in $\widetilde{\Sigma}_{\Timefunction}$:
\begin{align} \label{E:INTRODIVTHEOREMALONGTILDESIGMA}
\int_{\widetilde{\Sigma}_{\Timefunction}} 
	\left\lbrace
		|\pmb{\partial} \vortrenormalized|_{\topfirstfund}^2
		-
		(\projectedtransport_{\alpha}
		\projectedtransport \vortrenormalized^{\alpha})^2
	\right\rbrace
=
\underbrace{
\int_{\mathcal{S}_{\Timefunction}}
\spherenormal_{\alpha} J^{\alpha} 
}_{\mbox{\upshape Dangerous}}
+
\int_{\widetilde{\Sigma}_{\Timefunction}} F^2 
+ 
\frac{1}{2}
\int_{\widetilde{\Sigma}_{\Timefunction}}
	|d (\vortrenormalized_{\flat})|_{\topfirstfund}^2
+
\int_{\widetilde{\Sigma}_{\Timefunction}}
\cdots.
\end{align}
As we further describe below, the integrals $\int_{\widetilde{\Sigma}_{\Timefunction}}$
on RHS~\eqref{E:INTRODIVTHEOREMALONGTILDESIGMA} can be shown to have sufficient regularity,
after an integration with respect to $\Timefunction$. Even this step requires substantial new ideas
compared to the simple domains treated in Subsubsect.\,\ref{SSS:SIMPLERDOMAINS};
we will return to this issue below.
For now, we focus on how to handle the ``Dangerous'' $\mathcal{S}_{\Timefunction}$ integral on RHS~\eqref{E:INTRODIVTHEOREMALONGTILDESIGMA}.
Its presence is simply a manifestation of the fact that
in the standard approach to elliptic estimates on $\widetilde{\Sigma}_{\Timefunction}$, 
\emph{one must know the data along the boundary $\mathcal{S}_{\Timefunction}$} in order to treat the boundary integrals. 
That is, the standard elliptic approach does not allow one to control LHS~\eqref{E:INTRODIVTHEOREMALONGTILDESIGMA} 
without \emph{specifying} boundary data for $\vortrenormalized$ along $\mathcal{S}_{\Timefunction}$. Put differently:
\begin{quote}
	Since $\vortrenormalized$ is determined\footnote{The transport equation
	\eqref{E:RENORMALIZEDVORTICTITYTRANSPORTEQUATION} for $\vortrenormalized$ and our assumptions on the spacetime
	region $\mathcal{M}_T$ guarantee that $\vortrenormalized|_{\mathcal{S}_{\Timefunction}}$ 
	and $\partial \vortrenormalized|_{\mathcal{S}_{\Timefunction}}$ are evolutionarily determined by their initial data
	on $\widetilde{\Sigma}_0$.
	However, by itself, equation \eqref{E:RENORMALIZEDVORTICTITYTRANSPORTEQUATION} does not allow
	us to conclude the desired quantitative Sobolev control over $\vortrenormalized|_{\mathcal{S}_{\Timefunction}}$ 
	and $\partial \vortrenormalized|_{\mathcal{S}_{\Timefunction}}$
	that we need in order for \eqref{E:INTRODIVTHEOREMALONGTILDESIGMA} to be useful.} 
	by the initial data on $\widetilde{\Sigma}_0$, 
	we cannot ``specify'' its boundary data on $\mathcal{S}_{\Timefunction}$; the boundary data $\mathcal{S}_{\Timefunction}$ is evolutionarily 
	determined by the initial data, and for the identity \eqref{E:INTRODIVTHEOREMALONGTILDESIGMA} to be of any use,
	we must find a way to ``access quantitative 
	information about the boundary data.''
\end{quote}

Let us describe a natural attempt at how to control the dangerous integral
$\int_{\mathcal{S}_{\Timefunction}}
\spherenormal_{\alpha} J^{\alpha}$ on RHS~\eqref{E:INTRODIVTHEOREMALONGTILDESIGMA}, 
which in the end does not work.
Specifically, through careful geometric decompositions,
in the spirit of those carried out in the proof of Lemma~\ref{L:PRELIMINARYANALYSISOFBOUNDARYINTEGRAND},
it can be shown that
the boundary integrand $\spherenormal_{\alpha} J^{\alpha}$ 
can be re-expressed as terms of the schematic form $\vortrenormalized \cdot \angpartial \vortrenormalized$,
where the operator $\angpartial$ is $\mathcal{S}_{\Timefunction}$-tangent.  
Now formally, we have a fractional integration by parts identity of the schematic form
$\int_{\mathcal{S}_{\Timefunction}}
\vortrenormalized \cdot \angpartial \vortrenormalized
=
\int_{\mathcal{S}_{\Timefunction}}
	\angpartial^{1/2} \vortrenormalized 
	\cdot 
	\angpartial^{1/2} \vortrenormalized
+
\cdots
$,
and Sobolev trace estimates suggest that
$
\int_{\mathcal{S}_{\Timefunction}}
	\angpartial^{1/2} \vortrenormalized 
	\cdot 
	\angpartial^{1/2} \vortrenormalized
$
can be controlled in terms of
$
\int_{\widetilde{\Sigma}_{\Timefunction}}
	|\partial \vortrenormalized|_{\topfirstfund}^2 
$
(plus lower-order terms that we ignore here),
consistent with the strength of LHS~\eqref{E:INTRODIVTHEOREMALONGTILDESIGMA}.
However, we highlight the following key point:
\begin{quote}
	The Sobolev trace estimate mentioned above is at a critical regularity level,
	and the ``constant'' in the trace estimate could be large.
	This could prevent one from treating the
	boundary integral 
	$\int_{\mathcal{S}_{\Timefunction}}
	\vortrenormalized \cdot \angpartial \vortrenormalized$ 
	as an error term to be absorbed into LHS~\eqref{E:INTRODIVTHEOREMALONGTILDESIGMA}.
	In total, \emph{this calls into question the usefulness of the Hodge-type identity \eqref{E:INTRODIVTHEOREMALONGTILDESIGMA}}
	and suggests that the $\mathcal{S}_{\Timefunction}$ integral cannot be directly controlled using only elliptic theory, 
	prompting us to seek a new approach.
\end{quote}

To overcome the regularity difficulty described above, 
we adopt a different strategy to handle the boundary integrand
$\spherenormal_{\alpha} J^{\alpha}$.
First, in Lemma~\ref{L:PRELIMINARYANALYSISOFBOUNDARYINTEGRAND},
without using the compressible Euler equations,
we derive the following crucial geometric identity,\footnote{In obtaining \eqref{E:INTROCARICATUREOFPRELIMINARYANALYSISOFBOUNDARYINTEGRAND},
we have set the weight function $\weight$ from Lemma~\ref{L:PRELIMINARYANALYSISOFBOUNDARYINTEGRAND} equal to unity.} 
which we  
restate here in schematic form as follows, ignoring order unity coefficients but
respecting the overall sign of the important terms:
\begin{align} \label{E:INTROCARICATUREOFPRELIMINARYANALYSISOFBOUNDARYINTEGRAND}
\spherenormal_{\alpha} J^{\alpha}
& =
	-
		\modgen
			\left\lbrace
			|\angvortrenormalized|^2
		\right\rbrace 
		+
		\uspecialgen^{\alpha}
		\angvortrenormalized^{\beta}
		(\partial_{\alpha} \angvortrenormalized_{\beta} - \partial_{\beta} \angvortrenormalized_{\alpha})
		+
		\underline{\mbox{\upshape Tangential}}
		+
		\cdots,
\end{align}
where $\cdots$ denotes perfect $\mathcal{S}_{\Timefunction}$-divergences
(which therefore vanish when inserted into the $\mathcal{S}_{\Timefunction}$ integral on RHS~\eqref{E:INTRODIVTHEOREMALONGTILDESIGMA}).
In \eqref{E:INTROCARICATUREOFPRELIMINARYANALYSISOFBOUNDARYINTEGRAND}, 
the vectorfield $\modgen$ is $\gfour$-null,
$\underline{\mathcal{N}}_{\Timefunction}$-\emph{tangent}, and normalized by $\modgen \Timefunction = 1$,
i.e., it is a null generator adapted to $\Timefunction$ 
(and thus $\modgen = \frac{\partial}{\partial \Timefunction}$ relative to appropriate coordinates on $\underline{\mathcal{N}}_{\Timefunction}$).
Moreover, $\uspecialgen$ (see Def.\,\ref{D:SPECIALGENERATOR}) is $\underline{\mathcal{N}}_{\Timefunction}$-\emph{tangent},
$\angvortrenormalized$ is the $\gfour$-orthogonal projection of $\vortrenormalized$ onto $\mathcal{S}_{\Timefunction}$
(and thus $\angvortrenormalized$ is $\mathcal{S}_{\Timefunction}$-\emph{tangent}),
and the error term $\underline{\mbox{\upshape Tangential}}$
has the good properties described in Theorem~\ref{T:MAINIDSCHEMATICSTATEMENT},
e.g., it enjoys admissible regularity and the only derivatives of the density and velocity that appear are \emph{tangent} to $\underline{\mathcal{N}}_{\Timefunction}$.
Next, we substitute RHS~\eqref{E:INTROCARICATUREOFPRELIMINARYANALYSISOFBOUNDARYINTEGRAND} 
for the first integrand on RHS~\eqref{E:INTRODIVTHEOREMALONGTILDESIGMA} and then
integrate \eqref{E:INTRODIVTHEOREMALONGTILDESIGMA} with respect to $\Timefunction$
(see Remark~\ref{R:WHYINTEGRATEINTIME}),
use the fundamental theorem of calculus-type result
$
-
\int_{\underline{\mathcal{N}}_{\Timefunction}}
\modgen
			\left\lbrace
			|\angvortrenormalized|^2
		\right\rbrace 
=
-
\int_{\mathcal{S}_{\Timefunction}}
|\angvortrenormalized|^2
+
\int_{\mathcal{S}_0}
|\angvortrenormalized|^2
+
\cdots
$
(see Lemma~\ref{L:DIFFERENTIATIONANDINTEGRALIDENTITEISINVOLVINGST} for the details),
and use \eqref{E:INTROTRANSPORTSPECIFICVORTICITYCARICATURE} to algebraically
substitute for $\Transport \vortrenormalized$,
thereby obtaining the following identity for $\Timefunction \in [0,T]$, expressed in schematic form
(see \eqref{E:IDENTITYMAINQUADRATICFORMFORCONTROLLINGFIRSTDERIVATIVESOFSPECIFICVORTICITYANDENTROPYGRADIENT}
for a precise formula for the integrand of the spacetime integral $\int_{\mathcal{M}_{\Timefunction}} \cdots$ on the left-hand side):
\begin{align} \label{E:INTROMAINSPACETIMEINTEGRALIDENTITY}
\int_{\mathcal{M}_{\Timefunction}}
	&
	\left\lbrace
	|\Transport \vortrenormalized|^2
	+
	|\pmb{\partial} \vortrenormalized|_{\topfirstfund}^2
	-
	(\projectedtransport_{\alpha}
	\projectedtransport \vortrenormalized^{\alpha})^2
	\right\rbrace
+
\int_{\mathcal{S}_{\Timefunction}} 
|\angvortrenormalized|^2
	\\
& =
\int_{\mathcal{S}_0} 
|\angvortrenormalized|^2
+
\int_{\underline{\mathcal{N}}_{\Timefunction}}
	\uspecialgen^{\alpha}
	\angvortrenormalized^{\beta}
	(\partial_{\alpha} \vortrenormalized_{\beta} - \partial_{\beta} \vortrenormalized_{\alpha}) 
+
\int_{\underline{\mathcal{N}}_{\Timefunction}}
\underline{\mbox{\upshape Tangential}}
	\notag \\
& \ \
+
\int_{\mathcal{M}_{\Timefunction}} 
F^2 
+ 
\frac{1}{2}
\int_{\mathcal{M}_{\Timefunction}}
	|d (\vortrenormalized_{\flat})|_{\topfirstfund}^2
+
\int_{\mathcal{M}_{\Timefunction}}
|\vortrenormalized \cdot \partial v + \Transport v \cdot \GradEnt|^2
+
\int_{\mathcal{M}_{\Timefunction}}
\cdots.
\notag
\end{align}
Note that $\int_{\mathcal{S}_0} 
|\angvortrenormalized|^2$ is an ``initial data term'' that we consider known.
Since the timelike vectorfield $\Transport$ is transversal to $\widetilde{\Sigma}_{\Timefunction}$ 
(see Footnote~\ref{FN:TIMEFUNCTIONPASTDIRECTEDGRADIENT}),
it follows that the spacetime integrand on 
LHS~\eqref{E:INTROMAINSPACETIMEINTEGRALIDENTITY} is positive definite in $\pmb{\partial} \vortrenormalized$;
this explains the positivity of the quadratic form $\mathscr{Q}(\pmb{\partial} \vortrenormalized,\pmb{\partial} \vortrenormalized)$
in Theorem~\ref{T:MAINIDSCHEMATICSTATEMENT}; see Lemma~\ref{L:POSITIVITYPROPERTIESOFVARIOUSQUADRATICFORMS}
for a proof of the positivity.

All error integrals on RHS~\eqref{E:INTROMAINSPACETIMEINTEGRALIDENTITY}
except for
$
\int_{\underline{\mathcal{N}}_{\Timefunction}}
	\uspecialgen^{\alpha}
	\angvortrenormalized^{\beta}
	(\partial_{\alpha} \vortrenormalized_{\beta} - \partial_{\beta} \vortrenormalized_{\alpha}) 
$
can readily be shown to have the desired regularity and structural properties;
see, for example, the proof of Theorem~\ref{T:LOCALIZEDAPRIORIESTIMATES} for the details.
Thus, to finish the proof sketch of Theorem~\ref{T:MAINIDSCHEMATICSTATEMENT},
it remains for us
to explain how to handle the null hypersurface integral
$
\int_{\underline{\mathcal{N}}_{\Timefunction}}
	\uspecialgen^{\alpha}
	\angvortrenormalized^{\beta}
	(\partial_{\alpha} \vortrenormalized_{\beta} - \partial_{\beta} \vortrenormalized_{\alpha}) 
$.
This is the most difficult analysis in the paper.
While the integrand 
$\uspecialgen^{\alpha}
	\angvortrenormalized^{\beta}
	(\partial_{\alpha} \vortrenormalized_{\beta} - \partial_{\beta} \vortrenormalized_{\alpha}) 
$
has the desired $\underline{\mathcal{N}}_{\Timefunction}$-tangential differentiation structure,
it is not clear that it has sufficient regularity to be treated as an error term.
The difficulty is that the integral involves $\pmb{\partial} \vortrenormalized$ and is 
along the hypersurface $\underline{\mathcal{N}}_{\Timefunction}$, while
LHS~\eqref{E:INTROMAINSPACETIMEINTEGRALIDENTITY} is 
quadratic in $\pmb{\partial} \vortrenormalized$ and is
an integral over
the spacetime region $\mathcal{M}_{\Timefunction}$;
thus, for reasons similar to the ones given in the paragraph above \eqref{E:INTROCARICATUREOFPRELIMINARYANALYSISOFBOUNDARYINTEGRAND},
$
\int_{\underline{\mathcal{N}}_{\Timefunction}}
	\uspecialgen^{\alpha}
	\angvortrenormalized^{\beta}
	(\partial_{\alpha} \vortrenormalized_{\beta} - \partial_{\beta} \vortrenormalized_{\alpha}) 
$ cannot be controlled with hypersurface trace estimates.
We dedicate all of Sect.\,\ref{S:GEOMETRICDECOMPOSITIONSTEIDTOANTISYMMETRICPART}
to showing that the integrand
$\uspecialgen^{\alpha}
	\angvortrenormalized^{\beta}
	(\partial_{\alpha} \vortrenormalized_{\beta} - \partial_{\beta} \vortrenormalized_{\alpha})$
enjoys the desired structures. 
We now highlight the main steps in this analysis.
\begin{itemize}
 \item First, in Cor.\,\ref{C:SHARPDECOMPOSITIONOFANTISYMMETRICGRADIENTS},
	we use the new formulation of compressible Euler flow from Theorem~\ref{T:GEOMETRICWAVETRANSPORTSYSTEM} 
	to provide a geo-analytic decomposition of the two-form
	$
\partial_{\alpha} \vortrenormalized_{\beta} - \partial_{\beta} \vortrenormalized_{\alpha}
$,
i.e., the decomposition holds only for solutions. 
The main idea is to split the two-form into components tangent to $\underline{\mathcal{N}}_{\Timefunction}$ and components in the direction of $\Transport$,
and to separate out the Euclidean curl components 
$\Flatcurl \vortrenormalized$, which by 
\eqref{E:INTRODIVCURLTRANSPORTONSIGMAT}
can be independently controlled along $\underline{\mathcal{N}}_{\Timefunction}$.
The ``$\Transport$'' components can be treated with the transport equation \eqref{E:RENORMALIZEDVORTICTITYTRANSPORTEQUATION}
together with the simple fact that $\Transport^{\alpha} \vortrenormalized_{\alpha} = 0$,
since $\Transport$ is $\gfour$-orthogonal to $\Sigma_t$ (see \eqref{E:TRANSPORTONEFORMIDENTITY})
while $\vortrenormalized$ is $\Sigma_t$-tangent.
In total, this allows us to show that ``most pieces'' of the term
$\uspecialgen^{\alpha}
	\angvortrenormalized^{\beta}
	(\partial_{\alpha} \vortrenormalized_{\beta} - \partial_{\beta} \vortrenormalized_{\alpha})$
	exhibit the desired structures. However,
	there is one remaining ``difficult piece'' that requires special geometric treatment.
\item The ``difficult piece'' mentioned in the previous sentence is in fact generated by the entropy gradient term
$\Transport v \cdot \GradEnt$ on RHS~\eqref{E:INTROTRANSPORTSPECIFICVORTICITYCARICATURE} 
and thus is absent in the isentropic case $\Ent \equiv \mbox{\upshape const}$.
In the context of Cor.\,\ref{C:SHARPDECOMPOSITIONOFANTISYMMETRICGRADIENTS}, the difficult entropy gradient terms re-emerge
as the fifth and sixth terms on RHS~\eqref{E:KEYIDENTITYANTISYMMETRICPARTOFSPECIFICVORTICITYDUALGRADIENT},
specifically
$
-
		\upepsilon_{\alpha \beta \gamma \delta}
		 \uspecialgen^{\alpha} 
		 \angvortrenormalized^{\beta}
		(\Transport v^{\gamma}) 
		\GradEnt^{\delta}
		+
		\upepsilon_{\alpha \beta \gamma \delta}
			\uspecialgen^{\alpha} 
			\angvortrenormalized^{\beta} 
			\Transport^{\gamma}
			[\GradEnt^{\delta}
				(\Flatdiv v)
				-
				\GradEnt^a \partial_a v^{\delta}]
$,
where $\upepsilon_{\alpha \beta \gamma \delta}$ is the fully antisymmetric symbol normalized by
$\upepsilon_{0123}=1$. The difficult part of the analysis 
is showing that,
after contracting this combination of terms against $\uspecialgen^{\alpha}
	\angvortrenormalized^{\beta}$,
the resulting expression
involves only $\underline{\mathcal{N}}_{\Timefunction}$-tangential derivatives of the velocity and density;
as we have mentioned, $\underline{\mathcal{N}}_{\Timefunction}$-tangential derivatives of the velocity and density
are controllable using standard energy estimates for the wave equations 
\eqref{E:VELOCITYWAVEEQUATION}-\eqref{E:RENORMALIZEDDENSITYWAVEEQUATION}.
After a standard preliminary geometric decomposition provided by
Lemma~\ref{L:PRELIMINARYDECOMPOSITIONOFSUBTLETERMS},
in which we decompose all derivatives in this combination of terms
into their $\underline{\mathcal{N}}_{\Timefunction}$-tangential components and $\Transport$-parallel components
and exploit the compressible Euler formulation provided by Theorem~\ref{T:GEOMETRICWAVETRANSPORTSYSTEM},
we are left with precisely one product that needs to be carefully treated:\footnote{When $\underline{\mathcal{N}}_{\Timefunction}$ is spacelike,
the formula in Prop.\,\ref{P:KEYDETERMINANT} has $\sidehypnorm$ in place of $\uLunit$, where $\sidehypnorm$ is normal to
$\underline{\mathcal{N}}_{\Timefunction}$ and normalized by $\sidehypnorm t = 1$. When $\underline{\mathcal{N}}_{\Timefunction}$ 
is null, we have $\sidehypnorm = \uLunit$.}
$\upepsilon_{\alpha \beta \gamma \delta} 
				\uspecialgen^{\alpha}
				\angvortrenormalized^{\beta}
				\uLunit^{\gamma}
				(\GradEnt^{\delta} 
				+
				\GradEnt^a \uLunit_a
				\Transport^{\delta})
				\Transport \LogDensity
$,
where $\uLunit$ is a null generator of $\underline{\mathcal{N}}_{\Timefunction}$, normalized by $\uLunit t = 1$ (where $t$ is Cartesian time).
The difficulty is that the factor $\Transport \LogDensity$ involves a derivative of $\LogDensity$
in a direction transversal to $\underline{\mathcal{N}}_{\Timefunction}$;
Since the first derivatives of $\LogDensity$ can be controlled only with the wave equation \eqref{E:RENORMALIZEDDENSITYWAVEEQUATION},
and since wave equation energies along null hypersurfaces do not control transversal derivatives (see e.g.\ \eqref{E:NULLCASEWAVEFLUXESSEMICOERCIVE}),
this calls into question whether or not one can control the corresponding null hypersurface error integral
$
\int_{\underline{\mathcal{N}}_{\Timefunction}}
\upepsilon_{\alpha \beta \gamma \delta} 
				\uspecialgen^{\alpha}
				\angvortrenormalized^{\beta}
				\uLunit^{\gamma}
				(\GradEnt^{\delta} 
				+
				\GradEnt^a \uLunit_a
				\Transport^{\delta})
				\Transport \LogDensity
$.
\item In Prop.\,\ref{P:KEYDETERMINANT} ,
we use the detailed structure of the vectorfield $\uspecialgen$ (see Def.\,\ref{D:SPECIALGENERATOR})
to prove that this remaining difficult product
$\upepsilon_{\alpha \beta \gamma \delta} 
				\uspecialgen^{\alpha}
				\angvortrenormalized^{\beta}
				\uLunit^{\gamma}
				(\GradEnt^{\delta} 
				+
				\GradEnt^a \uLunit_a
				\Transport^{\delta})
				\Transport \LogDensity
$
in fact
\emph{completely vanishes};\footnote{When $\underline{\mathcal{N}}_{\Timefunction}$ is spacelike, this product does not vanish. In this case,
the identity \eqref{E:KEYDETERMINANT} provided by
Prop.\,\ref{P:KEYDETERMINANT} allows us to re-express it in terms of $\underline{\mathcal{N}}_{\Timefunction}$-tangential derivatives.}
see the identity \eqref{E:NULLCASEKEYDETERMINANT}.
\end{itemize}

This finishes our proof sketch of Theorem~\ref{T:MAINIDSCHEMATICSTATEMENT} for $\vortrenormalized$ 
on domains $\mathcal{M}$ of the type depicted in
Fig.\,\ref{F:SPACETIMEDOMAIN}.
The analysis in the case of the entropy gradient is similar but much simpler, mostly because
the error term
$\uspecialgen^{\alpha}
	\angvortrenormalized^{\beta}
	(\partial_{\alpha} \GradEnt_{\beta} - \partial_{\beta} \GradEnt_{\alpha})$
is much simpler
than the error
term 
$\uspecialgen^{\alpha}
	\angvortrenormalized^{\beta}
	(\partial_{\alpha} \vortrenormalized_{\beta} - \partial_{\beta} \vortrenormalized_{\alpha})$,
stemming from the much simpler structure of the identity
\eqref{E:KEYIDENTITYANTISYMMETRICPARTOFENTROPYGRADIENTDUALGRADIENT} for $d (\GradEnt_{\flat})$
compared to the identity \eqref{E:KEYIDENTITYANTISYMMETRICPARTOFSPECIFICVORTICITYDUALGRADIENT} for $d (\vortrenormalized_{\flat})$.

In the case of domains $\mathcal{M}$ covered by double-null foliations,
Theorem~\ref{T:MAINIDSCHEMATICSTATEMENT} can be proved using 
modifications of the arguments sketched above;
see Sect.\,\ref{S:DOUBLENULL} for the details.

\subsection{Paper outline}
\label{SS:PAPEROUTLINE}
The remainder of the paper is organized as follows.

\begin{itemize}
	\item In Sect.\,\ref{S:COMPRESSIBLEEULEREQUATIONS},
		we recall, as Theorem~\ref{T:GEOMETRICWAVETRANSPORTSYSTEM}, 
		the results of \cite{jS2019c}, 
		which provide the geometric formulation of compressible Euler flow that we use in proving our main results.
		The equations of Theorem~\ref{T:GEOMETRICWAVETRANSPORTSYSTEM} feature ``hyperbolic'' wave- and transport-parts,
		as well as ``elliptic-hyperbolic'' div-curl-transport parts.
		We also define some geometric tensors that play a fundamental role in the rest of the article.
	\item In Sect.\,\ref{S:SPACETIMEDOMAINS}, we state our assumptions 
			on the acoustical time function $\Timefunction$ that we use to foliate the spacetime regions $\mathcal{M}$ under study,
			and we state the assumptions on $\mathcal{M}$
			that we use to prove our main results. We then derive some basic properties
			of various tensors tied to the geometry of $\mathcal{M}$.
	\item In Sect.\,\ref{S:ELLIPTICHYPERBOLICDIVERGENCEIDENTITYANDPRELIMANALYSISOFOUNDARY}, 
		we first construct the coercive quadratic forms that are featured in our main integral identities,
		which are provided by Theorem~\ref{T:MAINREMARKABLESPACETIMEINTEGRALIDENTITY}.
		We then derive some divergence-form elliptic Hodge-type identities for the specific vorticity $\vortrenormalized$ and the entropy gradient $\GradEnt$.
		The identities are valid along the portions of the level sets of $\Timefunction$ that are contained in $\mathcal{M}$; 
		we denote these portions by $\widetilde{\Sigma}_{\Timefunction}$. 
		In the proof of Theorem~\ref{T:MAINREMARKABLESPACETIMEINTEGRALIDENTITY},
		we will integrate these divergence identities along $\widetilde{\Sigma}_{\Timefunction}$,
		which leads to boundary integrals along the topological spheres $\mathcal{S}_{\Timefunction} = \partial \widetilde{\Sigma}_{\Timefunction}$.
		The boundary integrals are seemingly dangerous from the point of view of regularity, and
		it is not apparent that they are controllable. However,
		Prop.\,\ref{P:STRUCTUREOFERRORINTEGRALS} and Theorem~\ref{T:STRUCTUREOFERRORTERMS}
		together show that after an integration with respect to $\Timefunction$
		and arguments that exploit the special geo-analytic structures exhibited by the compressible Euler formulation of
		Theorem~\ref{T:GEOMETRICWAVETRANSPORTSYSTEM},
		these dangerous-looking terms can be shown to be equal to error integrals along the lateral hypersurface $\underline{\mathcal{H}}$
		featuring only $\underline{\mathcal{H}}$-tangential derivatives of quantities that enjoy
		sufficient regularity;
		as we described in Subsect.\,\ref{SS:APPLICATIONS}, 
		these properties are crucial for the study of shocks.
	\item In Sect.\,\ref{S:GEOMETRICDECOMPOSITIONSTEIDTOANTISYMMETRICPART},
		we provide a series of geometric identities tied to the antisymmetric tensors
		$\partial_{\alpha} \vortrenormalized_{\beta} - \partial_{\beta} \vortrenormalized_{\alpha}$
		and
		$\partial_{\alpha} \GradEnt_{\beta} - \partial_{\beta} \GradEnt_{\alpha}$.
		The results of this section are crucial
		for the proof of Prop.\,\ref{P:STRUCTUREOFERRORINTEGRALS},
		where they are used to show that some
		difficult error terms involving these antisymmetric tensors
		are controllable from the point of view of regularity and thanks to their
		$\underline{\mathcal{H}}$-tangential derivative structure.
		Many aspects of the analysis in this section rely on the compressible Euler formulation of
		Theorem~\ref{T:GEOMETRICWAVETRANSPORTSYSTEM}.
	\item In Sect.\,\ref{S:VOLUMEFORMSANDINTEGRALS}, we define the geometric volume forms, area forms, and integrals 
		that we use in our main integral identities, and we derive some simple identities tied to them.
	\item In Sect.\,\ref{S:STRUCTUREOFLATERALERRORINTEGRALS}, we first prove
		Prop.\,\ref{P:STRUCTUREOFERRORINTEGRALS}, which yields identities for
		the error integrals along the lateral boundary hypersurface $\underline{\mathcal{H}}$
		mentioned above. We then prove the aforementioned Theorem~\ref{T:STRUCTUREOFERRORTERMS},
		which exhibits the remarkable structure of the integrands.
	\item In Sect.\,\ref{S:MAININTEGRALIDENTITIES}, we 
		prove Theorem~\ref{T:MAINREMARKABLESPACETIMEINTEGRALIDENTITY}, which provides the 
		main new integral identities verified by the first derivatives
		of the specific vorticity and the entropy gradient. 
		In Remark~\ref{R:HIGHLIGHTKEYSTRUCTURES}, we highlight some of the key structural properties of the identities,
		and in Theorem~\ref{T:STRUCTUREOFERRORTERMS}, we give a precise description of the most important of these properties.
		The proof of Theorem~\ref{T:MAINREMARKABLESPACETIMEINTEGRALIDENTITY} follows in a straightforward fashion from
		the divergence identities of Sect.\,\ref{S:ELLIPTICHYPERBOLICDIVERGENCEIDENTITYANDPRELIMANALYSISOFOUNDARY}
		and the identities for the lateral boundary hypersurface error integrals from
		Prop.\,\ref{P:STRUCTUREOFERRORINTEGRALS}.
	\item Sect.\,\ref{S:APRIORI} provides an application of the previous results.
		Specifically, we combine the results of the previous sections
		with standard applications of the geometric vectorfield method for wave equations
		to prove Theorem~\ref{T:LOCALIZEDAPRIORIESTIMATES}, 
		which yields a priori estimates for compressible Euler solutions,
		localized to compact acoustically globally hyperbolic spacetime subsets $\mathcal{M}$.
		The estimates exhibit the local gain in regularity for the specific vorticity and entropy gradient
		described in Point \textbf{I} of Subsect.\,\ref{SS:APPLICATIONS}.
		The proof of Theorem~\ref{T:LOCALIZEDAPRIORIESTIMATES}
		crucially relies on the special structures of the integral identities provided by Theorem~\ref{T:MAINREMARKABLESPACETIMEINTEGRALIDENTITY}
		and Theorem~\ref{T:STRUCTUREOFERRORTERMS},
		especially in the case that the lateral boundary $\underline{\mathcal{H}}$ is acoustically null.
		To shorten the exposition, in Sect.\,\ref{S:APRIORI}, we assume that the acoustical time function $\Timefunction$ 
		is equal to the Cartesian time function $t$.
	\item In Sect.\,\ref{S:DOUBLENULL}, we extend Theorem~\ref{T:MAINREMARKABLESPACETIMEINTEGRALIDENTITY} 
		yield analogous integral identities for spacetime regions covered by double-null foliations; 
		as we explained in Subsect.\,\ref{SS:APPLICATIONS}, 
		double-null foliations form the starting point for the study of 
		the characteristic initial value problem for compressible Euler flow.
		The main result of this section is Theorem~\ref{T:DOUBLENULLMAINTHEOREM}.
	\item In Appendix~\ref{A:APPENDIXFORBULK}, we summarize some of the key notation used throughout the article.
\end{itemize}

\section{The geometric formulation of $3D$ compressible Euler flow}
\label{S:COMPRESSIBLEEULEREQUATIONS}
In this section, we provide Theorem~\ref{T:GEOMETRICWAVETRANSPORTSYSTEM},
which recalls the new formulation of compressible Euler flow from \cite{jS2019c}.
The remarkable structures in this formulation play a fundamental role in 
the rest of the paper.
Before stating the theorem, 
we first define some non-standard fluid variables that
are prominently featured in its statement, specifically
modified fluid variables of Def.\,\ref{D:RENORMALIZEDCURLOFSPECIFICVORTICITY}.
We also define several geometric
tensors associated to the flow, notably the acoustical metric $\gfour$ of
Def.\,\ref{D:ACOUSTICALMETRIC}.
Moreover, in Subsubsect.\,\ref{SSS:STANDARDNULLFORMS}, to explain
the significance of the null form structures revealed by Theorem~\ref{T:GEOMETRICWAVETRANSPORTSYSTEM},
we recall the definitions and properties of null forms relative to $\gfour$.

\subsection{Additional geometric and analytic quantities associated to the flow}
\label{SS:ADDITIONALGEOANALYTICQUANTITIES}
In this subsection, we define various objects of analytic 
and physical significance that are needed for
the formulation of compressible Euler flow provided by
Theorem~\ref{T:GEOMETRICWAVETRANSPORTSYSTEM}.

\subsubsection{Modified fluid variables}
\label{SSS:MODIFIEDFLUIDVARIABLES}
The tensorfields $\VortVort$ and $\DivGradEnt$ 
in the next definition
are modified versions of $\curl \vortrenormalized$ and $\dive \GradEnt$.
They were discovered in \cites{jLjS2016a,jS2019c}
and play a fundamental role in our analysis.
Specifically, the equations of
Theorem~\ref{T:GEOMETRICWAVETRANSPORTSYSTEM}
show that $\VortVort$ and $\DivGradEnt$ satisfy 
transport equations whose source terms 
\textbf{i)} are one degree more differentiable than expected
and 
\textbf{ii)} exhibit remarkable null structures.
We exploit Property \textbf{i)} 
when deriving div-curl-transport estimates to exhibit a gain
in differentiability for $\vortrenormalized$
and $\GradEnt$; see Theorem~\ref{T:LOCALIZEDAPRIORIESTIMATES}. 
As is explained in \cites{gHsKjSwW2016,jS2016b,jLjS2016a,jS2018c,mDjS2019,jS2019c},
Property \textbf{ii)} is crucial for the study of shock formation without symmetry assumptions;
see also Subsubsect.\,\ref{SSS:STANDARDNULLFORMS}.

\begin{definition}[\textbf{Modified fluid variables}]
	\label{D:RENORMALIZEDCURLOFSPECIFICVORTICITY}
	We define the Cartesian components of the $\Sigma_t$-tangent vectorfield $\VortVort$ 
	and the scalar function $\DivGradEnt$ as follows,
	($i=1,2,3$):
	\begin{subequations}
	\begin{align} \label{E:RENORMALIZEDCURLOFSPECIFICVORTICITY}
		\VortVort^i
		& :=
			\exp(-\LogDensity) (\Flatcurl \vortrenormalized)^i
			+
			\exp(-3\LogDensity) \Speed^{-2} \frac{p_{;\Ent}}{\bar{\varrho}} \GradEnt^a \partial_a v^i
			-
			\exp(-3\LogDensity) \Speed^{-2} \frac{p_{;\Ent}}{\bar{\varrho}} (\Flatdiv v) \GradEnt^i,
				\\
		\DivGradEnt
		& := 
			\exp(-2 \LogDensity) \Flatdiv \GradEnt 
			-
			\exp(-2 \LogDensity) \GradEnt^a \partial_a \LogDensity,
			\label{E:RENORMALIZEDDIVOFENTROPY}
	\end{align}
	\end{subequations}
	where $\Flatcurl$ denotes the standard Euclidean curl operator on $\Sigma_t$
	and $\Flatdiv$ denotes the standard Euclidean divergence operator on $\Sigma_t$.
\end{definition}

\subsubsection{The acoustical metric $\gfour$, basic properties of $\gfour$ and $\Transport$, and classification of vectors and hypersurfaces}
\label{SSS:ACOUSTICALMETRIC}
The \emph{acoustical metric} drives the propagation of sound waves
and is featured prominently in Theorem~\ref{T:GEOMETRICWAVETRANSPORTSYSTEM}.

\begin{definition}[The acoustical metric\footnote{Other authors have defined the acoustical metric to be $\Speed^2 \gfour$.
	We prefer our definition because it implies that $(\gfour^{-1})^{00} = - 1$,
	which simplifies the presentation of many formulas.} 
and its inverse] 
\label{D:ACOUSTICALMETRIC}
We define the \emph{acoustical metric} $\gfour$ and the 
\emph{inverse acoustical metric} 
$\gfour^{-1}$ relative
to the Cartesian coordinates as follows:
\begin{subequations}
	\begin{align}
		\gfour
		& := 
		-  dt \otimes dt
			+ 
			\Speed^{-2} \sum_{a=1}^3(dx^a - v^a dt) \otimes (dx^a - v^a dt),
				\label{E:ACOUSTICALMETRIC} \\
		\gfour^{-1} 
		& := 
			- \Transport \otimes \Transport
			+ \Speed^2 \sum_{a=1}^3 \partial_a \otimes \partial_a.
			\label{E:INVERSEACOUSTICALMETRIC}
	\end{align}
\end{subequations}

\end{definition}

It is straightforward to check that
indeed, the $4 \times 4$ matrix with components $(\gfour^{-1})^{\alpha \beta}$
is the inverse of the $4 \times 4$ matrix with components $\gfour_{\alpha \beta}$.

Throughout, if $\mathbf{X}$ and $\mathbf{Y}$ are vectors, then
$\gfour(\mathbf{X},\mathbf{Y}) := \gfour_{\alpha \beta} \mathbf{X}^{\alpha} \mathbf{Y}^{\beta}$
denotes their inner product with respect to $\gfour$.

Most of the geometric constructions in this paper are tied to $\gfour$.
\begin{quote}
	Thus, in the rest of the paper, we lower and raise lowercase Greek ``spacetime'' indices	
	with $\gfour$ and $\gfour^{-1}$, e.g., $\Transport_{\alpha} := \gfour_{\alpha \beta} \Transport^{\beta}$.
\end{quote}

On a a few occasions, we will find it convenient to explicitly distinguish between a vectorfield and its $\gfour$-dual
one-form. Specifically, if $\mathbf{X}$ is a vectorfield, then 
we sometimes use the notation $\mathbf{X}_{\flat}$ to denote the corresponding $\gfour$-dual one-form,
i.e.,
\begin{align} \label{E:DUALONEFORMINDICES}
	(\mathbf{X}_{\flat})_{\alpha}
	& := \gfour_{\alpha \beta} \mathbf{X}^{\beta}.
\end{align}

We now provide the following basic definition, also tied to $\gfour$, 
which plays a key role in our analysis.

\begin{definition}[$\gfour$-spacelike, $\gfour$-timelike, and $\gfour$-null]
	\label{D:SPACELIKETIMELIKENULL}
	Vectors $\mathbf{X}$ are classified as follows:
	\begin{subequations}
	\begin{align}
		\gfour(\mathbf{X},\mathbf{X}) & < 0
		& & \gfour\mbox{\upshape-timelike},
			\\
		\gfour(\mathbf{X},\mathbf{X}) & = 0
		& & \gfour \mbox{\upshape-null},
			\\
		\gfour(\mathbf{X},\mathbf{X}) & > 0
		& & \gfour \mbox{\upshape-spacelike}.
	\end{align}
	\end{subequations}
	
	Hypersurfaces $\mathcal{H}$ of $\mathbb{R}^{1+3}$ 
	are classified as follows, where $\mathbf{H}$ denotes its $\gfour$-normal vectorfield:
	\begin{subequations}
	\begin{align}
		\gfour(\mathbf{H},\mathbf{H}) & < 0 \mbox{ at all points in } \mathcal{H}
		& & \gfour\mbox{\upshape-spacelike},
			\\
		\gfour(\mathbf{H},\mathbf{H}) & = 0 \mbox{ at all points in } \mathcal{H}
		& & \gfour \mbox{\upshape-null},
			\\
		\gfour(\mathbf{H},\mathbf{H}) & > 0 \mbox{ at all points in } \mathcal{H}
		& & \gfour \mbox{\upshape-timelike}.
	\end{align}
	\end{subequations}
	
	Moreover, if $\mathcal{S}$ is a co-dimension two submanifold of $\mathbb{R}^{1+3}$,
	we say that $\mathcal{S}$ is $\gfour$-spacelike if at each of its points, 
	all non-zero vectors $Y$ that are tangent to $\mathcal{S}$ verify $\gfour(Y,Y) > 0$.
\end{definition}

We close this subsection with a simple lemma that exhibits some basic properties of $\Transport$.

\begin{lemma}[Basic properties of $\Transport$]
\label{L:BASICPROPERTIESOFTRANSPORTVECTORFIELD}
The vectorfield $\Transport$ defined by \eqref{E:MATERIALVECTORVIELDRELATIVECTORECTANGULAR}
is $\gfour$-timelike and has $\gfour$-unit-length:
\begin{align} \label{E:TRANSPORTISUNITLENGTHANDTIMELIKE}
	\gfour(\Transport,\Transport)
	& = -1.
\end{align}

Moreover, relative to the Cartesian coordinates, we have
\begin{align} \label{E:TRANSPORTONEFORMIDENTITY}
	\Transport_{\alpha}
	& = -(dt)_{\alpha}
	=
	- \updelta_{\alpha}^0,
\end{align}
where $\updelta_{\alpha}^{\beta}$ is the Kronecker delta.
Thus, $\Transport$ is $\gfour$-orthogonal to $\Sigma_t$, i.e., $\gfour(\Transport,V) = 0$
for all vectorfields $\SigmatTan$ that are tangent to $\Sigma_t$.

Finally, we have $\Transport^0 = 1$, which implies in particular that $\Transport$ is future-directed\footnote{If
$\mathbf{X}$ is $\gfour$-timelike or $\gfour$-null, then we say that $\mathbf{X}$ is future-directed
if its Cartesian component $\mathbf{X}^0 = \mathbf{X} t$ is positive. \label{FN:FUTUREDIRECTED}}
\end{lemma}

\begin{proof}
	The lemma follows from straightforward calculations
	relative to the Cartesian coordinates, 
	based on 
	\eqref{E:MATERIALVECTORVIELDRELATIVECTORECTANGULAR}
	and
	\eqref{E:ACOUSTICALMETRIC}.
\end{proof}

\subsubsection{Null forms relative to $\gfour$}
\label{SSS:STANDARDNULLFORMS}
The statement of Theorem~\ref{T:GEOMETRICWAVETRANSPORTSYSTEM} refers to 
``null forms relative to $\gfour$,'' where $\gfour$ is the acoustical metric from Def.\,\ref{D:ACOUSTICALMETRIC}.
We will now briefly describe  
their importance in the study of shock waves without
symmetry assumptions, since this line of investigation is a primary motivating factor for the results of this paper.
By definition, null forms relative to $\gfour$ are linear combinations 
(with coefficients depending on $\LogDensity$, $v$, $\Ent$, $\vortrenormalized$, and $\GradEnt$ -- but \emph{not} their derivatives)
of the \emph{standard null forms relative to} $\gfour$, which we define in Def.\,\ref{D:NULLFORMS}.
Since Klainerman's foundational work \cite{sK1984}, it has been understood that (quadratic) null form nonlinearities
are ``weak'' in the sense that, at least in the setting of wave-like PDEs, 
they often allow for proofs of small-data global existence.
That is, null forms are quadratic terms exhibiting special cancellations that allow for small global solutions.
Null forms are also important in other contexts, such as the study of low regularity well-posedness \cites{sKmM1993,sKmM1994,sKmM1995}.
We clarify that in fact, there are different classes of null forms, and that Klainerman's notion of a null
form was adapted to the Minkowski metric, that is, to the geometry of special relativity.
Perhaps surprisingly, \emph{appropriately defined} null forms also play a crucial role in the theory of shock formation.
By ``appropriately defined,'' we mean that the null forms in the theory of shocks do not coincide\footnote{The antisymmetric null forms defined in 
\eqref{E:QALPHABETANULLFORM}
\emph{do} appear in Klainerman's framework \cite{sK1984}, but the null form
$\mathfrak{Q}^{(\gfour)}$ defined \eqref{E:Q0NULLFORM} does not; the analog of
$\mathfrak{Q}^{(\gfour)}$ in Klainerman's framework is $\mathfrak{Q}^{(m)}$,
where $m$ is the Minkowski metric.} 
with the null forms from Klainerman's framework \cite{sK1984}; see the next paragraph for further discussion.
In the context of shock waves, the appropriately defined null forms are also ``weak'' in the sense that, at least in certain solution regimes,
they are not strong enough to prevent the formation of shocks. That is, in compressible fluid mechanics,
shocks are singularities driven by
``strong'' derivative-quadratic Riccati-type terms,
and null forms relative to $\gfour$ do not contain any such Riccati-type interaction;
and see \cites{jLjS2016a,jS2018c,jS2019c} for further discussion.

Let us clarify the phrase ``appropriately defined'' from the previous paragraph.
In the study of shock formation without symmetry assumptions,
\emph{the precise nonlinear structure of the null forms is crucial}, in particular 
more important than it is in the context of
proving small-data global existence. That is, proofs of small-data global existence (say, for wave equations on $\mathbb{R}^{1+3}$) 
are typically stable under the addition of higher-order nonlinearities to the equations, 
such as perturbing a null form
by adding derivative-cubic nonlinearities. In contrast, in the context of shock waves, 
the needed ``special cancellations'' 
(which render the null form weak) become visible only when one decomposes the derivative-quadratic terms
relative to the characteristic surfaces
(i.e., acoustically null hypersurfaces in the context of compressible Euler flow), 
\emph{which, in quasilinear problems, are evolutionarily determined by the solution}. 
Thus, since in the context of compressible Euler flow, acoustically null hypersurfaces
are determined by $\gfour$, the relevant ``special cancellations'' are inextricably tied to $\gfour$.
For this reason, we speak of ``null forms relative to $\gfour$.''
Specifically, the null form $\mathfrak{Q}^{(\gfour)}$ defined in \eqref{E:Q0NULLFORM} 
explicitly depends on $\gfour$.
In the context of compressible fluid mechanics, the special cancellations
can be described as follows: if $\mathfrak{Q}(\partial \phi, \partial \widetilde{\phi})$ is a null form relative to $\gfour$
(in particular $\mathfrak{Q}$ is derivative-quadratic)
and $\underline{\mathcal{N}}$ is \emph{any} hypersurface that is null\footnote{Acoustically null hypersurfaces have $\gfour$-normals $\uLunit$
that are null, i.e., $\gfour(\uLunit,\uLunit) = 0$.} 
(i.e., characteristic) 
relative to $\gfour$, then there holds a decomposition of the schematic form
\begin{align} \label{E:NULLFORMRELATIVETOGSCHEMATICGOODPROPERTIES}
\mathfrak{Q}(\partial \phi, \partial \widetilde{\phi}) 
= 
\mathcal{T} \phi \cdot \pmb{\partial} \widetilde{\phi}
+
\mathcal{T}  \widetilde{\phi} \cdot \pmb{\partial} \phi,
\end{align}
where $\mathcal{T}$ denotes a derivative in a direction \emph{tangent} to  $\underline{\mathcal{N}}$
and $\pmb{\partial}$ denotes a generic derivative. 

The connection between the decomposition of null forms relative to $\gfour$ highlighted in \eqref{E:NULLFORMRELATIVETOGSCHEMATICGOODPROPERTIES}
and shock formation is as follows: 
in all solution regimes for the compressible Euler equations in which stable shock formation without symmetry assumptions has been shown,
one can construct a foliation of spacetime by acoustically null hypersurfaces $\underline{\mathcal{N}}$ such that
the $\mathcal{T}$-derivatives of the solution are 
\emph{much less singular} than its derivatives in directions transversal to $\underline{\mathcal{N}}$.
In fact, in all known results, it is precisely the transversal derivatives of the solution that blow up, as in the model case of
Burgers' equation. Thus, null forms are \emph{linear} in the terms that blow up.
Therefore, they are much less singular than generic quadratic terms 
$\pmb{\partial} \phi \cdot \pmb{\partial} \widetilde{\phi}$,
which are what drive the formation of the shock.
This explains why null forms represent ``weak'' nonlinearities in the context of shock formation,
and also clarifies why shock formation proofs are typically unstable under perturbing the equations
by generic higher-order nonlinearities, such as derivative-cubic ones:
if the blowup is driven by derivative-quadratic terms, then generic derivative-cubic (or higher-order) terms
would be expected to blow up \emph{at an even faster rate}, possibly radically altering the nature
of the singularity or even preventing it altogether.
The importance of null forms relative to $\gfour$ in the context of proving shock formation 
is further explained in \cites{gHsKjSwW2016,jS2016b,jLjS2016a,jS2018c,mDjS2019,jS2019c}.

We close our discussion by highlighting a connection
between the good properties of null forms relative to $\gfour$ in the context of shock formation
and the results of the present paper:
\begin{quote}
	In our main integral identities, in the case that the lateral boundary is acoustically null
	(let us refer to it as $\underline{\mathcal{N}}$),
	the integrals along $\underline{\mathcal{N}}$ \emph{involve only tangential derivatives} $\mathcal{T}$;
	see \eqref{E:ERRORTERMSSCHEMATICSTRUCTURENULLCASE} for the precise statement.
	This structure is crucial for controlling these error integrals in the context of shock-forming solutions,
	in analogy with the way that null forms relative to $\gfour$
	lead to ``weak'' (i.e., controllable) error terms.
\end{quote}

We now define the standard null forms relative to $\gfour$. 

\begin{definition}[Standard null forms relative to $\gfour$]
	\label{D:NULLFORMS}
	The standard null forms $\mathfrak{Q}^{(\gfour)}(\cdot,\cdot)$
	(relative to $\gfour$)
	and
	$\mathfrak{Q}_{(\alpha \beta)}(\cdot,\cdot)$,
	($0 \leq \alpha < \beta \leq 3$),
	act on pairs $(\phi,\widetilde{\phi})$
	of scalar-valued functions as follows:
	\begin{subequations}
		\begin{align}
		\mathfrak{Q}^{(\gfour)}(\partial \phi, \partial \widetilde{\phi})
		&:= (\gfour^{-1})^{\alpha \beta} \partial_{\alpha} \phi \partial_{\beta} \widetilde{\phi},
			\label{E:Q0NULLFORM} \\
		\mathfrak{Q}_{(\alpha \beta)}(\partial \phi, \partial \widetilde{\phi})
		& := \partial_{\alpha} \phi \partial_{\beta} \widetilde{\phi}
				-
			 \partial_{\alpha} \widetilde{\phi} \partial_{\beta} \phi.
			\label{E:QALPHABETANULLFORM}
		\end{align}
	\end{subequations}
\end{definition}

\subsubsection{Covariant wave operator}
\label{SSS:COVARIANTWAVEOPERATOR}
The statement of Theorem~\ref{T:GEOMETRICWAVETRANSPORTSYSTEM} refers to 
the covariant wave operator $\square_{\gfour}$
of the acoustical metric $\gfour$, defined below in Def.\,\ref{D:COVWAVEOP}.
The main significance of covariant wave operators
is that sophisticated geo-analytic technology has been developed for such operators.
It allows one to construct commutator and multiplier
vectorfields that are dynamically adapted to $\square_{\gfour}$.
It turns out that this technology is crucial for the study of
shocks without symmetry assumptions, in particular for deriving
energy estimates with controllable error terms both for the solution
and its higher derivatives; we refer readers to
\cites{gHsKjSwW2016,jS2016b,jLjS2016a,jS2018c,mDjS2019,jS2019c} for further discussion of these issues.
The technology is also important for the study of low-regularity solutions in the context of quasilinear problems;
we refer readers to \cite{mDcLgMjS2019} for further discussion.
In the present article, in our derivation of energy identities,
we will use only a basic version of the multiplier method, which we review in
Subsect.\,\ref{SS:VECTORFIELDMULTIPLIERMETHOD}.

\begin{definition}[Covariant wave operator]
\label{D:COVWAVEOP}
The covariant wave operator $\square_{\gfour}$ 
acts on scalar-valued functions $\phi$ according to the following formula:
\begin{align} \label{E:WAVEOPERATORARBITRARYCOORDINATES}
\square_{\gfour} \phi
	:= \frac{1}{\sqrt{|\mbox{\upshape det} \gfour|}}
	\partial_{\alpha}
	\left\lbrace
			\sqrt{|\mbox{\upshape det} \gfour|} (\gfour^{-1})^{\alpha \beta}
			\partial_{\beta} \phi
	\right\rbrace.
\end{align}
\end{definition}

\begin{remark}[Coordinate invariance of $\square_{\gfour}$]
	It is a standard fact that RHS~\eqref{E:WAVEOPERATORARBITRARYCOORDINATES} is coordinate invariant.
\end{remark}

\subsubsection{Some notation}
\label{SSS:STATESPACEVARIABLEDIFFERENTIATION}
We use the following notation in our statement of Theorem~\ref{T:GEOMETRICWAVETRANSPORTSYSTEM}.

\begin{notation}[\textbf{Differentiation with respect to state-space variables via semicolons}]
	\label{N:STATESPACEDIFFERENTIATION}
	If $f = f(\LogDensity,\Ent)$ is a scalar function, then
	we use the following notation to denote partial differentiation with respect to
	$\LogDensity$ and $\Ent$:
	$
	\displaystyle
	f_{;\LogDensity} 
	:= \frac{\partial f}{\partial \LogDensity}
	$
	and
	$
	\displaystyle
	f_{;\Ent} 
	:= \frac{\partial f}{\partial \Ent}
	$.
	Moreover, 
	$
	\displaystyle
	f_{;\LogDensity;\Ent} 
	:= \frac{\partial^2 f}{\partial \Ent \partial \LogDensity}
	$,
	and we use similar notation for other higher-order partial derivatives of $f$
	with respect to $\LogDensity$ and $\Ent$.
\end{notation}

\subsection{The geometric wave-transport-divergence-curl formulation of the compressible Euler equations}
Our main results fundamentally rely on the following formulation of the compressible Euler equations,
derived in \cite{jS2019c}.

\begin{theorem} \cite{jS2019c}*{Theorem 1; The geometric wave-transport-divergence-curl formulation of the compressible Euler equations}
	\label{T:GEOMETRICWAVETRANSPORTSYSTEM}
	Let $\bar{\varrho} > 0$ be any constant background density,\footnote{Recall that the definition \eqref{E:LOGDENSITY} of $\LogDensity$ depends on $\bar{\varrho}$.}
	and assume that $(\LogDensity,v^1,v^2,v^3,\Ent)$ is
	a $C^3$ solution\footnote{We have made the $C^3$ assumption only for convenience, 
	i.e., so that all of the quantities on the left- and right-hand sides of the equations of
	Theorem~\ref{T:GEOMETRICWAVETRANSPORTSYSTEM} are at least continuous. In applications, one can make
	sense of the equations and solutions in a distributional sense 
	under weaker regularity assumptions
	(for example, in suitable Sobolev spaces).} 
	to the compressible Euler equations 
	\eqref{E:TRANSPORTDENSRENORMALIZEDRELATIVECTORECTANGULAR}-\eqref{E:ENTROPYTRANSPORT}
	in three spatial dimensions under an arbitrary equation of state $p = p(\varrho,\Ent)$ with positive sound speed 
	$\Speed$ (see \eqref{E:SOUNDSPEED}).
	Let $\Transport$ be the material derivative vectorfield defined in \eqref{E:MATERIALVECTORVIELDRELATIVECTORECTANGULAR},
	let $\gfour$ be the acoustical metric from Def.\,\ref{D:ACOUSTICALMETRIC},
	and let $\VortVort$ and $\DivGradEnt$
	be the modified fluid variables from Def.\,\ref{D:RENORMALIZEDCURLOFSPECIFICVORTICITY}.
	Then the scalar-valued functions
	$\LogDensity$
	and
	$v^i$,
	$\vortrenormalized^i$,
$\Ent$,
$\GradEnt^i$,
$\Flatdiv \vortrenormalized$,
$\VortVort^i$,
$\DivGradEnt$,
and
$(\Flatcurl \GradEnt)^i$,
	($i=1,2,3$),
also solve the following equations,
	where $\upepsilon_{ijk}$ is the fully antisymmetric symbol normalized by $\upepsilon_{123}=1$ 
	and \textbf{the Cartesian component functions $v^i$ are 
	treated as scalar-valued functions
	under covariant differentiation on LHS~\eqref{E:VELOCITYWAVEEQUATION}}:

\medskip

\noindent \underline{\textbf{\upshape Covariant wave equations}}
	\begin{subequations}
	\begin{align}
		\square_{\gfour} v^i
		& = 
			- 
			\Speed^2 \exp(2 \LogDensity) \VortVort^i
			+ 
			\mathfrak{Q}_{(v)}^i
			+ 
			\mathfrak{L}_{(v)}^i,
			\label{E:VELOCITYWAVEEQUATION}	\\
	\square_{\gfour} \LogDensity
	& = 
		-
		\exp(\LogDensity) \frac{p_{;\Ent}}{\bar{\varrho}} \DivGradEnt
		+
		\mathfrak{Q}_{(\LogDensity)}
		+
		\mathfrak{L}_{(\LogDensity)},
			\label{E:RENORMALIZEDDENSITYWAVEEQUATION} 
				\\
	\square_{\gfour} \Ent
	& = 
		\Speed^2 \exp(2 \LogDensity)  \DivGradEnt
		+
		\mathfrak{L}_{(\Ent)}.
	\label{E:ENTROPYWAVEEQUATION}
\end{align}
\end{subequations}

\medskip

\noindent \underline{\textbf{\upshape Transport equations}}
\begin{subequations}
\begin{align}	\Transport \vortrenormalized^i
	& = \mathfrak{L}_{(\vortrenormalized)}^i,
		\label{E:RENORMALIZEDVORTICTITYTRANSPORTEQUATION}
		\\
	\Transport \Ent
	& = 0,	
		\label{E:ENTROPYTRANSPORTMAINSYSTEM}
			\\
	\Transport \GradEnt^i
	& = \mathfrak{L}_{(\GradEnt)}^i.
		\label{E:GRADENTROPYTRANSPORT}
	\end{align}
	\end{subequations}

\medskip	
	
\noindent \underline{\textbf{\upshape Transport-divergence-curl system for the specific vorticity}}
\begin{subequations}
\begin{align} \label{E:FLATDIVOFRENORMALIZEDVORTICITY}
	\Flatdiv \vortrenormalized
	& = 
		\mathfrak{L}_{(\Flatdiv \vortrenormalized)},
		\\
\Transport 
\VortVort^i
& = 
		- 
		2 \updelta_{jk} \upepsilon_{iab} \exp(-\LogDensity) (\partial_a v^j) \partial_b \vortrenormalized^k
		+
		\upepsilon_{ajk}
		\exp(-\LogDensity)
		(\partial_a v^i) 
		\partial_j \vortrenormalized^k
		\label{E:EVOLUTIONEQUATIONFLATCURLRENORMALIZEDVORTICITY} 
			\\
& \ \
		+ 
		\exp(-3 \LogDensity) \Speed^{-2} \frac{p_{;\Ent}}{\bar{\varrho}} 
		\left\lbrace
			(\Transport \GradEnt^a) \partial_a v^i
			-
			(\Transport v^i) \partial_a \GradEnt^a
		\right\rbrace
		\notag \\
	& \ \
		+
		\exp(-3 \LogDensity) \Speed^{-2} \frac{p_{;\Ent}}{\bar{\varrho}}  
		\left\lbrace
			(\Transport v^a) \partial_a \GradEnt^i
			- 
			( \Transport \GradEnt^i) \partial_a v^a
		\right\rbrace
		\notag \\
& \ \
	+
	\mathfrak{Q}_{(\VortVort)}^i
	+	
	\mathfrak{L}_{(\VortVort)}^i.
	\notag 
\end{align}	
\end{subequations}

\medskip

\noindent \underline{\textbf{\upshape Transport-divergence-curl system for the entropy gradient}}
\begin{subequations}
\begin{align} 	
\Transport \DivGradEnt
	& =  
			2 \exp(-2 \LogDensity) 
			\left\lbrace
				(\partial_a v^a) \partial_b \GradEnt^b
				-
				(\partial_a \GradEnt^b) \partial_b v^a
			\right\rbrace
			+
			\exp(-\LogDensity) \updelta_{ab} (\Flatcurl \vortrenormalized)^a \GradEnt^b
			\label{E:TRANSPORTFLATDIVGRADENT}
				\\
	& \ \
			+
			\mathfrak{Q}_{(\DivGradEnt)},
			\notag 
			\\
	(\Flatcurl \GradEnt)^i & = 0.
	\label{E:CURLGRADENTVANISHES}
\end{align}
\end{subequations}

	Above, 
	$\mathfrak{Q}_{(v)}^i$,
	$\mathfrak{Q}_{(\LogDensity)}$, 
	$\mathfrak{Q}_{(\VortVort)}^i$,
	and
	$\mathfrak{Q}_{(\DivGradEnt)}$
	are\footnote{The terms on the first four lines of
	RHS~\eqref{E:EVOLUTIONEQUATIONFLATCURLRENORMALIZEDVORTICITY} 
	and the first product
	on RHS~\eqref{E:TRANSPORTFLATDIVGRADENT} are also null forms relative to $\gfour$.
	We have explicitly displayed these null forms since in applications, they 
	are more difficult to treat than 
	$\mathfrak{Q}_{(v)}^i$,
	$\mathfrak{Q}_{(\LogDensity)}$, 
	$\mathfrak{Q}_{(\VortVort)}^i$,
	and
	$\mathfrak{Q}_{(\DivGradEnt)}$;
	from the point of view of regularity,
	the explicitly displayed null forms
	need to be treated with 
	div-curl-transport identities.
	\label{FN:MORENULLFORMS}} 
	the \textbf{null forms relative to} $\gfour$ defined by
	\begin{subequations}
		\begin{align}
		\mathfrak{Q}_{(v)}^i	
		& := 	-
					\left\lbrace
						1
						+
						\Speed^{-1} \Speed_{;\LogDensity}
					\right\rbrace
					(\gfour^{-1})^{\alpha \beta} (\partial_{\alpha} \LogDensity) \partial_{\beta} v^i,
			\label{E:VELOCITYNULLFORM} \\
		\mathfrak{Q}_{(\LogDensity)}
		& := 
		- 
		3 \Speed^{-1} \Speed_{;\LogDensity} 
		(\gfour^{-1})^{\alpha \beta} (\partial_{\alpha} \LogDensity) \partial_{\beta} \LogDensity
		+ 
		\left\lbrace
			(\partial_a v^a) \partial_b v^b
			-
			(\partial_a v^b) \partial_b v^a
		\right\rbrace,
			\label{E:DENSITYNULLFORM}
				\\
	\mathfrak{Q}_{(\VortVort)}^i
	& :=
		\exp(-3 \LogDensity) \Speed^{-2} \frac{p_{;\Ent}}{\bar{\varrho}}  \GradEnt^i
		\left\lbrace
			(\partial_a v^b) \partial_b v^a
			-
			(\partial_a v^a) \partial_b v^b
		\right\rbrace
			\label{E:RENORMALIZEDVORTICITYCURLNULLFORM} \\
& \ \
		+ 
		\exp(-3 \LogDensity) \Speed^{-2} \frac{p_{;\Ent}}{\bar{\varrho}}
		\GradEnt^b 
		\left\lbrace
			(\partial_a v^a) \partial_b v^i 
			- 
			(\partial_a v^i) \partial_b v^a 
		\right\rbrace
		\notag \\
& \ \
		+ 
		2 \exp(-3 \LogDensity) \Speed^{-2} \frac{p_{;\Ent}}{\bar{\varrho}}
		\GradEnt^a
		\left\lbrace
			(\partial_a \LogDensity)  \Transport v^i
			 - 
		  (\partial_a v^i) \Transport \LogDensity
		\right\rbrace
		\notag \\
	&  \ \
			+ 
			2 \exp(-3 \LogDensity) \Speed^{-3} \Speed_{;\LogDensity} \frac{p_{;\Ent}}{\bar{\varrho}}
			\GradEnt^a 
			\left\lbrace
				(\partial_a \LogDensity)  \Transport v^i
				- 
				(\partial_a v^i) \Transport \LogDensity
			\right\rbrace
		\notag \\
	& \ \
		+ 
		\exp(-3 \LogDensity) \Speed^{-2} \frac{p_{;\Ent;\LogDensity}}{\bar{\varrho}}
		\GradEnt^a 
		\left\lbrace
			(\partial_a v^i)
			\Transport \LogDensity 
			- 
			(\partial_a \LogDensity) \Transport v^i
		\right\rbrace
		\notag \\
	& \ \
		+
		\exp(-3 \LogDensity) \Speed^{-2} \frac{p_{;\Ent;\LogDensity}}{\bar{\varrho}} \GradEnt^i
		\left\lbrace
			(\Transport v^a) \partial_a \LogDensity
			-
			 (\Transport \LogDensity) \partial_a v^a
		\right\rbrace
		\notag \\
	& \ \
		+
		2 \exp(-3 \LogDensity) \Speed^{-2} \frac{p_{;\Ent}}{\bar{\varrho}} \GradEnt^i
		\left\lbrace
			(\Transport \LogDensity) \partial_a v^a 
			- 
			(\Transport v^a) \partial_a \LogDensity
		\right\rbrace
		\notag \\
	& \ \
		 	+
			2 \exp(-3 \LogDensity) \Speed^{-3} \Speed_{;\LogDensity} \frac{p_{;\Ent}}{\bar{\varrho}} \GradEnt^i
			\left\lbrace
				(\Transport \LogDensity) \partial_a v^a
		 		- 
		 		(\Transport v^a) \partial_a \LogDensity
		 	\right\rbrace,
		 		\notag
		 		\\
\label{E:DIVENTROPYGRADIENTNULLFORM}
	\mathfrak{Q}_{(\DivGradEnt)}
	& :=
		2 \exp(-2 \LogDensity) 
		\GradEnt^a 
		\left\lbrace
			(\partial_a v^b) \partial_b \LogDensity
			-
			(\partial_a \LogDensity)
			\partial_b v^b 
		\right\rbrace.
\end{align}
\end{subequations}
	In addition, the terms
	$\mathfrak{L}_{(v)}^i$,
	$\mathfrak{L}_{(\LogDensity)}$,
	$\mathfrak{L}_{(\Ent)}$,
	$\mathfrak{L}_{(\vortrenormalized)}^i$,
	$\mathfrak{L}_{(\GradEnt)}^i$,
		$\mathfrak{L}_{(\Flatdiv \vortrenormalized)}$,
		and
	$\mathfrak{L}_{(\VortVort)}^i$,
	which are at most linear in the derivatives of the unknowns, are defined as follows:
	\begin{subequations}
	\begin{align} \label{E:VELOCITYILINEARORBETTER} 
		\mathfrak{L}_{(v)}^i
		& := 
		2 \exp(\LogDensity) \upepsilon_{iab} (\Transport v^a) \vortrenormalized^b
		-
		\frac{p_{;\Ent}}{\bar{\varrho}} \upepsilon_{iab} \vortrenormalized^a \GradEnt^b
			\\
	& \ \
		- 
		\frac{1}{2} \exp(-\LogDensity) \frac{p_{;\LogDensity;\Ent}}{\bar{\varrho}} \GradEnt^a \partial_a v^i
			\notag \\
		& \ \
		- 
		2 \exp(-\LogDensity) \Speed^{-1} \Speed_{;\LogDensity} \frac{p_{;\Ent}}{\bar{\varrho}} 
		(\Transport \LogDensity) \GradEnt^i
		+
		\exp(-\LogDensity) \frac{p_{;\Ent;\LogDensity}}{\bar{\varrho}} (\Transport \LogDensity) \GradEnt^i,
			\notag  \\
		\mathfrak{L}_{(\LogDensity)}
		& :=
		-
		\frac{5}{2} \exp(-\LogDensity) \frac{p_{;\Ent;\LogDensity}}{\bar{\varrho}} \GradEnt^a \partial_a \LogDensity 
		-
		\exp(-\LogDensity) \frac{p_{;\Ent;\Ent}}{\bar{\varrho}} \updelta_{ab} \GradEnt^a \GradEnt^b,
		 \label{E:DENSITYLINEARORBETTER}
		 	\\
		\mathfrak{L}_{(\Ent)}
		& := 
			\Speed^2  \GradEnt^a \partial_a \LogDensity
			- 
			\Speed \Speed_{;\LogDensity} \GradEnt^a \partial_a \LogDensity
			- 
			\Speed \Speed_{;\Ent} \updelta_{ab} \GradEnt^a \GradEnt^b,
			\label{E:ENTROPYLINEARORBETTER} \\
		\mathfrak{L}_{(\vortrenormalized)}^i
		& := 
		\vortrenormalized^a \partial_a v^i
		-
		\exp(-2 \LogDensity) \Speed^{-2} \frac{p_{;\Ent}}{\bar{\varrho}} \upepsilon_{iab} (\Transport v^a) \GradEnt^b,
		\label{E:SPECIFICVORTICITYLINEARORBETTER}
			\\
		\mathfrak{L}_{(\GradEnt)}^i
		& :=
			- \GradEnt^a \partial_a v^i
			+ 
			\upepsilon_{iab} \exp(\LogDensity) \vortrenormalized^a \GradEnt^b,
			\label{E:ENTROPYGRADIENTLINEARORBETTER}
				\\
	\mathfrak{L}_{(\Flatdiv \vortrenormalized)}
		& := - \vortrenormalized^a \partial_a \LogDensity,
		\label{E:RENORMALIZEDVORTICITYDIVLINEARORBETTER} \\
	\mathfrak{L}_{(\VortVort)}^i	 		
	& :=
		 		2 \exp(-3 \LogDensity) \Speed^{-3} \Speed_{;\Ent} \frac{p_{;\Ent}}{\bar{\varrho}} 
				(\Transport v^i) \updelta_{ab} \GradEnt^a \GradEnt^b
				\label{E:RENORMALIZEDVORTICITYCURLLINEARORBETTER} \\
		& \ \
			-
				2 \exp(-3 \LogDensity) \Speed^{-3} \Speed_{;\Ent} \frac{p_{;\Ent}}{\bar{\varrho}} 
				\updelta_{ab} \GradEnt^a (\Transport v^b) \GradEnt^i
				\notag \\
		& \ \
			+ 
			\exp(-3 \LogDensity) \Speed^{-2} \frac{p_{;\Ent;\Ent}}{\bar{\varrho}} \updelta_{ab} (\Transport v^a) \GradEnt^b \GradEnt^i
			\notag \\
		& \ \ 
			- 
			\exp(-3 \LogDensity) \Speed^{-2} \frac{p_{;\Ent;\Ent}}{\bar{\varrho}} (\Transport v^i) \updelta_{ab} \GradEnt^a \GradEnt^b.
			\notag 
		\end{align}
	\end{subequations}
\end{theorem}

\section{The spacetime regions and their topology and geometry}
\label{S:SPACETIMEDOMAINS}
In this section, we state our assumptions on the spacetime region $\mathcal{M}$ on which we will derive our main integral identities,
as well as the acoustical time function $\Timefunction$ that foliates $\mathcal{M}$. We then define a collection of geometric
tensors associated to $\mathcal{M}$ and exhibit their basic properties.

\subsection{The spacetime region $\mathcal{M}$, the acoustical time function $\Timefunction$, and related constructions}
\label{SS:DOMAINANDTIMEFUNCTIONETC}
Until Sect.\,\ref{S:DOUBLENULL},
our results concern compressible Euler solutions 
on subsets of spacetime, denoted by $\mathcal{M}$ and depicted in Fig.\,\ref{F:SPACETIMEDOMAIN}.
We assume that $\mathcal{M}$ is a compact, connected manifold with corners;
the purpose of the latter assumption is that it allows us to use Stokes' theorem 
(more precisely, in the context of the present paper, the divergence theorem) on $\mathcal{M}$.
We assume that $\partial \mathcal{M}$ (i.e., the boundary of $\mathcal{M}$, viewed as a subset of $\mathbb{R}^{1+3}$) 
can be decomposed\footnote{The union \eqref{E:SPLITTINGOFBOUNDARYOFREGION} is not disjoint since, as we describe below, our assumptions
imply that $\underline{\mathcal{H}}$ intersects $\widetilde{\Sigma}_{\Timefunction}$ in two-dimensional submanifolds.} 
as
\begin{align} \label{E:SPLITTINGOFBOUNDARYOFREGION}
	\partial \mathcal{M}
	& = \widetilde{\Sigma}_0 
			\cup
			\widetilde{\Sigma}_T
			\cup
			\underline{\mathcal{H}},
\end{align}
where $\widetilde{\Sigma}_T$ is the top boundary of $\mathcal{M}$,
$\widetilde{\Sigma}_0$ is the bottom boundary of $\mathcal{M}$, 
$\underline{\mathcal{H}}$ is the lateral boundary of $\mathcal{M}$,
and just below, we explain the meaning of the subscripts on the symbol ``$\widetilde{\Sigma}$.''
We assume that for some $T > 0$, $\mathcal{M}$ 
is foliated by $\gfour$-spacelike hypersurface portions 
$\widetilde{\Sigma}_{\Timefunction}$ for $\Timefunction \in [0,T]$. 
More precisely, we assume that $\Timefunction$ 
is a smooth \emph{acoustical time function} (not necessarily equal to Cartesian time) on an open, connected subset $\mathscr{O}$ of $\mathbb{R}^{1+3}$ containing $\mathcal{M}$. By ``acoustical time function,'' we mean a time function in the sense of Lorentzian geometry 
(see \cite{rW1984}*{Section~8.2} for background material on time functions),
where the Lorentzian metric is the acoustical metric $\gfour$.
Specifically, we assume that on $\mathscr{O}$, $\Timefunction$ has non-vanishing gradient
and that $\Dfour \Timefunction$ is past-directed,\footnote{By ``past-directed,'' we mean that $\Transport \Timefunction > 0$, i.e., $\Timefunction$ increases
along the integral curves of $\Transport$, like Cartesian time does (since $\Transport t = 1$). 
By \eqref{E:TRANSPORTONEFORMIDENTITY}, 
this is equivalent to the assumption that the opposite of the $\gfour$-dual of the gradient of $\Timefunction$, namely 
$- (\gfour^{-1})^{\alpha \beta} \partial_{\beta} \Timefunction$,
is a future-directed vectorfield in the sense of Footnote~\ref{FN:FUTUREDIRECTED}.
\label{FN:TIMEFUNCTIONPASTDIRECTEDGRADIENT}}  
where $\Dfour \Timefunction$ denotes the gradient one-form of $\Timefunction$,
and\footnote{The assumption $(\gfour^{-1})^{\alpha \beta} \partial_{\alpha} \Timefunction \partial_{\beta} \Timefunction < 0$ implies that $\widetilde{\Sigma}_{\Timefunction}$
is $\gfour$-spacelike in the sense of Def.\,\ref{D:SPACELIKETIMELIKENULL}.} 
$(\gfour^{-1})^{\alpha \beta} \partial_{\alpha} \Timefunction \partial_{\beta} \Timefunction < 0$.
We then set
\begin{align} \label{E:WIDETILDESIGMAISTHEINTERSECTIOFDOMAINWITHLEVELSETOFTIMEFUNCTION}
	\widetilde{\Sigma}_{\Timefunction'}
	& := \mathcal{M} \cap \lbrace (t,x^1,x^2,x^3) \in \mathscr{O} \ | \ \Timefunction(t,x^1,x^2,x^3) = \Timefunction' \rbrace.
\end{align}
Note that our assumptions imply that\footnote{Throughout, we abuse notation by using the symbol ``$\Timefunction$'' to denote both the acoustical time function and the values that it takes on;
the precise meaning of the symbol will be clear from context.} 
\begin{align} \label{E:DOMAINFOLIATEDBYSPACELIKEHYPERSURFACES}
	\mathcal{M}
	& = \cup_{\Timefunction \in [0,T]} \widetilde{\Sigma}_{\Timefunction}.
\end{align}
We assume that $\underline{\mathcal{H}}$ is the intersection of
$\mathcal{M}$ with a smooth, three-dimensional embedded submanifold of $\mathscr{O}$,
and that $\underline{\mathcal{H}}$ is $\gfour$-spacelike at all of its points or $\gfour$-null at all of its points
(in the sense of Def.\,\ref{D:SPACELIKETIMELIKENULL}).
This implies, in particular, that
$\underline{\mathcal{H}}$ is transversal\footnote{That is, at every point in $\underline{\mathcal{H}} \cap \widetilde{\Sigma}_{\Timefunction}$, 
the normal of $\widetilde{\Sigma}_{\Timefunction}$ is $\gfour$-timelike
and therefore cannot be parallel to the normal of $\underline{\mathcal{H}}$.} 
to the level sets of the acoustical time function $\Timefunction$. 
Finally, we assume that for $\Timefunction \in [0,T]$, $\widetilde{\Sigma}_{\Timefunction}$ 
is diffeomorphic to the closed unit ball in $\mathbb{R}^3$.
This implies, in particular, that the boundary of $\widetilde{\Sigma}_{\Timefunction}$, viewed as a subset
of the $\Timefunction$-level set of the acoustical time function, 
and which we denote by $\partial \widetilde{\Sigma}_{\Timefunction}$,
is diffeomorphic to $\mathbb{S}^2$. 

The following subsets of spacetime, associated to $\mathcal{M}$, 
will play a fundamental role in the ensuing discussion.
\begin{definition}[Subsets of spacetime]
For $0 \leq \Timefunction \leq T$, we define
\begin{subequations}
\begin{align} 
	\mathcal{S}_{\Timefunction}
	& := \underline{\mathcal{H}} \cap \widetilde{\Sigma}_{\Timefunction},
		\label{E:SPHEREDEF} \\
	\underline{\mathcal{H}}_{\Timefunction}
	& := \underline{\mathcal{H}} \cap \mathcal{M}_{\Timefunction},
	\label{E:LATERALHYPERSURFACEPORTION}
		\\
\mathcal{M}_{\Timefunction}
	& := \cup_{\Timefunction'\in [0,\Timefunction]} \widetilde{\Sigma}_{\Timefunction'}.
	\label{E:TIMETRUNCATEDDOMAIN}
\end{align}
\end{subequations}
\end{definition}
From the above definitions, it follows that $\mathcal{M} = \mathcal{M}_T$;
see Fig.\,\ref{F:SPACETIMEDOMAIN}.
Moreover, we note that $\underline{\mathcal{H}} = \underline{\mathcal{H}}_T$
and that for $\Timefunction \in [0,T]$, we have
\begin{align} \label{E:LATERALNULLHYPERSURFACEFOLIATEDBYSPHERES}
	\underline{\mathcal{H}}_{\Timefunction} = \cup_{\Timefunction' \in [0,\Timefunction]} \mathcal{S}_{\Timefunction'}.
\end{align}
We refer to either of $\underline{\mathcal{H}}$ or
$\underline{\mathcal{H}}_{\Timefunction}$ as the ``lateral hypersurface''
or the ``lateral boundary'' of $\mathcal{M}$.

From the above assumptions, in the language of manifolds with corners,
it follows that points in $\mathcal{S}_T \cup \mathcal{S}_0$ are index\footnote{By definition, 
a point of index $k$ is contained in a subset $\mathcal{D}_k$ of $\mathcal{M}$ such that $\mathcal{D}_k$ is
diffeomorphic to a neighborhood of the origin in $[0,\infty)^k \times \mathbb{R}^{4-k}$; the case $k = 0$
corresponds to the standard notion of a differentiable manifold.} $2$,
that points in $\partial \mathcal{M} \backslash (\mathcal{S}_T \cup \mathcal{S}_0)$
are index $1$, and that the remaining points in $\mathcal{M}$ (which belong to its interior)
are index $0$.

We also note that since $\underline{\mathcal{H}}$ is transversal to the level
sets of the acoustical time function, the following identity holds for $\Timefunction \in [0,T]$:
\begin{align} \label{E:BOUNDARYOFTILDESIGMAISSPHERE}
	\partial \widetilde{\Sigma}_{\Timefunction}
	& = \mathcal{S}_{\Timefunction}.
\end{align}
Thus, in view of our assumption that $\widetilde{\Sigma}_{\Timefunction}$ is diffeomorphic to the closed unit ball in $\mathbb{R}^3$,
it follows that $\mathcal{S}_{\Timefunction}$ is diffeomorphic to $\mathbb{S}^2$.

Finally, we note that the above assumptions imply that for $\Timefunction \in [0,T]$,
we have
\begin{align} \label{E:BOUNDARYOFTRUNCATEDREGION}
	\partial \mathcal{M}_{\Timefunction}
	& := \mbox{\upshape the boundary of } \mathcal{M}_{\Timefunction} \mbox{ \upshape in } \mathbb{R}^{1+3}
		= \widetilde{\Sigma}_0 \cup \widetilde{\Sigma}_{\Timefunction} \cup \underline{\mathcal{H}}_{\Timefunction}.
	\end{align}

See Example~\ref{EX:EXAMPLEOFDOMAIN} below for a canonical example of a family of spacetime regions that satisfy our assumptions:
truncated backwards sound cones.

\subsection{\texorpdfstring{The vectorfields $\tophypnorm$, $\modtophypnorm$, $\sidehypnorm$, $\spherenormal$, $\gen$, and $\modgen$,
and the scalar function $\lapsemodgen$}{Geometric vectorfields and conventions in the null case}}
In this subsection, we define a collection of geometric vectorfields associated to $\mathcal{M}$. 
We also introduce alternate notation that we often use when the lateral boundary $\underline{\mathcal{H}}$ is $\gfour$-null.

\begin{definition}[The vectorfields $\tophypnorm$, $\modtophypnorm$, $\sidehypnorm$, $\spherenormal$, $\gen$, and $\modgen$,
and the scalar function $\lapsemodgen$]
\label{D:HYPNORMANDSPHEREFORMDEFS}
We define $\tophypnorm$ to be the vectorfield that is $\gfour$-orthogonal to $\widetilde{\Sigma}_{\Timefunction}$ and
normalized by 
\begin{align} \label{E:FUTURENORMALTOTOPHYPERSURFACE}
	\tophypnorm t & = 1.
\end{align} 
Note that $\tophypnorm$ is $\gfour$-timelike since $\widetilde{\Sigma}_{\Timefunction}$ is $\gfour$-spacelike by assumption.

We next define $\modtophypnorm$ to be the vectorfield that is $\gfour$-orthogonal to $\widetilde{\Sigma}_{\Timefunction}$
(i.e., parallel to $\tophypnorm$) and
normalized by
\begin{align} \label{E:TOPHYPNORMNORMALIZEDAGAINSTTIMEFUNCTION}
	\modtophypnorm \Timefunction
	& = 1,
\end{align}
where $\Timefunction$ is the acoustical time function from the beginning of Sect.\,\ref{S:SPACETIMEDOMAINS}.

We define $\sidehypnorm$ to be the $\gfour$-normal to $\underline{\mathcal{H}}$,
normalized by 
\begin{align} \label{E:FUTURENORMALTOHYPERSURFACE}
	\sidehypnorm t & = 1.
\end{align} 

We define $\spherenormal$ to be the $\gfour$-unit outer normal
	to $\mathcal{S}_{\Timefunction}$ in $\widetilde{\Sigma}_{\Timefunction}$. In particular, 
	$\spherenormal$ is tangent to $\widetilde{\Sigma}_{\Timefunction}$, $\gfour$-normal to $\mathcal{S}_{\Timefunction}$,
	and satisfies
	\begin{align}  \label{E:SPHERENORMALISUNITLENGTH}
		\gfour(\spherenormal,\spherenormal)
		& = 1.
	\end{align}
	
	We define $\gen$ to be the vectorfield that is tangent to $\underline{\mathcal{H}}$,
	$\gfour$-orthogonal to $\mathcal{S}_{\Timefunction}$,
	and normalized by 
	\begin{align} \label{E:GENERATOROFHYPERSURFACE}
		\gen t = 1.
	\end{align} 
	
	Finally, we define\footnote{Since $\gen$ and the $\gfour$-dual of $- \Dfour \Timefunction$ are both future-directed in the sense of Footnote~\ref{FN:FUTUREDIRECTED}
	(the former by \eqref{E:GENERATOROFHYPERSURFACE} and the latter by assumption),
	it follows that $\gen \Timefunction > 0$.}
	the scalar function
	\begin{align} \label{E:LAPSEFORMODGENCORRESPONDINGTOTIMEFUNCTION}
		\lapsemodgen
		& := \frac{1}{\gen \Timefunction}
	\end{align}
	and the vectorfield
	\begin{align} \label{E:NORMALIZEDAGAINSTTIMEFUNCTIONGENERATOROFHYPERSURFACE}
		\modgen
		& := \lapsemodgen \gen. 
	\end{align}
	
\end{definition}

Note that by \eqref{E:TRANSPORTONEFORMIDENTITY}, 
\eqref{E:FUTURENORMALTOTOPHYPERSURFACE} is equivalent to
\begin{align} \label{E:EQUIVALENTFUTURENORMALTOTOPHYPERSURFACE}
	\gfour(\tophypnorm,\Transport) & = -1,
\end{align}
\eqref{E:FUTURENORMALTOHYPERSURFACE} is equivalent to
\begin{align} \label{E:EQUIVALENTFUTURENORMALTOHYPERSURFACE}
	\gfour(\sidehypnorm,\Transport) & = -1,
\end{align} 
and \eqref{E:GENERATOROFHYPERSURFACE} is equivalent to
\begin{align} \label{E:EQUIVALENTGENERATOROFHYPERSURFACE}
	\gfour(\gen,\Transport) = -1.
\end{align} 

Note also that
	\eqref{E:LAPSEFORMODGENCORRESPONDINGTOTIMEFUNCTION}-\eqref{E:NORMALIZEDAGAINSTTIMEFUNCTIONGENERATOROFHYPERSURFACE}
imply that
\begin{align} \label{E:SIDEGENERATORNORMALIZEDAGAINSTTIMEFUNCTION}
	\modgen \Timefunction
	& = 1.
\end{align}

We also note that since the acoustical time function $\tau$ is constant along each hypersurface $\widetilde{\Sigma}_{\Timefunction}$, 
it follows from \eqref{E:TRANSPORTONEFORMIDENTITY} and \eqref{E:FUTURENORMALTOTOPHYPERSURFACE} that
\begin{align} \label{E:TOPHYPNORMINTERMSOFGRADIENTOFTIMEFUNCTION}
	\tophypnorm^{\alpha}
	& = \frac{(\gfour^{-1})^{\alpha \beta} \partial_{\beta} \Timefunction}{(\gfour^{-1})^{\kappa \lambda} \partial_{\kappa} t \partial_{\lambda} \Timefunction}
	=
	\frac{- (\gfour^{-1})^{\alpha \beta} \partial_{\beta} \Timefunction}{\Transport \Timefunction}.
\end{align}

\begin{remark}
	For setting up the geometry, 
	we find it convenient to normalize various vectorfields with respect to
	Cartesian time $t$, as we did in 
	\eqref{E:FUTURENORMALTOTOPHYPERSURFACE},
	\eqref{E:FUTURENORMALTOHYPERSURFACE},
	and \eqref{E:GENERATOROFHYPERSURFACE}. Nonetheless, our geometric identities will be able 
	to accommodate foliations of spacetime regions with respect to arbitrary smooth acoustical time functions $\Timefunction$.
\end{remark}

\begin{convention}[$\underline{\mathcal{N}}$ vs.\, $\underline{\mathcal{H}}$ and $\uLunit$ vs.\, $\sidehypnorm$]	
	\label{C:NULLCASE}
	If the lateral hypersurface $\underline{\mathcal{H}}$ is $\gfour$-null, we often
	refer to this as the ``null case.''
	In the null case, we often use the
	alternate notation $\underline{\mathcal{N}}$ in place of $\underline{\mathcal{H}}$, 
	$\underline{\mathcal{N}}_{\Timefunction}$ in place of $\underline{\mathcal{H}}_{\Timefunction}$,
	etc. Moreover, in the null case, we often use the notation
	\begin{align} \label{E:ULNIT}
		\uLunit
	\end{align}
	in place of $\sidehypnorm$ since in Lorentzian geometry, $\uLunit$ is common notation for an ``ingoing'' null vector
	(where $\uLunit$ will be ``ingoing'' thanks to our assumptions in Subsect.\,\ref{SS:ASSUMPTIONSONSPACETIMEREGION}).
\end{convention}

\subsection{The positivity of $\uposinnerproduct$, $\seconduposinnerproduct$, and $\lapsemodgen$ and some consequences}
\label{SS:ASSUMPTIONSONSPACETIMEREGION}
In this subsection, we introduce the scalar functions $\uposinnerproduct$ and $\seconduposinnerproduct$,
whose assumed positivity, together with the positivity of $\lapsemodgen$ (see \eqref{E:POSITIVITYOFLAPSEMODGEN}), 
is crucial for all of our main results. We also discuss some geometric and topological consequences of 
the positivity.

Specifically, we make the following assumptions on $\mathcal{M}$.
\begin{itemize}
	\item If $\underline{\mathcal{H}}$ is either $\gfour$-spacelike or $\gfour$-null, 
		then we assume that the scalar functions 
	$\uposinnerproduct$
	and
	$\seconduposinnerproduct$
	are such that
		\begin{subequations}
		\begin{align} \label{E:INGOINGCONDITION}
			\gfour(\spherenormal,\sidehypnorm) 
			& : = - \uposinnerproduct,
			&
			&
			\uposinnerproduct
			> 0,
				\\
			\gfour(\gen,\tophypnorm) 
			& := - \seconduposinnerproduct,
			&
			&
			\seconduposinnerproduct
			> 0.
			\label{E:SECONDINGOINGCONDITION}
		\end{align}
		\end{subequations}
		Note that \eqref{E:EQUIVALENTFUTURENORMALTOHYPERSURFACE} and the fact that 
		$\Transport \Timefunction > 0$ (see Footnote~\ref{FN:TIMEFUNCTIONPASTDIRECTEDGRADIENT}) 
		together imply that $\sidehypnorm \Timefunction > 0$,  
		i.e., along each sphere $\mathcal{S}_{\Timefunction} = \underline{\mathcal{H}} \cap \widetilde{\Sigma}_{\Timefunction}$,
		$\sidehypnorm$
		(which by definition is $\gfour$-orthogonal to $\underline{\mathcal{H}}$) 
		points to the future of $\widetilde{\Sigma}_{\Timefunction}$.
		The assumption \eqref{E:INGOINGCONDITION} is tantamount to the assumption that
		$\underline{\mathcal{H}}$ is in fact \emph{ingoing} to the future in the sense that when $\underline{\mathcal{H}}$ is $\gfour$-spacelike,
		$\sidehypnorm$ 
		points outward to $\mathcal{M}$. 
		To explain why $\sidehypnorm$ is outward-pointing when $\underline{\mathcal{H}}$ is $\gfour$-spacelike, we first note that
		at a given point $q \in \underline{\mathcal{H}}$, 
		the set of vectors belonging to the tangent space of spacetime at $q$ 
		that are \emph{not} tangent to $\underline{\mathcal{H}}$
		is equal to the disjoint union of two connected components:
		$\lbrace \mathbf{X} \ | \ \gfour(\mathbf{X},\sidehypnorm) < 0 \rbrace \cup \lbrace \mathbf{X} \ | \ \gfour(\mathbf{X},\sidehypnorm) > 0 \rbrace$.
		One of these components is the set of inward-pointing vectors to $\mathcal{M}$ at $q$,
		and the other is the set of outward-pointing vectors to $\mathcal{M}$ at $q$.
		Since $\sidehypnorm$ is $\gfour$-timelike by assumption (and thus $\gfour(\sidehypnorm,\sidehypnorm) < 0$), \eqref{E:INGOINGCONDITION} guarantees that
		$\sidehypnorm$ and $\spherenormal$ belong to the same connected component.
		Thus, since along each sphere $\mathcal{S}_{\Timefunction}$,
		$\spherenormal$ points outward to $\mathcal{M}$ by assumption, we conclude that when $\underline{\mathcal{H}}$ is $\gfour$-spacelike,
		$\sidehypnorm$ also points outward to $\mathcal{M}$.
		Similarly, since $\tophypnorm$ points outward to $\mathcal{M}_{\Timefunction}$ along each point in $\widetilde{\Sigma}_{\Timefunction}$,
		\eqref{E:SECONDINGOINGCONDITION}
		is tantamount to the assumption that along each sphere $\mathcal{S}_{\Timefunction} \subset \widetilde{\Sigma}_{\Timefunction}$,
		the generator $\gen$ of $\underline{\mathcal{H}}$ (which is tangent to $\underline{\mathcal{H}}$)
		also points outward to $\mathcal{M}_{\Timefunction}$. 
\end{itemize}

We next observe that 
\eqref{E:LAPSEFORMODGENCORRESPONDINGTOTIMEFUNCTION},
\eqref{E:TOPHYPNORMINTERMSOFGRADIENTOFTIMEFUNCTION},
\eqref{E:SECONDINGOINGCONDITION},
and the assumption $\Transport \Timefunction > 0$ (see Footnote~\ref{FN:TIMEFUNCTIONPASTDIRECTEDGRADIENT})
imply that
\begin{align} \label{E:POSITIVITYOFLAPSEMODGEN}
		\lapsemodgen
		& =
		\frac{1}
		{\gen \Timefunction}
		= \frac{1}{- \gfour(\gen,\tophypnorm) \Transport \Timefunction}
		= \frac{1}{\seconduposinnerproduct \Transport \Timefunction}
		> 0.
\end{align}

From the perspective of analysis, the positivity of $\lapsemodgen$, $\uposinnerproduct$, and $\seconduposinnerproduct$
is important for the coerciveness of some key terms in our integral identities;
see, for example, the $\mathcal{S}_{\Timefunction}$ integrals 
on LHSs~\eqref{E:SPACETIMEREMARKABLEIDENTITYSPECIFICVORTICITY} and \eqref{E:SPACETIMEREMARKABLEIDENTITYENTROPYGRADIENT}.

\begin{example}[Canonical examples: truncated backwards sound cones]
	\label{EX:EXAMPLEOFDOMAIN}
	Here we provide canonical examples of spacetime regions $\mathcal{M}$ with $\gfour$-null lateral boundaries
	to which the results of this paper apply; 
	such regions arise in the study of shock formation.
	For convenience, we consider the case in which the acoustical time function $\Timefunction$ is equal to the Cartesian time function $t$.
	Let $\mathcal{S}_0$ be any embedded two-dimensional submanifold of $\Sigma_0$ that is diffeomorphic to $\mathbb{S}^2$
	(for example, $\mathcal{S}_0$ could be equal to $\mathbb{S}^2 \subset \Sigma_0 \simeq \mathbb{R}^3$). 
	At each $q \in \mathcal{S}_0$, there is a unique vector
	$\underline{\ell}_q \in T_q \mathcal{S}_0$ that is $\gfour$-null,
	future-directed,
	$\gfour$-orthogonal to $\mathcal{S}_0$, 
	and inwards-pointing in the sense that
	its $\gfour$-orthogonal projection onto $\Sigma_0$ points inwards to $\mathcal{S}_0$.
	Next, for each fixed $q \in \mathcal{S}_0$, 
	we construct the null geodesic curve $\upgamma_q : I_q \rightarrow \mathbb{R}^{1+3}$
	with initial data $\upgamma_q(0) = q$ and $\dot{\upgamma}_q(0) = \underline{\ell}_q$,
	where $I_q = [0,A_q]$ is $q$-dependent interval of parameter-time. 
	That is, with $\upgamma_q = \upgamma_q(\uplambda)$,
	$\dot{\upgamma}_q(\uplambda) := \frac{d}{d \uplambda} \upgamma_q(\uplambda)$,
	and with $\Dfour$ denoting the Levi--Civita connection of $\gfour$ (see Subsect.\,\ref{SS:LEVICIVITACONNECTIONS}),
	we solve the geodesic equation 
	$\Dfour_{\dot{\upgamma}_q} \dot{\upgamma}_q = 0$
	with initial conditions $\upgamma_q(0) = q$
	and $\dot{\upgamma}_q(0) = \underline{\ell}_q$.
	Assuming that the compressible Euler solution (on which $\gfour$ depends) is smooth, 
	standard existence and uniqueness theory for ODEs with parameter-dependent initial conditions
	and the compactness of $\mathcal{S}_0$
	together imply that the interval $I_q$ can be chosen to be uniform over $q$ (let's refer to the uniform interval as ``$I$''),
	that there is a $T > 0$ such that $\min_{q \in \mathcal{S}_0} \max_{\uplambda \in I} \upgamma_q^0(\uplambda) > T$
	(where $\upgamma_q^0(\uplambda)$ is the Cartesian time component of the point $\upgamma_q(\uplambda)$),
	and such that the set $\underline{\mathcal{H}} := \lbrace \upgamma_q(I) \ | \ q \in \mathcal{S}_0 \rbrace \cap [0,T] \times \mathbb{R}^3$
	is an embedded three-dimensional manifold-with-boundary.
	Moreover, the results of \cite{mDcLgMjS2019}*{Section~9} 
	can be used to show\footnote{More precisely, \cite{mDcLgMjS2019}*{Section~9} addressed the existence of outgoing $\gfour$-null cones
	emanating from a point, but the results can readily be extended so as to apply to the present example.}
	that $\underline{\mathcal{H}}$ is $\gfour$-null, 
	and that $\underline{\mathcal{H}}$ is the lateral boundary of a region of $\mathcal{M}$
	(in this case a truncated backwards sound cone with a flat top and bottom) 
	satisfying the assumptions stated in Subsect.\,\ref{SS:DOMAINANDTIMEFUNCTIONETC}
	(where for this example, $\Timefunction = t$).
\end{example}

\begin{remark}[The results could be extended to substantially more general spacetime regions]
	\label{R:MOREGENERALDOMAINS}
	The results of this paper could be extended to apply to substantially more general spacetime regions $\mathcal{M}$,
	and it is only for convenience and concreteness that we have made the precise assumptions stated above.
	For example, $\mathcal{M}$ need not be compact and could ``extend to spatial infinity'' 
	(e.g., $\mathcal{M}$ could be the portion of the exterior of an outgoing sound cone that lies in between two constant-time hyperplanes).
	The most crucial assumptions are that the lateral boundary $\underline{\mathcal{H}}$ is $\gfour$-spacelike or $\gfour$-null,
	that positivity properties in the spirit of \eqref{E:INGOINGCONDITION}-\eqref{E:SECONDINGOINGCONDITION} hold
	(these are needed to ensure the coerciveness of our integral identities),
	and that $\mathcal{S}_{\Timefunction} = \underline{\mathcal{H}} \cap \widetilde{\Sigma}_{\Timefunction}$ is a closed manifold
	(this last assumption is helpful in the sense that it guarantees that no boundary terms occur when we integrate by parts over $\mathcal{S}_{\Timefunction}$ in the 
	first step of the proof of Prop.\,\ref{P:STRUCTUREOFERRORINTEGRALS}).
\end{remark}

\begin{remark}[The regions $\mathcal{M}_T$ are acoustically globally hyperbolic]
	\label{R:ACOUSTICALLYGLOBALLYHYPERBOLIC}
	Although it is not directly needed in the paper,
	we can now explain why the spacetime regions $\mathcal{M} = \mathcal{M}_T$ that we study are globally hyperbolic with respect to
	the acoustical metric $\gfour$. That is, we will show that $\widetilde{\Sigma}_0$ is a Cauchy hypersurface in $\mathcal{M}_T$.
	We will consider in detail the case where the lateral boundary $\mathcal{H}_T$ is $\gfour$-spacelike; 
	the $\gfour$-null case can be addressed using similar arguments.
	More precisely, we will show that every past-inextendible future-directed $\gfour$-causal curve
	$\upgamma$ contained in $\mathcal{M}_T$
	must intersect $\widetilde{\Sigma}_0$;
	see \cite{rW1984}*{Chapter~8} for background material on causality, and note that causal curves do not have to be differentiable.
	We recall that we are considering only smooth fluid solutions (see Remark~\ref{R:SMOOTHNESSNOTNEEDED}),
	and thus $\Transport$, $\gfour$, etc.\ are smooth on $\mathcal{M}_T$.
	We can assume that $\upgamma$ is parametrized by\footnote{We can assume this because
	$\upgamma$ is $\gfour$-causal
	and because 
	\eqref{E:MATERIALVECTORVIELDRELATIVECTORECTANGULAR}
	and
	\eqref{E:INVERSEACOUSTICALMETRIC} imply that the gradient of $t$ is $\gfour$-timelike.} 
	$t$, that is,
	there exists an interval $I$ such that
	the domain of $\upgamma$ is $I$ and such that for $t \in I$,
	$\upgamma^0(t) = t$
	and $\upgamma(t) \in \mathcal{M}_T$.
	We argue by contradiction, assuming that $\upgamma$ has a past endpoint
	$q \in \mathcal{M}_T$ such that $q \notin \widetilde{\Sigma}_0$.
	It is straightforward to see that $\upgamma$ can be continuously extended so that $q = \upgamma(a)$,
	where $a$ is the left-endpoint of the closure of $I$,
	and that $\upgamma(a)$ must be a boundary point
	of $\mathcal{M}_T$ not lying in $\widetilde{\Sigma}_0$,
	that is,
	$\upgamma(a) \in (\mathcal{H}_T \backslash \widetilde{\Sigma}_0) \cup \widetilde{\Sigma}_T$.
	From the discussion just below \eqref{E:INGOINGCONDITION},
	we see that if $\upgamma(a) \in \mathcal{H}_T \backslash \widetilde{\Sigma}_0$,
	then $\sidehypnorm|_{\upgamma(a)}$ points outward to $\mathcal{M}_T$ at $\upgamma(a)$.
	It follows that there is an $\epsilon > 0$ such that
	we can extend $\upgamma$ as a causal curve such that relative
	to the Cartesian coordinates, 
	$\dot{\upgamma}^{\alpha}(t) \equiv \sidehypnorm^{\alpha}|_{\upgamma(a)}$ for $t \in [a-\epsilon,a]$.
	This contradicts the assumption that $\upgamma$ is inextendible.
	Similarly, from the discussion just below \eqref{E:INGOINGCONDITION}, we see that
	 if $\upgamma(a) \in \widetilde{\Sigma}_T$,
	then $\tophypnorm|_{\upgamma(a)}$ points outward to $\mathcal{M}_T$ at $\upgamma(a)$,
	and there is an $\epsilon > 0$ such that we can extend $\upgamma$ as a causal curve such that relative
	to the Cartesian coordinates, 
	$\dot{\upgamma}^{\alpha}(t) \equiv \tophypnorm^{\alpha}|_{\upgamma(a)}$ for $t \in [a-\epsilon,a]$,
	again contradicting the assumption that $\upgamma$ is inextendible.
	In total, we have shown that
	$\upgamma(a) \in \widetilde{\Sigma}_0$
	as desired.
	
\end{remark}

\subsection{Additional geometric quantities associated to $\mathcal{M}$}
\label{SS:ADDITIONALGEOMETRICQUANTITIES}
In this subsection, we define some additional geometric quantities 
that we use to prove our main results,
and we prove a simple lemma that yields some identities.

\begin{definition}[$\lengthofgen$, $\lengthofmodgen$, $\lengthoftophypnorm$, $\lengthofmodtophypnorm$, $\lengthofsidehypnorm$, 
$\hat{\tophypnorm}$, and $\hat{\sidehypnorm}$]
\label{D:LENGTHOFVARIOUSVECTORFIELDSETC}
We define $\lengthofgen$ to be the following scalar function,
which is positive\footnote{The identity \eqref{E:NORMALISGENERATORINNULLCASE} implies that $\lengthofgen = 0$ 
when $\sidehypnorm$ is $\gfour$-null.} 
when $\sidehypnorm$ is $\gfour$-timelike (because in this case 
$\underline{\mathcal{H}}$ is $\gfour$-spacelike and thus
$\gen$ is $\gfour$-spacelike
with $\gfour(\gen,\gen) > 0$): 
\begin{align} \label{E:LENGTHOFSIDEGEN}
	\lengthofgen
	& := \sqrt{\gfour(\gen,\gen)}.
\end{align}

Similarly, we define $\lengthofmodgen$ to be the following scalar function,
which is positive when $\underline{\mathcal{H}}$ is $\gfour$-spacelike:
\begin{align} \label{E:LENGTHOFMODSIDEGEN}
	\lengthofmodgen
	& := \sqrt{\gfour(\modgen,\modgen)}.
\end{align}

In addition,
we define $\lengthoftophypnorm > 0$ to be the following scalar function,
which is positive because $\widetilde{\Sigma}_{\Timefunction}$ is $\gfour$-spacelike:
\begin{align} \label{E:LENGTHOFTOPHYPNORM}
	\lengthoftophypnorm
	& := \sqrt{-\gfour(\tophypnorm,\tophypnorm)}.
\end{align}

Similarly, we define $\lengthofmodtophypnorm > 0$ by
\begin{align} \label{E:LENGTHOFTOPHYPNORMNORMALIZEDAGAINSTTIMEFUNCTION}
	\lengthofmodtophypnorm
	& := \sqrt{- \gfour(\modtophypnorm,\modtophypnorm)}. 
\end{align}

Next,
we define $\lengthofsidehypnorm$ to be the following scalar function,
which is positive when $\sidehypnorm$ is $\gfour$-timelike
and vanishing when $\sidehypnorm$ is $\gfour$-null:
\begin{align} \label{E:LENGTHOFHYPNORM}
	\lengthofsidehypnorm
	& := \sqrt{-\gfour(\sidehypnorm,\sidehypnorm)}.
\end{align}

In addition, we define $\hat{\tophypnorm}$ to be the following vectorfield
($\hat{\tophypnorm}$ is the $\gfour$-unit future-directed (see Footnote~\ref{FN:FUTUREDIRECTED}) 
normal to $\widetilde{\Sigma}_{\Timefunction}$):
	\begin{align} \label{E:TOPUNITHYPERSURFACENORMAL}
		\hat{\tophypnorm}^{\alpha}
		& := \frac{\tophypnorm^{\alpha}}{\lengthoftophypnorm}.
	\end{align}	
 
Finally, when $\sidehypnorm$ is $\gfour$-timelike, we define $\hat{\sidehypnorm}$ to be the following vectorfield
($\hat{\sidehypnorm}$ is the $\gfour$-unit future-directed normal to $\underline{\mathcal{H}}$):
	\begin{align} \label{E:UNITHYPERSURFACENORMAL}
		\hat{\sidehypnorm}^{\alpha}
		& := \frac{\sidehypnorm^{\alpha}}{\lengthofsidehypnorm}.
	\end{align}	

\end{definition}

\begin{lemma}[Some convenient identities]
	\label{L:SOMECONVENIENTIDENTITIES}
	Assume that $\underline{\mathcal{H}}$ is $\gfour$-spacelike,
	let $\lengthoftophypnorm > 0$ be the scalar function defined in \eqref{E:LENGTHOFTOPHYPNORM},
	let $\lengthofsidehypnorm > 0$ be the scalar function defined in \eqref{E:LENGTHOFHYPNORM},
	let $\hat{\tophypnorm}$ be the vectorfield defined in \eqref{E:TOPUNITHYPERSURFACENORMAL},
	and let $\hat{\sidehypnorm}$ be the vectorfield defined in \eqref{E:UNITHYPERSURFACENORMAL}.
	Then the following identities hold:
	\begin{align} 
	\gfour(\Transport,\hat{\tophypnorm}) 
	& 
	= - \frac{1}{\lengthoftophypnorm},
		\label{E:INNERPRODUCTOFTRANPORTANDFUTUREUNITNORMALTOTOPHYPERSURFACE} \\
	\gfour(\Transport,\hat{\sidehypnorm}) 
	& 
	= - \frac{1}{\lengthofsidehypnorm}.
	\label{E:INNERPRODUCTOFTRANPORTANDFUTUREUNITNORMALTOHYPERSURFACE}
\end{align} 

Moreover, let $\sidehypnorm$, $\gen$, and $\tophypnorm$
be the vectorfields from Def.\,\ref{D:HYPNORMANDSPHEREFORMDEFS},
let $\lengthofgen \geq 0$ be the scalar function from Def.\,\ref{D:LENGTHOFVARIOUSVECTORFIELDSETC},
and let $\seconduposinnerproduct > 0$ be the scalar function defined in \eqref{E:SECONDINGOINGCONDITION}.
Then the following identity holds:
\begin{align} \label{E:SIDEHYPNORMINTERMSOFGENANDTOPHYPNORM}
	\sidehypnorm
	& = 
		\frac{\seconduposinnerproduct}{\seconduposinnerproduct + \lengthofgen^2} \gen
		+
		\frac{\lengthofgen^2}{\seconduposinnerproduct + \lengthofgen^2} \tophypnorm.
\end{align}

In addition,
\begin{align} \label{E:INNERPRODUCTOFSIDENORMANDTOPNORM}
	\gfour(\sidehypnorm,\tophypnorm)
	& = - 
			\frac{\seconduposinnerproduct^2 + \lengthofgen^2 \lengthoftophypnorm^2}{\seconduposinnerproduct + \lengthofgen^2}.
\end{align}

Furthermore,
\begin{align} \label{E:RATIOLENGTHOFSIDEHYPNORMLENGTHOFGEN}
	\frac{\lengthofsidehypnorm^2}{\lengthofgen^2}
	& = \frac{\lengthofgen^2 \lengthoftophypnorm^2 + \seconduposinnerproduct^2}{(\seconduposinnerproduct + \lengthofgen^2)^2}.
\end{align}
	
Finally, in the $\gfour$-null case (i.e., $\underline{\mathcal{H}} = \underline{\mathcal{N}}$ and $\uLunit = \sidehypnorm$),
we have
\begin{align} \label{E:NORMALISGENERATORINNULLCASE}
	\uLunit
	& = \gen.
\end{align}

\end{lemma}	

\begin{proof}
	\eqref{E:INNERPRODUCTOFTRANPORTANDFUTUREUNITNORMALTOTOPHYPERSURFACE} 
	is a simple consequence of \eqref{E:EQUIVALENTFUTURENORMALTOTOPHYPERSURFACE} and \eqref{E:TOPUNITHYPERSURFACENORMAL}.
	Similarly,
	\eqref{E:INNERPRODUCTOFTRANPORTANDFUTUREUNITNORMALTOHYPERSURFACE}
	is a simple consequence of \eqref{E:EQUIVALENTFUTURENORMALTOHYPERSURFACE} and \eqref{E:UNITHYPERSURFACENORMAL}.
	
	To prove \eqref{E:SIDEHYPNORMINTERMSOFGENANDTOPHYPNORM},
	we first note that $\gen$ and $\tophypnorm$ belong to the $\gfour$-orthogonal complement of $\mathcal{S}_{\Timefunction}$
	and are linearly independent
	(since $\tophypnorm$ is $\gfour$-timelike while $\gen$ is not).
	It follows that the two-dimensional subspace 
	$\mbox{\upshape span} \lbrace \gen, \tophypnorm \rbrace$
	is the $\gfour$-orthogonal complement of $\mathcal{S}_{\Timefunction}$.
	Therefore, since $\sidehypnorm$ is also $\gfour$-orthogonal to $\mathcal{S}_{\Timefunction}$,
	there are scalar functions $a_1$ and $a_2$ such that
	$\sidehypnorm = a_1 \gen + a_2 \tophypnorm$. Taking the $\gfour$-inner product 
	of each side of this equation with respect to $\Transport$
	and using \eqref{E:EQUIVALENTFUTURENORMALTOTOPHYPERSURFACE}-\eqref{E:EQUIVALENTGENERATOROFHYPERSURFACE},
	we find that $a_1 + a_2 =1$.
	Next, taking the $\gfour$-inner product of the identity
	with respect to $\gen$
	and using
	\eqref{E:SECONDINGOINGCONDITION},
	\eqref{E:LENGTHOFSIDEGEN},
	and the relation $\gfour(\sidehypnorm,\gen)=0$,
	we find that
	$0 = a_1 \lengthofgen^2 - a_2 \seconduposinnerproduct$.
	Solving the two equations for $a_1$ and $a_2$, we conclude \eqref{E:SIDEHYPNORMINTERMSOFGENANDTOPHYPNORM}.
	
	To prove \eqref{E:INNERPRODUCTOFSIDENORMANDTOPNORM},
	we simply take the $\gfour$-inner product of each side of \eqref{E:SIDEHYPNORMINTERMSOFGENANDTOPHYPNORM}
	with respect to $\tophypnorm$
	and use 
	\eqref{E:SECONDINGOINGCONDITION}
	and
	\eqref{E:LENGTHOFTOPHYPNORM}.

	To prove \eqref{E:RATIOLENGTHOFSIDEHYPNORMLENGTHOFGEN},
	we take the $\gfour$-inner product of each side of \eqref{E:SIDEHYPNORMINTERMSOFGENANDTOPHYPNORM}
	with respect to itself and use
	\eqref{E:SECONDINGOINGCONDITION},
	\eqref{E:LENGTHOFSIDEGEN},
	\eqref{E:LENGTHOFTOPHYPNORM},
	and
	\eqref{E:LENGTHOFHYPNORM},
	and then carry out straightforward algebraic computations.
	
The identity \eqref{E:NORMALISGENERATORINNULLCASE} holds because in the $\gfour$-null case,
$\uLunit$ is $\gfour$-orthogonal to itself (where $\uLunit$ is alternate notation for $\sidehypnorm$) 
and is therefore $\underline{\mathcal{N}}$-tangent;
thus, $\uLunit$ satisfies all of the conditions from Def.\,\ref{D:HYPNORMANDSPHEREFORMDEFS}
that uniquely define $\gen$.
	
\end{proof}

\subsection{The vectorfields $\utang$ and $\uspecialgen$}
\label{SS:THEVECTORFIELDSTANGANDSPECIALGEN}
The following two vectorfields are featured prominently in the ensuing analysis.

\begin{definition}[The vectorfields $\utang$ and $\uspecialgen$]
	\label{D:SPECIALGENERATOR}
Let $\Transport$ be the vectorfield defined in \eqref{E:MATERIALVECTORVIELDRELATIVECTORECTANGULAR},
let $\spherenormal$ and $\gen$ be the vectorfields from Def.\,\ref{D:HYPNORMANDSPHEREFORMDEFS},
and let $\uposinnerproduct > 0$ and $\seconduposinnerproduct > 0$
be the scalar functions defined in \eqref{E:INGOINGCONDITION}-\eqref{E:SECONDINGOINGCONDITION}.
We define the vectorfield $\utang$ by
		\begin{align} \label{E:STUTANGENTPARTOFTRANSPORT}
			\utang
			& :=
			\Transport
			-
			\frac{1}{\seconduposinnerproduct}
			\gen 
			- 
			\frac{1}{\uposinnerproduct}
			\spherenormal,
		\end{align}
	and the vectorfield $\uspecialgen$ by
	\begin{align} 
		\uspecialgen
		& := 
				\Transport
				-
				\frac{1}{\uposinnerproduct}
				\spherenormal.
				\label{E:SPECIALGENERATOR}
	\end{align}
\end{definition}

\subsection{Algebraic identities in which the sign matters}
\label{SS:ALGEBRAICSIGNMATTERS}
The following lemma provides algebraic identities relating various vectorfields tied to the 
geometry of $\mathcal{M}$. In order for our main results to useful, it is crucial that
the scalar functions $\seconduposinnerproduct$ and $\uposinnerproduct$ 
on RHS~\eqref{E:TRANSPORTDECOMPOSITION} are positive (by assumption -- see Subsect.\,\ref{SS:ASSUMPTIONSONSPACETIMEREGION}).
The positivity is in particular necessary for the coercivity of the boundary terms in
our main integral identities \eqref{E:SPACETIMEREMARKABLEIDENTITYSPECIFICVORTICITY}-\eqref{E:SPACETIMEREMARKABLEIDENTITYENTROPYGRADIENT}.

\begin{lemma}[Properties of $\utang$ and connections between $\Transport$, $\spherenormal$, $\gen$, $\uspecialgen$, and $\utang$]
\label{L:KEYIDBETWEENVARIOUSVECTORFIELDS}
The vectorfield $\utang$ from Def.\,\ref{D:SPECIALGENERATOR}
is $\mathcal{S}_{\Timefunction}$-tangent,
while the vectorfield $\uspecialgen$ from Def.\,\ref{D:SPECIALGENERATOR} is $\underline{\mathcal{H}}$-tangent.
Moreover, the following identities hold
along $\underline{\mathcal{H}}$:
\begin{align} \label{E:SPECIALGENERATORIDENTITY}
	\uspecialgen
	& = \frac{1}{\seconduposinnerproduct}
			\gen
			+ 
			\utang,
\end{align}

\begin{align} \label{E:TRANSPORTDECOMPOSITION}
	\Transport
	& = \frac{1}{\seconduposinnerproduct}
			\gen
			+ 
			\frac{1}{\uposinnerproduct}
			\spherenormal
			+ 
			\utang
		= \uspecialgen
			+
			\frac{1}{\uposinnerproduct}
			\spherenormal.
\end{align}

\end{lemma}

\begin{proof}
	The identities \eqref{E:SPECIALGENERATORIDENTITY}-\eqref{E:TRANSPORTDECOMPOSITION} 
	are direct consequences of Def.\,\ref{D:SPECIALGENERATOR}.
	Moreover, once we show that $\utang$ is $\mathcal{S}_{\Timefunction}$-tangent,
	the $\underline{\mathcal{H}}$-tangent nature of $\uspecialgen$ follows trivially from
	\eqref{E:SPECIALGENERATORIDENTITY}.
	
	It remains for us to prove that $\utang$ is $\mathcal{S}_{\Timefunction}$-tangent.
	We first note that because the surfaces
	$\widetilde{\Sigma}_{\Timefunction}$ and $\underline{\mathcal{H}}$ are transversal by assumption,
	their $\gfour$-normal vectors, which are $\tophypnorm$ and $\sidehypnorm$ respectively, cannot be parallel.
	Moreover, because $\tophypnorm$ and $\sidehypnorm$
	are $\gfour$-orthogonal to $\mathcal{S}_{\Timefunction} = \underline{\mathcal{H}} \cap \widetilde{\Sigma}_{\Timefunction}$,
	it follows that the two-dimensional subspace 
	$\mbox{\upshape span} \lbrace \tophypnorm, \sidehypnorm \rbrace$
	is the $\gfour$-orthogonal complement of $\mathcal{S}_{\Timefunction}$.
	Thus, in view of \eqref{E:STUTANGENTPARTOFTRANSPORT},
	the $\mathcal{S}_{\Timefunction}$-tangent property of $\utang$ will follow once we
	show that the vectorfield
	$
	\Transport
			-
			\frac{1}{\seconduposinnerproduct}
			\gen
			- 
			\frac{1}{\uposinnerproduct}
			\spherenormal
	$
	has vanishing $\gfour$-inner product with $\tophypnorm$ and $\sidehypnorm$.
	These inner products are easy to compute 
	using \eqref{E:EQUIVALENTFUTURENORMALTOTOPHYPERSURFACE}-\eqref{E:EQUIVALENTFUTURENORMALTOHYPERSURFACE}
	and \eqref{E:INGOINGCONDITION}-\eqref{E:SECONDINGOINGCONDITION}
	as well as the relations 
	$
	0
	=
	\gfour(\sidehypnorm,\gen) 
	=
	\gfour(\tophypnorm,\spherenormal)
	$.
	
\end{proof}

\subsection{First fundamental forms and projections}
\label{SS:FIRSTFUNDANDPROJECTIONS}
Having described our assumptions on $\mathcal{M}$ and its boundary,
and having constructed geometric vectorfields adapted to these sets,
we now define some additional geometric tensorfields that are adapted to them.
Specifically, we will define various first fundamental forms and projection operators.
These standard geometric objects
will play an important role in the formulation and proof of our main integral identities.

\begin{definition}[First fundamental forms and projections]
	\label{D:FIRSTFUNDAMENTALFORMSANDPROJECTIONS}
	Let $\Transport$ be the vectorfield defined in \eqref{E:MATERIALVECTORVIELDRELATIVECTORECTANGULAR},
	let $\spherenormal$ be the vectorfield from Def.\,\ref{D:HYPNORMANDSPHEREFORMDEFS},
	and let
	$\hat{\tophypnorm}$
	and
	$\hat{\sidehypnorm}$ be the vectorfields from Def.\,\ref{D:LENGTHOFVARIOUSVECTORFIELDSETC},
	where $\hat{\sidehypnorm}$ is defined only when $\underline{\mathcal{H}}$ is $\gfour$-spacelike.
	We define the following symmetric type $\binom{0}{2}$
	tensorfields, where $g$ and $\topfirstfund$ are defined on $\mathcal{M}$,
	while $\sidefirstfund$ and $\gsphere$ are defined on 
	$\underline{\mathcal{H}}$,
	and $\sidefirstfund$ is defined only when $\underline{\mathcal{H}}$ is $\gfour$-spacelike:
	\begin{subequations}
	\begin{align} \label{E:FIRSTFUNDAMENTALFORMSIGMAT}
		g_{\alpha \beta}
		& :=  \gfour_{\alpha \beta} 
				+ 
				\Transport_{\alpha} \Transport_{\beta},
				\\
		\topfirstfund_{\alpha \beta}
		& :=  \gfour_{\alpha \beta} 
				+ 
				\hat{\tophypnorm}_{\alpha} \hat{\tophypnorm}_{\beta},
				\label{E:TOPFIRSTFUNDAMENTALFORMHYPERSURFACE} 
				\\
		\sidefirstfund_{\alpha \beta}
		& :=  \gfour_{\alpha \beta} 
				+ 
				\hat{\sidehypnorm}_{\alpha} \hat{\sidehypnorm}_{\beta},
				\label{E:FIRSTFUNDAMENTALFORMHYPERSURFACE} 
				\\
		\gsphere_{\alpha \beta}
		& :=  \gfour_{\alpha \beta} 
				+ 
				\hat{\tophypnorm}_{\alpha} \hat{\tophypnorm}_{\beta}
				-
				\spherenormal_{\alpha} \spherenormal_{\beta}.
				\label{E:SPHEREFIRSTFUNDAMENTAL}
	\end{align}
	\end{subequations}
	
	We define the following symmetric type $\binom{2}{0}$
	tensorfields, where $g^{-1}$ and $\topfirstfund^{-1}$ are defined on $\mathcal{M}$,
	while $\sidefirstfund^{-1}$ and $\gsphere^{-1}$ are defined on 
	$\underline{\mathcal{H}}$,
	and $\sidefirstfund^{-1}$ is defined only when $\underline{\mathcal{H}}$ is $\gfour$-spacelike:
	\begin{subequations}
	\begin{align} \label{E:INVERSEFIRSTFUNDAMENTALFORMSIGMAT}
		(g^{-1})^{\alpha \beta}
		& :=  (\gfour^{-1})^{\alpha \beta} 
				+ 
				\Transport^{\alpha} \Transport^{\beta},
				\\
			(\topfirstfund^{-1})^{\alpha \beta}
		& :=  (\gfour^{-1})^{\alpha \beta} 
				+ 
				\hat{\tophypnorm}^{\alpha} \hat{\tophypnorm}^{\beta},
				\label{E:TOPINVERSEFIRSTFUNDAMENTALFORMHYPERSURFACE} 
				\\
		(\sidefirstfund^{-1})^{\alpha \beta}
		& :=  (\gfour^{-1})^{\alpha \beta} 
				+ 
				\hat{\sidehypnorm}^{\alpha} \hat{\sidehypnorm}^{\beta},
				\label{E:INVERSEFIRSTFUNDAMENTALFORMHYPERSURFACE} 
				\\
		(\gsphere^{-1})^{\alpha \beta}
		& :=  (\gfour^{-1})^{\alpha \beta} 
				+ 
				\hat{\tophypnorm}^{\alpha} \hat{\tophypnorm}^{\beta}
				-
				\spherenormal^{\alpha} \spherenormal^{\beta}.
				\label{E:INVERSESPHEREFIRSTFUNDAMENTAL}
	\end{align}
	\end{subequations}
	
	Finally, we define the following type $\binom{1}{1}$ tensorfields,
	where $\Sigmatproject$ and $\topproject$ are defined on $\mathcal{M}$,
	while $\sideproject$ and $\sphereproject$ are defined on 
	$\underline{\mathcal{H}}$,
	and $\sideproject$ is defined only when $\underline{\mathcal{H}}$ is $\gfour$-spacelike:
	\begin{subequations}
	\begin{align} 
			\Sigmatproject_{\ \beta}^{\alpha} 
		& 	:=  
				\updelta_{\ \beta}^{\alpha} 
				+ 
				\Transport^{\alpha} \Transport_{\beta},
				\label{E:SIGMATPROJECT} \\
		\topproject_{\ \beta}^{\alpha} 
		& 	: =  
				\updelta_{\ \beta}^{\alpha}
				+ 
				\hat{\tophypnorm}^{\alpha} \hat{\tophypnorm}_{\beta},
				\label{E:TOPPROJECT} \\
		\sideproject_{\ \beta}^{\alpha} 
		& 	: =  
				\updelta_{\ \beta}^{\alpha}
				+ 
				\hat{\sidehypnorm}^{\alpha} \hat{\sidehypnorm}_{\beta},
				\label{E:HPROJECT} \\
			\sphereproject_{\ \beta}^{\alpha} 
			& 
			:= \updelta_{\ \beta}^{\alpha}
			+
			\hat{\tophypnorm}^{\alpha} \hat{\tophypnorm}_{\beta}
			-
			\spherenormal^{\alpha}
			\spherenormal_{\beta}.
				\label{E:SPHEREPROJECTION}
	\end{align}	
	\end{subequations}
\end{definition}

In the following lemma, we record some basic properties of the tensorfields from Def.\,\ref{D:FIRSTFUNDAMENTALFORMSANDPROJECTIONS}.
We omit the proof, which is a routine consequence of the definitions.

\begin{lemma}[Basic properties of the tensorfields from Definition~\ref{D:FIRSTFUNDAMENTALFORMSANDPROJECTIONS}]
	\label{L:BASICPROPSOFFUNDAMENTALFORMSANDPROJECTIONS}
	$g$ is the first fundamental form of $\Sigma_t$ in the following sense:
	\begin{subequations}
	\begin{align}
		g(X,Y) 
		& = \gfour(X,Y) && \mbox{\upshape for all $\Sigma_t$-tangent vectorfields $X$ and $Y$},
			\label{E:FIRSTFUNDSIGMATEQUALSSPACETIMEMETRICONSIGMAT} \\
		g(\Transport,\mathbf{X})
		& = 0,
		&& \mbox{\upshape for all vectorfields $\mathbf{X}$},
		\label{E:FIRSTFUNDSIGMATISORTHOGONALTOTRANSPORT}
	\end{align}
	\end{subequations}
	where $\Transport$ is the vectorfield defined in \eqref{E:MATERIALVECTORVIELDRELATIVECTORECTANGULAR}
	(it is the unit $\gfour$-normal to $\Sigma_t$). In particular, $g$ is a 
	Riemannian metric (i.e., a positive definite quadratic form) on $\Sigma_t$,
	and $g$ is a positive semi-definite quadratic form on all vectorfields.
	
	Similarly, 
	$\topfirstfund$ is the first fundamental form of $\widetilde{\Sigma}_{\Timefunction}$ in the following sense:
	\begin{subequations}
	\begin{align}
		\topfirstfund(X,Y) 
		& = \gfour(X,Y) && \mbox{\upshape for all $\widetilde{\Sigma}_{\Timefunction}$-tangent vectorfields $X$ and $Y$},
			\label{E:TOPHYPFIRSTFUNDAGREESWITHSPACETIMEMETRICONTANGENTVECTORS} \\
		\topfirstfund(\tophypnorm,\mathbf{X})
		& = 0,
		&& \mbox{\upshape for all vectorfields $\mathbf{X}$},
	\end{align}
	\end{subequations}
	where $\tophypnorm$ is the vectorfield from Def.\,\ref{D:HYPNORMANDSPHEREFORMDEFS}
	(it is $\gfour$-normal to $\widetilde{\Sigma}_{\Timefunction}$). In particular, 
	$\topfirstfund$ is a Riemannian metric on $\widetilde{\Sigma}_{\Timefunction}$,
	and $\topfirstfund$ is a positive semi-definite quadratic form on all vectorfields.
	
	Similarly, when $\underline{\mathcal{H}}$ is $\gfour$-spacelike,
	$\sidefirstfund$ is the first fundamental form of $\underline{\mathcal{H}}$ in the following sense:
	\begin{subequations}
	\begin{align}
		\sidefirstfund(X,Y) 
		& = \gfour(X,Y) && \mbox{\upshape for all $\underline{\mathcal{H}}$-tangent vectorfields $X$ and $Y$},
			\label{E:HYPFIRSTFUNDAGREESWITHSPACETIMEMETRICONTANGENTVECTORS} \\
		\sidefirstfund(\sidehypnorm,\mathbf{X})
		& = 0,
		&& \mbox{\upshape for all vectorfields $\mathbf{X}$},
	\end{align}
	\end{subequations}
	where $\sidehypnorm$ is the vectorfield from Def.\,\ref{D:HYPNORMANDSPHEREFORMDEFS}
	(it is $\gfour$-normal to $\underline{\mathcal{H}}$). In particular,
	when $\underline{\mathcal{H}}$ is $\gfour$-spacelike,
	$\sidefirstfund$ is a Riemannian metric on $\underline{\mathcal{H}}$,
	and $\sidefirstfund$ is a positive semi-definite quadratic form on all vectorfields.
	
	Similarly, 
	$\gsphere$ is the first fundamental form of $\mathcal{S}_{\Timefunction}$ in the following sense:
	\begin{subequations}
	\begin{align}
		\gsphere(X,Y) 
		& = \gfour(X,Y) && \mbox{\upshape for all $\mathcal{S}_{\Timefunction}$-tangent vectorfields $X$ and $Y$},
			\label{E:GSPHEREAGREESWITHGONSTTANGENTVECTORFIELDS} \\
		\gsphere(V,\mathbf{X})
		& = 0
		&&
		\mbox{\upshape for all vectorfields $\mathbf{X}$ if $\SigmatTan \in \mbox{\upshape span} \lbrace \tophypnorm, \spherenormal \rbrace$},
			\label{E:GSPHEREVANISHESONSPANOFTOPHYPNORMANDSPHERENORMAL}
	\end{align}
	\end{subequations}
	where $\spherenormal$ is the vectorfield from Def.\,\ref{D:HYPNORMANDSPHEREFORMDEFS}
	(and thus $\mbox{\upshape span} \lbrace \tophypnorm, \spherenormal \rbrace$
	is the space of vectorfields that is $\gfour$-orthogonal to $\mathcal{S}_{\Timefunction}$).
	In particular, since $\mathcal{S}_{\Timefunction}$ is a submanifold of the $\gfour$-spacelike submanifold $\widetilde{\Sigma}_{\Timefunction}$,
	$\gsphere$ is a Riemannian metric on $\mathcal{S}_{\Timefunction}$,
	and $\gsphere$ is a positive semi-definite quadratic form on all vectorfields.
	
	In addition, $\sphereproject$ is the $\gfour$-orthogonal projection onto $\mathcal{S}_{\Timefunction}$ in the following sense:
	\begin{subequations}
	\begin{align}
		\sphereproject_{\ \beta}^{\alpha} X^{\beta} 
		& =X^{\alpha},
		& \sphereproject_{\ \beta}^{\alpha} X_{\alpha} & = X_{\beta},
		& \mbox{\upshape if $X$ is $\mathcal{S}_{\Timefunction}$-tangent},
			\label{E:STPROJECTIONISIDENTITYONST} \\
	\sphereproject_{\ \beta}^{\alpha} \SigmatTan^{\beta} & = 0,
	& \sphereproject_{\ \beta}^{\alpha} \SigmatTan_{\alpha} & = 0,
	&\mbox{\upshape if $\SigmatTan \in \mbox{\upshape span} \lbrace \tophypnorm, \spherenormal \rbrace$},
		\label{E:STPROJECTIONANNIHILATESNORMALS} \\
	\sphereproject_{\ \gamma}^{\alpha}
	\sphereproject_{\ \beta}^{\gamma}
	& =
	\sphereproject_{\ \beta}^{\alpha}.
	\label{E:STPROJECTIONSQUAREDEQUALSSTPROJECTION}
	\end{align}
	\end{subequations}
	
	Moreover, $\gsphere^{-1}$ is the inverse first fundamental form of $\mathcal{S}_{\Timefunction}$ in the sense that
	$(\gsphere^{-1})^{\alpha \kappa} \gsphere_{\kappa \beta} = \sphereproject_{\ \beta}^{\alpha}$.
	In particular, when restricted to tensors tangent to $\mathcal{S}_{\Timefunction}$, $(\gsphere^{-1})^{\alpha \kappa} \gsphere_{\kappa \beta}$ is the identity.
	In an analogous fashion, $g^{-1}$ is the inverse first fundamental form of $\Sigma_t$,
	$\Sigmatproject$ is the $\gfour$-orthogonal projection onto $\Sigma_t$,
	$\topfirstfund^{-1}$ is the inverse first fundamental form of $\widetilde{\Sigma}_{\Timefunction}$,
	$\topproject$ is the $\gfour$-orthogonal projection onto $\widetilde{\Sigma}_{\Timefunction}$,
	and, when $\underline{\mathcal{H}}$ is $\gfour$-spacelike,
	$\sidefirstfund^{-1}$ is the inverse first fundamental form of $\underline{\mathcal{H}}$
	and $\sideproject$ is the $\gfour$-orthogonal projection onto $\underline{\mathcal{H}}$.
\end{lemma}

Our identities involve projections of tensorfields onto $\mathcal{S}_{\Timefunction}$ and $\underline{\mathcal{H}}$,
which we now define.

\begin{definition}[Projections of tensorfields]
	\label{D:PROJECTIONSOFTENSORFIELDS}
	If $\pmb{\upxi}_{\beta_1 \cdots \beta_n}^{\alpha_1 \cdots \alpha_m}$ is a type $\binom{m}{n}$
	tensorfield, then 
	$\Sigmatproject \pmb{\upxi}$ denotes the $\gfour$-orthogonal projection of $\pmb{\upxi}$ onto $\Sigma_t$,
	defined by
	$(\Sigmatproject \pmb{\upxi})_{\beta_1 \cdots \beta_n}^{\alpha_1 \cdots \alpha_m}
	:=
	\Sigmatproject_{\ \gamma_1}^{\alpha_1}
	\cdots
	\Sigmatproject_{\ \gamma_m}^{\alpha_m}
	\Sigmatproject_{\ \beta_1}^{\delta_1}
	\cdots 
	\Sigmatproject_{\ \beta_n}^{\delta_n}
	\pmb{\upxi}_{\delta_1 \cdots \delta_n}^{\gamma_1 \cdots \gamma_m}
	$.
	Similarly, we denote the $\gfour$-orthogonal projections of $\pmb{\upxi}$ onto 
	$\widetilde{\Sigma}_{\Timefunction}$,
	$\underline{\mathcal{H}}$, 
	and $\mathcal{S}_{\Timefunction}$
	by 
	$\topproject \pmb{\upxi}$,
	$\sideproject \pmb{\upxi}$, 
	and $\sphereproject \pmb{\upxi}$ respectively; these are defined as above, but with
	$\topproject$,
	$\sideproject$, 
	and $\sphereproject$ respectively in the role of $\Sigmatproject$,
	and $\sideproject \pmb{\upxi}$ is defined only when
	$\underline{\mathcal{H}}$ is $\gfour$-spacelike.
	
\end{definition}

\begin{definition}[$\Sigma_t$-, $\widetilde{\Sigma}_{\Timefunction}$-, $\underline{\mathcal{H}}$-, $\mathcal{S}_{\Timefunction}$-, and
$\underline{\mathcal{N}}$-tangent tensorfields]
\label{D:TANGENTTENSORFIELDS}
If $\pmb{\upxi}_{\beta_1 \cdots \beta_n}^{\alpha_1 \cdots \alpha_m}$ is a type $\binom{m}{n}$
tensorfield, then we say that $\pmb{\upxi}$ is $\Sigma_t$-tangent if $\pmb{\upxi} = \Sigmatproject \pmb{\upxi}$.
Similarly, we say that $\pmb{\upxi}$ is $\widetilde{\Sigma}_{\Timefunction}$-tangent if $\pmb{\upxi} = \topproject \pmb{\upxi}$
and we say that
$\pmb{\upxi}$ is $\mathcal{S}_{\Timefunction}$-tangent if $\pmb{\upxi} = \sphereproject \pmb{\upxi}$.
Moreover, if $\underline{\mathcal{H}}$ is $\gfour$-spacelike, 
we say that $\pmb{\upxi}$ is $\underline{\mathcal{H}}$-tangent if $\pmb{\upxi} = \sideproject \pmb{\upxi}$.

In the case that $\underline{\mathcal{H}} = \underline{\mathcal{N}}$ is $\gfour$-null,
if $q \in \mathcal{S}_{\Timefunction} \subset \underline{\mathcal{N}}$ and $\uLunit|_q$ denotes $\uLunit$ at $q$,
then we say that the vectorfield $\mathbf{X} \in T_q \mathcal{M}$ is $\underline{\mathcal{N}}$-tangent at $q$ if
$\mathbf{X} \in \mbox{\upshape span} \lbrace \uLunit|_q \rbrace \oplus T_q \mathcal{S}_{\Timefunction}$.
We typically avoid explicitly referencing the point $q$ when there is no danger of confusion.
Similarly, we say that the one-form\footnote{For any positive integers $m$ and $n$,
the definition of $\underline{\mathcal{N}}$-tangent
extends in a natural fashion to type $\binom{m}{n}$ tensorfields $\pmb{\upxi}$;
we do not need the extended definition in the present article.} 
$\pmb{\upxi}$ is $\underline{\mathcal{N}}$-tangent at $q$ if
its $\gfour$-dual vectorfield $\mathbf{X}^{\alpha} := (\gfour^{-1})^{\alpha \beta} \pmb{\upxi}_{\beta}$
is $\underline{\mathcal{N}}$-tangent at $q$. 

In addition, we sometimes use the following alternate notation for the $\gfour$-orthogonal 
	projection of a vectorfield $\SigmatTan$ onto $\mathcal{S}_{\Timefunction}$:
	\begin{align} \label{E:ALTERNATENOTATIONFORSPHEREPROJECTEDTENSORFIELDS}
		\angV 
		& := \sphereproject \SigmatTan,
	\end{align}
	i.e., $\angVarg{\alpha} := \sphereproject_{\ \beta}^{\alpha} \SigmatTan^{\beta}$.
\end{definition}

\begin{convention}[Restricting $g$, $\gsphere$, and $\sphereproject$ to $\Sigma_t$-tangent vectorfields
and displaying only spatial indices]
\label{C:DISPLAYONLYSPATIAL}
In the rest of the article, we will often adopt the point of view that our formulas are statements about
the Cartesian components of tensors, 
even though many of the formulas could be given a coordinate
invariant interpretation; see also Remark~\ref{R:AVOIDINGCHRISTOFFEL}.
Moreover, in much of the remaining discussion, we will only have to consider the action
of $g$, $\gsphere$, and $\sphereproject$ on $\Sigma_t$-tangent vectorfields.
Thus, when working with $\Sigma_t$-tangent tensors, 
we typically only display their spatial components, 
since the $0$ component (i.e., the ``Cartesian time component'')
vanishes. For example, we have $\gfour(\vortrenormalized,\vortrenormalized) = g_{ab} \vortrenormalized^a \vortrenormalized^b$.
As a second example, we note that $\angGradEnt^{\alpha} = \sphereproject_{\ b}^{\alpha} \GradEnt^b$.
\end{convention}

\begin{convention}[Lowering and raising Cartesian spatial indices with $g$ and $g^{-1}$]
	\label{C:LOWERANDRAISE}
	In the remainder of the paper, we often lower and raise indices
	of $\Sigma_t$-tangent tensors with $g$ and $g^{-1}$.
	For example, we have $\vortrenormalized_a = g_{ab} \vortrenormalized^b$.
	For Greek spacetime indices, we will continue to use the conventions 
	for lowering and raising
	stated in Subsubsect.\,\ref{SSS:ACOUSTICALMETRIC}.
	Note that there is no danger of confusion in the sense that 
	\eqref{E:ACOUSTICALMETRIC} and \eqref{E:TRANSPORTONEFORMIDENTITY} 
	imply that when $\SigmatTan$ is $\Sigma_t$-tangent,
	we have $\gfour_{a \beta} \SigmatTan^{\beta} = g_{ab} \SigmatTan^b$.
	We caution, however, that $\SigmatTan_0$ is generally not equal to $0$;
	see \eqref{E:LOWERINDICESOFSIGMATTANGENTVECTORFIELDS}.
\end{convention}

\begin{definition}[Projections of Cartesian coordinate partial derivative vectorfields]
	\label{D:PROJECTIONSOFCARTESIANCOORDINATEVECTORFIELDS}
	For $\alpha = 0,1,2,3$, we define the following vectorfields
	relative to the Cartesian coordinates:
	\begin{subequations}
	\begin{align}
		\toppartialarg{\alpha}
		& := \topproject_{\ \alpha}^{\beta} \partial_{\beta},
		&	
		\sidepartialarg{\alpha}
		& := \sideproject_{\ \alpha}^{\beta} \partial_{\beta},
		&
	\angpartialarg{\alpha}
	& := \sphereproject_{\ \alpha}^{\beta} \partial_{\beta},
		\label{E:PROJECTIONSOFCARTESIANVECTORFIELDS} 
		\\
		\toppartialuparg{\alpha}
		& := (\topfirstfund^{-1})^{\alpha \beta} \partial_{\beta},
		&	
		\sidepartialuparg{\alpha}
		& := (\sidefirstfund^{-1})^{\alpha \beta} \partial_{\beta},
		&
	\angpartialuparg{\alpha}
	& := (\gsphere^{-1})^{\alpha \beta} \partial_{\beta},
	\label{E:RAISEDPROJECTIONSOFCARTESIANVECTORFIELDS}
\end{align}
\end{subequations}
where $\sidepartialarg{\alpha}$ and $\sidepartialuparg{\alpha}$ are defined only when $\underline{\mathcal{H}}$ is $\gfour$-spacelike.
\end{definition}	

For future use, we note that
\eqref{E:TOPPROJECT},
\eqref{E:SPHEREPROJECTION},
and \eqref{E:PROJECTIONSOFCARTESIANVECTORFIELDS}
imply the following vectorfield identity:
\begin{align} \label{E:TOPPARTIALINTERMSOFOUTERNORMALDERIVATIVEANDANGPARTIAL}
	\toppartialarg{\alpha}
	& =
		\spherenormal_{\alpha}
		\spherenormal
		+
		\angpartialarg{\alpha}.
\end{align}

Moreover, 
when $\underline{\mathcal{H}}$ is $\gfour$-spacelike, 
it follows in a straightforward fashion from
\eqref{E:LENGTHOFSIDEGEN},
\eqref{E:HPROJECT},
\eqref{E:SPHEREPROJECTION},
the fact that $\gen$ is $\gfour$-orthogonal to $\mathcal{S}_{\Timefunction}$,
and \eqref{E:PROJECTIONSOFCARTESIANVECTORFIELDS}
that the following identity holds: 
\begin{align} \label{E:SIDEPARTIALINTERMSOFOUTERNORMALDERIVATIVEANDANGPARTIAL}
	\sidepartialarg{\alpha}
	& =
		\frac{\gen_{\alpha}}{\lengthofgen^2}		
		\gen
		+
		\angpartialarg{\alpha}.
\end{align}

\subsection{Levi--Civita connections and $\mathcal{S}_{\Timefunction}$-divergence}
\label{SS:LEVICIVITACONNECTIONS}
We refer readers to \cite{rW1984} for basic background on Levi--Civita connections in the context of Lorentzian geometry.
Our ensuing discussion will involve the Levi--Civita connection of $\gfour$, which we denote by $\Dfour$.
It will also involve the Levi--Civita connection of $g$ (viewed as a Riemannian metric on $\Sigma_t$),
which we denote by $\nabla$, 
the Levi--Civita connection of $\topfirstfund$ (viewed as a Riemannian metric on $\widetilde{\Sigma}_{\Timefunction}$),
which we denote by $\widetilde{\nabla}$, 
the Levi--Civita connection of $\gsphere$
(viewed as a Riemannian metric on $\mathcal{S}_{\Timefunction}$),
which we denote by $\angD$,
and, in case that $\underline{\mathcal{H}}$ is $\gfour$-spacelike,
the Levi--Civita connection of $\sidefirstfund$ (viewed as a Riemannian metric on $\underline{\mathcal{H}}$),
which we denote by $\underline{\nabla}$.
We recall the following basic facts from differential geometry: 
if $\pmb{\xi}$ is $\Sigma_t$-tangent, then $\nabla \pmb{\xi} = \Sigmatproject (\Dfour \pmb{\xi})$,
where $(\Dfour \pmb{\xi})_{\beta_1 \beta_2 \cdots \beta_{n+1}}^{\alpha_1 \cdots \alpha_m} 
= 
\Dfour_{\beta_1} \pmb{\upxi}_{\beta_2 \cdots \beta_{n+1}}^{\alpha_1 \cdots \alpha_m}$ is the covariant derivative of $\pmb{\xi}$ 
with respect to $\Dfour$;
if $\pmb{\xi}$ is $\widetilde{\Sigma}_{\Timefunction}$-tangent, then $\widetilde{\nabla} \pmb{\xi} = \topproject (\Dfour \pmb{\xi})$;
if $\pmb{\xi}$ is $\mathcal{S}_{\Timefunction}$-tangent, then $\angD \pmb{\xi} = \sphereproject (\Dfour \pmb{\xi})$;
and if $\underline{\mathcal{H}}$ is $\gfour$-spacelike
and $\pmb{\xi}$ is $\underline{\mathcal{H}}$-tangent, 
then $\underline{\nabla} \pmb{\xi} = \sideproject (\Dfour \pmb{\xi})$.

All of our formulas involving indices 
should be interpreted as formulas relative to the
Cartesian coordinates, even though many of them could be re-expressed in a coordinate invariant form.
For example, if $\mathbf{V}$ is a vectorfield and $\Chfour_{\beta \ \gamma}^{\ \alpha}$
denotes a Christoffel symbol of $\gfour$ relative to the Cartesian coordinates, then
\begin{subequations}
\begin{align}
	\Dfour_{\beta} \mathbf{V}^{\alpha}
	& = \partial_{\beta} \mathbf{V}^{\alpha}
		+
		\Chfour_{\beta \ \gamma}^{\ \alpha} \mathbf{V}^{\gamma},
		&&
			\label{E:SPACETIMELEVICIVITAACTIONONVECTORFIELD} \\
	\Chfour_{\beta \ \gamma}^{\ \alpha}
	& := \frac{1}{2}
				(\gfour^{-1})^{\alpha \kappa}
				\left\lbrace
					\partial_{\beta} \gfour_{\kappa \gamma}
					+
					\partial_{\gamma} \gfour_{\beta \kappa}
					-
					\partial_{\kappa} \gfour_{\beta \gamma}
				\right\rbrace.
				\label{E:CHRISTOFFELFOUR}
\end{align}
\end{subequations}
We lower and raise the indices of $\Chfour_{\beta \ \gamma}^{\ \alpha}$ with $\gfour$ and $\gfour^{-1}$.
For example, $\Chfour_{\beta \alpha \gamma} := \gfour_{\alpha \kappa} \Chfour_{\beta \ \gamma}^{\ \kappa}$.

Similarly, if $\SigmatTan$ is a $\Sigma_t$-tangent vectorfield and $\Gamma_{b \ c}^{\ a}$
denotes a Christoffel symbols of $g$ relative to the Cartesian spatial coordinates, then
\begin{subequations}
\begin{align}
	\nabla_b \SigmatTan^a
	& = \partial_b \SigmatTan^a
		+
		\Gamma_{b \ c}^{\ a} \SigmatTan^c,
		&&
			\label{E:SIGMATLEVICIVITAACTIONONVECTORFIELD} \\
	\Gamma_{b \ c}^{\ a}
	& := \frac{1}{2}
				(g^{-1})^{ak}
				\left\lbrace
					\partial_b g_{kc}
					+
					\partial_c g_{bk}
					-
					\partial_k g_{bc}
				\right\rbrace.
				\label{E:CHRISTOFFELSIGMAT}
\end{align}
\end{subequations}

A few of our formulas involve the $\mathcal{S}_{\Timefunction}$-divergence,
i.e., the divergence of $\mathcal{S}_{\Timefunction}$-tangent vectorfields with respect
to the connection $\angD$.
Relative to arbitrary local coordinates $(\vartheta^1,\vartheta^2)$
on $\mathcal{S}_{\Timefunction}$, with $Y = Y^A \frac{\partial}{\partial \vartheta^A}$, 
we have $\angdiv Y = \angD_A Y^A$. 

\subsection{Expressions for the divergence of various vectorfields}
\label{SS:DIVERGNCESOFVARIOUSVECTORFIELDS}
For future use, in the next lemma, 
we provide an expression for the $\widetilde{\nabla}$-divergence of $\widetilde{\Sigma}_{\Timefunction}$-tangent vectorfields
and an expression for the $\angD$-divergence of $\mathcal{S}_{\Timefunction}$-tangent vectorfields.

\begin{lemma}[Expressions for the divergence of various vectorfields]
	\label{L:EXPRESSIONFORWIDETILDESIGMAGRADIENTANDDIVERGENCEINCOORDINATES}
	Let $J$ be a $\widetilde{\Sigma}_{\Timefunction}$-tangent vectorfield defined on
	$\mathcal{M}$. Then relative to the Cartesian coordinates,
	we have the following identities:
	\begin{subequations}
	\begin{align} \label{E:EXPRESSIONFORWIDETILDESIGMAGRADIENTINCOORDINATES}
		\widetilde{\nabla}_{\alpha} J^{\beta}
		& = 
		\toppartialarg{\alpha} J^{\beta}
		-
		\hat{\tophypnorm}^{\beta}
		 J^{\gamma}
		\toppartialarg{\alpha} \hat{\tophypnorm}_{\gamma}
			\\
		& \ \
		+
		\frac{1}{2}
		(\topfirstfund^{-1})^{\delta \beta}
		J^{\kappa}
		\toppartialarg{\alpha} \gfour_{\delta \kappa}
		+
		\frac{1}{2}
		\topproject_{\ \alpha}^{\gamma} 
		(\topfirstfund^{-1})^{\delta \beta}
		J \gfour_{\gamma \delta}
		-
		\frac{1}{2}
		\topproject_{\ \alpha}^{\gamma} 
		J^{\kappa}
		\toppartialuparg{\beta}
		\gfour_{\gamma \kappa},
		\notag
			\\
		\widetilde{\nabla}_{\alpha} J^{\alpha}
		& 
		= 
		\toppartialarg{\alpha} 
		J^{\alpha}
		+
		\frac{1}{2}
		 (\topfirstfund^{-1})^{\alpha \beta}
		J \gfour_{\alpha \beta}
			\label{E:EXPRESSIONFORWIDETILDESIGMADIVERGENCEINCOORDINATES} \\
	& = 
		\partial_{\alpha} J^{\alpha}
		-
		J_{\alpha}
		\hat{\tophypnorm} \hat{\tophypnorm}^{\alpha} 
		-
		(\hat{\tophypnorm} \gfour_{\alpha \beta})
		J^{\alpha}
		\hat{\tophypnorm}^{\beta}
		+
		\frac{1}{2}
		 (\topfirstfund^{-1})^{\alpha \beta}
		J \gfour_{\alpha \beta}.
			\notag
	\end{align}
	\end{subequations}
	
	Moreover, 
	if $Y$ is an $\mathcal{S}_{\Timefunction}$-tangent vectorfield defined on
	$\mathcal{M}$, then relative to the Cartesian coordinates,
	we have the following identities:
	\begin{subequations}
	\begin{align} \label{E:EXPRESSIONFORSPHEREGRADIENTINCOORDINATES}
		\angD_{\alpha} Y^{\beta}
		& = 
		\angpartialarg{\alpha} Y^{\beta}
		-
		\hat{\tophypnorm}^{\beta}
		Y^{\gamma}
		\angpartialarg{\alpha} \hat{\tophypnorm}_{\gamma}
		+
		\spherenormal^{\beta}
		Y^{\gamma}
		\angpartialarg{\alpha} \spherenormal_{\gamma}
			\\
		& \ \
		+
		\frac{1}{2}
		(\gsphere^{-1})^{\delta \beta}
		Y^{\kappa}
		\angpartialarg{\alpha} \gfour_{\delta \kappa}
		+
		\frac{1}{2}
		\sphereproject_{\ \alpha}^{\gamma} 
		(\gsphere^{-1})^{\delta \beta}
		Y \gfour_{\gamma \delta}
		-
		\frac{1}{2}
		\sphereproject_{\ \alpha}^{\gamma} 
		Y^{\kappa}
		\angpartialuparg{\beta}
		\gfour_{\gamma \kappa},
		\notag
			\\
		\angdiv Y
		=
		\angD_{\alpha} Y^{\alpha}
		& 
		= 
		\angpartialarg{\alpha} 
		Y^{\alpha}
		+
		\frac{1}{2}
		(\gsphere^{-1})^{\alpha \beta}
		Y \gfour_{\alpha \beta}.
			\label{E:EXPRESSIONFORSPHEREDIVERGENCEINCOORDINATES} 
	\end{align}
	\end{subequations}
	
	Finally, in the particular case $Y = \angpartialarg{\kappa}$,
	we have the following identity relative to the Cartesian coordinates:
	\begin{align} \label{E:EXPRESSIONFORSPHEREDIVERGENCEOFANGPARTIALVECTORFIELDSINCOORDINATES} 
	\angdiv \angpartialarg{\kappa}
		& 
		= 
		\hat{\tophypnorm}_{\kappa}
	\angpartialarg{\alpha} \hat{\tophypnorm}^{\alpha}
	-
	\spherenormal_{\kappa} 
	\angpartialarg{\alpha} \spherenormal^{\alpha}
		+
		\frac{1}{2}
		(\gsphere^{-1})^{\alpha \beta}
		\angpartialarg{\kappa} \gfour_{\alpha \beta}.
	\end{align}
\end{lemma}

\begin{proof}
	We first prove \eqref{E:EXPRESSIONFORWIDETILDESIGMAGRADIENTINCOORDINATES}.
	As we mentioned in Subsect.\,\ref{SS:LEVICIVITACONNECTIONS} (see also \cite{rW1984}),
	we have
	$
		\widetilde{\nabla}_{\alpha} J^{\beta}
		=
		\topproject_{\ \alpha}^{\gamma} \topproject_{\ \delta}^{\beta} \Dfour_{\gamma} J^{\delta}
	$,
	where $\topproject$ is the $\widetilde{\Sigma}_{\Timefunction}$ projection tensorfield
	(see \eqref{E:TOPPROJECT}).
	Thus, in view of 
	\eqref{E:TOPPROJECT},
	\eqref{E:PROJECTIONSOFCARTESIANVECTORFIELDS}-\eqref{E:RAISEDPROJECTIONSOFCARTESIANVECTORFIELDS},
	\eqref{E:SPACETIMELEVICIVITAACTIONONVECTORFIELD}-\eqref{E:CHRISTOFFELFOUR}, 
	and the assumption that $J$ is $\widetilde{\Sigma}_{\Timefunction}$-tangent
	(which implies in particular that $J^{\alpha} \hat{\tophypnorm}_{\alpha} = 0$
	and
	$(\partial_{\beta} J^{\alpha}) \hat{\tophypnorm}_{\alpha} = - J^{\alpha} \partial_{\beta} \tophypnorm_{\alpha}$),
	we compute that relative to the Cartesian coordinates, we have
	\begin{align} \label{E:FIRSTSTEPEXPRESSIONFORWIDETILDESIGMADIVERGENCEINCOORDINATES}
		\widetilde{\nabla}_{\alpha} J^{\beta}
		& =
		\topproject_{\ \alpha}^{\gamma} 
		\topproject_{\ \delta}^{\beta}
		\partial_{\gamma} J^{\delta}
		+
		\topproject_{\ \alpha}^{\gamma} 
		(\topfirstfund^{-1})^{\delta \beta}
		\Chfour_{\gamma \delta \kappa}
		J^{\kappa}
			\\
		& =
		\toppartialarg{\alpha} J^{\beta}
		-
		\hat{\tophypnorm}^{\beta}
		 J^{\gamma}
		\toppartialarg{\alpha} \hat{\tophypnorm}_{\gamma}
			\notag \\
		& \ \
		+
		\frac{1}{2}
		(\topfirstfund^{-1})^{\delta \beta}
		J^{\kappa}
		\toppartialarg{\alpha} \gfour_{\delta \kappa}
		+
		\frac{1}{2}
		\topproject_{\ \alpha}^{\gamma} 
		(\topfirstfund^{-1})^{\delta \beta}
		J \gfour_{\gamma \delta}
		-
		\frac{1}{2}
		\topproject_{\ \alpha}^{\gamma} 
		J^{\kappa}
		\toppartialuparg{\beta}
		\gfour_{\gamma \kappa}
		\notag
	\end{align}
	as desired.
	
	To prove \eqref{E:EXPRESSIONFORWIDETILDESIGMADIVERGENCEINCOORDINATES},
	we trace \eqref{E:EXPRESSIONFORWIDETILDESIGMAGRADIENTINCOORDINATES} over the indices $\alpha$ and $\beta$
	and carry out straightforward computations.
	
	The identities \eqref{E:EXPRESSIONFORSPHEREGRADIENTINCOORDINATES}-\eqref{E:EXPRESSIONFORSPHEREDIVERGENCEINCOORDINATES}
	can be proved using similar arguments that rely on the
	identity 
	$
		\angD_{\alpha} J^{\beta}
		=
		\sphereproject_{\ \alpha}^{\gamma} \sphereproject_{\ \delta}^{\beta} \Dfour_{\gamma} J^{\delta}
	$
	and the expression \eqref{E:SPHEREPROJECTION}; we omit the details.
	
	\eqref{E:EXPRESSIONFORSPHEREDIVERGENCEOFANGPARTIALVECTORFIELDSINCOORDINATES} 
	follows from \eqref{E:EXPRESSIONFORSPHEREDIVERGENCEINCOORDINATES}
	and the following identity for Cartesian components,
	which follows from
	\eqref{E:SPHEREPROJECTION}
	and
	\eqref{E:PROJECTIONSOFCARTESIANVECTORFIELDS}:
	$
	(\angpartialarg{\kappa})^{\alpha}
	=
	\updelta_{\kappa}^{\alpha}
	+
	\hat{\tophypnorm}_{\kappa}
	\hat{\tophypnorm}^{\alpha}
	-
	\spherenormal_{\kappa}
	\spherenormal^{\alpha}
	$.

\end{proof}

\subsection{Gradients, Euclidean metrics, and pointwise norms with respect to various metrics}
\label{SS:PROPERTIESOFMETRICS}
Our ensuing analysis involves several kinds of gradients of tensorfields and 
pointwise norms of tensorfields with respect to various Riemannian metrics.
In this subsection, we provide the relevant definitions.
We also define the standard Euclidean metrics on $\mathbb{R}^{1+3}$ and $\Sigma_t$.
We also remind the reader that, as is stated in Convention~\ref{C:DISPLAYONLYSPATIAL}, 
we typically only display the Cartesian spatial indices of $\Sigma_t$-tangent tensorfields.

\subsubsection{Gradients}
\label{SSS:GRADIENTS}
Let $\pmb{\upxi}$ be a type $\binom{m}{n}$ tensorfield with Cartesian components 
$\pmb{\upxi}_{\beta_1 \cdots \beta_n}^{\alpha_1 \cdots \alpha_m}$.
\begin{itemize}
\item $\pmb{\partial} \pmb{\upxi}$ denotes the type $\binom{m}{n+1}$ tensorfield with Cartesian components 
	$\partial_{\beta_1} \pmb{\upxi}_{\beta_2 \cdots \beta_{n+1}}^{\alpha_1 \cdots \alpha_m}$.
	\item $\toppartial \pmb{\upxi}$ denotes the type $\binom{m}{n+1}$ tensorfield with Cartesian components 
	$\toppartialarg{\beta_1} \pmb{\upxi}_{\beta_2 \cdots \beta_{n+1}}^{\alpha_1 \cdots \alpha_m}$,
	where $\lbrace \toppartialarg{\alpha} \rbrace_{\alpha=0,1,2,3}$ denotes the 
	$\widetilde{\Sigma}_{\Timefunction}$-tangent vectorfields from \eqref{E:PROJECTIONSOFCARTESIANVECTORFIELDS}.
	\item When $\underline{\mathcal{H}}$ is $\gfour$-spacelike,
	$\sidepartial \pmb{\upxi}$ denotes the type $\binom{m}{n+1}$ tensorfield with Cartesian components 
	$\sidepartialarg{\beta_1} \pmb{\upxi}_{\beta_2 \cdots \beta_{n+1}}^{\alpha_1 \cdots \alpha_m}$,
	where $\lbrace \sidepartialarg{\alpha} \rbrace_{\alpha=0,1,2,3}$ denotes the 
	$\underline{\mathcal{H}}$-tangent vectorfields from \eqref{E:PROJECTIONSOFCARTESIANVECTORFIELDS}.
	\item $\angpartial \pmb{\upxi}$ denotes the type $\binom{m}{n+1}$ tensorfield with Cartesian spatial components 
	$\angpartialarg{\beta_1} \pmb{\upxi}_{\beta_2 \cdots \beta_{n+1}}^{\alpha_1 \cdots \alpha_m}$,
	where $\lbrace \angpartialarg{\alpha} \rbrace_{\alpha=0,1,2,3}$ denotes the 
	$\mathcal{S}_{\Timefunction}$-tangent vectorfields from \eqref{E:PROJECTIONSOFCARTESIANVECTORFIELDS}.
\end{itemize}

\begin{itemize}
\item Given any $\Sigma_t$-tangent type $\binom{m}{n}$ tensorfield $\upxi$ with Cartesian components 
$\upxi_{b_1 \cdots b_n}^{a_1 \cdots a_m}$,
$\partial \upxi$ denotes the $\Sigma_t$-tangent type $\binom{m}{n+1}$ tensorfield 
with Cartesian spatial components 
$\partial_{b_1} \upxi_{b_2 \cdots b_{n+1}}^{a_1 \cdots a_m}$.
\end{itemize}

\subsubsection{The Euclidean metrics on $\mathbb{R}^{1+3}$ and $\Sigma_t$}
We let $\euc$ denote the standard Euclidean metric on $\mathbb{R}^{1+3}$, i.e., relative to the Cartesian coordinates, 
$\euc_{\alpha \beta} = \updelta_{\alpha \beta}$, where $\updelta_{\alpha \beta}$ is the Kronecker delta. Moreover, $\euc^{-1}$ denotes the corresponding 
inverse metric. Similarly, $\euct_{ij} = \updelta_{ij}$ denotes the standard Euclidean metric on $\Sigma_t$, and $\euct^{-1}$ denotes 
the corresponding inverse metric.

\subsubsection{Pointwise seminorms and norms}
\label{SSS:GRADIENTSANDPOINTWISENORMS}

\begin{itemize}
\item Given any spacetime tensorfield $\upxi_{\beta_1 \cdots \beta_n}^{\alpha_1 \cdots \alpha_m}$,
\[|\upxi|_{\euc}
	:= 
	\sqrt{\euc_{\alpha_1 \gamma_1} \cdots \euc_{\alpha_m \gamma_m} (\euc^{-1})^{\beta_1 \delta_1} \cdots (\euc^{-1})^{\beta_n \delta_n}
	\upxi_{\beta_1 \cdots \beta_n}^{\alpha_1 \cdots \alpha_m} \upxi_{\delta_1 \cdots \delta_n}^{\gamma_1 \cdots \gamma_m}}
\]
denotes its pointwise norm with respect to $\euc$. 
\item Given any $\Sigma_t$-tangent spacetime tensorfield $\upxi_{b_1 \cdots b_n}^{a_1 \cdots a_m}$,
\[
	|\upxi|_g
	:= \sqrt{g_{a_1 c_1} \cdots g_{a_m c_m} (g^{-1})^{b_1 d_1} \cdots (g^{-1})^{b_n d_n}
	\upxi_{b_1 \cdots b_n}^{a_1 \cdots a_m} \upxi_{d_1 \cdots d_n}^{c_1 \cdots c_m}}
\]
denotes its pointwise norm with respect to $g$. We define the pointwise norm of $\upxi$ with respect
to the Euclidean metric $\euct$ in an analogous fashion.
For example, relative to the Cartesian spatial 
coordinates, we have $|\partial \vortrenormalized|_{\euct}^2 = \sum_{a,b=1}^3 (\partial_a \vortrenormalized^b)^2$.
\item Given any type $\binom{m}{n}$ tensorfield 
$\upxi_{\beta_1 \cdots \beta_n}^{\alpha_1 \cdots \alpha_m}$,
\[
	|\upxi|_{\topfirstfund}
	:= 
	\sqrt{\topfirstfund_{\alpha_1 \gamma_1} \cdots \topfirstfund_{\alpha_m \gamma_m} (\topfirstfund^{-1})^{\beta_1 \delta_1} 
	\cdots (\topfirstfund^{-1})^{\beta_n \delta_n}
	\upxi_{\beta_1 \cdots \beta_n}^{\alpha_1 \cdots \alpha_m} \upxi_{\delta_1 \cdots \delta_n}^{\gamma_1 \cdots \gamma_m}}
\]
denotes its pointwise seminorm with respect to $\topfirstfund$.
Note that $|\cdot|_{\topfirstfund}$ is a norm on $\widetilde{\Sigma}_{\Timefunction}$-tangent tensorfields.
\item When $\underline{\mathcal{H}}$ is $\gfour$-spacelike, 
given any type $\binom{m}{n}$ tensorfield 
$\upxi_{\beta_1 \cdots \beta_n}^{\alpha_1 \cdots \alpha_m}$,
\[
	|\upxi|_{\sidefirstfund}
	:= 
	\sqrt{\sidefirstfund_{\alpha_1 \gamma_1} \cdots \sidefirstfund_{\alpha_m \gamma_m} (\sidefirstfund^{-1})^{\beta_1 \delta_1} \cdots 
	(\sidefirstfund^{-1})^{\beta_n \delta_n}
	\upxi_{\beta_1 \cdots \beta_n}^{\alpha_1 \cdots \alpha_m} \upxi_{\delta_1 \cdots \delta_n}^{\gamma_1 \cdots \gamma_m}}
\]
denotes its pointwise seminorm with respect to $\sidefirstfund$.
Note that $|\cdot|_{\sidefirstfund}$ is a norm on $\underline{\mathcal{H}}$-tangent tensorfields.
\item Given any type $\binom{m}{n}$ tensorfield $\upxi_{\beta_1 \cdots \beta_n}^{\alpha_1 \cdots \alpha_m}$,
\[
|\upxi|_{\gsphere}
:= 
\sqrt{\gsphere_{\alpha_1 \gamma_1} \cdots \gsphere_{\alpha_m \gamma_m} (\gsphere^{-1})^{\beta_1 \delta_1} \cdots (\gsphere^{-1})^{\beta_n \delta_n}
\upxi_{\beta_1 \cdots \beta_n}^{\alpha_1 \cdots \alpha_m} \upxi_{\delta_1 \cdots \delta_n}^{\gamma_1 \cdots \gamma_m}}
\]
denotes its pointwise seminorm with respect to $\gsphere$.
Note that $|\cdot|_{\gsphere}$ is a norm on $\mathcal{S}_{\Timefunction}$-tangent tensorfields.
\end{itemize}

If $\vec{\varphi}$ is an array of $\Sigma_t$-tangent tensorfields, 
then $|\vec{\varphi}|_{\euct}^2$ denotes the sum of the squares of the norms $|\cdot|_{\euct}$
of the elements of the array. Norms of arrays with respect to other metrics are defined in an analogous fashion, 
e.g., if $\underline{\mathcal{H}}$ is $\gfour$-spacelike
and if $\vec{\varphi}$ is an array of $\underline{\mathcal{H}}$-tangent tensorfields, 
then $|\vec{\varphi}|_{\sidefirstfund}^2$ denotes the sum of the squares of the seminorms $|\cdot|_{\sidefirstfund}$
of the elements of the array.

\section{The coercive quadratic form, the elliptic-hyperbolic divergence identities, and preliminary analysis of the boundary integrands}
\label{S:ELLIPTICHYPERBOLICDIVERGENCEIDENTITYANDPRELIMANALYSISOFOUNDARY}
Our first main goal in this section is to define the coercive quadratic form $\mathscr{Q}$
featured in our main integral identities (i.e., in Theorem~\ref{T:MAINREMARKABLESPACETIMEINTEGRALIDENTITY})
and to exhibit its coerciveness properties.
Our second main result in this section is Lemma~\ref{L:BASICDIVERGENCEIDENTTIY},
which provides the elliptic Hodge-type divergence identity
along $\widetilde{\Sigma}_{\Timefunction}$
that forms the starting point for the proof of Theorem~\ref{T:MAINREMARKABLESPACETIMEINTEGRALIDENTITY}.
Our third main result in this section is Lemma~\ref{L:PRELIMINARYANALYSISOFBOUNDARYINTEGRAND},
in which we provide a preliminary analysis of the boundary integrand terms in the elliptic Hodge-type identities.
The lemma shows 
that most boundary integrand terms involve only $\underline{\mathcal{H}}$-tangential derivatives
of terms enjoying a compatible amount of regularity.
The remaining boundary terms
also enjoy these same good properties, 
but the proof requires 
an integration with respect to $\Timefunction$ and substantial additional arguments 
that, unlike the results of the present section, 
exploit the special properties of the compressible Euler formulation
provided by Theorem~\ref{T:GEOMETRICWAVETRANSPORTSYSTEM};
see Prop.\,\ref{P:STRUCTUREOFERRORINTEGRALS} and Theorem~\ref{T:STRUCTUREOFERRORTERMS}
for the detailed statements showing that \emph{all} boundary integrand terms have these good properties.

\subsection{A solution-adapted coercive quadratic form}
\label{SS:SOLUTIONADPATEDCOERCIVEQUADRATICFORM}

\subsubsection{A $\widetilde{\Sigma}_{\Timefunction}$-tangent vectorfield that arises in the analysis}
\label{SSS:PROJECTIONOFPONTTOTILDESIGMA}
The coercive quadratic form that we provide in Def.\,\ref{D:QUADRATICFORMSFORCONTROLLINGFIRSTDERIVATIVESOFSPECIFICVORTICITYANDENTROPYGRADIENT}
involves the $\widetilde{\Sigma}_{\Timefunction}$-tangent vectorfield $\projectedtransport$ from the following definition.

\begin{definition}[A rescaled, $\widetilde{\Sigma}_{\Timefunction}$-projected version of $\Transport$]
\label{D:PROJECTIONOFPONTTOTILDESIGMA}
	Let $\lengthoftophypnorm > 0$ be the scalar function defined in \eqref{E:LENGTHOFTOPHYPNORM},
	and let $\topproject$ be the $\gfour$-orthogonal projection onto $\widetilde{\Sigma}_{\Timefunction}$ defined in \eqref{E:TOPPROJECT}.
	We define $\projectedtransport$ to be the $\widetilde{\Sigma}_{\Timefunction}$-tangent vectorfield with the following
	Cartesian components:
\begin{align} \label{E:PROJECTIONOFTRANSPORTONTTOTILDESIGMA}
		\projectedtransport^{\alpha}
		& 
		:= \lengthoftophypnorm (\topproject \Transport)^{\alpha}
		= 
		\lengthoftophypnorm \topproject_{\ \beta}^{\alpha} \Transport^{\beta}.
\end{align}
\end{definition}

In the next lemma, we derive some simple identities involving $\projectedtransport$.

\begin{lemma}[Identities involving $\projectedtransport$]
\label{L:IDENTITIESINVOLVINGSIGMATILDEPROJECTEDTRANSPORT}
Let $\projectedtransport$ be vectorfield
defined by \eqref{E:PROJECTIONOFTRANSPORTONTTOTILDESIGMA},
let $\lengthoftophypnorm > 0$ be the scalar function defined in \eqref{E:LENGTHOFTOPHYPNORM},
let $\tophypnorm$ be the vectorfield from Def.\,\ref{D:HYPNORMANDSPHEREFORMDEFS},
and let $\hat{\tophypnorm}$ be the vectorfield from Def.\,\ref{D:LENGTHOFVARIOUSVECTORFIELDSETC}.
Then the following identities hold:
\begin{align} \label{E:SIGMATILDEPROJECTEDTRANSPORTIDENTITY}
\projectedtransport^{\alpha}
& 
=
\lengthoftophypnorm
\Transport^{\alpha}
-
\frac{1}{\lengthoftophypnorm}
\tophypnorm^{\alpha}
=
\lengthoftophypnorm
\Transport^{\alpha}
-
\hat{\tophypnorm}^{\alpha},
\end{align}
	
	\begin{align}  \label{E:LENGTHOFPROJECTIONOFPONTTOTILDESIGMA}
	|\projectedtransport|_{\topfirstfund}^2
	&:=
	\topfirstfund(\projectedtransport,\projectedtransport)
	= 1 - \lengthoftophypnorm^2
	< 1.
\end{align}

	Moreover, the following identities hold, where
	$\topfirstfund^{-1}$ is the inverse first fundamental form of $\widetilde{\Sigma}_{\Timefunction}$ from 
	Def.\,\ref{D:FIRSTFUNDAMENTALFORMSANDPROJECTIONS},
	$g^{-1}$ is the inverse first fundamental form of $\Sigma_t$ from 
	Def.\,\ref{D:FIRSTFUNDAMENTALFORMSANDPROJECTIONS},
	and
	$\Sigmatproject$ is the
	$\gfour$-orthogonal projection onto $\Sigma_t$ from Def.\,\ref{D:FIRSTFUNDAMENTALFORMSANDPROJECTIONS}:
	\begin{subequations}
	\begin{align} \label{E:DECOMPOSETILDEMETRICINTOSIGMATMETRICMINUSATANGENTIALTENSORPRODUCT}
	(\topfirstfund^{-1})^{\alpha \beta}
	& 
	= 
	(g^{-1})^{\alpha \beta}
	+
	\projectedtransport^{\alpha} \projectedtransport^{\beta}
	-
	\lengthoftophypnorm \Transport^{\alpha} \projectedtransport^{\beta}
	-
	\lengthoftophypnorm \projectedtransport^{\alpha} \Transport^{\beta} 
	-
	(1 - \lengthoftophypnorm^2)
	\Transport^{\alpha}
	\Transport^{\beta},
		\\
	\topproject_{\ \beta}^{\alpha}
	& 
	= 
	\Sigmatproject_{\ \beta}^{\alpha}
	+
	\projectedtransport^{\alpha} \projectedtransport_{\beta}
	-
	\lengthoftophypnorm \Transport^{\alpha} \projectedtransport_{\beta}
	-
	\lengthoftophypnorm \projectedtransport^{\alpha} \Transport_{\beta} 
	-
	(1 - \lengthoftophypnorm^2)
	\Transport^{\alpha}
	\Transport_{\beta}.
	\label{E:DECOMPOSETILDEPROJECTIONINTOSIGMATPROJECTIONMINUSATANGENTIALTENSORPRODUCT}
\end{align}
\end{subequations}

Finally, if $\SigmatTan$ is a $\Sigma_t$-tangent vectorfield defined on $\mathcal{M}$, 
then the following identity holds:
\begin{align} \label{E:SIGMATILDEDIVERGENCEIDENTITYFORSIGMATTANGENTVECTORFIELDS}
	\topproject_{\ \beta}^{\alpha}
	\partial_{\alpha} \SigmatTan^{\beta}
	 & =
		\partial_a \SigmatTan^a
		+
	 \projectedtransport_{\alpha}
		\projectedtransport \SigmatTan^{\alpha}
	-
	\lengthoftophypnorm \projectedtransport_{\alpha}
	 \Transport \SigmatTan^{\alpha}.
\end{align}
\end{lemma}	

\begin{proof}
	\eqref{E:SIGMATILDEPROJECTEDTRANSPORTIDENTITY} is a simple consequence of definition \eqref{E:PROJECTIONOFTRANSPORTONTTOTILDESIGMA},
	\eqref{E:EQUIVALENTFUTURENORMALTOTOPHYPERSURFACE},
	\eqref{E:TOPPROJECT},
	and 
	\eqref{E:TOPUNITHYPERSURFACENORMAL}.
	\eqref{E:LENGTHOFPROJECTIONOFPONTTOTILDESIGMA} then follows from
	\eqref{E:SIGMATILDEPROJECTEDTRANSPORTIDENTITY},
	\eqref{E:TRANSPORTISUNITLENGTHANDTIMELIKE},
	\eqref{E:EQUIVALENTFUTURENORMALTOTOPHYPERSURFACE},
	and \eqref{E:LENGTHOFTOPHYPNORM}.
	
	To prove \eqref{E:DECOMPOSETILDEMETRICINTOSIGMATMETRICMINUSATANGENTIALTENSORPRODUCT}, we first use
	\eqref{E:INVERSEFIRSTFUNDAMENTALFORMSIGMAT}
	and
	\eqref{E:TOPINVERSEFIRSTFUNDAMENTALFORMHYPERSURFACE}
	to express
	$
	(\topfirstfund^{-1})^{\alpha \beta}
	=
	(g^{-1})^{\alpha \beta}
	-
	\Transport^{\alpha} \Transport^{\beta}
	+
	\hat{\tophypnorm}^{\alpha} \hat{\tophypnorm}^{\beta}
	$.
	We then use \eqref{E:SIGMATILDEPROJECTEDTRANSPORTIDENTITY}
	to express the two factors of $\hat{\tophypnorm}$ 
	in terms of $\projectedtransport$ and $\Transport$,
	which in total yields \eqref{E:DECOMPOSETILDEMETRICINTOSIGMATMETRICMINUSATANGENTIALTENSORPRODUCT}.
	
	\eqref{E:DECOMPOSETILDEPROJECTIONINTOSIGMATPROJECTIONMINUSATANGENTIALTENSORPRODUCT} follows from 
	lowering the index $\beta$ in \eqref{E:DECOMPOSETILDEMETRICINTOSIGMATMETRICMINUSATANGENTIALTENSORPRODUCT} with $\gfour$.
	
	\eqref{E:SIGMATILDEDIVERGENCEIDENTITYFORSIGMATTANGENTVECTORFIELDS}
	follows from \eqref{E:DECOMPOSETILDEPROJECTIONINTOSIGMATPROJECTIONMINUSATANGENTIALTENSORPRODUCT},
	\eqref{E:SIGMATPROJECT},
	and the fact that by \eqref{E:TRANSPORTONEFORMIDENTITY}, 
	$\Transport_{\beta} \partial_{\alpha} \SigmatTan^{\beta} = 0$ when $\SigmatTan$ is $\Sigma_t$-tangent.
\end{proof}

\subsubsection{A solution-adapted coercive quadratic form}
\label{SSS:SOLUTIONADAPATEDCOERCIVEQUADRATICFORMS}
We now define the quadratic form $\mathscr{Q}$ that we use to control the first derivatives of the
specific vorticity and entropy gradient. $\mathscr{Q}$ is the main integrand factor on the left-hand 
side of the integral identities provided by Theorem~\ref{T:MAINREMARKABLESPACETIMEINTEGRALIDENTITY}.
In Lemma~\ref{L:POSITIVITYPROPERTIESOFVARIOUSQUADRATICFORMS}, we exhibit the coerciveness of $\mathscr{Q}$.

\begin{definition}[A solution-adapted coercive quadratic form for controlling the first derivatives of $\vortrenormalized$ and $\GradEnt$]
\label{D:QUADRATICFORMSFORCONTROLLINGFIRSTDERIVATIVESOFSPECIFICVORTICITYANDENTROPYGRADIENT}
We define the quadratic form $\mathscr{Q}(\cdot,\cdot)$
on type $\binom{1}{1}$ tensorfields $\mathbf{U}_{\ \alpha}^{\beta}$ as follows relative to the Cartesian coordinates,
where $\topfirstfund$ and $\topfirstfund^{-1}$ are respectively the first fundamental form 
and inverse first fundamental form
of $\widetilde{\Sigma}_{\Timefunction}$ from 
Def.\,\ref{D:FIRSTFUNDAMENTALFORMSANDPROJECTIONS}:
\begin{align} \label{E:MAINQUADRATICFORMFORCONTROLLINGFIRSTDERIVATIVESOFSPECIFICVORTICITYANDENTROPYGRADIENT}
		\mathscr{Q}(\mathbf{U},\mathbf{U})
		& 
		:=
		|\topproject \mathbf{U}|_{\topfirstfund}^2
		-
		(\projectedtransport_{\beta} \projectedtransport^{\alpha} \mathbf{U}_{\ \alpha}^{\beta})^2
			\\
	& \ \
		+
		\topfirstfund_{\gamma \delta} (\Transport^{\alpha} \mathbf{U}_{\ \alpha}^{\gamma}) (\Transport^{\beta} \mathbf{U}_{\ \beta}^{\delta})
		+
		(\topfirstfund^{-1})^{\alpha \beta} (\Transport_{\gamma} \mathbf{U}_{\ \alpha}^{\gamma}) (\Transport_{\delta} \mathbf{U}_{\ \beta}^{\delta})
		+
		(\Transport_{\beta} \Transport^{\alpha} \mathbf{U}_{\ \alpha}^{\beta})^2.
		\notag
\end{align}

\end{definition}

\subsubsection{The positive definiteness of $\mathscr{Q}(\cdot,\cdot)$}
\label{SSS:COERCIVENESSOFQUADRATICFORM}
In the next lemma, we exhibit the positive definite nature of $\mathscr{Q}(\cdot,\cdot)$.

\begin{lemma}[Positivity properties of the quadratic form]
	\label{L:POSITIVITYPROPERTIESOFVARIOUSQUADRATICFORMS}
	Recall that $\widetilde{\Sigma}_{\Timefunction}$ is $\gfour$-spacelike by assumption. Then 
	as a consequence, the quadratic form $\mathscr{Q}(\mathbf{U},\mathbf{U})$
	from Def.\,\ref{D:QUADRATICFORMSFORCONTROLLINGFIRSTDERIVATIVESOFSPECIFICVORTICITYANDENTROPYGRADIENT}
	is positive definite on the space of type $\binom{1}{1}$ tensorfields $\mathbf{U}_{\ \alpha}^{\beta}$.
	
	Moreover, if $\SigmatTan$ is a $\Sigma_t$-tangent vectorfield, then the following identity holds
	(i.e., with $\partial_{\alpha} \SigmatTan^{\beta}$ in the role of $\mathbf{U}_{\ \alpha}^{\beta}$),
	where
	where $\topproject$ 
	is the $\gfour$-orthogonal projection onto $\widetilde{\Sigma}_{\Timefunction}$
	from 
	Def.\,\ref{D:FIRSTFUNDAMENTALFORMSANDPROJECTIONS}
	and $\topfirstfund$ the first fundamental form 
	of $\widetilde{\Sigma}_{\Timefunction}$ from 
	Def.\,\ref{D:FIRSTFUNDAMENTALFORMSANDPROJECTIONS}:
	\begin{align} \label{E:IDENTITYMAINQUADRATICFORMFORCONTROLLINGFIRSTDERIVATIVESOFSPECIFICVORTICITYANDENTROPYGRADIENT}
		\mathscr{Q}(\pmb{\partial} \SigmatTan,\pmb{\partial} \SigmatTan)
		& 
		=
		|\topproject(\pmb{\partial} \SigmatTan)|_{\topfirstfund}^2
		-
		(\projectedtransport_{\alpha} \projectedtransport \SigmatTan^{\alpha})^2
		+
		\topfirstfund_{\alpha \beta} (\Transport \SigmatTan^{\alpha}) (\Transport \SigmatTan^{\beta}).
\end{align}
\end{lemma}

\begin{proof}
	\eqref{E:IDENTITYMAINQUADRATICFORMFORCONTROLLINGFIRSTDERIVATIVESOFSPECIFICVORTICITYANDENTROPYGRADIENT} 
	is a straightforward consequence of definition \eqref{E:MAINQUADRATICFORMFORCONTROLLINGFIRSTDERIVATIVESOFSPECIFICVORTICITYANDENTROPYGRADIENT} 
	and the identity \eqref{E:TRANSPORTONEFORMIDENTITY},
	which in particular implies that
	$\Transport_{\beta} \partial_{\alpha} \SigmatTan^{\beta} = 0$ when $\SigmatTan$ is $\Sigma_t$-tangent.
	
	To prove the positivity of $\mathscr{Q}$,
	we first use
	$\topfirstfund$-Cauchy--Schwarz,
	\eqref{E:PROJECTIONOFTRANSPORTONTTOTILDESIGMA},
	and
	\eqref{E:LENGTHOFPROJECTIONOFPONTTOTILDESIGMA}
	to deduce
	that the first two terms 
	$|\topproject \mathbf{U}|_{\topfirstfund}^2
		-
		(\projectedtransport_{\beta} \projectedtransport^{\alpha} \mathbf{U}_{\ \alpha}^{\beta})^2$
	on RHS~\eqref{E:MAINQUADRATICFORMFORCONTROLLINGFIRSTDERIVATIVESOFSPECIFICVORTICITYANDENTROPYGRADIENT} 
	are collectively positive semi-definite on the space of type $\binom{1}{1}$ tensorfields $\mathbf{U}_{\ \alpha}^{\beta}$
	and positive definite on the subspace of such tensorfields that are tangent to $\widetilde{\Sigma}_{\Timefunction}$.
	Moreover, the each of the last three products on RHS~\eqref{E:MAINQUADRATICFORMFORCONTROLLINGFIRSTDERIVATIVESOFSPECIFICVORTICITYANDENTROPYGRADIENT}  
	are manifestly positive semi-definite on the space of type $\binom{1}{1}$ tensorfields.
	In particular, for all $\mathbf{U}$, we have $\mathscr{Q}(\mathbf{U},\mathbf{U}) \geq 0$.
	Thus, to demonstrate
	the desired positive definiteness of $\mathscr{Q}(\mathbf{U},\mathbf{U})$,
	it suffices to show that $\mathscr{Q}(\mathbf{U},\mathbf{U}) = 0 \implies \mathbf{U} = 0$.

	To proceed, we assume that $\mathscr{Q}(\mathbf{U},\mathbf{U}) = 0$.
	From the discussion in the previous paragraph,
	it follows that
		$|\topproject \mathbf{U}|_{\topfirstfund}^2
		-
		(\projectedtransport_{\beta} \projectedtransport^{\alpha} \mathbf{U}_{\ \alpha}^{\beta})^2
		= 0$
	and that $\topproject \mathbf{U} = 0$.
	From this identity,
	\eqref{E:SIGMATILDEPROJECTEDTRANSPORTIDENTITY},
	and the $\widetilde{\Sigma}_{\Timefunction}$-tangent property of $\projectedtransport$,
	we see that contractions of $\mathbf{U}$ against
	$\Transport$ are equal, up to a scalar function multiple,
	to contractions against $\hat{\tophypnorm}$.
	From this fact,
	the fact that
	the third term 
	$
	\topfirstfund_{\gamma \delta} 
	(\Transport^{\alpha} \mathbf{U}_{\ \alpha}^{\gamma}) 
	(\Transport^{\beta} \mathbf{U}_{\ \beta}^{\delta})
	$
	on RHS~\eqref{E:MAINQUADRATICFORMFORCONTROLLINGFIRSTDERIVATIVESOFSPECIFICVORTICITYANDENTROPYGRADIENT} must also vanish,
	and the fact that $\topfirstfund$ is positive definite on the space of $\widetilde{\Sigma}_{\Timefunction}$-tangent vectorfields
	and the fact that $\topfirstfund$ vanishes when contracted with $\hat{\tophypnorm}$, 
	we see that the vectorfield with components $\Transport^{\alpha} \mathbf{U}_{\ \alpha}^{\beta}$ must be proportional to 
	$\hat{\tophypnorm}^{\beta}$.
	From these facts,
	the fact that
	the last term 
	$(\Transport_{\beta} \Transport^{\alpha} \mathbf{U}_{\ \alpha}^{\beta})^2$
	on RHS~\eqref{E:MAINQUADRATICFORMFORCONTROLLINGFIRSTDERIVATIVESOFSPECIFICVORTICITYANDENTROPYGRADIENT} must also vanish,
	and the fact that $\Transport_{\kappa} \hat{\tophypnorm}^{\kappa} \neq 0$ (see \eqref{E:INNERPRODUCTOFTRANPORTANDFUTUREUNITNORMALTOTOPHYPERSURFACE}),
	we find that $\hat{\tophypnorm}^{\alpha} \mathbf{U}_{\ \alpha}^{\beta} = 0$.
	Similar reasoning, based on exploiting the positive semi-definiteness of
	the fourth and last terms on RHS~\eqref{E:MAINQUADRATICFORMFORCONTROLLINGFIRSTDERIVATIVESOFSPECIFICVORTICITYANDENTROPYGRADIENT},
	leads to the identity 
	$\hat{\tophypnorm}_{\beta} \mathbf{U}_{\ \alpha}^{\beta} = 0$.
	We have therefore shown that the $\gfour$-orthogonal projection of $\mathbf{U}$ onto $\widetilde{\Sigma}_{\Timefunction}$ vanishes,
	and that any contraction of $\mathbf{U}$ against the unit normal to $\widetilde{\Sigma}_{\Timefunction}$ (namely $\hat{\tophypnorm}$) vanishes.
	This implies that $\mathbf{U} = 0$, which completes the proof of the lemma. 
	
\end{proof}

\subsection{Definition of the elliptic-hyperbolic currents}
\label{SS:DEFINITIONSOFHYPERBOLICELLIPTICCURRENT}
The $\widetilde{\Sigma}_{\Timefunction}$-tangent vectorfields $J$ in the next definition play a key role in our analysis.
We refer to them as ``elliptic-hyperbolic currents.'' We motivate this terminology as follows:
even though Lemma~\ref{L:BASICDIVERGENCEIDENTTIY} shows 
that $J$ is tied to elliptic-type identities along $\widetilde{\Sigma}_{\Timefunction}$, 
the complete set of structures that we need to control the boundary terms 
become manifest only in Prop.\,\ref{P:STRUCTUREOFERRORINTEGRALS} and Theorem~\ref{T:STRUCTUREOFERRORTERMS}, 
after we integrate the elliptic identities in time
and exploit some special structural features found in the
``\emph{hyperbolic} part'' of the equations of Theorem~\ref{T:GEOMETRICWAVETRANSPORTSYSTEM}.

\begin{definition}[Elliptic-hyperbolic current]
\label{D:ELLIPTICHYPERBOLICCURRENT}
Given a $\Sigma_t$-tangent vectorfield $\SigmatTan$ (which in our forthcoming applications will be equal to $\vortrenormalized$ or $\GradEnt$), 
we define $J[\SigmatTan]$ to be the $\widetilde{\Sigma}_{\Timefunction}$-tangent vectorfield 
with the following Cartesian components,
where $\topproject$ is the $\widetilde{\Sigma}_{\Timefunction}$ projection tensorfield
defined in \eqref{E:TOPPROJECT}:
\begin{align} \label{E:NEWELLIPTICHYPERBOLICCURRENT}
	J^{\alpha}[\SigmatTan]
	& := \SigmatTan^{\gamma} \topproject_{\ \gamma}^{\lambda} \topproject_{\ \kappa}^{\alpha} \partial_{\lambda} \SigmatTan^{\kappa}
		-
		\SigmatTan^{\gamma} \topproject_{\ \gamma}^{\alpha} 
		\topproject_{\ \lambda}^{\kappa}
		\partial_{\kappa} \SigmatTan^{\lambda}
		= \SigmatTan^{\gamma} \topproject_{\ \kappa}^{\alpha} \toppartialarg{\gamma} \SigmatTan^{\kappa}
			-
		\SigmatTan^{\gamma} \topproject_{\ \gamma}^{\alpha} 
		\toppartialarg{\lambda} 
		\SigmatTan^{\lambda},
	\end{align}
	where to obtain the last equality in \eqref{E:NEWELLIPTICHYPERBOLICCURRENT}, we used \eqref{E:PROJECTIONSOFCARTESIANVECTORFIELDS}.
\end{definition}

\begin{remark}[{$J[\SigmatTan]$ is $\widetilde{\Sigma}_{\Timefunction}$-tangent}]
	We stress that $J[\SigmatTan]$ is $\widetilde{\Sigma}_{\Timefunction}$-tangent, even though $\SigmatTan$ is $\Sigma_t$-tangent.
\end{remark}

\subsection{The elliptic divergence identity}
\label{SS:ELLIPTICDIVERGENCE}
In the next lemma, we derive a covariant divergence identity for the elliptic-hyperbolic current $J[\SigmatTan]$.
With future applications in mind, we have allowed for the presence of a ``weight function''
$\weight$ in the identity.
We have carefully organized the structure of the ``main terms'' in the divergence identity so that later on,
with the help of Lemma~\ref{L:POSITIVITYPROPERTIESOFVARIOUSQUADRATICFORMS}, we will be able to show
that the quadratic form $\mathscr{Q}(\pmb{\partial} \SigmatTan,\pmb{\partial} \SigmatTan)$ 
(see in particular \eqref{E:IDENTITYMAINQUADRATICFORMFORCONTROLLINGFIRSTDERIVATIVESOFSPECIFICVORTICITYANDENTROPYGRADIENT})
can be used to derive coercive, spacetime $L^2$-type control over the first derivatives of $\SigmatTan$.
We have carefully organized the structure of the ``error terms'' in the divergence identity \eqref{E:NEWSTANDARDDIVERGENCEIDENTITYFORELLIPTICHYPERBOLICCURRENT} 
so that later on,
with the help of the equations of Theorem~\ref{T:GEOMETRICWAVETRANSPORTSYSTEM},
we will be able to see that these terms exhibit the remarkable structures that are needed in applications.
In the proof of our main integral identities, namely Theorem~\ref{T:MAINREMARKABLESPACETIMEINTEGRALIDENTITY}, 
we will use the divergence identity
when we apply the divergence theorem.

\begin{lemma}[Covariant divergence identity along $\widetilde{\Sigma}_{\Timefunction}$ satisfied by {$J[\SigmatTan]$}]
	\label{L:BASICDIVERGENCEIDENTTIY}
	Let $\SigmatTan$ be a $\Sigma_t$-tangent vectorfield defined on $\mathcal{M}$,
	let $J[\SigmatTan]$ be the $\widetilde{\Sigma}_{\Timefunction}$-tangent vectorfield 
	from Def.\,\ref{D:ELLIPTICHYPERBOLICCURRENT}, and let $\weight$ be an arbitrary scalar function.  
	Then the following
	divergence identity holds 
	relative to the Cartesian coordinates:
	\begin{align} \label{E:NEWSTANDARDDIVERGENCEIDENTITYFORELLIPTICHYPERBOLICCURRENT}
			\weight
			|\pmb{\partial} \SigmatTan|_{\topfirstfund}^2
			-
			\weight
			(\projectedtransport_{\alpha}
			\projectedtransport \SigmatTan^{\alpha})^2
			& 
			:=
			\weight
			(\topfirstfund^{-1})^{\alpha \gamma}
			\topfirstfund_{\beta \delta}
			(\partial_{\alpha} \SigmatTan^{\beta})
			\partial_{\gamma} \SigmatTan^{\delta}
			-
			\weight
			(\projectedtransport_{\alpha}
			\projectedtransport \SigmatTan^{\alpha})^2
				\\
		& =
			\widetilde{\nabla}_{\alpha} \left(\weight J^{\alpha}[\SigmatTan] \right)
			\notag	\\
		& \ \
			+
			\weight
			\mathfrak{J}_{(\widetilde{Antisym})}[\pmb{\partial} V,\pmb{\partial} V]
			+
			\weight
			\mathfrak{J}_{(Div)}[\pmb{\partial} V,\pmb{\partial} V]
			+
			\weight
			\mathfrak{J}_{(Coeff)}[\SigmatTan,\pmb{\partial} \SigmatTan]
				\notag \\
		& \ \
			+
			\mathfrak{J}_{(\pmb{\partial} \weight)}[\SigmatTan,\pmb{\partial} \SigmatTan],
				\notag
		\end{align}
		where
		\begin{align} \label{E:ANTISYMMETRICERRORTERMDIVERGENCEOFSIGMATILDECURRENT}
			\mathfrak{J}_{(\widetilde{Antisym})}[\pmb{\partial} V,\pmb{\partial} V]
			& :=
			\frac{1}{2}
			|d (\SigmatTan_{\flat})|_{\topfirstfund}^2
			=
			\frac{1}{2}
			(\topfirstfund^{-1})^{\alpha \gamma}
			(\topfirstfund^{-1})^{\beta \delta}
			\left\lbrace
				\partial_{\alpha} \SigmatTan_{\beta}
				-
				\partial_{\beta} \SigmatTan_{\alpha}
			\right\rbrace
			\left\lbrace
				\partial_{\gamma} \SigmatTan_{\delta}
				-
				\partial_{\delta} \SigmatTan_{\gamma}
			\right\rbrace,
		\end{align}
		
		\begin{align} \label{E:DIVSYMMETRICERRORTERMDIVERGENCEOFSIGMATILDECURRENT}
			\mathfrak{J}_{(Div)}[\pmb{\partial} V,\pmb{\partial} V]
			& :=
			(\partial_a \SigmatTan^a)^2
			+
			\lengthoftophypnorm^2
			(\projectedtransport_{\alpha} \Transport \SigmatTan^{\alpha})^2
				\notag \\
			& \ \
			+
			2 
			(\projectedtransport_{\alpha}\projectedtransport \SigmatTan^{\alpha})
			\partial_a \SigmatTan^a
			-
			2
			\lengthoftophypnorm 
			(\projectedtransport_{\alpha} \Transport \SigmatTan^{\alpha})
			\projectedtransport_{\alpha}\projectedtransport \SigmatTan^{\alpha}
			-
			2
			\lengthoftophypnorm 
			(\partial_a \SigmatTan^a)
			\projectedtransport_{\alpha} 
			\Transport \SigmatTan^{\alpha},
	\end{align}
	
	\begin{align} \label{E:THIRDERRORTERMDIVERGENCEOFSIGMATILDECURRENT}
		\mathfrak{J}_{(Coeff)}[\SigmatTan,\pmb{\partial} \SigmatTan]
		& =
			\SigmatTan^{\alpha} 
			\gfour_{\beta \gamma}
			(\toppartialarg{\alpha} \hat{\tophypnorm}^{\beta})
			\hat{\tophypnorm} \SigmatTan^{\gamma}
			-
			\SigmatTan_{\alpha} 
			(\toppartialarg{\beta} \hat{\tophypnorm}^{\alpha})
			\hat{\tophypnorm} \SigmatTan^{\beta}
				\\
		& \ \
				+
		\SigmatTan^{\alpha} 
		\hat{\tophypnorm}_{\alpha} 
		(\toppartialarg{\beta} \hat{\tophypnorm}^{\beta}) 
		\toppartialarg{\gamma} \SigmatTan^{\gamma}
			-
			\SigmatTan^{\alpha} \hat{\tophypnorm}_{\alpha} 
			(\toppartialarg{\beta} \hat{\tophypnorm}^{\gamma}) 
			\toppartialarg{\gamma} 
			\SigmatTan^{\beta} 
				\notag \\
		& \ \
		+
		\SigmatTan^{\alpha} 
		\hat{\tophypnorm}_{\beta}
		(\toppartialarg{\alpha} \hat{\tophypnorm}^{\gamma})
		\toppartialarg{\gamma} \SigmatTan^{\beta}
			-
			\SigmatTan^{\alpha} 
			\hat{\tophypnorm}_{\beta} 
			(\toppartialarg{\gamma} \hat{\tophypnorm}^{\gamma})  
			\toppartialarg{\alpha} \SigmatTan^{\beta}
		\notag
			\\
		&  \ \
			+
			\SigmatTan^{\alpha} 
			\hat{\tophypnorm}^{\beta}
			(\toppartialarg{\alpha} \gfour_{\beta \gamma})
			\hat{\tophypnorm} \SigmatTan^{\gamma}
			-
			\SigmatTan^{\alpha} 
			\hat{\tophypnorm}^{\beta}
			(\toppartialarg{\gamma} \gfour_{\alpha \beta})
			\hat{\tophypnorm} \SigmatTan^{\gamma}
				\notag \\
		& \ \
			+
			\frac{1}{2}
			\SigmatTan^{\alpha} 
			\hat{\tophypnorm}_{\beta}
			\hat{\tophypnorm}^{\gamma}
			\hat{\tophypnorm}^{\delta}
			(\toppartialarg{\alpha} \gfour_{\gamma \delta})			
			\hat{\tophypnorm} \SigmatTan^{\beta}
			-
			\frac{1}{2}
			\SigmatTan^{\alpha} \hat{\tophypnorm}_{\alpha} 
			\hat{\tophypnorm}^{\beta}
			\hat{\tophypnorm}^{\gamma}
			(\toppartialarg{\delta} \gfour_{\beta \gamma}) 
			\hat{\tophypnorm} \SigmatTan^{\delta} 
				\notag \\
	& \ \
		+
		2
		\SigmatTan^{\alpha}
		(\toppartialarg{\beta} \gfour_{\alpha \gamma})
		\toppartialuparg{\gamma} \SigmatTan^{\beta}
		-
		2
		\topproject_{\ \beta}^{\alpha}
		\SigmatTan^{\gamma}
		(\toppartialuparg{\delta} \gfour_{\alpha \gamma}) 
		\toppartialarg{\delta} \SigmatTan^{\beta}
				\notag \\
	& \ \
		+
		\frac{1}{2}
		(\topfirstfund^{-1})^{\alpha \beta}
		\SigmatTan^{\gamma} 
		(\toppartialarg{\gamma} \gfour_{\alpha \beta})
		\toppartialarg{\delta} \SigmatTan^{\delta}
		-
		\frac{1}{2}
		(\topfirstfund^{-1})^{\alpha \beta}
		\SigmatTan^{\gamma} 
		(\toppartialarg{\delta} \gfour_{\alpha \beta})
		\toppartialarg{\gamma} \SigmatTan^{\delta}
		\notag
			\\
		& \ \
			+
			\SigmatTan^{\alpha} 
			\SigmatTan^{\beta}
			(\toppartialuparg{\gamma} \gfour_{\alpha \delta})
			(\toppartialuparg{\delta} \gfour_{\beta \gamma})
		-
		\SigmatTan^{\alpha} 
		\SigmatTan^{\beta}
		(\topfirstfund^{-1})^{\gamma \delta} 
		(\toppartialuparg{\kappa} \gfour_{\alpha \gamma}) 
		\toppartialarg{\kappa} \gfour_{\beta \delta},
				\notag
	\end{align}
	and
	\begin{align} \label{E:DIVERGENCEIDELLIPTICHYPERBOLICCURRENTBULKERRORTERMWITHWEIGHTDERIVATIVES}
		\mathfrak{J}_{(\pmb{\partial} \weight)}[\SigmatTan,\pmb{\partial} \SigmatTan]
		:= - J[\SigmatTan] \weight 
		=
		\SigmatTan^{\kappa}  
		(\toppartialarg{\kappa} \weight)
		\toppartialarg{\lambda} \SigmatTan^{\lambda}
		-
		\SigmatTan^{\kappa} 
		(\toppartialarg{\lambda} \weight)
		\toppartialarg{\kappa} \SigmatTan^{\lambda}.
	\end{align}
	
	\end{lemma}
	
	\begin{proof}
	In view of
	\eqref{E:NEWELLIPTICHYPERBOLICCURRENT}
	and
	\eqref{E:DIVERGENCEIDELLIPTICHYPERBOLICCURRENTBULKERRORTERMWITHWEIGHTDERIVATIVES},
	we see that once we have proven the identity \eqref{E:NEWSTANDARDDIVERGENCEIDENTITYFORELLIPTICHYPERBOLICCURRENT} 
	in the case $\weight = 1$, 
	the identity in the case of a general $\weight$
	then follows from the
	simple algebraic identity
	$
	\weight \widetilde{\nabla}_{\alpha} J^{\alpha}[\SigmatTan]
	=
	\widetilde{\nabla}_{\alpha} \left(\weight J^{\alpha}[\SigmatTan] \right)
	-
	J[\SigmatTan] \weight
	$.
	
	It remains for us to prove \eqref{E:NEWSTANDARDDIVERGENCEIDENTITYFORELLIPTICHYPERBOLICCURRENT} in the case $\weight = 1$.
	In our computations, we will silently use the following simple matrix inverse differentiation identity,
	which follows from differentiating the identity $(\gfour^{-1})^{\gamma \kappa} \gfour_{\kappa \delta} = \updelta_{\ \delta}^{\gamma}$,
	where $\updelta_{\ \delta}^{\gamma}$ is the Kronecker delta:
	\begin{align} \label{E:SIMPLEMATRIXINVERSEDIFFERENTIATIONIDENTITY}
		\partial_{\alpha} (\gfour^{-1})^{\gamma \delta}
		& = - 
		(\gfour^{-1})^{\gamma \kappa}
		(\gfour^{-1})^{\delta \lambda}
		\partial_{\alpha} \gfour_{\kappa \lambda}.
	\end{align}
	We now apply $\partial_{\alpha}$ to the first equality in \eqref{E:NEWELLIPTICHYPERBOLICCURRENT} and use \eqref{E:TOPPROJECT}
	to compute the following identity, which in particular shows that terms involving the second derivatives of
	$\SigmatTan$ cancel:
	\begin{align} \label{E:FIRSTSTEPNEWSTANDARDDIVERGENCEIDENTITYFORELLIPTICHYPERBOLICCURRENT}
		\partial_{\alpha} J^{\alpha}[\SigmatTan]
		& = 
			\topproject_{\ \beta}^{\delta} \topproject_{\ \gamma}^{\alpha} 
			(\partial_{\alpha} \SigmatTan^{\beta}) 
			\partial_{\delta} \SigmatTan^{\gamma}
			-
			(\topproject_{\ \lambda}^{\kappa} \partial_{\kappa} \SigmatTan^{\lambda})^2
				\\
		& \ \
			+
			\SigmatTan^{\beta} \hat{\tophypnorm}_{\beta} \topproject_{\ \gamma}^{\alpha} (\partial_{\alpha} \hat{\tophypnorm}^{\delta}) 
			\partial_{\delta} \SigmatTan^{\gamma}
			+
			\hat{\tophypnorm}^{\delta} 
			\topproject_{\ \gamma}^{\alpha}
			(\partial_{\alpha} \hat{\tophypnorm}^{\beta})
			\SigmatTan_{\beta} 
			\partial_{\delta} \SigmatTan^{\gamma}
			+
			\hat{\tophypnorm}^{\delta} 
			\topproject_{\ \gamma}^{\alpha}
			(\partial_{\alpha} \gfour_{\beta \kappa})
			\hat{\tophypnorm}^{\kappa}
			\SigmatTan^{\beta} 
			\partial_{\delta} \SigmatTan^{\gamma}
				\notag \\
		& \ \
			+
			\SigmatTan^{\beta} \topproject_{\ \beta}^{\delta} 
			(\partial_{\alpha} \hat{\tophypnorm}^{\alpha})  
			\hat{\tophypnorm}_{\gamma} 
			\partial_{\delta} \SigmatTan^{\gamma}
			+
			\SigmatTan^{\beta} \topproject_{\ \beta}^{\delta} 
			\gfour_{\gamma \kappa}
			(\hat{\tophypnorm}^{\alpha} \partial_{\alpha} \hat{\tophypnorm}^{\gamma}) 
			\partial_{\delta} \SigmatTan^{\kappa}
			+
			\SigmatTan^{\beta} \topproject_{\ \beta}^{\delta} 
			(\hat{\tophypnorm}^{\alpha} \partial_{\alpha} \gfour_{\gamma \kappa}) 
			\hat{\tophypnorm}^{\kappa}
			\partial_{\delta} \SigmatTan^{\gamma}
				\notag \\
		& \ \
		-
		\SigmatTan^{\gamma} 
		\hat{\tophypnorm}_{\gamma} 
		(\partial_{\alpha} \hat{\tophypnorm}^{\alpha}) 
		\topproject_{\ \lambda}^{\kappa}
		\partial_{\kappa} \SigmatTan^{\lambda}
		-
		\SigmatTan_{\gamma} 
		(\hat{\tophypnorm}^{\alpha} \partial_{\alpha} \hat{\tophypnorm}^{\gamma})
		\topproject_{\ \lambda}^{\kappa}
		\partial_{\kappa} \SigmatTan^{\lambda}
		-
		\SigmatTan^{\gamma} 
		\hat{\tophypnorm}^{\beta}
		(\hat{\tophypnorm}^{\alpha} \partial_{\alpha} \gfour_{\gamma \beta})
		\topproject_{\ \lambda}^{\kappa}
		\partial_{\kappa} \SigmatTan^{\lambda}
			\notag \\
	& \ \
		-
		\SigmatTan^{\gamma} 
		\topproject_{\ \gamma}^{\alpha}
		(\partial_{\alpha} \hat{\tophypnorm}^{\kappa})
		\hat{\tophypnorm}_{\lambda}
		\partial_{\kappa} \SigmatTan^{\lambda}
		-
		\SigmatTan^{\gamma} 
		\topproject_{\ \gamma}^{\alpha}
		\hat{\tophypnorm}^{\kappa}
		\gfour_{\lambda \beta}
		(\partial_{\alpha} \hat{\tophypnorm}^{\beta})
		\partial_{\kappa} \SigmatTan^{\lambda}
		-
		\SigmatTan^{\gamma} 
		\topproject_{\ \gamma}^{\alpha}
		\hat{\tophypnorm}^{\kappa}
		\hat{\tophypnorm}^{\beta}
		(\partial_{\alpha} \gfour_{\lambda \beta})
		\partial_{\kappa} \SigmatTan^{\lambda}.
		\notag
	\end{align}
	
	Next, using that $\SigmatTan^{\beta} = (\gfour^{-1})^{\beta \kappa} \SigmatTan_{\kappa}$
	and
	$\SigmatTan^{\gamma} = (\gfour^{-1})^{\gamma \lambda} \SigmatTan_{\lambda}$,
	and using the simple identity
	\begin{align*}
	(\topfirstfund^{-1})^{\beta \delta} 
				(\topfirstfund^{-1})^{\alpha \gamma} 
				(\partial_{\alpha} \SigmatTan_{\beta}) 
				\partial_{\delta} \SigmatTan_{\gamma}
	&
	=
	(\topfirstfund^{-1})^{\beta \delta} 
				(\topfirstfund^{-1})^{\alpha \gamma} 
				(\partial_{\alpha} \SigmatTan_{\beta}) 
				\partial_{\gamma} \SigmatTan_{\delta}
			\\
	& \ \
	+
	\frac{1}{2}
				(\topfirstfund^{-1})^{\beta \delta} 
				(\topfirstfund^{-1})^{\alpha \gamma} 
				\left\lbrace
					\partial_{\alpha} \SigmatTan_{\beta}
					-
					\partial_{\beta} \SigmatTan_{\alpha}
				\right\rbrace
				\left\lbrace
					\partial_{\delta} \SigmatTan_{\gamma}
					-
					\partial_{\gamma} \SigmatTan_{\delta}
				\right\rbrace,
	\end{align*}
	we rewrite the first product on RHS~\eqref{E:FIRSTSTEPNEWSTANDARDDIVERGENCEIDENTITYFORELLIPTICHYPERBOLICCURRENT} as follows:
	\begin{align} \label{E:SECONDSTEPNEWSTANDARDDIVERGENCEIDENTITYFORELLIPTICHYPERBOLICCURRENT}
			\topproject_{\ \beta}^{\delta} 
			\topproject_{\ \gamma}^{\alpha} 
			(\partial_{\alpha} \SigmatTan^{\beta}) 
			\partial_{\delta} \SigmatTan^{\gamma}
			& = 	
				(\topfirstfund^{-1})^{\beta \delta} 
				(\topfirstfund^{-1})^{\alpha \gamma} 
				(\partial_{\alpha} \SigmatTan_{\beta}) 
				\partial_{\gamma} \SigmatTan_{\delta}
				+
				\frac{1}{2}
				(\topfirstfund^{-1})^{\beta \delta} 
				(\topfirstfund^{-1})^{\alpha \gamma} 
				\left\lbrace
					\partial_{\alpha} \SigmatTan_{\beta}
					-
					\partial_{\beta} \SigmatTan_{\alpha}
				\right\rbrace
				\left\lbrace
					\partial_{\delta} \SigmatTan_{\gamma}
					-
					\partial_{\gamma} \SigmatTan_{\delta}
				\right\rbrace
				\\
			& \ \
			-
			(\topfirstfund^{-1})^{\beta \delta} 
			\topproject_{\ \gamma}^{\alpha} 
			\SigmatTan^{\kappa}
			(\partial_{\alpha} \gfour_{\beta \kappa})
			\partial_{\delta} \SigmatTan^{\gamma}
			-
			\topproject_{\ \beta}^{\delta} 
			(\topfirstfund^{-1})^{\alpha \gamma} 
			(\partial_{\alpha} \SigmatTan^{\beta}) 
			\SigmatTan^{\lambda}
			\partial_{\delta} \gfour_{\gamma \lambda}
				\notag \\
			& \ \
			-
			(\topfirstfund^{-1})^{\beta \delta} 
			(\topfirstfund^{-1})^{\alpha \gamma} 
			(\partial_{\alpha} \gfour_{\beta \kappa})
			(\partial_{\delta} \gfour_{\gamma \lambda})
			\SigmatTan^{\kappa} \SigmatTan^{\lambda}.
				\notag
	\end{align}
	
	Next, using that $\SigmatTan_{\beta} = \gfour_{\beta \kappa} \SigmatTan^{\kappa}$
	and
	$\SigmatTan_{\delta} = \gfour_{\delta \lambda} \SigmatTan^{\lambda}$,
	we rewrite the first product on RHS~\eqref{E:SECONDSTEPNEWSTANDARDDIVERGENCEIDENTITYFORELLIPTICHYPERBOLICCURRENT} as follows:
	\begin{align} \label{E:THIRDSTEPNEWSTANDARDDIVERGENCEIDENTITYFORELLIPTICHYPERBOLICCURRENT}
			(\topfirstfund^{-1})^{\beta \delta} 
			(\topfirstfund^{-1})^{\alpha \gamma} 
			(\partial_{\alpha} \SigmatTan_{\beta}) 
			\partial_{\gamma} \SigmatTan_{\delta}
	& = 
			(\topfirstfund^{-1})^{\alpha \gamma}
			\topfirstfund_{\beta \delta}
			(\partial_{\alpha} \SigmatTan^{\beta})
			\partial_{\gamma} \SigmatTan^{\delta}
				\\
	& \ \
		+
		\topproject_{\ \lambda}^{\beta}
		(\topfirstfund^{-1})^{\alpha \gamma} 
		\SigmatTan^{\kappa}
		(\partial_{\alpha} \gfour_{\beta \kappa}) 
		\partial_{\gamma} \SigmatTan^{\lambda}
		+
		\topproject_{\ \kappa}^{\delta} 
		(\topfirstfund^{-1})^{\alpha \gamma} 
		(\partial_{\alpha} 	\SigmatTan^{\kappa}) 
		\SigmatTan^{\lambda}
		\partial_{\gamma} \gfour_{\delta \lambda}
			\notag \\
	& \ \
		+
		(\topfirstfund^{-1})^{\beta \delta} 
		(\topfirstfund^{-1})^{\alpha \gamma} 
		(\partial_{\alpha} \gfour_{\beta \kappa}) 
		(\partial_{\gamma} \gfour_{\delta \lambda})
		\SigmatTan^{\kappa} \SigmatTan^{\lambda}.
		\notag
	\end{align}
	
	Next, using \eqref{E:SIGMATILDEDIVERGENCEIDENTITYFORSIGMATTANGENTVECTORFIELDS},
	we rewrite the second product on RHS~\eqref{E:FIRSTSTEPNEWSTANDARDDIVERGENCEIDENTITYFORELLIPTICHYPERBOLICCURRENT}
	as follows:
	\begin{align} \label{E:FOURTHSTEPNEWSTANDARDDIVERGENCEIDENTITYFORELLIPTICHYPERBOLICCURRENT}
		(\topproject_{\ \lambda}^{\kappa} \partial_{\kappa} \SigmatTan^{\lambda})^2
		& =
		(\projectedtransport_{\alpha} \projectedtransport \SigmatTan^{\alpha})^2
		+
		(\partial_a \SigmatTan^a)^2
		+
		\lengthoftophypnorm^2
		(\projectedtransport_{\alpha} \Transport \SigmatTan^{\alpha})^2
			\\
		& \ \
			+
			2 
			(\projectedtransport_{\alpha}\projectedtransport \SigmatTan^{\alpha})
			\partial_a \SigmatTan^a
			-
			2
			\lengthoftophypnorm 
			(\projectedtransport_{\alpha} \Transport \SigmatTan^{\alpha})
			\projectedtransport_{\alpha}\projectedtransport \SigmatTan^{\alpha}
			-
			2
			\lengthoftophypnorm 
			(\partial_a \SigmatTan^a)
			\projectedtransport_{\alpha} 
			\Transport \SigmatTan^{\alpha}.
			\notag
	\end{align}
	
	Combining
	\eqref{E:FIRSTSTEPNEWSTANDARDDIVERGENCEIDENTITYFORELLIPTICHYPERBOLICCURRENT}-\eqref{E:FOURTHSTEPNEWSTANDARDDIVERGENCEIDENTITYFORELLIPTICHYPERBOLICCURRENT},
	and also using \eqref{E:EXPRESSIONFORWIDETILDESIGMADIVERGENCEINCOORDINATES} 
	as well as \eqref{E:NEWELLIPTICHYPERBOLICCURRENT}
	and \eqref{E:PROJECTIONSOFCARTESIANVECTORFIELDS}-\eqref{E:RAISEDPROJECTIONSOFCARTESIANVECTORFIELDS},
	and rearranging terms and relabeling indices,
	we arrive at the desired identity \eqref{E:NEWSTANDARDDIVERGENCEIDENTITYFORELLIPTICHYPERBOLICCURRENT}
	in the case $\weight = 1$ 
	(note that the term
		$\mathfrak{J}_{(\pmb{\partial} \weight)}[\SigmatTan,\pmb{\partial} \SigmatTan]$
		defined in
	\eqref{E:DIVERGENCEIDELLIPTICHYPERBOLICCURRENTBULKERRORTERMWITHWEIGHTDERIVATIVES}
	vanishes in this case),
	but in place of the expression for the error term $\mathfrak{J}_{(Coeff)}[\SigmatTan,\pmb{\partial} \SigmatTan]$
	stated in \eqref{E:THIRDERRORTERMDIVERGENCEOFSIGMATILDECURRENT},
	we instead have the following expression
	involving the Cartesian partial derivative vectorfields $\partial_{\alpha}$
	\emph{and} the $\widetilde{\Sigma}_{\Timefunction}$-projected vectorfields $\toppartialarg{\alpha}$ defined in \eqref{E:PROJECTIONSOFCARTESIANVECTORFIELDS}:
	\begin{align} \label{E:MOREPROOFTHIRDERRORTERMDIVERGENCEOFSIGMATILDECURRENT}
		\mathfrak{J}_{(Coeff)}[\SigmatTan,\pmb{\partial} \SigmatTan]
		& =
			-
			\SigmatTan^{\alpha} \hat{\tophypnorm}_{\alpha} 
			(\toppartialarg{\beta} \hat{\tophypnorm}^{\gamma}) 
			\partial_{\gamma} 
			\SigmatTan^{\beta} 
			-
			\SigmatTan_{\alpha} 
			(\toppartialarg{\beta} \hat{\tophypnorm}^{\alpha})
			\hat{\tophypnorm} \SigmatTan^{\beta}
				\\
		& \ \
			-
			\SigmatTan^{\alpha} 
			\hat{\tophypnorm}_{\beta} 
			(\partial_{\gamma} \hat{\tophypnorm}^{\gamma})  
			\toppartialarg{\alpha} \SigmatTan^{\beta}
			-
			\SigmatTan^{\alpha} 
			\gfour_{\beta \gamma}
			(\hat{\tophypnorm} \hat{\tophypnorm}^{\beta}) 
			\toppartialarg{\alpha} \SigmatTan^{\gamma}
			\notag \\
		& \ \
		+
		\SigmatTan^{\alpha} 
		\hat{\tophypnorm}_{\alpha} 
		(\partial_{\beta} \hat{\tophypnorm}^{\beta}) 
		\toppartialarg{\gamma} \SigmatTan^{\gamma}
		+
		\SigmatTan_{\alpha} 
		(\hat{\tophypnorm} \hat{\tophypnorm}^{\alpha})
		\toppartialarg{\beta} \SigmatTan^{\beta}
		\notag \\
	& \ \
		+
		\SigmatTan^{\alpha} 
		\hat{\tophypnorm}_{\beta}
		(\toppartialarg{\alpha} \hat{\tophypnorm}^{\gamma})
		\partial_{\gamma} \SigmatTan^{\beta}
		+
		\SigmatTan^{\alpha} 
		\gfour_{\beta \gamma}
		(\toppartialarg{\alpha} \hat{\tophypnorm}^{\beta})
		\hat{\tophypnorm} \SigmatTan^{\gamma}
		\notag
			\\
			& \ \
		+
		\SigmatTan^{\alpha} 
		\topfirstfund_{\beta \gamma} 
		(\hat{\tophypnorm} \hat{\tophypnorm}^{\beta})
		\toppartialarg{\alpha} \SigmatTan^{\gamma}
		-
		\SigmatTan^{\alpha} 
		\topfirstfund_{\alpha \beta}
		(\hat{\tophypnorm} \hat{\tophypnorm}^{\beta})
		\toppartialarg{\gamma} \SigmatTan^{\gamma}
		\notag \\
		&  \ \
			-
			\SigmatTan^{\alpha} 
			\hat{\tophypnorm}^{\beta}
			(\toppartialarg{\gamma} \gfour_{\alpha \beta})
			\hat{\tophypnorm} \SigmatTan^{\gamma}
			-
			\SigmatTan^{\alpha} 
			\hat{\tophypnorm}^{\beta}
			(\hat{\tophypnorm} \gfour_{\beta \gamma}) 
			\toppartialarg{\alpha} \SigmatTan^{\gamma}
			+
		\SigmatTan^{\alpha} 
		\hat{\tophypnorm}^{\beta}
		(\hat{\tophypnorm} \gfour_{\alpha \beta})
		\toppartialarg{\gamma} \SigmatTan^{\gamma}
		+
		\SigmatTan^{\alpha} 
		\hat{\tophypnorm}^{\beta}
		(\toppartialarg{\alpha} \gfour_{\beta \gamma})
		\hat{\tophypnorm} \SigmatTan^{\gamma}	
			\notag
			\\
		& \ \
		+
		2
		\SigmatTan^{\alpha}
		(\toppartialarg{\beta} \gfour_{\alpha \gamma})
		\toppartialuparg{\gamma} \SigmatTan^{\beta}
		-
		2
		\topproject_{\ \beta}^{\alpha}
		\SigmatTan^{\gamma}
		(\toppartialuparg{\delta} \gfour_{\alpha \gamma}) 
		\toppartialarg{\delta} \SigmatTan^{\beta}
				\notag \\
		&  \ \
		+
		\SigmatTan^{\alpha} 
		\topproject_{\ \gamma}^{\beta} 
		\hat{\tophypnorm}^{\delta}
		(\hat{\tophypnorm} \gfour_{\beta \delta})
		\toppartialarg{\alpha} \SigmatTan^{\gamma}
		-
		\SigmatTan^{\alpha} 
		\topproject_{\ \alpha}^{\beta} 
		\hat{\tophypnorm}^{\gamma}
		(\hat{\tophypnorm} \gfour_{\beta \gamma})
		\toppartialarg{\delta} \SigmatTan^{\delta}
		\notag \\
	& \ \
		+
		\frac{1}{2}
		(\topfirstfund^{-1})^{\alpha \beta}
		\SigmatTan^{\gamma} 
		(\toppartialarg{\gamma} \gfour_{\alpha \beta})
		\toppartialarg{\delta} \SigmatTan^{\delta}
		-
		\frac{1}{2}
		(\topfirstfund^{-1})^{\alpha \beta}
		\SigmatTan^{\gamma} 
		(\toppartialarg{\delta} \gfour_{\alpha \beta})
		\toppartialarg{\gamma} \SigmatTan^{\delta}
		\notag
			\\
		& \ \
			+
			\SigmatTan^{\alpha} 
			\SigmatTan^{\beta}
			(\toppartialuparg{\gamma} \gfour_{\alpha \delta})
			(\toppartialuparg{\delta} \gfour_{\beta \gamma})
		-
		\SigmatTan^{\alpha} 
		\SigmatTan^{\beta}
		(\topfirstfund^{-1})^{\gamma \delta} 
		(\toppartialuparg{\kappa} \gfour_{\alpha \gamma}) 
		\toppartialarg{\kappa} \gfour_{\beta \delta}.
				\notag
	\end{align}
	To complete the proof, it remains for us to show that 
	$\mbox{RHS~\eqref{E:MOREPROOFTHIRDERRORTERMDIVERGENCEOFSIGMATILDECURRENT}} = \mbox{RHS~\eqref{E:THIRDERRORTERMDIVERGENCEOFSIGMATILDECURRENT}}$.
	This can be shown through straightforward calculations 
	(which in particular lead to the cancellation of many terms on RHS~\eqref{E:MOREPROOFTHIRDERRORTERMDIVERGENCEOFSIGMATILDECURRENT}),
	based on splitting the vectorfield $\partial_{\gamma}$ 
	on RHS~\eqref{E:MOREPROOFTHIRDERRORTERMDIVERGENCEOFSIGMATILDECURRENT}
	into its $\widetilde{\Sigma}_{\Timefunction}$-tangential and $\widetilde{\Sigma}_{\Timefunction}$-orthogonal parts via
	the identity
		\begin{align}
		\partial_{\gamma}
		 & = \toppartialarg{\gamma}
				-
				\hat{\tophypnorm}_{\gamma}
				\hat{\tophypnorm},
	\end{align}
	which follows from definition \eqref{E:PROJECTIONSOFCARTESIANVECTORFIELDS},
	from using the identity
	\begin{align} \label{E:COORDINATEDERIVATIVEOFTOPNORMALCONTRACTEDWITHTOPNORMAL}
		\hat{\tophypnorm}_{\gamma}
		\partial_{\alpha} 
		\hat{\tophypnorm}^{\gamma}
		& = 
		- \frac{1}{2}
			\hat{\tophypnorm}^{\beta} \hat{\tophypnorm}^{\gamma}
			(\partial_{\alpha} \gfour_{\beta \gamma}),
	\end{align}
	which follows from differentiating the relation
	$\gfour_{\kappa \lambda} \hat{\tophypnorm}^{\kappa} \hat{\tophypnorm}^{\lambda} = -1$,
	from using \eqref{E:TOPFIRSTFUNDAMENTALFORMHYPERSURFACE}
	to decompose the factors of $\topfirstfund$ in the terms
	$
	\SigmatTan^{\alpha} 
		\topfirstfund_{\beta \gamma} 
		(\hat{\tophypnorm} \hat{\tophypnorm}^{\beta})
		\toppartialarg{\alpha} \SigmatTan^{\gamma}
		-
		\SigmatTan^{\alpha} 
		\topfirstfund_{\alpha \beta}
		(\hat{\tophypnorm} \hat{\tophypnorm}^{\beta})
		\toppartialarg{\gamma} \SigmatTan^{\gamma}
	$
	on RHS~\eqref{E:MOREPROOFTHIRDERRORTERMDIVERGENCEOFSIGMATILDECURRENT},
	and from using 
	\eqref{E:TOPPROJECT} 
	to decompose
	the factors of $\topproject$
	in the terms
	$
	\SigmatTan^{\alpha} 
		\topproject_{\ \gamma}^{\beta} 
		\hat{\tophypnorm}^{\delta}
		(\hat{\tophypnorm} \gfour_{\beta \delta})
		\toppartialarg{\alpha} \SigmatTan^{\gamma}
		-
		\SigmatTan^{\alpha} 
		\topproject_{\ \alpha}^{\beta} 
		\hat{\tophypnorm}^{\gamma}
		(\hat{\tophypnorm} \gfour_{\beta \gamma})
		\toppartialarg{\delta} \SigmatTan^{\delta}
	$
	on RHS~\eqref{E:MOREPROOFTHIRDERRORTERMDIVERGENCEOFSIGMATILDECURRENT}.
	
	\end{proof}

\subsection{Preliminary analysis of the boundary integrand}
\label{SS:PRELIMINARYDECOMPOSITIONOFOUNDARYINTEGRAND}
When we apply the divergence theorem on $\widetilde{\Sigma}_{\Timefunction}$ 
to the current $J^{\alpha}[\SigmatTan]$
defined in \eqref{E:NEWELLIPTICHYPERBOLICCURRENT},
we will encounter boundary terms on $\mathcal{S}_{\Timefunction}$.
In Prop.\,\ref{P:STRUCTUREOFERRORINTEGRALS} and Theorem~\ref{T:STRUCTUREOFERRORTERMS}, 
we show that after integration with respect to $\Timefunction$, 
the boundary terms involve derivatives of various quantities 
\emph{only in $\underline{\mathcal{H}}$-tangential directions}.
In the next lemma, namely Lemma~\ref{L:PRELIMINARYANALYSISOFBOUNDARYINTEGRAND}, 
we perform some preliminary analysis that essentially shows all terms
have the desired structure,
except for the term
$
		2 
		\weight
		\uposinnerproduct
		\uspecialgen^{\alpha}
		\angV^{\beta}
		(\partial_{\alpha} \SigmatTan_{\beta} - \partial_{\beta} \SigmatTan_{\alpha})
$
on RHS~\eqref{E:PRELIMINARYDECOMPOFBOUNDARYINTEGRAND},
which is much more difficult to handle; we dedicate all of
Sect.\,\ref{S:GEOMETRICDECOMPOSITIONSTEIDTOANTISYMMETRICPART}
to understanding the structure of this term.
The proof of Lemma~\ref{L:PRELIMINARYANALYSISOFBOUNDARYINTEGRAND} relies on careful geometric decompositions,
but unlike the analysis of Sect.\,\ref{S:GEOMETRICDECOMPOSITIONSTEIDTOANTISYMMETRICPART}
and the proof of Prop.\,\ref{P:STRUCTUREOFERRORINTEGRALS},
it does not rely on the formulation of compressible Euler flow provided by Theorem~\ref{T:GEOMETRICWAVETRANSPORTSYSTEM}.

\begin{lemma}[Preliminary analysis of the boundary integrand]
	\label{L:PRELIMINARYANALYSISOFBOUNDARYINTEGRAND}
	Let $\SigmatTan$ be a $\Sigma_t$-tangent vectorfield defined on $\mathcal{M}$,
	let $\angV$ be its $\gfour$-orthogonal projection onto $\mathcal{S}_{\Timefunction}$ 
	(see Defs.\,\ref{D:PROJECTIONSOFTENSORFIELDS} and \ref{D:TANGENTTENSORFIELDS}),
	let $J[\SigmatTan]$ be the $\widetilde{\Sigma}_{\Timefunction}$-tangent vectorfield 
	from Def.\,\ref{D:ELLIPTICHYPERBOLICCURRENT}, and let $\weight$ be an arbitrary scalar function. 
	Let $\uspecialgen$ and $\utang$ be the vectorfields from
	Def.\,\ref{D:SPECIALGENERATOR} and Lemma~\ref{L:KEYIDBETWEENVARIOUSVECTORFIELDS}.
	The following identity holds along $\mathcal{S}_{\Timefunction}$,
	where on RHS~\eqref{E:PRELIMINARYDECOMPOFBOUNDARYINTEGRAND},
	$\angdiv \angpartialarg{\beta}$ denotes the $\angD$-divergence
	of the $\mathcal{S}_{\Timefunction}$-tangent vectorfield 
	$\angpartialarg{\beta}$ (as in \eqref{E:EXPRESSIONFORSPHEREDIVERGENCEOFANGPARTIALVECTORFIELDSINCOORDINATES}):
	\begin{align} \label{E:PRELIMINARYDECOMPOFBOUNDARYINTEGRAND}
		\weight \spherenormal_{\alpha} J^{\alpha}[\SigmatTan]
		& = 
			-
		\modgen
		\left\lbrace
			\weight
			\frac{\uposinnerproduct}{\seconduposinnerproduct \lapsemodgen} 
			|\SigmatTan|_{\gsphere}^2
		\right\rbrace
		+
		2 
		\weight
		\uposinnerproduct
		\uspecialgen^{\alpha}
		\angV^{\beta}
		(\partial_{\alpha} \SigmatTan_{\beta} - \partial_{\beta} \SigmatTan_{\alpha})
			\\
		& \ \
		+
		\left\lbrace
		\modgen
		\left[\weight \frac{\uposinnerproduct}{\seconduposinnerproduct \lapsemodgen} (\gsphere^{-1})^{\alpha \beta} \right]
		\right\rbrace
		\SigmatTan_{\alpha} 
	  \SigmatTan_{\beta}
		+
		\left\lbrace
		\utang
			\left[\weight \uposinnerproduct (\gsphere^{-1})^{\alpha \beta} \right]
		\right\rbrace
		\SigmatTan_{\alpha} 
	  \SigmatTan_{\beta}
		-
		2 \weight \SigmatTan_{\alpha} (\angV \Transport^{\alpha})
		\notag \\
		& \ \
				+
			\weight
			\uposinnerproduct 
			|\angV|_{\gsphere}^2
			\angdiv \utang 
		+
		\weight
		\spherenormal_{\alpha} \SigmatTan^{\alpha}  
		\SigmatTan^{\beta}
		\angdiv \angpartialarg{\beta}
		+
		\weight
		\SigmatTan_{\alpha}
		\angV \spherenormal^{\alpha}	
			\notag \\
		& \ \
		+
		\weight
		\SigmatTan^{\alpha}
		\spherenormal^{\beta}
		\angV \gfour_{\alpha \beta}
		+
		\SigmatTan_{\alpha} \spherenormal^{\alpha}
		\angV \weight
		\notag
			\\
		& \ \
		-
		\angdiv
		\left\lbrace
			\weight
			\uposinnerproduct 
			|\angV|_{\gsphere}^2
			\utang 
		\right\rbrace
		-
		\angdiv
		\left\lbrace
			\weight
			\spherenormal_{\alpha} \SigmatTan^{\alpha}  
			\angV 
		\right\rbrace.
		\notag
	\end{align}
	\end{lemma}
	
	\begin{proof}
		First, using \eqref{E:NEWELLIPTICHYPERBOLICCURRENT}
		and
		\eqref{E:TOPPARTIALINTERMSOFOUTERNORMALDERIVATIVEANDANGPARTIAL}
		and the fact that $\topproject \spherenormal = \spherenormal$,
		we compute that
		\begin{align} \label{E:FIRSTSTEPPRELIMINARYDECOMPOFBOUNDARYINTEGRAND}
		\weight
		\spherenormal_{\alpha} J^{\alpha}[\SigmatTan]
		& 
		=
		\weight
		\spherenormal_{\gamma} 
		\SigmatTan^{\beta} 
		\sphereproject_{\ \beta}^{\delta} 
		\angpartialarg{\delta} \SigmatTan^{\gamma}
		-
		\weight
		\spherenormal_{\gamma} 
		\SigmatTan^{\gamma} 
		\sphereproject_{\ \lambda}^{\kappa}
		\angpartialarg{\kappa} \SigmatTan^{\lambda}.
	\end{align}
	Differentiating by parts on the last product 
	$
	-
		\weight
		\spherenormal_{\gamma} 
		\SigmatTan^{\gamma} 
		\sphereproject_{\ \lambda}^{\kappa}
		\angpartialarg{\kappa} \SigmatTan^{\lambda}
	$
	on RHS~\eqref{E:FIRSTSTEPPRELIMINARYDECOMPOFBOUNDARYINTEGRAND}
	and using the simple identity 
		$\angV = \SigmatTan^{\beta} \angpartialarg{\beta}
		=
		\SigmatTan^{\lambda} 
		\sphereproject_{\ \lambda}^{\kappa}
		\angpartialarg{\kappa}
		$
		(which follows from \eqref{E:ALTERNATENOTATIONFORSPHEREPROJECTEDTENSORFIELDS}, 
		\eqref{E:STPROJECTIONSQUAREDEQUALSSTPROJECTION}, and definition \eqref{E:PROJECTIONSOFCARTESIANVECTORFIELDS}),
		we compute that
		\begin{align} \label{E:SECONDSTEPPRELIMINARYDECOMPOFBOUNDARYINTEGRAND}
		\mbox{RHS}~\eqref{E:FIRSTSTEPPRELIMINARYDECOMPOFBOUNDARYINTEGRAND}
		& =
		2
		\weight
		\spherenormal_{\gamma} 
		\SigmatTan^{\beta} 
		\sphereproject_{\ \beta}^{\delta} 
		\angpartialarg{\delta} \SigmatTan^{\gamma}
		-
		\angdiv
		\left\lbrace
			\weight
			\spherenormal_{\alpha} \SigmatTan^{\alpha}  
			\angV 
		\right\rbrace
		+
		\weight
		\spherenormal_{\alpha} \SigmatTan^{\alpha}  
		\SigmatTan^{\beta}
		\angdiv \angpartialarg{\beta}
		+
		\SigmatTan^{\alpha}
		\angV (\weight \spherenormal_{\alpha}),
		\end{align}
	where we stress that on RHS~\eqref{E:SECONDSTEPPRELIMINARYDECOMPOFBOUNDARYINTEGRAND},
	we are viewing $\angpartialarg{\beta}$ as an $\mathcal{S}_{\Timefunction}$-tangent vectorfield.
	Expressing the last factor on RHS~\eqref{E:SECONDSTEPPRELIMINARYDECOMPOFBOUNDARYINTEGRAND}
	as $\spherenormal_{\alpha} = \gfour_{\alpha \beta} \spherenormal^{\beta}$,
	we compute that
	\begin{align} \label{E:THIRDSTEPPRELIMINARYDECOMPOFBOUNDARYINTEGRAND}
		\mbox{RHS}~\eqref{E:SECONDSTEPPRELIMINARYDECOMPOFBOUNDARYINTEGRAND}
		& =
		2
		\weight
		\SigmatTan^{\beta} 
		\sphereproject_{\ \beta}^{\delta} 
		\spherenormal^{\alpha}
		\angpartialarg{\delta} \SigmatTan_{\alpha}
		-
		\angdiv
		\left\lbrace
			\weight
			\spherenormal_{\alpha} \SigmatTan^{\alpha}  
			\angV 
		\right\rbrace
		+
		\weight
		\spherenormal_{\alpha} \SigmatTan^{\alpha}  
		\SigmatTan^{\beta}
		\angdiv \angpartialarg{\beta}
		+
		\weight
		\SigmatTan_{\alpha}
		\angV \spherenormal^{\alpha}
			\\
	& \ \
		+
		\weight
		\SigmatTan^{\alpha}
		\spherenormal^{\beta}
		\angV \gfour_{\alpha \beta}
		+
		\SigmatTan_{\alpha} \spherenormal^{\alpha}
		\angV \weight.
		\notag
\end{align}
Next, we use \eqref{E:TRANSPORTDECOMPOSITION}
to substitute for the factor $\spherenormal^{\alpha}$ in the first product on RHS~\eqref{E:THIRDSTEPPRELIMINARYDECOMPOFBOUNDARYINTEGRAND},
which allows us to rewrite the factor as follows:
\begin{align} \label{E:FOURTHSTEPPRELIMINARYDECOMPOFBOUNDARYINTEGRAND}
		2 
		\weight
		\SigmatTan^{\beta} 
		\sphereproject_{\ \beta}^{\delta} 
		\spherenormal^{\alpha}
		\angpartialarg{\delta} \SigmatTan_{\alpha}
		& = 
		2 
		\weight
		\uposinnerproduct \SigmatTan^{\beta} \sphereproject_{\ \beta}^{\delta} \Transport^{\alpha}
		\angpartialarg{\delta} \SigmatTan_{\alpha}
		-
		2 
		\weight
		\frac{\uposinnerproduct}{\seconduposinnerproduct} \SigmatTan^{\beta} \sphereproject_{\ \beta}^{\delta} \gen^{\alpha}
		\angpartialarg{\delta} \SigmatTan_{\alpha}
		-
		2 
		\weight
		\uposinnerproduct \SigmatTan^{\beta} \sphereproject_{\ \beta}^{\delta} \utang^{\alpha}
		\angpartialarg{\delta} \SigmatTan_{\alpha}.
\end{align}
Next, since $\Transport$ is $\gfour$-orthogonal to $\Sigma_t$, we have $\Transport^{\alpha} \SigmatTan_{\alpha} = 0$,
and by differentiating this relation, we obtain the identity
that $\Transport^{\alpha} \angpartialarg{\delta} \SigmatTan_{\alpha} = - (\angpartialarg{\delta} \Transport^{\alpha}) \SigmatTan_{\alpha}$.
Using this identity
to remove the derivatives off the factor $\SigmatTan_{\alpha}$ in the first term on RHS~\eqref{E:FOURTHSTEPPRELIMINARYDECOMPOFBOUNDARYINTEGRAND},
we deduce that
\begin{align} \label{E:FIFTHSTEPPRELIMINARYDECOMPOFBOUNDARYINTEGRAND}
		2 
		\weight
		\SigmatTan^{\beta} 
		\sphereproject_{\ \beta}^{\delta} 
		\spherenormal^{\alpha}
		\angpartialarg{\delta} \SigmatTan_{\alpha}
		& = 
		-
		2 \weight 
		\uposinnerproduct  \SigmatTan_{\alpha} \angV \Transport^{\alpha}
		-
		2 
		\weight
		\frac{\uposinnerproduct}{\seconduposinnerproduct} \SigmatTan^{\beta} \sphereproject_{\ \beta}^{\delta} \gen^{\alpha}
		\angpartialarg{\delta} \SigmatTan_{\alpha}
		-
		2 
		\weight
		\uposinnerproduct \SigmatTan^{\beta} \sphereproject_{\ \beta}^{\delta} \utang^{\alpha}
		\angpartialarg{\delta} \SigmatTan_{\alpha}.
\end{align}	
Next, we use straightforward algebraic calculations
to rewrite the last two products on RHS~\eqref{E:FIFTHSTEPPRELIMINARYDECOMPOFBOUNDARYINTEGRAND}
as follows, where we take into account \eqref{E:SPECIALGENERATORIDENTITY}:
\begin{align} \label{E:SIXTHSTEPPRELIMINARYDECOMPOFBOUNDARYINTEGRAND}
	-
		2 
		\weight
		\frac{\uposinnerproduct}{\seconduposinnerproduct} \SigmatTan^{\beta} \sphereproject_{\ \beta}^{\delta} \gen^{\alpha}
		\angpartialarg{\delta} \SigmatTan_{\alpha}
		-
		2 
		\weight
		\uposinnerproduct \SigmatTan^{\beta} \sphereproject_{\ \beta}^{\delta} \utang^{\alpha}
		\angpartialarg{\delta} \SigmatTan_{\alpha}
		&
		=
		-
		2 
		\weight
		\frac{\uposinnerproduct}{\seconduposinnerproduct} 
		(\gsphere^{-1})^{\beta \delta}
		\SigmatTan_{\beta}
		\gen \SigmatTan_{\delta}
		-
		2 
		\weight
		\uposinnerproduct
		(\gsphere^{-1})^{\beta \delta}
		 \SigmatTan_{\beta}
		\utang \SigmatTan_{\alpha}
			\\
	& \ \
		+ 
		2 
		\weight
		\uposinnerproduct
		\uspecialgen^{\alpha} \angV^{\beta} 
		\left\lbrace
			\partial_{\alpha} \SigmatTan_{\beta} 
			- 
			\partial_{\beta} \SigmatTan_{\alpha}
		\right\rbrace.
		\notag
\end{align}
Using \eqref{E:NORMALIZEDAGAINSTTIMEFUNCTIONGENERATOROFHYPERSURFACE},
differentiating by parts on the first two products on RHS~\eqref{E:SIXTHSTEPPRELIMINARYDECOMPOFBOUNDARYINTEGRAND},
and using the simple identity $|\SigmatTan|_{\gsphere}^2 = (\gsphere^{-1})^{\alpha \beta} \SigmatTan_{\alpha} \SigmatTan_{\beta}$,
we rewrite \eqref{E:SIXTHSTEPPRELIMINARYDECOMPOFBOUNDARYINTEGRAND} as follows:
\begin{align} \label{E:SEVENTHSTEPPRELIMINARYDECOMPOFBOUNDARYINTEGRAND}
		&
		-
		2 
		\weight
		\frac{\uposinnerproduct}{\seconduposinnerproduct} \SigmatTan^{\beta} \sphereproject_{\ \beta}^{\delta} \gen^{\alpha}
		\angpartialarg{\delta} \SigmatTan_{\alpha}
		-
		2 
		\weight
		\uposinnerproduct \SigmatTan^{\beta} \sphereproject_{\ \beta}^{\delta} \utang^{\alpha}
		\angpartialarg{\delta} \SigmatTan_{\alpha}
			\\
		&
		=
		-
		\modgen
		\left\lbrace
			\weight
			\frac{\uposinnerproduct}{\seconduposinnerproduct \lapsemodgen} 
			|\SigmatTan|_{\gsphere}^2
		\right\rbrace
		-
		\angdiv
		\left\lbrace
			\weight
			\uposinnerproduct 
			|\angV|_{\gsphere}^2
			\utang 
		\right\rbrace
	\notag \\
		& \ \
		+
		\left\lbrace
		\modgen
			\left[\weight \frac{\uposinnerproduct}{\seconduposinnerproduct \lapsemodgen} (\gsphere^{-1})^{\alpha \beta} \right]
		\right\rbrace
		\SigmatTan_{\alpha} 
	  \SigmatTan_{\beta}
		+
		\left\lbrace
		\utang
			\left[\weight \uposinnerproduct (\gsphere^{-1})^{\alpha \beta} \right]
		\right\rbrace
		\SigmatTan_{\alpha} 
	  \SigmatTan_{\beta}
			+
			\weight
			\uposinnerproduct 
			|\angV|_{\gsphere}^2
			\angdiv \utang 
		\notag
			\\
	& \ \
		+ 
		2 
		\weight
		\uposinnerproduct
		\uspecialgen^{\alpha} \angV^{\beta} 
		\left\lbrace
			\partial_{\alpha} \SigmatTan_{\beta} 
			- 
			\partial_{\beta} \SigmatTan_{\alpha}
		\right\rbrace.
		\notag
\end{align}
Finally, combining \eqref{E:FIRSTSTEPPRELIMINARYDECOMPOFBOUNDARYINTEGRAND}-\eqref{E:SEVENTHSTEPPRELIMINARYDECOMPOFBOUNDARYINTEGRAND},
we conclude the desired identity \eqref{E:PRELIMINARYDECOMPOFBOUNDARYINTEGRAND}.	
	
\end{proof}

\section{Additional geometric decompositions tied to $\partial_{\alpha} \SigmatTan_{\beta} - \partial_{\beta} \SigmatTan_{\alpha}$}
\label{S:GEOMETRICDECOMPOSITIONSTEIDTOANTISYMMETRICPART}
In order to derive our main results, we need to uncover some subtle structures
found in the term
$
2 
		\weight
		\uposinnerproduct
		\uspecialgen^{\alpha}
		\angV^{\beta}
		(\partial_{\alpha} \SigmatTan_{\beta} - \partial_{\beta} \SigmatTan_{\alpha})
$
on the right-hand side of the boundary term identity \eqref{E:PRELIMINARYDECOMPOFBOUNDARYINTEGRAND}.
Although this term has the desired feature that it involves only $\underline{\mathcal{H}}$-tangential derivatives of $\SigmatTan$,
as written, it \emph{appears} to have insufficient regularity for applications. 
The reason is that our forthcoming analysis
(specifically the proof of Prop.\,\ref{P:STRUCTUREOFERRORINTEGRALS} -- 
see the second integral on RHS~\eqref{E:INTEGRATEDPRELIMINARYDECOMPOFBOUNDARYINTEGRAND})
involves the integral of 
$
2 
		\weight
		\uposinnerproduct
		\uspecialgen^{\alpha}
		\angV^{\beta}
		(\partial_{\alpha} \SigmatTan_{\beta} - \partial_{\beta} \SigmatTan_{\alpha})
$
along the lateral hypersurface $\underline{\mathcal{H}}$,
and the difficulty is that for $V \in \lbrace \vortrenormalized, \GradEnt \rbrace$,
we have no control over even $\underline{\mathcal{H}}$-tangential first derivatives of $\SigmatTan$ in $L^2$ along $\underline{\mathcal{H}}$;
this is consistent with the fact that 
our main integral identities, which are provided by Theorem~\ref{T:MAINREMARKABLESPACETIMEINTEGRALIDENTITY}, yield only \emph{spacetime} $L^2$ control
over the first derivatives of $\SigmatTan$.
To overcome this difficulty, we will use the compressible Euler equations to show that
$
2 
		\weight
		\uposinnerproduct
		\uspecialgen^{\alpha}
		\angV^{\beta}
		(\partial_{\alpha} \SigmatTan_{\beta} - \partial_{\beta} \SigmatTan_{\alpha})
$
can be rewritten
as \emph{terms that involve only $\underline{\mathcal{H}}$-tangential derivatives of quantities with sufficient regularity}.
In this section, we provide various geo-analytic decompositions
that in total reveal the structures of interest.
The main results in this section are 
Lemma~\ref{L:TRANSPORTLOGDENSITYISTANGENTIALEXCEPTINNULLCASE},
Cor.\,\ref{C:SHARPDECOMPOSITIONOFANTISYMMETRICGRADIENTS},
and
Prop.\,\ref{P:KEYDETERMINANT}. 

\begin{remark}[Exploiting the special structure of the compressible Euler equations]
\label{R:EXPLOITINGSPECIALSTRUCTUREOFCOMPRESSIBLEEULER}
Lemma~\ref{L:TRANSPORTLOGDENSITYISTANGENTIALEXCEPTINNULLCASE},
Cor.\,\ref{C:SHARPDECOMPOSITIONOFANTISYMMETRICGRADIENTS},
and
Prop.\,\ref{P:KEYDETERMINANT} are the main results in the paper in which 
we crucially exploit the special properties of the compressible Euler formulation
provided by Theorem~\ref{T:GEOMETRICWAVETRANSPORTSYSTEM} \emph{and} the precise structure of equations
\eqref{E:TRANSPORTDENSRENORMALIZEDRELATIVECTORECTANGULAR}-\eqref{E:TRANSPORTVELOCITYRELATIVECTORECTANGULAR}
(where the latter two equations are part of the standard first-order formulation of compressible Euler flow).
\end{remark}

\subsection{Some preliminary geometric decompositions}
\label{SS:PRELIMINARYGEOMETRICDECOMPOSITIONSFORDIFFICULTTERM}

\subsubsection{The vectorfield $\urescalednewgenminushypnorm$}
\label{SSS:RESCALEDVERSIONOFHYPNORMMINUSNEWGEN}
In our decompositions, we will encounter the vectorfield $\urescalednewgenminushypnorm$,
which we now define.

\begin{definition}[The vectorfield $\urescalednewgenminushypnorm$]
	\label{D:RESCALEDVERSIONOFHYPNORMMINUSNEWGEN}
	Let $\tophypnorm$ and $\sidehypnorm$ be the vectorfields
	from Def.\,\ref{D:HYPNORMANDSPHEREFORMDEFS}, and let $\seconduposinnerproduct > 0$ be the scalar function
	from \eqref{E:SECONDINGOINGCONDITION}. 
	Along $\underline{\mathcal{H}}$, we define $\urescalednewgenminushypnorm$ to be the following vectorfield:
	\begin{align} \label{E:REGULARFORMRESCALEDVERSIONOFHYPNORMMINUSNEWGEN}
		\urescalednewgenminushypnorm
		& := \frac{\sidehypnorm - \tophypnorm}{\seconduposinnerproduct}.
	\end{align}
	
\end{definition}

\begin{lemma}[Properties of $\urescalednewgenminushypnorm$]
	\label{L:PROPERTIESOFRESCALEDGENERATORMINUSSIDEHYPNORM}
	The vectorfield $\urescalednewgenminushypnorm$ from Def.\,\ref{D:RESCALEDVERSIONOFHYPNORMMINUSNEWGEN},
	which is defined along $\underline{\mathcal{H}}$,
	is $\Sigma_t$-tangent (i.e., $\urescalednewgenminushypnorm t = 0$) and $\gfour$-spacelike.

	\end{lemma}

\begin{proof}
	From \eqref{E:FUTURENORMALTOTOPHYPERSURFACE} and \eqref{E:FUTURENORMALTOHYPERSURFACE},
	we see that $(\sidehypnorm - \tophypnorm)t = 1 - 1 = 0$.
	In view of \eqref{E:REGULARFORMRESCALEDVERSIONOFHYPNORMMINUSNEWGEN}, we
	conclude that $\urescalednewgenminushypnorm$ is $\Sigma_t$-tangent as desired.

\end{proof}

\subsubsection{A geometric decomposition of the Cartesian coordinate partial derivative vectorfield $\partial_{\alpha}$}
\label{SSS:NEWDECOMPOSITIONOFCOORDINATEPARTIALDERIVATIVEVECTORFIELDS}
In our forthcoming analysis, we will often decompose
the Cartesian partial derivative vectorfield $\partial_{\alpha}$
into a part that is parallel to $\Transport$ and a part that is
$\underline{\mathcal{H}}$-tangent. 
In the next lemma, we provide the decomposition.

\begin{lemma}[Decomposition of $\partial_{\alpha}$ into $\Transport$-parallel and $\underline{\mathcal{H}}$-tangential components]
	\label{L:NEWDECOMPOSITIONOFCOORDINATEPARTIALDERIVATIVEVECTORFIELDS}
	Let 
	$\sidehypnorm$
	and
	$\gen$
	be the vectorfields from Def.\,\ref{D:HYPNORMANDSPHEREFORMDEFS},
	and let $\urescalednewgenminushypnorm$ be the vectorfield from
	Def.\,\ref{D:RESCALEDVERSIONOFHYPNORMMINUSNEWGEN}.
	For $\alpha = 0,1,2,3$, let $\utandecompvectorfielddownarg{\alpha}$ be the vectorfield on $\underline{\mathcal{H}}$ 
	defined by the following identity relative to the Cartesian coordinates:
	\begin{align} \label{E:NEWDECOMPOSITIONOFCOORDINATEPARTIALDERIVATIVEVECTORFIELDS}
		\partial_{\alpha}
		& = 
			-
			\sidehypnorm_{\alpha} \Transport
			+
			\urescalednewgenminushypnorm_{\alpha} \gen
			+
			\utandecompvectorfielddownarg{\alpha}.
	\end{align}
	Then $\utandecompvectorfielddownarg{\alpha}$ is $\mathcal{S}_{\Timefunction}$-tangent.

\end{lemma}

\begin{proof}
	We first claim that for every sphere $\mathcal{S}_{\Timefunction} \subset \underline{\mathcal{H}}$ and every
	$q \in \mathcal{S}_{\Timefunction}$, $T_q \mathcal{M}$ (which is the tangent space to $\mathcal{M}$ at $q$) 
	enjoys the direct sum decomposition
	$T_q \mathcal{M} = \mbox{\upshape span} \lbrace \Transport|_q, \gen|_q \rbrace \oplus T_q \mathcal{S}_{\Timefunction}$
	where $\mathbf{X}|_q$ denotes the vectorfield $\mathbf{X}$ evaluated at $q$
	(and the spaces in the direct sum are not necessarily $\gfour$-orthogonal).
	To prove the claim, we first note that since $\gen$ is $\underline{\mathcal{H}}$-tangent by construction,
	since $\mathcal{S}_{\Timefunction} = \underline{\mathcal{H}} \cap \widetilde{\Sigma}_{\Timefunction}$,
	and since $\gen$ is $\gfour$-orthogonal to $\mathcal{S}_{\Timefunction}$, 
	it follows that $T_q \underline{\mathcal{H}} = \mbox{\upshape span} \lbrace \gen|_q \rbrace \oplus T_q \mathcal{S}_{\Timefunction}$.
	Since $\underline{\mathcal{H}}$ is transversal to $\widetilde{\Sigma}_{\Timefunction}$, it follows that
	$\gen$ is transversal to $\widetilde{\Sigma}_{\Timefunction}$.
	Moreover, since $\Transport$ is $\gfour$-timelike while $\widetilde{\Sigma}_{\Timefunction}$ is $\gfour$-spacelike,
	it follows that $\Transport$ is transversal to $\widetilde{\Sigma}_{\Timefunction}$.
	Moreover, $\Transport$ and $\gen$ are linearly independent since  $\Transport$ is $\gfour$-timelike and $\gen$ is $\gfour$-spacelike 
	or null. We have therefore shown that
	$\lbrace \Transport|_q, \gen|_q \rbrace$
	is a two dimensional subspace of $T_q \mathcal{M}$
	that is transversal to $\widetilde{\Sigma}_{\Timefunction}$ at $q$.
	Since $\mathcal{S}_{\Timefunction}$ is a two-dimensional submanifold of $\widetilde{\Sigma}_{\Timefunction}$,
	we conclude the claim.
	
	From the claim, it follows that 
	there exist scalar functions
	$a_1$ and $a_2$
	and an $\mathcal{S}_{\Timefunction}$-tangent vectorfield $\utandecompvectorfielddownarg{\alpha}$
	such that the following vectorfield identity holds relative to the Cartesian coordinates:
	$\partial_{\alpha}
	=
	a_1 \Transport
	+
	a_2 \gen
	+
	\utandecompvectorfielddownarg{\alpha}
	$.
	Taking the $\gfour$-inner product of this identity with respect to 
	$\sidehypnorm$ and using \eqref{E:EQUIVALENTFUTURENORMALTOHYPERSURFACE}
	and the fact that $\sidehypnorm$ is $\gfour$-orthogonal to $\gen$ and $\utandecompvectorfielddownarg{\alpha}$,
	we find that $\sidehypnorm_{\alpha} = - a_1$ and thus $a_1 = - \sidehypnorm_{\alpha}$ as desired.
	Similarly, taking the $\gfour$-inner product of the identity with respect to 
	$\tophypnorm$ and using  
	\eqref{E:EQUIVALENTFUTURENORMALTOTOPHYPERSURFACE},
	\eqref{E:SECONDINGOINGCONDITION},
	and the fact that $\tophypnorm$ is $\gfour$-orthogonal to $\utandecompvectorfielddownarg{\alpha}$,
	we find that
	$\tophypnorm_{\alpha} = - a_1 - a_2 \seconduposinnerproduct$
	and thus $a_2 = \frac{\sidehypnorm_{\alpha} - \tophypnorm_{\alpha}}{\seconduposinnerproduct}$.
	In view of \eqref{E:REGULARFORMRESCALEDVERSIONOFHYPNORMMINUSNEWGEN},
	we conclude \eqref{E:NEWDECOMPOSITIONOFCOORDINATEPARTIALDERIVATIVEVECTORFIELDS}.
\end{proof}

In the next lemma, we show that for compressible Euler solutions, 
the term $\lengthofsidehypnorm^2 \Transport \LogDensity$, which involves a derivative of $\LogDensity$
in an $\underline{\mathcal{H}}$-transversal\footnote{$\Transport$ is transversal to $\underline{\mathcal{H}}$ because $\Transport$ is
$\gfour$-timelike while $\underline{\mathcal{H}}$ is $\gfour$-spacelike or $\gfour$-null.} direction, 
can be expressed in terms of $\underline{\mathcal{H}}$-tangential derivatives of the solution.
The lemma plays a crucial role in the proof of Prop.\,\ref{P:KEYDETERMINANT},
where it allows us to eliminate some $\underline{\mathcal{H}}$-transversal 
derivatives found in the term
$
2 	\weight
		\uposinnerproduct
		\uspecialgen^{\alpha}
		\angV^{\beta}
		(\partial_{\alpha} \SigmatTan_{\beta} - \partial_{\beta} \SigmatTan_{\alpha})
$
on the right-hand side of the boundary term identity \eqref{E:PRELIMINARYDECOMPOFBOUNDARYINTEGRAND}
in the case $\SigmatTan=\vortrenormalized$.
We stress that the identity proved in the lemma degenerates as $\underline{\mathcal{H}}$ becomes null, 
that is, as the factor
$\lengthofsidehypnorm$
on LHS~\eqref{E:TRANSPORTLOGDENSITYISTANGENTIALEXCEPTINNULLCASE} converges to $0$.

\begin{lemma}[Expression for $\lengthofsidehypnorm^2 \Transport \LogDensity$ in terms of $\underline{\mathcal{H}}$-tangential derivatives]
	\label{L:TRANSPORTLOGDENSITYISTANGENTIALEXCEPTINNULLCASE}
	Let $\sidehypnorm$ and $\gen$ be the vectorfields
	from Def.\,\ref{D:HYPNORMANDSPHEREFORMDEFS},
	let $\urescalednewgenminushypnorm$ be the vectorfield from from Def.\,\ref{D:RESCALEDVERSIONOFHYPNORMMINUSNEWGEN},
	and let $\lbrace \utandecompvectorfielddownarg{a} \rbrace_{a = 1,2,3}$ be the $\mathcal{S}_{\Timefunction}$-tangent vectorfields from
	Lemma~\ref{L:NEWDECOMPOSITIONOFCOORDINATEPARTIALDERIVATIVEVECTORFIELDS}.
	Let $\lengthofsidehypnorm$ be the scalar function from Def.\,\ref{D:LENGTHOFVARIOUSVECTORFIELDSETC}.
	Then 
	for smooth solutions (see Remark~\ref{R:SMOOTHNESSNOTNEEDED})
	to the compressible Euler equations \eqref{E:TRANSPORTDENSRENORMALIZEDRELATIVECTORECTANGULAR}-\eqref{E:ENTROPYTRANSPORT} on $\mathcal{M}$, 
	the following identities hold along $\underline{\mathcal{H}}$:
	\begin{align} \label{E:TRANSPORTLOGDENSITYISTANGENTIALEXCEPTINNULLCASE}
		\lengthofsidehypnorm^2 \Transport \LogDensity
		& = 
				-
				\left\lbrace
					\urescalednewgenminushypnorm_a \gen v^a
					+
					\sidehypnorm_a \urescalednewgenminushypnorm^a \gen \LogDensity
					+
					\utandecompvectorfielddownarg{a} v^a
					+
					(g^{-1})^{ab} \sidehypnorm_a \utandecompvectorfielddownarg{b} \LogDensity
					+
					\exp(-\LogDensity) \frac{p_{;\Ent}}{\bar{\varrho}} \sidehypnorm_a \GradEnt^a
				\right\rbrace.
	\end{align}
\end{lemma}

\begin{proof}
			First, 
			using \eqref{E:TRANSPORTDENSRENORMALIZEDRELATIVECTORECTANGULAR}
			and
			\eqref{E:NEWDECOMPOSITIONOFCOORDINATEPARTIALDERIVATIVEVECTORFIELDS},
			we compute that
			\begin{align} \label{E:FIRSTSTEPTRANSPORTLOGDENSITYISTANGENTIALEXCEPTINNULLCASE}
			\Transport \LogDensity 
				& = - \partial_a v^a 
				= 
				\sidehypnorm_a \Transport v^a 
				-
				\urescalednewgenminushypnorm_a \gen v^a
				-
				\utandecompvectorfielddownarg{a} v^a.
			\end{align}
		We then use \eqref{E:TRANSPORTVELOCITYRELATIVECTORECTANGULAR}
		and
		\eqref{E:NEWDECOMPOSITIONOFCOORDINATEPARTIALDERIVATIVEVECTORFIELDS}
		to express the first product on RHS~\eqref{E:FIRSTSTEPTRANSPORTLOGDENSITYISTANGENTIALEXCEPTINNULLCASE}
		as follows:
		\begin{align} \label{E:SECONDSTEPTRANSPORTLOGDENSITYISTANGENTIALEXCEPTINNULLCASE}
		\sidehypnorm_a \Transport v^a 
		& = - \Speed^2 \sidehypnorm_a \partial_a \LogDensity
				- 
				\exp(-\LogDensity) \frac{p_{;\Ent}}{\bar{\varrho}} \sidehypnorm_a \GradEnt^a
			\\
		& = 
		\Speed^2 \sidehypnorm_a \sidehypnorm_a \Transport \LogDensity
		-
		\Speed^2 \sidehypnorm_a \urescalednewgenminushypnorm_a \gen \LogDensity
		- 
		\Speed^2 \sidehypnorm_a \utandecompvectorfielddownarg{a} \LogDensity
		- 
		\exp(-\LogDensity) \frac{p_{;\Ent}}{\bar{\varrho}} \sidehypnorm_a \GradEnt^a.
		\notag
	\end{align}
	Next, using \eqref{E:ACOUSTICALMETRIC},
	the fact that $\sidehypnorm^0 = 1$ (i.e., \eqref{E:FUTURENORMALTOHYPERSURFACE}),
	and the fact that $v$ is $\Sigma_t$-tangent,
	we compute that
	\begin{align} \label{E:SIDEHYPNORMSPATIALLOWEREDINTERMSOFUPPER}
		\sidehypnorm_a 
		& = - v_a + \Speed^{-2} \sidehypnorm^a
		=
		- \Speed^{-2} v^a + \Speed^{-2} \sidehypnorm^a.
	\end{align}
	Using 
	\eqref{E:ACOUSTICALMETRIC}, 
	\eqref{E:FUTURENORMALTOHYPERSURFACE}, 
	and \eqref{E:SIDEHYPNORMSPATIALLOWEREDINTERMSOFUPPER},
	we compute that
	relative to the Cartesian coordinates, we have
	$
	\Speed^2 \sidehypnorm_a \sidehypnorm_a
	= v^a v_a (\sidehypnorm^0)^2
			+
			\Speed^{-2} \sidehypnorm^a \sidehypnorm^a
			-
			2 \sidehypnorm^0 v_a \sidehypnorm^a
			= \gfour_{\alpha \beta} \sidehypnorm^{\alpha} \sidehypnorm^{\beta}
				+
				1
	$.
	From this identity and \eqref{E:LENGTHOFHYPNORM},
	we find that
	\begin{align} \label{E:THIRDSTEPTRANSPORTLOGDENSITYISTANGENTIALEXCEPTINNULLCASE}
	\Speed^2 \sidehypnorm_a \sidehypnorm_a
	& = 1 - \lengthofsidehypnorm^2.
\end{align}
Substituting RHS~\eqref{E:THIRDSTEPTRANSPORTLOGDENSITYISTANGENTIALEXCEPTINNULLCASE} for the
relevant factors in the first product on RHS~\eqref{E:SECONDSTEPTRANSPORTLOGDENSITYISTANGENTIALEXCEPTINNULLCASE},
we express the first product on RHS~\eqref{E:FIRSTSTEPTRANSPORTLOGDENSITYISTANGENTIALEXCEPTINNULLCASE}
as follows:
\begin{align} \label{E:FOURTHSTEPTRANSPORTLOGDENSITYISTANGENTIALEXCEPTINNULLCASE}
		\sidehypnorm_a \Transport v^a 
		& = 
		(1 - \lengthofsidehypnorm^2)\Transport \LogDensity
		-
		\Speed^2 \sidehypnorm_a \urescalednewgenminushypnorm_a \gen \LogDensity
		- 
		\Speed^2 \sidehypnorm_a \utandecompvectorfielddownarg{a} \LogDensity
		- 
		\exp(-\LogDensity) \frac{p_{;\Ent}}{\bar{\varrho}} \sidehypnorm_a \GradEnt^a.
\end{align}
	Finally, substituting RHS~\eqref{E:FOURTHSTEPTRANSPORTLOGDENSITYISTANGENTIALEXCEPTINNULLCASE}
	for the first product on RHS~\eqref{E:FIRSTSTEPTRANSPORTLOGDENSITYISTANGENTIALEXCEPTINNULLCASE},
	noting that \eqref{E:INVERSEACOUSTICALMETRIC} implies that
	$(g^{-1})^{ab} = \Speed^2 \updelta^{ab}$ (where $\updelta^{ab}$ is the standard Kronecker delta),
	and carrying out straightforward algebraic computations, we conclude \eqref{E:TRANSPORTLOGDENSITYISTANGENTIALEXCEPTINNULLCASE}.
	
\end{proof}

\subsection{Expressions for $\partial_{\alpha} \vortrenormalized_{\beta} - \partial_{\beta} \vortrenormalized_{\alpha}$ 
and $\partial_{\alpha} \GradEnt_{\beta} - \partial_{\beta} \GradEnt_{\alpha}$}
In this subsection, 
we derive identities for the antisymmetric gradient tensorfields 
$\partial_{\alpha} \vortrenormalized_{\beta} - \partial_{\beta} \vortrenormalized_{\alpha}$ 
and $\partial_{\alpha} \GradEnt_{\beta} - \partial_{\beta} \GradEnt_{\alpha}$.
The identities hold only for compressible Euler solutions,
and their proof relies on the precise structure of the equations of Theorem~\ref{T:GEOMETRICWAVETRANSPORTSYSTEM}.
The identities play a crucial role in revealing the good structure of the term
$
2 	\weight
		\uposinnerproduct
		\uspecialgen^{\alpha}
		\angV^{\beta}
		(\partial_{\alpha} \SigmatTan_{\beta} - \partial_{\beta} \SigmatTan_{\alpha})
$
on the right-hand side of the boundary term identity \eqref{E:PRELIMINARYDECOMPOFBOUNDARYINTEGRAND}.

\subsubsection{Simple identities for $\Sigma_t$-tangent vectorfields}
\label{SSS:SIMPLEIDENTITIESFORSIGMATTANGENTVECTORFIELDS}
We start by deriving some identities for $\Sigma_t$-tangent vectorfields.

\begin{lemma}[Identities for $\Sigma_t$-tangent vectorfields]
	\label{L:SIMPLEIDENTITIESFORSIGMATTANGENTVECTORFIELDS}
	Let $\SigmatTan$ be a $\Sigma_t$-tangent vectorfield. Then relative to the Cartesian coordinates,
	the following identities hold for $\alpha, \beta = 0,1,2,3$ and $i,j=1,2,3$, 
	where $\Speed = \Speed(\LogDensity,\Ent)$ is the speed of sound,
	$\upepsilon_{ijk}$ denotes the fully antisymmetric symbol normalized by $\upepsilon_{123} = 1$,
	and $\upepsilon_{\alpha \beta \gamma \delta}$ denotes the fully antisymmetric symbol normalized by 
	$\upepsilon_{0123} = 1$:
	\begin{align} \label{E:LOWERINDICESOFSIGMATTANGENTVECTORFIELDS}
		\SigmatTan_0
		& = - g_{ab} v^a \SigmatTan^b,
		&
		\SigmatTan_i & = \Speed^{-2} \SigmatTan^i,
	\end{align}
	
	\begin{subequations}
	\begin{align}
		\partial_i \SigmatTan_j
		-
		\partial_j \SigmatTan_i
		& = 
			\Speed^{-2} \upepsilon_{ija} \Flatcurl \SigmatTan^a
			-
			2 (\partial_i \ln \Speed) \SigmatTan_j
			+
			2 (\partial_j \ln \Speed) \SigmatTan_i,
					\label{E:SPATIALPARTANTISYMMETRICGRADIENTOFDUALTOSIGMATTTANGENTVECTORFIELD} \\
	\Transport^{\beta} \partial_{\alpha} \SigmatTan_{\beta}
	& = - \SigmatTan_{\beta} (\partial_{\alpha} \Transport^{\beta})
		= - \SigmatTan_a \partial_{\alpha} v^a,
			\label{E:SWITCHTHEDERIVATIVESTOTRANSPORTFORSIGMATTANGENTVECTORFIELDS} \\
	\partial_t \SigmatTan_i
	-
	\partial_i \SigmatTan_0
	& =\Speed^{-2} \Transport \SigmatTan^i
			- 
			2 (\partial_t \ln \Speed) \SigmatTan_i
			+
			\SigmatTan_a \partial_i v^a
				\label{E:MIXEDTIMESPACEANTISYMMETRICGRADIENTOFDUALTOSIGMATTTANGENTVECTORFIELD} \\
		& \ \
			+
			v^a
			\left\lbrace
				\Speed^{-2}
				\upepsilon_{iab}
				(\Flatcurl \SigmatTan)^b
				-
				2
				(\partial_i \ln \Speed)
				\SigmatTan_a
			\right\rbrace,
			\notag
	\end{align}
	\end{subequations}
	
	\begin{align} \label{E:MAINDECOMPOSITIONOFANTISYMMETRICPARTOFGRADIENTOFSIGMATTANGENTONEFORM}
		\partial_{\alpha} \SigmatTan_{\beta}
		-
		\partial_{\beta} \SigmatTan_{\alpha}
			&   
				= 
				2 (\partial_{\beta} \ln \Speed) \SigmatTan_{\alpha} 
				- 
				2 (\partial_{\alpha} \ln \Speed) \SigmatTan_{\beta}
				+
				\updelta_{\alpha}^0 \SigmatTan_a \partial_{\beta} v^a
				-
				\updelta_{\beta}^0 \SigmatTan_a \partial_{\alpha} v^a
				\\
				& \ \
				+
				\left\lbrace
				\updelta_{\alpha}^0 
				\gfour_{\beta \gamma}
				-
				\updelta_{\beta}^0 
				\gfour_{\alpha \gamma}
				\right\rbrace
				\Transport \SigmatTan^{\gamma}
			+
			\Speed^{-2}
			\upepsilon_{\alpha \beta \gamma \delta}
			\Transport^{\gamma}
			(\Flatcurl \SigmatTan)^{\delta}.
			\notag
	\end{align}

\end{lemma}

\begin{proof}
		\eqref{E:LOWERINDICESOFSIGMATTANGENTVECTORFIELDS}-\eqref{E:MIXEDTIMESPACEANTISYMMETRICGRADIENTOFDUALTOSIGMATTTANGENTVECTORFIELD}  
		from straightforward calculations relative to the Cartesian coordinates 
		based on the identity $\partial_i \SigmatTan^j - \partial_j \SigmatTan^i = \upepsilon_{ija} (\Flatcurl \SigmatTan)^a$
		for $\Sigma_t$-tangent vectorfields $\SigmatTan$,
		the fact that 
		$\Transport^{\alpha} \SigmatTan_{\alpha} = \gfour(\Transport,\SigmatTan) = 0$ for $\Sigma_t$-tangent vectorfields $\SigmatTan$ 
		(see \eqref{E:TRANSPORTONEFORMIDENTITY}),
		\eqref{E:ACOUSTICALMETRIC},
		the decomposition $\partial_t = \Transport - v^a \partial_a$  (see \eqref{E:MATERIALVECTORVIELDRELATIVECTORECTANGULAR}),
		and the fact that $g_{ij} = \gfour_{ij} = \Speed^{-2} \updelta_{ij}$,
		where $\updelta_{ij}$ is the Kronecker delta (see \eqref{E:ACOUSTICALMETRIC} and \eqref{E:FIRSTFUNDSIGMATEQUALSSPACETIMEMETRICONSIGMAT}).
		
		\eqref{E:MAINDECOMPOSITIONOFANTISYMMETRICPARTOFGRADIENTOFSIGMATTANGENTONEFORM} is just a
		combining of
		\eqref{E:SPATIALPARTANTISYMMETRICGRADIENTOFDUALTOSIGMATTTANGENTVECTORFIELD} and \eqref{E:MIXEDTIMESPACEANTISYMMETRICGRADIENTOFDUALTOSIGMATTTANGENTVECTORFIELD}
		that takes into account the
		fact that $\Transport =  \partial_t + v^a \partial_a$,
		the identity $\SigmatTan_0 = - v^a \SigmatTan_a$ 
		(which follows from \eqref{E:LOWERINDICESOFSIGMATTANGENTVECTORFIELDS} and \eqref{E:ACOUSTICALMETRIC}),
		and the form \eqref{E:ACOUSTICALMETRIC}
		of the acoustical metric $\gfour$ relative to the Cartesian coordinates.
		
\end{proof}

\subsubsection{Specializing the identities to compressible Euler solutions}
\label{SSS:SPECIALIZINGIDENTITIESTOCOMPRESSIBLEULERSOLUTIONS}
In the next corollary, we specialize the results of Lemma~\ref{L:SIMPLEIDENTITIESFORSIGMATTANGENTVECTORFIELDS}
to $\vortrenormalized$ and $\GradEnt$. Unlike the proof of the lemma, the proof of
the corollary relies on the equations of Theorem~\ref{T:GEOMETRICWAVETRANSPORTSYSTEM}.

\begin{corollary}[Sharp decomposition of $\partial_{\alpha} \vortrenormalized_{\beta}
		-
		\partial_{\beta} \vortrenormalized_{\alpha}$
		and
		$
		\partial_{\alpha} \GradEnt_{\beta}
		-
		\partial_{\beta} \GradEnt_{\alpha}
		$]
	\label{C:SHARPDECOMPOSITIONOFANTISYMMETRICGRADIENTS}
	For smooth solutions (see Remark~\ref{R:SMOOTHNESSNOTNEEDED})
	to the compressible Euler equations \eqref{E:TRANSPORTDENSRENORMALIZEDRELATIVECTORECTANGULAR}-\eqref{E:ENTROPYTRANSPORT}
	on $\mathcal{M}$,
	the following identity holds, 
	where $\VortVort$ is the $\Sigma_t$-tangent modified fluid variable from \eqref{E:RENORMALIZEDCURLOFSPECIFICVORTICITY},
	$\vortrenormalized_{\flat}$ is the $\gfour$-dual one-form of $\vortrenormalized$ (see \eqref{E:DUALONEFORMINDICES}),
	and $d \vortrenormalized_{\flat}$ denotes the exterior derivative of $\vortrenormalized_{\flat}$:
	\begin{subequations}
	\begin{align} \label{E:KEYIDENTITYANTISYMMETRICPARTOFSPECIFICVORTICITYDUALGRADIENT}
		(d \vortrenormalized_{\flat})_{\alpha \beta}
		:=
		\partial_{\alpha} \vortrenormalized_{\beta}
		-
		\partial_{\beta} \vortrenormalized_{\alpha}
		& = 
				2 (\partial_{\beta} \ln \Speed) \vortrenormalized_{\alpha} 
				- 
				2 (\partial_{\alpha} \ln \Speed) \vortrenormalized_{\beta}
				+
				2 \updelta_{\alpha}^0 \vortrenormalized_a \partial_{\beta} v^a
				-
				2 \updelta_{\beta}^0 \vortrenormalized_a \partial_{\alpha} v^a
				\\
			& \ \
				-
				\Speed^{-4} 
				\exp(-2 \LogDensity) 
				\frac{p_{;\Ent}}{\bar{\varrho}}
				\upepsilon_{\alpha \beta \gamma \delta}
				(\Transport v^{\gamma}) 
				\GradEnt^{\delta}
					\notag
					\\
			& \ \
				+
			\Speed^{-4} 
			\exp(-2 \LogDensity) 
			\frac{p_{;\Ent}}{\bar{\varrho}}
			\upepsilon_{\alpha \beta \gamma \delta}
			 \Transport^{\gamma}
			[\GradEnt^{\delta}
				(\partial_a v^a)
				-
				\GradEnt^a \partial_a v^{\delta}]
			\notag
				\\
		& \ \
			+
			\Speed^{-2}
			\exp(\LogDensity)
			\upepsilon_{\alpha \beta \gamma \delta}
			\Transport^{\gamma}
			\VortVort^{\delta}.
			\notag
\end{align}
	
	Moreover, the following identity holds:
	\begin{align} \label{E:KEYIDENTITYANTISYMMETRICPARTOFENTROPYGRADIENTDUALGRADIENT}
		(d \GradEnt_{\flat})_{\alpha \beta}
		:=
		\partial_{\alpha} \GradEnt_{\beta}
		-
		\partial_{\beta} \GradEnt_{\alpha}
			&   
				= 
				2 (\partial_{\beta} \ln \Speed) \GradEnt_{\alpha} 
				- 
				2 (\partial_{\alpha} \ln \Speed) \GradEnt_{\beta}.
	\end{align}
	\end{subequations}
	
\end{corollary}

\begin{proof}
	To prove \eqref{E:KEYIDENTITYANTISYMMETRICPARTOFSPECIFICVORTICITYDUALGRADIENT},
	we consider \eqref{E:MAINDECOMPOSITIONOFANTISYMMETRICPARTOFGRADIENTOFSIGMATTANGENTONEFORM}
	with $\vortrenormalized$ in the role of $\SigmatTan$.
	We next note that definition \eqref{E:SPECIFICVORTICITY} implies that
	RHS~\eqref{E:SPECIFICVORTICITYLINEARORBETTER} 
	(which is equal to RHS~\eqref{E:RENORMALIZEDVORTICTITYTRANSPORTEQUATION})
	can alternatively be expressed as
	$
	\vortrenormalized_a \partial_i v^a
	+
	\exp(\LogDensity) \upepsilon_{aij} \vortrenormalized^a \vortrenormalized^j
	-
	\exp(-2 \LogDensity) \Speed^{-2} \frac{p_{;\Ent}}{\bar{\varrho}} \upepsilon_{iab} (\Transport v^a) \GradEnt^b
	=
	\vortrenormalized_a \partial_i v^a
	-
	\exp(-2 \LogDensity) \Speed^{-2} \frac{p_{;\Ent}}{\bar{\varrho}} \upepsilon_{iab} (\Transport v^a) \GradEnt^b
	$.
	We now substitute this ``alternative'' version of RHS~\eqref{E:RENORMALIZEDVORTICTITYTRANSPORTEQUATION} for the term
	$\Transport \vortrenormalized^{\gamma}$ on RHS~\eqref{E:MAINDECOMPOSITIONOFANTISYMMETRICPARTOFGRADIENTOFSIGMATTANGENTONEFORM}
	when $\gamma = i \in \lbrace 1,2,3 \rbrace$,
	and in the case $\gamma = 0$, we use the simple identity
	$\Transport \vortrenormalized^0 = 0$. 
	Next, we use \eqref{E:RENORMALIZEDCURLOFSPECIFICVORTICITY} to algebraically substitute
	for the factor $(\Flatcurl \vortrenormalized)^{\delta}$ on RHS~\eqref{E:MAINDECOMPOSITIONOFANTISYMMETRICPARTOFGRADIENTOFSIGMATTANGENTONEFORM}
	in terms of remaining terms in \eqref{E:RENORMALIZEDCURLOFSPECIFICVORTICITY}
	when $\delta = i \in \lbrace 1,2,3 \rbrace$,
	and in the case $\gamma = 0$, we use the fact that $(\Flatcurl \vortrenormalized)^0  = 0$.
	From these steps,
	the fact that $\Transport = \partial_t + v^a \partial_a$, 
	the form \eqref{E:ACOUSTICALMETRIC}
	of the acoustical metric $\gfour$ relative to the Cartesian coordinates,
	and straightforward algebraic calculations,
	we arrive at \eqref{E:KEYIDENTITYANTISYMMETRICPARTOFSPECIFICVORTICITYDUALGRADIENT}.
	
	A similar argument yields \eqref{E:KEYIDENTITYANTISYMMETRICPARTOFENTROPYGRADIENTDUALGRADIENT},
	where we use \eqref{E:GRADENTROPYTRANSPORT} for substitution,
	we observe that definition \eqref{E:SPECIFICVORTICITY} implies that RHS~\eqref{E:ENTROPYGRADIENTLINEARORBETTER} can alternatively be expressed as
	$-\GradEnt_a \partial_i v^a$,
	and we also use the simple identity $\Flatcurl \GradEnt = 0$ (see \eqref{E:CURLGRADENTVANISHES}).
	
\end{proof}

\subsection{Preliminary decomposition of the most subtle term on RHS~\eqref{E:PRELIMINARYDECOMPOFBOUNDARYINTEGRAND}}
\label{SS:MOSTSUBTLEENTROPYTERM}
In the next lemma, namely Lemma~\ref{L:PRELIMINARYDECOMPOSITIONOFSUBTLETERMS}, 
we provide a preliminary decomposition of the most subtle part of the term
$
2 	\weight
		\uposinnerproduct
		\uspecialgen^{\alpha}
		\angV^{\beta}
		(\partial_{\alpha} \vortrenormalized_{\beta} - \partial_{\beta} \vortrenormalized_{\alpha})
$,
which is the second term
on RHS~\eqref{E:PRELIMINARYDECOMPOFBOUNDARYINTEGRAND}
in the case $\SigmatTan = \vortrenormalized$.
Specifically, in the lemma, we decompose the part of 
$
2 	\weight
		\uposinnerproduct
		\uspecialgen^{\alpha}
		\angvortrenormalized^{\beta}
		(\partial_{\alpha} \vortrenormalized_{\beta} - \partial_{\beta} \vortrenormalized_{\alpha})
$
that corresponds to the terms
\begin{align} \label{E:HARDTERMS}
-
				\Speed^{-4} 
				\exp(-2 \LogDensity) 
				\frac{p_{;\Ent}}{\bar{\varrho}}
				\upepsilon_{\alpha \beta \gamma \delta}
				(\Transport v^{\gamma}) 
				\GradEnt^{\delta}
		+
			\Speed^{-4} 
			\exp(-2 \LogDensity) 
			\frac{p_{;\Ent}}{\bar{\varrho}}
			\upepsilon_{\alpha \beta \gamma \delta}
			 \Transport^{\gamma}
			[\GradEnt^{\delta}
				(\partial_a v^a)
				-
				\GradEnt^a \partial_a v^{\delta}]
\end{align}
on RHS~\eqref{E:KEYIDENTITYANTISYMMETRICPARTOFSPECIFICVORTICITYDUALGRADIENT};
more precisely, in the lemma, we ignore the overall factor of
$
\Speed^{-4} 
				\exp(-2 \LogDensity) 
				\frac{p_{;\Ent}}{\bar{\varrho}}
$
in the previous expression.
Later, in Prop.\,\ref{P:KEYDETERMINANT}, 
with the help of Lemma~\ref{L:PRELIMINARYDECOMPOSITIONOFSUBTLETERMS}, 
we will show that for compressible Euler solutions,
the special combination of terms in \eqref{E:HARDTERMS}
	can be expressed in terms of $\underline{\mathcal{H}}$-tangential derivatives of
	of the fluid solution variables.
The ``remaining part'' of the term
$
2 	\weight
		\uposinnerproduct
		\uspecialgen^{\alpha}
		\angvortrenormalized^{\beta}
		(\partial_{\alpha} \vortrenormalized_{\beta} - \partial_{\beta} \vortrenormalized_{\alpha})
$,
as well as the entirety of the term
$
2 
		\uposinnerproduct
		\uspecialgen^{\alpha}
		\angGradEnt^{\beta}
		(\partial_{\alpha} \GradEnt_{\beta} - \partial_{\beta} \GradEnt_{\alpha})
$,
will be easy to treat, thanks to the decompositions provided by Cor.\,\ref{C:SHARPDECOMPOSITIONOFANTISYMMETRICGRADIENTS}.
		
\begin{lemma}[Preliminary geometric decomposition of the most subtle terms on RHS~\eqref{E:KEYIDENTITYANTISYMMETRICPARTOFSPECIFICVORTICITYDUALGRADIENT}]	
	\label{L:PRELIMINARYDECOMPOSITIONOFSUBTLETERMS}
	Let $\sidehypnorm$ and $\gen$ be the vectorfields
	from Def.\,\ref{D:HYPNORMANDSPHEREFORMDEFS},
	let $\uspecialgen$ be the $\underline{\mathcal{H}}$-tangent vectorfield
	from \eqref{E:SPECIALGENERATORIDENTITY},
	let $\urescalednewgenminushypnorm$ be the vectorfield from Def.\,\ref{D:RESCALEDVERSIONOFHYPNORMMINUSNEWGEN}, 
	and let $\lbrace \utandecompvectorfielddownarg{a} \rbrace_{a = 1,2,3}$ be the $\mathcal{S}_{\Timefunction}$-tangent vectorfields from
	Lemma~\ref{L:NEWDECOMPOSITIONOFCOORDINATEPARTIALDERIVATIVEVECTORFIELDS}.
	For smooth solutions (see Remark~\ref{R:SMOOTHNESSNOTNEEDED})
	to the compressible Euler equations \eqref{E:TRANSPORTDENSRENORMALIZEDRELATIVECTORECTANGULAR}-\eqref{E:ENTROPYTRANSPORT} on $\mathcal{M}$,
	the following identity holds along $\underline{\mathcal{H}}$:
	\begin{align}
		&
		\upepsilon_{\alpha \beta \gamma \delta}
		 \uspecialgen^{\alpha} 
		 \angvortrenormalized^{\beta}
		\left\lbrace
		-
		(\Transport v^{\gamma}) 
		\GradEnt^{\delta}
		+
		\Transport^{\gamma}
			[\GradEnt^{\delta}
				(\partial_a v^a)
				-
				\GradEnt^a \partial_a v^{\delta}]
		\right\rbrace
					\label{E:PRELIMINARYDECOMPOSITIONOFSUBTLETERMS} 
					\\
	& =
		-
		\upepsilon_{\alpha \beta \gamma \delta}
		 \uspecialgen^{\alpha} 
		 \angvortrenormalized^{\beta}
		\sidehypnorm^{\gamma}
		(\GradEnt^{\delta} + \GradEnt^a \sidehypnorm_a \Transport^{\delta})
		\Transport \LogDensity
			\notag 
				\\
	& \ \
		+
		\upepsilon_{\alpha \beta \gamma \delta}
		 \uspecialgen^{\alpha} 
		 \angvortrenormalized^{\beta}
		\urescalednewgenminushypnorm^{\gamma}
		\GradEnt^{\delta}
		\gen \LogDensity
		-
		\GradEnt^a \sidehypnorm_a
		\upepsilon_{\alpha \beta \gamma \delta}
		\uspecialgen^{\alpha} 
		\angvortrenormalized^{\beta} 
		\Transport^{\gamma}
		\urescalednewgenminushypnorm^{\delta} 
		\gen \LogDensity
			\notag 
			\\
	 & \ \
		+
		\Speed^2
		\upepsilon_{\alpha \beta ab}
		\uspecialgen^{\alpha}  
		\angvortrenormalized^{\beta}
		(\utandecompvectorfielddownarg{a} \LogDensity)
		\GradEnt^b
		-
		\Speed^2
		\GradEnt^a \sidehypnorm_a
		\upepsilon_{\alpha \beta \gamma d}
		\uspecialgen^{\alpha} 
		\angvortrenormalized^{\beta} 
		\Transport^{\gamma}
		\utandecompvectorfielddownarg{d} \LogDensity
		\notag
			\\
	& \ \
		-
		\GradEnt^a \urescalednewgenminushypnorm_a
		\upepsilon_{\alpha \beta \gamma \delta}
		\uspecialgen^{\alpha} 
		\angvortrenormalized^{\beta} 
		\Transport^{\gamma}
		\gen v^{\delta}
		-
		\upepsilon_{\alpha \beta \gamma \delta}
		\uspecialgen^{\alpha} 
		\angvortrenormalized^{\beta} 
		\Transport^{\gamma}
		\GradEnt^a \utandecompvectorfielddownarg{a} v^{\delta}
		\notag
			\\
	& \ \
			-
		\exp(-\LogDensity) \frac{p_{;\Ent}}{\bar{\varrho}}
		\GradEnt^a \sidehypnorm_a
		\upepsilon_{\alpha \beta \gamma \delta}
		\uspecialgen^{\alpha} 
		\angvortrenormalized^{\beta} 
		\Transport^{\gamma}
		\GradEnt^{\delta}.
		\notag
	\end{align}
\end{lemma}

\begin{proof}
	First, using \eqref{E:TRANSPORTVELOCITYRELATIVECTORECTANGULAR},
	\eqref{E:NEWDECOMPOSITIONOFCOORDINATEPARTIALDERIVATIVEVECTORFIELDS},
	\eqref{E:SIDEHYPNORMSPATIALLOWEREDINTERMSOFUPPER},
	the fact that $\urescalednewgenminushypnorm$ is $\Sigma_t$-tangent (see Lemma~\ref{L:PROPERTIESOFRESCALEDGENERATORMINUSSIDEHYPNORM}),
	and the form \eqref{E:ACOUSTICALMETRIC}
	of the acoustical metric $\gfour$ relative to the Cartesian coordinates,
	we compute that
	\begin{align} \label{E:FIRSTSTEPPRELIMINARYDECOMPOSITIONOFSUBTLETERMS}
		\Transport v^j
		& = 
			- \Speed^2 \partial_j \LogDensity
			- 
			\exp(-\LogDensity) \frac{p_{;\Ent}}{\bar{\varrho}} \GradEnt^j
				\\
		& = 
			(\sidehypnorm^j - v^j) \Transport \LogDensity
			- 
			\urescalednewgenminushypnorm^j \gen \LogDensity
			- 
			\Speed^2 \utandecompvectorfielddownarg{j} \LogDensity
			- 
			\exp(-\LogDensity) \frac{p_{;\Ent}}{\bar{\varrho}} \GradEnt^j.
				\notag
	\end{align}
	
	Next, using \eqref{E:NEWDECOMPOSITIONOFCOORDINATEPARTIALDERIVATIVEVECTORFIELDS},
	\eqref{E:FIRSTSTEPPRELIMINARYDECOMPOSITIONOFSUBTLETERMS},
	and \eqref{E:SIDEHYPNORMSPATIALLOWEREDINTERMSOFUPPER},
	we compute that
	\begin{align} \label{E:SECONDSTEPPRELIMINARYDECOMPOSITIONOFSUBTLETERMS}
		\partial_i v^j
		& = - 
				\sidehypnorm_i \Transport v^j
				+
				\urescalednewgenminushypnorm_i \gen v^j
				+
				\utandecompvectorfielddownarg{i} v^j
					\\
		& = 
			-
			\sidehypnorm_i (\sidehypnorm^j - v^j) \Transport \LogDensity
			+ 
			\Speed^2 \sidehypnorm_i \urescalednewgenminushypnorm_j \gen \LogDensity
			+ 
			\Speed^2 \sidehypnorm_i \utandecompvectorfielddownarg{j} \LogDensity
			+
			\urescalednewgenminushypnorm_i \gen v^j
			+
			\utandecompvectorfielddownarg{i} v^j
			+
			\exp(-\LogDensity) \frac{p_{;\Ent}}{\bar{\varrho}} \sidehypnorm_i \GradEnt^j.
			\notag
	\end{align}
	Contracting \eqref{E:SECONDSTEPPRELIMINARYDECOMPOSITIONOFSUBTLETERMS} against $\GradEnt^i$, we deduce that
	\begin{align} \label{E:THIRDSTEPPRELIMINARYDECOMPOSITIONOFSUBTLETERMS}
		\GradEnt^a \partial_a v^j
		& = - 
				\GradEnt^a \sidehypnorm_a \Transport v^j
				+
				\GradEnt^a \urescalednewgenminushypnorm_a \gen v^j
				+
				\GradEnt^a \utandecompvectorfielddownarg{a} v^j
					\\
		& = 
			-
			\GradEnt^a \sidehypnorm_a (\sidehypnorm^j - v^j) \Transport \LogDensity
			+ 
			\Speed^2 \GradEnt^a \sidehypnorm_a \urescalednewgenminushypnorm_j \gen \LogDensity
			+ 
			\Speed^2 \GradEnt^a \sidehypnorm_a \utandecompvectorfielddownarg{j} \LogDensity
				\notag \\
		&  \ \
			+
			\GradEnt^a \urescalednewgenminushypnorm_a \gen v^j
			+
			\GradEnt^a \utandecompvectorfielddownarg{a} v^j
			+
			\exp(-\LogDensity) \frac{p_{;\Ent}}{\bar{\varrho}} \GradEnt^a \sidehypnorm_a \GradEnt^j.
			\notag
	\end{align}
	
	Using \eqref{E:FIRSTSTEPPRELIMINARYDECOMPOSITIONOFSUBTLETERMS} to substitute for the spatial components of the factor $\Transport v^{\gamma}$ 
	on LHS~\eqref{E:PRELIMINARYDECOMPOSITIONOFSUBTLETERMS},
	using \eqref{E:TRANSPORTDENSRENORMALIZEDRELATIVECTORECTANGULAR} to replace the factor $\partial_a v^a$ on LHS~\eqref{E:PRELIMINARYDECOMPOSITIONOFSUBTLETERMS}
	with $- \Transport \LogDensity$,
	using
	\eqref{E:THIRDSTEPPRELIMINARYDECOMPOSITIONOFSUBTLETERMS}
	to substitute for the spatial components of the last term 
	$\GradEnt^a \partial_a v^{\delta}$
	on LHS~\eqref{E:PRELIMINARYDECOMPOSITIONOFSUBTLETERMS},
	using the identities $\Transport^j = v^j$ and
	$0 = \GradEnt^0 = \urescalednewgenminushypnorm^0 = v^0 = \sidehypnorm^0 - \Transport^0$,
	and taking into account the form \eqref{E:ACOUSTICALMETRIC}
	of the acoustical metric $\gfour$ relative to the Cartesian coordinates,
	we deduce the following identity:
	\begin{align}
		&
		-
		\upepsilon_{\alpha \beta \gamma \delta}
		 \uspecialgen^{\alpha} 
		 \angvortrenormalized^{\beta}
		(\Transport v^{\gamma}) 
		\GradEnt^{\delta}
		+
		\upepsilon_{\alpha \beta \gamma \delta}
			\uspecialgen^{\alpha} 
			\angvortrenormalized^{\beta} 
			\Transport^{\gamma}
			[\GradEnt^{\delta}
				(\partial_a v^a)
				-
				\GradEnt^a \partial_a v^{\delta}]
					\label{E:FOURTHSTEPPRELIMINARYDECOMPOSITIONOFSUBTLETERMS}  \\
	& =
		\left\lbrace
		-
		\upepsilon_{\alpha \beta \gamma \delta}
		 \uspecialgen^{\alpha} 
		 \angvortrenormalized^{\beta}
		(\sidehypnorm^{\gamma} - \Transport^{\gamma})
		\GradEnt^{\delta}
		-
		\upepsilon_{\alpha \beta \gamma \delta}
		\uspecialgen^{\alpha} 
		\angvortrenormalized^{\beta}
		\Transport^{\gamma}
		\GradEnt^{\delta}
		+
		\upepsilon_{\alpha \beta \gamma \delta}
		\uspecialgen^{\alpha} 
		\angvortrenormalized^{\beta}
		\Transport^{\gamma}
		\GradEnt^a \sidehypnorm_a
		(\sidehypnorm^{\delta} - \Transport^{\delta})
		\right\rbrace
		\Transport \LogDensity
			\notag \\
		& \ \
		+
		\upepsilon_{\alpha \beta \gamma \delta}
		 \uspecialgen^{\alpha} 
		 \angvortrenormalized^{\beta}
		\urescalednewgenminushypnorm^{\gamma}
		\GradEnt^{\delta}
		\gen \LogDensity
		-
		\GradEnt^a \sidehypnorm_a
		\upepsilon_{\alpha \beta \gamma \delta}
		\uspecialgen^{\alpha} 
		\angvortrenormalized^{\beta} 
		\Transport^{\gamma}
		\urescalednewgenminushypnorm^{\delta} 
		\gen \LogDensity
			\notag 
			\\
	 & \ \
		+
		\Speed^2
		\upepsilon_{\alpha \beta ab}
		\uspecialgen^{\alpha}  
		\angvortrenormalized^{\beta}
		(\utandecompvectorfielddownarg{a} \LogDensity)
		\GradEnt^b
		-
		\Speed^2
		\GradEnt^a \sidehypnorm_a
		\upepsilon_{\alpha \beta \gamma d}
		\uspecialgen^{\alpha} 
		\angvortrenormalized^{\beta} 
		\Transport^{\gamma}
		\utandecompvectorfielddownarg{d} \LogDensity
		\notag
			\\
	& \ \
		-
		\GradEnt^a \urescalednewgenminushypnorm_a
		\upepsilon_{\alpha \beta \gamma \delta}
		\uspecialgen^{\alpha} 
		\angvortrenormalized^{\beta} 
		\Transport^{\gamma}
		\gen v^{\delta}
		-
		\upepsilon_{\alpha \beta \gamma \delta}
		\uspecialgen^{\alpha} 
		\angvortrenormalized^{\beta} 
		\Transport^{\gamma}
		\GradEnt^a \utandecompvectorfielddownarg{a} v^{\delta}
		\notag
			\\
	& \ \
			-
		\exp(-\LogDensity) \frac{p_{;\Ent}}{\bar{\varrho}}
		\GradEnt^a \sidehypnorm_a
		\upepsilon_{\alpha \beta \gamma \delta}
		\uspecialgen^{\alpha} 
		\angvortrenormalized^{\beta} 
		\Transport^{\gamma}
		\GradEnt^{\delta}.
		\notag
		\end{align}
		The desired identity \eqref{E:PRELIMINARYDECOMPOSITIONOFSUBTLETERMS} 
		now follows as a simple
		algebraic consequence of \eqref{E:FOURTHSTEPPRELIMINARYDECOMPOSITIONOFSUBTLETERMS}.
\end{proof}

\subsection{Remarkable geometric structure of the error term $
		-
		\upepsilon_{\alpha \beta \gamma \delta}
		 \uspecialgen^{\alpha} 
		 \angvortrenormalized^{\beta}
		\sidehypnorm^{\gamma}
		(\GradEnt^{\delta} + \GradEnt^a \sidehypnorm_a \Transport^{\delta})
		\Transport \LogDensity
$}
\label{SS:REMARKABLESTRUCTUREOFMOSTSUBTLETERM}
Recall that in Lemma~\ref{L:PRELIMINARYDECOMPOSITIONOFSUBTLETERMS},
we provided a preliminary decomposition of the most subtle part of the error term
$
2 
		\weight
		\uposinnerproduct
\uspecialgen^{\alpha}
			\angV^{\beta}
			(\partial_{\alpha} \SigmatTan_{\beta} - \partial_{\beta} \SigmatTan_{\alpha})$
on RHS~\eqref{E:PRELIMINARYDECOMPOFBOUNDARYINTEGRAND}
(see also RHS~\eqref{E:INTEGRATEDPRELIMINARYDECOMPOFBOUNDARYINTEGRAND}).
The subtle part arises in the case $\SigmatTan = \vortrenormalized$,
and we provided the decomposition of it in equation \eqref{E:PRELIMINARYDECOMPOSITIONOFSUBTLETERMS}.
All products on RHS~\eqref{E:PRELIMINARYDECOMPOSITIONOFSUBTLETERMS} manifestly involve only $\underline{\mathcal{H}}$-tangential
derivatives of the fluid solution, 
except for the first product
$
		-
		\upepsilon_{\alpha \beta \gamma \delta}
		 \uspecialgen^{\alpha} 
		 \angvortrenormalized^{\beta}
		\sidehypnorm^{\gamma}
		(\GradEnt^{\delta} + \GradEnt^a \sidehypnorm_a \Transport^{\delta})
		\Transport \LogDensity
$.
In Prop.\,\ref{P:KEYDETERMINANT},
we show that for compressible Euler solutions,
this remaining product can be re-expressed in terms of $\underline{\mathcal{H}}$-tangential
derivatives of the solution. Moreover, in the crucial case in which the
lateral boundary $\underline{\mathcal{H}}$ is $\gfour$-null, \emph{the product completely vanishes}.
The results of the proposition are of fundamental importance for our main results.
The proof relies on the precise structure of some of the transport equations in Theorem~\ref{T:GEOMETRICWAVETRANSPORTSYSTEM}.

\subsubsection{A decomposition of the entropy gradient}
\label{SSS:ENTROPYVECTORFIELDKEYTENSORIALDECOMPOSITION}
In our proof of Prop.\,\ref{P:KEYDETERMINANT}, we will use 
the simple geometric decomposition of $\GradEnt$ provided by the following lemma.

\begin{lemma}[Decomposition of the entropy gradient]
	\label{L:ENTROPYVECTORFIELDKEYTENSORIALDECOMPOSITION}
	Let $\LeftoverGradEnt$ be the vectorfield whose Cartesian components $\LeftoverGradEnt^{\alpha}$ are defined
	by the following equation, where $\GradEnt$ is the ($\Sigma_t$-tangent)
	entropy gradient vectorfield,
	the vectorfields $\tophypnorm$ and $\sidehypnorm$
	are as in Def.\,\ref{D:HYPNORMANDSPHEREFORMDEFS},
	and $\seconduposinnerproduct > 0$ is the scalar function
	defined in \eqref{E:SECONDINGOINGCONDITION}:
	\begin{align} \label{E:ENTROPYVECTORFIELDKEYTENSORIALDECOMPOSITION}
		\GradEnt^{\alpha}
		& =
			-
			\GradEnt^a \sidehypnorm_a
			\Transport^{\alpha}
			+
			\left\lbrace
				\frac{\GradEnt^a \sidehypnorm_a}{\seconduposinnerproduct}
				- 
				\frac{\GradEnt^a \tophypnorm_a}{\seconduposinnerproduct}
			\right\rbrace
			\gen^{\alpha}
			+
			\LeftoverGradEnt^{\alpha}.
	\end{align}
	Then $\LeftoverGradEnt$ is $\mathcal{S}_{\Timefunction}$-tangent.
\end{lemma}

\begin{proof}
	Contracting each side of \eqref{E:ENTROPYVECTORFIELDKEYTENSORIALDECOMPOSITION}
	against $\sidehypnorm_{\alpha}$, using that $\gfour(\sidehypnorm,\gen) = 0$,
	and using \eqref{E:EQUIVALENTFUTURENORMALTOHYPERSURFACE},
	we find that $\gfour(\LeftoverGradEnt,\sidehypnorm) = 0$.
	Next, contracting each side of \eqref{E:ENTROPYVECTORFIELDKEYTENSORIALDECOMPOSITION}
	against $\tophypnorm_{\alpha}$ and using 
	\eqref{E:EQUIVALENTFUTURENORMALTOTOPHYPERSURFACE}
	and
	\eqref{E:SECONDINGOINGCONDITION},
	we find that $\gfour(\LeftoverGradEnt,\tophypnorm) = 0$. Thus, $\LeftoverGradEnt$ is $\gfour$-orthogonal
	to $\mbox{\upshape span} \lbrace \tophypnorm, \sidehypnorm \rbrace$, which
	is equal to the $\gfour$-orthogonal
	complement of $\mathcal{S}_{\Timefunction}$
	(as we noted in the proof of Lemma~\ref{L:KEYIDBETWEENVARIOUSVECTORFIELDS}). 
	Thus, $\LeftoverGradEnt$ is $\mathcal{S}_{\Timefunction}$-tangent as desired.
\end{proof}	

\subsubsection{An $\mathcal{S}_{\Timefunction}$-tangent vectorfield arising in the analysis}
\label{SSS:VECTORFIELDINKEYDETERMINANT}
In our proof of Prop.\,\ref{P:KEYDETERMINANT}, we will encounter the vectorfield featured in the following definition.

\begin{definition}[An $\mathcal{S}_{\Timefunction}$-tangent vectorfield arising in Prop.\,\ref{P:KEYDETERMINANT}]
		\label{D:VECTORFIELDINKEYDETERMINANT}
			Let $\utang$ be the $\mathcal{S}_{\Timefunction}$-tangent vectorfield from Lemma~\ref{L:KEYIDBETWEENVARIOUSVECTORFIELDS},
			let $\LeftoverGradEnt$ be the $\mathcal{S}_{\Timefunction}$-tangent vectorfield
			from Lemma~\ref{L:ENTROPYVECTORFIELDKEYTENSORIALDECOMPOSITION},
			let $\tophypnorm$ and $\sidehypnorm$
			be the vectorfields from Def.\,\ref{D:HYPNORMANDSPHEREFORMDEFS},
			and let $\seconduposinnerproduct > 0$ be the scalar function
			defined in \eqref{E:SECONDINGOINGCONDITION}.
			We define the $\mathcal{S}_{\Timefunction}$-tangent vectorfield $\keydetvectorfield$ as follows:
		\begin{align} \label{E:VECTORFIELDINKEYDETERMINANT}
			\keydetvectorfield^{\alpha}
			& :=
			\left\lbrace
				\frac{\GradEnt^a \sidehypnorm_a}{\seconduposinnerproduct}
				- 
				\frac{\GradEnt^a \tophypnorm_a}{\seconduposinnerproduct}
			\right\rbrace
			\utang^{\alpha} 
		-
		\frac{1}{\seconduposinnerproduct}
		\LeftoverGradEnt^{\alpha}.
		\end{align}
\end{definition}

\subsubsection{The main proposition revealing the structure of the term 
$\upepsilon_{\alpha \beta \gamma \delta} 
				\uspecialgen^{\alpha}
				\angvortrenormalized^{\beta}
				\sidehypnorm^{\gamma}
				(\GradEnt^{\delta} 
				+
				\GradEnt^a \sidehypnorm_a
				\Transport^{\delta})
				\Transport \LogDensity$
on RHS~\eqref{E:PRELIMINARYDECOMPOSITIONOFSUBTLETERMS}}
In the next proposition, we provide the main result of Subsect.\,\ref{SS:REMARKABLESTRUCTUREOFMOSTSUBTLETERM}.
We show that for compressible Euler solutions, the first product 
$
\upepsilon_{\alpha \beta \gamma \delta} 
				\uspecialgen^{\alpha}
				\angvortrenormalized^{\beta}
				\sidehypnorm^{\gamma}
				(\GradEnt^{\delta} 
				+
				\GradEnt^a \sidehypnorm_a
				\Transport^{\delta})
				\Transport \LogDensity
$
on RHS~\eqref{E:PRELIMINARYDECOMPOSITIONOFSUBTLETERMS}
can be expressed in terms of the $\underline{\mathcal{H}}$-tangential derivatives of the fluid solution
and moreover, that the product vanishes when $\underline{\mathcal{H}}$ is $\gfour$-null.
Lemma~\ref{L:TRANSPORTLOGDENSITYISTANGENTIALEXCEPTINNULLCASE} plays a key role in the proof.

\begin{proposition}[The key determinant-product calculation]
		\label{P:KEYDETERMINANT}
		Let $\sidehypnorm$ and $\gen$ be the vectorfields from
		Def.\,\ref{D:HYPNORMANDSPHEREFORMDEFS},
		let $\uspecialgen$ be the $\underline{\mathcal{H}}$-tangent vectorfield
		from \eqref{E:SPECIALGENERATORIDENTITY}, 
		let $\urescalednewgenminushypnorm$ be the vectorfield from Def.\,\ref{D:RESCALEDVERSIONOFHYPNORMMINUSNEWGEN},
		let $\lbrace \utandecompvectorfielddownarg{a} \rbrace_{a=1,2,3}$ be the $\mathcal{S}_{\Timefunction}$-tangent vectorfields from
		Lemma~\ref{L:NEWDECOMPOSITIONOFCOORDINATEPARTIALDERIVATIVEVECTORFIELDS},
		and let $\keydetvectorfield$ be the $\mathcal{S}_{\Timefunction}$-tangent vectorfield
		from Def.\,\ref{D:VECTORFIELDINKEYDETERMINANT}.
		Let $\lengthofgen \geq 0$ and $\lengthoftophypnorm > 0$
		be the scalar functions from
		Def.\,\ref{D:LENGTHOFVARIOUSVECTORFIELDSETC},
		and let $\seconduposinnerproduct > 0$ be the scalar function
		defined in \eqref{E:SECONDINGOINGCONDITION}.
		Consider a smooth solution (see Remark~\ref{R:SMOOTHNESSNOTNEEDED})
		to the compressible Euler equations \eqref{E:TRANSPORTDENSRENORMALIZEDRELATIVECTORECTANGULAR}-\eqref{E:ENTROPYTRANSPORT} on $\mathcal{M}$.
		We define
		\begin{align} \label{E:SIGNOFMAINDETERMINANTTERM}
			\upsigma
			& :=
		\mathop{\mathrm{sgn}}
		\left(
			\upepsilon_{\alpha \beta \gamma \delta} 
				\uspecialgen^{\alpha}
				\angvortrenormalized^{\beta}
				\sidehypnorm^{\gamma}
				(\GradEnt^{\delta} 
				+
				\GradEnt^a \sidehypnorm_a
				\Transport^{\delta})
		\right),
		\end{align}
		where $\mathop{\mathrm{sgn}}(0) := 0$ and for $r \in \mathbb{R}$ with $r \neq 0$, $\mathop{\mathrm{sgn}}(r) := \frac{r}{|r|}$.
		If $\underline{\mathcal{H}}$ is $\gfour$-spacelike,
		then along $\underline{\mathcal{H}}$,
		the following identity holds relative to the Cartesian coordinates:
		\begin{align} \label{E:KEYDETERMINANT}
				&
				\upepsilon_{\alpha \beta \gamma \delta} 
				\uspecialgen^{\alpha}
				\angvortrenormalized^{\beta}
				\sidehypnorm^{\gamma}
				(\GradEnt^{\delta} 
				+
				\GradEnt^a \sidehypnorm_a
				\Transport^{\delta})
				\Transport \LogDensity
					\\
				& =
		-
		\upsigma
		\Speed^3
		\frac{\seconduposinnerproduct + \lengthofgen^2}{\sqrt{\seconduposinnerproduct^2 + \lengthofgen^2 \lengthoftophypnorm^2}}
		\sqrt{\mbox{\upshape det}
		\begin{pmatrix}
		\gsphere(\angvortrenormalized,\angvortrenormalized) & 
		\gsphere(\angvortrenormalized,\keydetvectorfield)
			\\
		\gsphere(\angvortrenormalized,\keydetvectorfield) & \gsphere(\keydetvectorfield,\keydetvectorfield)
	\end{pmatrix}}
		\notag \\
	& 
	\ \ \ \ \ \
	\times
	\left\lbrace
			\urescalednewgenminushypnorm_a \gen v^a
				+
				\sidehypnorm_a \urescalednewgenminushypnorm^a \gen \LogDensity
				+
				\utandecompvectorfielddownarg{a} v^a
				+ 
				(g^{-1})^{ab} \sidehypnorm_a \utandecompvectorfielddownarg{b} \LogDensity
				+
			  \exp(-\LogDensity) \frac{p_{;\Ent}}{\bar{\varrho}} \sidehypnorm_a \GradEnt^a
	\right\rbrace.
	\notag
		\end{align}	
		In particular, all derivatives of $\LogDensity$ and $\lbrace v^a \rbrace_{a=1,2,3}$ on RHS~\eqref{E:KEYDETERMINANT} 
		are $\underline{\mathcal{H}}$-tangential.
		
		Moreover, if $\underline{\mathcal{H}}$ is $\gfour$-null,
		then we have (see Convention~\ref{C:NULLCASE}):
		\begin{align} \label{E:NULLCASEKEYDETERMINANT}
		\upepsilon_{\alpha \beta \gamma \delta} 
				\uspecialgen^{\alpha}
				\angvortrenormalized^{\beta}
				\sidehypnorm^{\gamma} 
				(\GradEnt^{\delta} 
				+
				\GradEnt^a \sidehypnorm_a
				\Transport^{\delta})
				\Transport \LogDensity
				:=
				\uspecialgen^{\alpha}
				\angvortrenormalized^{\beta}
				\uLunit^{\gamma} 
				(\GradEnt^{\delta} 
				+
				\GradEnt^a \uLunit_a
				\Transport^{\delta})
				\Transport \LogDensity
				& = 0
				= \mbox{\upshape RHS~\eqref{E:KEYDETERMINANT}}.
		\end{align}	
		
\end{proposition}

\begin{proof}
First, using
\eqref{E:SPECIALGENERATORIDENTITY}, 
\eqref{E:ENTROPYVECTORFIELDKEYTENSORIALDECOMPOSITION},
\eqref{E:VECTORFIELDINKEYDETERMINANT},
and the fact that $\upepsilon_{\alpha \beta \gamma \delta}$ must vanish when contracted in three or more slots
against vectorfields tangent to the two-dimensional submanifold $\mathcal{S}_{\Timefunction}$,
we compute that
\begin{align}  \label{E:FIRSTSTEPKEYDETERMINANT}
	\upepsilon_{\alpha \beta \gamma \delta} 
				\uspecialgen^{\alpha}
				\angvortrenormalized^{\beta}
				\sidehypnorm^{\gamma} 
				(\GradEnt^{\delta} 
				+
				\GradEnt^a \sidehypnorm_a
				\Transport^{\delta})
	& = 
			\upepsilon_{\alpha \beta \gamma \delta}
			\sidehypnorm^{\alpha}
			\gen^{\beta}
			\keydetvectorfield^{\gamma} 
			\angvortrenormalized^{\delta}.
\end{align}

Next we note the following standard fact:
for arbitrary sets of four vectorfields
$\lbrace X_{(1)}, X_{(2)}, X_{(3)}, X_{(4)} \rbrace$,
$
\left\lbrace
	\upepsilon_{\alpha \beta \gamma \delta} X_{(1)}^{\alpha} X_{(2)}^{\beta} X_{(3)}^{\gamma} X_{(4)}^{\delta}
\right\rbrace^2
$
is equal to $(\mbox{\upshape det} \gfour)^{-1}$ 
(where the determinant is taken relative to the Cartesian coordinates)
times the determinant of the $4 \times 4$ matrix whose $(A,B)$ entry is $\gfour(X_{(A)},X_{(B)})$.
Moreover, using \eqref{E:ACOUSTICALMETRIC}, we compute that $\mbox{\upshape det} \gfour = - \Speed^{-6}$.
From these facts, 
\eqref{E:RATIOLENGTHOFSIDEHYPNORMLENGTHOFGEN},
the relations $\gfour(\sidehypnorm,\sidehypnorm) = - \lengthofsidehypnorm^2$,
$\gfour(\gen,\gen) = \lengthofgen^2$,
$\gfour(\sidehypnorm,\gen)=0$,
and the fact that $\sidehypnorm$ and $\gen$ are $\gfour$-orthogonal to the $\mathcal{S}_{\Timefunction}$-tangent
vectorfields $\angvortrenormalized$ and $\keydetvectorfield$,
we express the square of RHS~\eqref{E:FIRSTSTEPKEYDETERMINANT} as follows:
\begin{align} \label{E:SECONDSTEPKEYDETERMINANT}
		\left\lbrace
		\upepsilon_{\alpha \beta \gamma \delta} 
		\sidehypnorm^{\alpha}
		\gen^{\beta} 
		\angvortrenormalized^{\gamma}
		\keydetvectorfield^{\delta}
		\right\rbrace^2
	& =
	\Speed^6
	\mbox{\upshape det}
	\begin{pmatrix}
		\lengthofsidehypnorm^2 & 0 & 0 & 0\\
		0 & \lengthofgen^2 & 0 & 0 \\
		0 & 0 & \gsphere(\angvortrenormalized,\angvortrenormalized) & 
		\gsphere(\angvortrenormalized,\keydetvectorfield)
			\\
		0 & 0 & \gsphere(\angvortrenormalized,\keydetvectorfield) & 
		\gsphere(\keydetvectorfield,\keydetvectorfield)
	\end{pmatrix}
		\\
	& = \Speed^6 
		\lengthofsidehypnorm^4
		 \frac{(\seconduposinnerproduct + \lengthofgen^2)^2}{\seconduposinnerproduct^2 + \lengthofgen^2 \lengthoftophypnorm^2}
		\mbox{\upshape det}
		\begin{pmatrix}
		\gsphere(\angvortrenormalized,\angvortrenormalized) & 
		\gsphere(\angvortrenormalized,\keydetvectorfield)
			\\
		\gsphere(\angvortrenormalized,\keydetvectorfield) & \gsphere(\keydetvectorfield,\keydetvectorfield)
	\end{pmatrix}.
	\notag
\end{align}
From 
\eqref{E:TRANSPORTLOGDENSITYISTANGENTIALEXCEPTINNULLCASE},
\eqref{E:FIRSTSTEPKEYDETERMINANT},
\eqref{E:SECONDSTEPKEYDETERMINANT},
and the fact that RHS~\eqref{E:SECONDSTEPKEYDETERMINANT} vanishes 
when the lateral boundary $\underline{\mathcal{H}}$ is $\gfour$-null (since $\lengthofsidehypnorm = 0$ in this case),
we conclude the desired relations \eqref{E:KEYDETERMINANT} and \eqref{E:NULLCASEKEYDETERMINANT}.
\end{proof}

\section{Volume forms, area forms, and integrals}
\label{S:VOLUMEFORMSANDINTEGRALS}
In this section, we define the geometric volume forms and area forms
featured in our localized energy-flux identities
and derive some simple identities tied to them.

\subsection{Volume forms, area forms, and integrals}
\label{SS:VOLUMEFORMSANDINTEGRALS}

\subsubsection{Definitions of the volume and area forms}
\label{SSS:DEFINITIONSOFVOLUMEFORMSANDINTEGRALS}

\begin{definition}[Geometric volume forms and area forms]
\label{D:VOLUMEFORMS}
	Let $\gfour$ be the acoustical metric of Def.\,\ref{D:ACOUSTICALMETRIC},
	let $g$ be the first fundamental form of $\Sigma_t$,
	let $\topfirstfund$ be the first fundamental form of $\widetilde{\Sigma}_{\Timefunction}$,
	let $\sidefirstfund$ be the first fundamental form of $\underline{\mathcal{H}}$ when $\underline{\mathcal{H}}$ is $\gfour$-spacelike,
	and let $\gsphere$ be the first fundamental form of $\mathcal{S}_{\Timefunction}$;
	see Subsect.\,\ref{SS:FIRSTFUNDANDPROJECTIONS} for the definitions and properties of the latter four tensors
	and the beginning of Sect.\,\ref{S:SPACETIMEDOMAINS} for a description of the acoustical time function $\Timefunction$.
	We define the following volume forms\footnote{Throughout the article, we blur the distinction between the forms themselves, which are antisymmetric tensors,
	and the corresponding integration measures they induce on the relevant manifolds; the precise meaning will be clear from context.} and area forms:
	\begin{itemize}
	\item $d \varpi_{\gfour}$ denotes the canonical volume form\footnote{For example, relative to arbitrary coordinates 
		$\lbrace y^{\alpha} \rbrace_{\alpha = 0,1,2,3}$ on $\mathcal{M}$, we have
		$d \varpi_{\gfour} = \sqrt{|\mbox{\upshape det} \gfour|} dy^0 dy^1 dy^2 dy^3$,
		while relative to arbitrary coordinates 
		$\lbrace \vartheta^A \rbrace_{A=1,2}$ on $\mathcal{S}_{\Timefunction}$, we have
		$d \varpi_{\gsphere} = \sqrt{\mbox{\upshape det} \gsphere} d \vartheta^1 d \vartheta^1$.}
		induced by $\gfour$ on $\mathcal{M}$.
	\item $d \varpi_g$ denotes the canonical volume form induced by $g$ on $\Sigma_t$.
	\item $d \varpi_{\topfirstfund}$ denotes the canonical volume form induced by $\topfirstfund$ on $\widetilde{\Sigma}_{\Timefunction}$.
	\item $d \varpi_{\gsphere}$ denote the canonical area form induced by $\gsphere$ on $\mathcal{S}_{\Timefunction}$.
	\item In the case that $\underline{\mathcal{H}}$ is spacelike,
		$d \varpi_{\sidefirstfund}$ denotes the canonical volume form induced by $\sidefirstfund$ on $\underline{\mathcal{H}}$.
	\item In the case that the lateral boundary $\underline{\mathcal{H}}$ 
		is $\gfour$-null (in which case we use the alternate notation $\underline{\mathcal{N}} = \underline{\mathcal{H}}$), 
		we endow\footnote{By considering ``a limit as the spacelike hypersurface $\underline{\mathcal{H}}$
		becomes null,'' one can infer that the volume form
		$d \varpi_{\gsphere} \, d \Timefunction'$ on the limiting null hypersurface is the ``correct'' form for use in the divergence theorem; 
		see the proof of Prop.\,\ref{P:ENERGYFLUXID} for further discussion.} 
	$\underline{\mathcal{N}} = \cup_{\Timefunction' \in [0,T]} \mathcal{S}_{\Timefunction'}$ 
		with the volume form
		$d \varpi_{\gsphere} \, d \Timefunction'$, where
		$d \varpi_{\gsphere}$ is the area form induced by $\gsphere$ on $\mathcal{S}_{\Timefunction'}$.
	\end{itemize}
\end{definition}

\subsubsection{Identities for the volume and area forms}
\label{SSS:IDENTITIESFORVOLUMEANDAREAFORMS}
In the next lemma, we provide some identities satisfied by the forms from Def.\,\ref{D:VOLUMEFORMS}.

\begin{lemma}[Identities involving the volume forms on $\mathcal{M}$, $\Sigma_t$, $\widetilde{\Sigma}_{\Timefunction}$, and $\underline{\mathcal{H}}$]
	\label{L:IDENTITIESFORVOLUMEFORMS}
	Let $(t',x^1,x^2,x^3)$ be the Cartesian coordinates.
	Then the following identities hold:
	\begin{align} \label{E:SPACTIMEVOLUMEFORMANDSIGMATVOLUMEFORMRELATIVETOCARTESIANCOORDINATES}
			d \varpi_g  
			& = \Speed^{-3} dx^1 dx^2 dx^3,
			&
			d \varpi_{\gfour} 
			& 
			= \Speed^{-3} dx^1 dx^2 dx^3 dt'
			=
			d \varpi_{\gfour} = d \varpi_g dt'.
	\end{align}
	
	Moreover, let $\Timefunction'$ be the acoustical time function from the beginning of Sect.\,\ref{S:SPACETIMEDOMAINS}
	and let $\lengthofmodtophypnorm > 0 $ be the scalar function defined in
	\eqref{E:LENGTHOFTOPHYPNORMNORMALIZEDAGAINSTTIMEFUNCTION}.
	Then the following identity holds:
	\begin{align} \label{E:SPACETIMEVOLUMEFORMEXPRESSIONWITHRESPECTTOTIMEFUNCTION}
		d \varpi_{\gfour} 
		& = \lengthofmodtophypnorm d \varpi_{\topfirstfund} d \Timefunction'.
	\end{align}
	
	Finally, if $\underline{\mathcal{H}}$ is $\gfour$-spacelike,
	then with $\lengthofmodgen > 0$ denoting the scalar function
	from \eqref{E:LENGTHOFMODSIDEGEN}, we have
	\begin{align} \label{E:SIDEHYPERSURFACEVOLUMEFORMEXPRESSIONWITHRESPECTTOTIMEFUNCTIONSPACELIKECASE}
		d \varpi_{\sidefirstfund} 
		& = \lengthofmodgen d \varpi_{\gsphere} d \Timefunction'.
	\end{align}
\end{lemma}

\begin{proof}
		Relative to arbitrary coordinates on $\mathcal{M} = \mathcal{M}_T$, we have
		$d \varpi_{\gfour} = \sqrt{|\mbox{\upshape det} \gfour|} dy^0 dy^1 dy^2 dy^3$.
		In the special case of the Cartesian coordinates $(t,x^1,x^2,x^3)$, 
		the desired identity \eqref{E:SPACTIMEVOLUMEFORMANDSIGMATVOLUMEFORMRELATIVETOCARTESIANCOORDINATES} 
		for $d \varpi_{\gfour}$ follows from a straightforward computation 
		based on \eqref{E:ACOUSTICALMETRIC}.
		To obtain the desired identity \eqref{E:SPACTIMEVOLUMEFORMANDSIGMATVOLUMEFORMRELATIVETOCARTESIANCOORDINATES}
		for $d \varpi_g$, we will (consistent with Convention~\ref{C:DISPLAYONLYSPATIAL}) 
		use the symbol ``$g$'' to denote
		the restriction of the first fundamental form of $\Sigma_t$ to $\Sigma_t$-tangent vectors.
		Then by \eqref{E:ACOUSTICALMETRIC}, relative to the Cartesian spatial coordinates
		$(x^1,x^2,x^3)$ on $\Sigma_t$, we have that $g = \Speed^{-2} \sum_{a=1}^3 dx^a \otimes dx^a$,
		and the desired identity for $d \varpi_g  = \sqrt{\mbox{\upshape det} g} dx^1 dx^2 dx^3$
		follows easily.
		
		To prove \eqref{E:SIDEHYPERSURFACEVOLUMEFORMEXPRESSIONWITHRESPECTTOTIMEFUNCTIONSPACELIKECASE},
		we first fix arbitrary local coordinates $(\vartheta^1,\vartheta^2)$
		on the initial sphere $\mathcal{S}_0$. We will now explain how we propagate these coordinates to all
		of $\underline{\mathcal{H}}$. Abusing notation, we will also denote the propagated coordinates by $(\vartheta^1,\vartheta^2)$. 
		Specifically, we propagate the coordinates by solving the transport equations
		$\modgen \vartheta^A = 0$, $(A=1,2)$, where the $\underline{\mathcal{H}}$-tangent vectorfield 
		$\modgen$ is defined by \eqref{E:NORMALIZEDAGAINSTTIMEFUNCTIONGENERATOROFHYPERSURFACE}
		and the initial conditions for $\vartheta^A$ are the ones specified on $\mathcal{S}_0$;
		this is possible since \eqref{E:SIDEGENERATORNORMALIZEDAGAINSTTIMEFUNCTION} implies that $\modgen$ 
		is transversal to the spheres $\mathcal{S}_{\Timefunction} \subset \widetilde{\Sigma}_{\Timefunction}$.
		Thus, for $\Timefunction \in [0,T]$, the restriction of $(\vartheta^1,\vartheta^2)$ to $\mathcal{S}_{\Timefunction}$ forms a local coordinate
		system on $\mathcal{S}_{\Timefunction}$.
		Moreover, relative to the coordinates $(\Timefunction,\vartheta^1,\vartheta^2)$ on $\underline{\mathcal{H}}$, we have 
		(again by \eqref{E:SIDEGENERATORNORMALIZEDAGAINSTTIMEFUNCTION})
		$\modgen = \frac{\partial}{\partial \Timefunction}$, and the condition \eqref{E:SIDEGENERATORNORMALIZEDAGAINSTTIMEFUNCTION} plus the fact that
		$\frac{\partial}{\partial \vartheta^A}|_{\mathcal{S}_0}$ is tangent to $\mathcal{S}_0$
		together ensure that $\frac{\partial}{\partial \vartheta^A}$ is tangent to $\mathcal{S}_{\Timefunction}$ for $\Timefunction \in [0,T]$ and $A=1,2$.
		We next recall that $\modgen$ is $\gfour$-orthogonal to $\mathcal{S}_{\Timefunction}$ by construction and, 
		by \eqref{E:LENGTHOFMODSIDEGEN},  
		that it verifies
		$\sidefirstfund(\modgen,\modgen) = \gfour(\modgen,\modgen) = 
		\lengthofmodgen^2 > 0$ (since $\underline{\mathcal{H}}$ is spacelike by assumption).
		In total, it follows that relative to the coordinates $(\Timefunction,\vartheta^1,\vartheta^2)$ 
		on $\underline{\mathcal{H}}$, we have
		$\sidefirstfund = \lengthofmodgen^2 d \Timefunction \otimes d \Timefunction + \gsphere_{AB} d \vartheta^A \otimes d \vartheta^B$
		and $\gsphere = \gsphere_{AB} d \vartheta^A \otimes d \vartheta^B$ 
		(here we are viewing $\gsphere$ as a Riemannian metric on $\mathcal{S}_{\Timefunction}$),
		where $\gsphere_{AB} := \gsphere(\frac{\partial}{\partial \vartheta^A},\frac{\partial}{\partial \vartheta^B})$
		and we are using Einstein summation convention for the capital Latin indices, which vary over $1,2$.
		Thus, 
		$
		d \varpi_{\sidefirstfund}
		=
		\sqrt{\mbox{\upshape det} \sidefirstfund} d \vartheta^1 d \vartheta^2 d \Timefunction 
		= 
		\lengthofmodgen \sqrt{\mbox{\upshape det} \gsphere} d \vartheta^1 d \vartheta^2 d \Timefunction
		= \lengthofmodgen d \varpi_{\gsphere} d \Timefunction
		$,
		which is the desired identity \eqref{E:SIDEHYPERSURFACEVOLUMEFORMEXPRESSIONWITHRESPECTTOTIMEFUNCTIONSPACELIKECASE}.
		
		The identity \eqref{E:SPACETIMEVOLUMEFORMEXPRESSIONWITHRESPECTTOTIMEFUNCTION} can be proved via similar arguments, as we now explain.
		We first fix arbitrary coordinates $\lbrace y^1,y^2,y^3 \rbrace$ on $\widetilde{\Sigma}_0$ 
		and propagate them to $\mathcal{M}$
		by solving the transport equation $\modtophypnorm y^A = 0$,
		where $\modtophypnorm$ is as in Def.\,\ref{D:HYPNORMANDSPHEREFORMDEFS}
		and the initial conditions for $\lbrace y^1,y^2,y^3 \rbrace$ are the ones specified on $\widetilde{\Sigma}_0$.
		We stress that this procedure allows us to extend $\lbrace y^1,y^2,y^3 \rbrace$ to all of $\mathcal{M}_T$
		since for $\Timefunction \in (0,T]$, every maximally extended integral curve of $\modtophypnorm$ contained in $\mathcal{M}_{\Timefunction}$ 
		must intersect $\widetilde{\Sigma}_0$ at its past endpoint. To see this, we only have to rule out the 
		possibilities that for $\Timefunction \in (0,T]$, the past endpoint of a maximally extended integral
		curve of $\modtophypnorm$ contained in $\mathcal{M}_{\Timefunction}$
		lies in its lateral boundary $\underline{\mathcal{H}}_{\Timefunction}$
		or in its top boundary $\widetilde{\Sigma}_{\Timefunction}$.
		These two possibilities are straightforward to rule out
		based on the discussion below \eqref{E:SECONDINGOINGCONDITION},
		the fact that $\modtophypnorm$ is a positive scalar function multiple of $\tophypnorm$,
		(this follows from the fact that $\modtophypnorm$ and $\tophypnorm$ are parallel,
		\eqref{E:LENGTHOFTOPHYPNORMNORMALIZEDAGAINSTTIMEFUNCTION}, 
		the fact that $\tophypnorm$ is $\gfour$-timelike and future-directed,
		and the fact that the gradient of $\Timefunction$ is $\gfour$-timelike and past-directed),
		and the fact that $\sidehypnorm$ is $\gfour$-timelike and future-directed,
		which in total imply that for $\Timefunction \in (0,T]$, 
		$\modtophypnorm$ points outward to $\mathcal{M}_{\Timefunction}$ along its the top boundary $\widetilde{\Sigma}_{\Timefunction}$
		and outward to $\mathcal{M}_{\Timefunction}$ along its lateral boundary $\underline{\mathcal{H}}_{\Timefunction}$.
		Next, using \eqref{E:TOPHYPNORMNORMALIZEDAGAINSTTIMEFUNCTION},
		\eqref{E:LENGTHOFTOPHYPNORMNORMALIZEDAGAINSTTIMEFUNCTION},
		and the fact that $\modtophypnorm$ is
		$\gfour$-orthogonal to $\widetilde{\Sigma}_{\Timefunction}$,
		we find that relative to the coordinates $\lbrace \Timefunction, y^1,y^2,y^3 \rbrace$ on $\mathcal{M}$,
		we have $\modtophypnorm = \frac{\partial}{\partial \Timefunction}$
		and
		$\gfour = \lengthofmodtophypnorm^2 d \Timefunction \otimes d \Timefunction + \topfirstfund_{AB} d y^A \otimes d y^B$,
		where $\topfirstfund_{AB} := \topfirstfund(\frac{\partial}{\partial y^A},\frac{\partial}{\partial y^B})$
		and the capital Latin indices now vary over $1,2,3$.
		From this identity for $\gfour$, the desired identity
		\eqref{E:SPACETIMEVOLUMEFORMEXPRESSIONWITHRESPECTTOTIMEFUNCTION}
		readily follows.
\end{proof}

\subsubsection{Integrals with respect to the geometric volume and area forms}
\label{SSS:INTGEGRALSWITHRESPECTTOGEOMETRICVOLUMEANDAREAFORMS}
Until Sect.\,\ref{S:DOUBLENULL}, we define all of our integrals
relative to the volume forms of Def.\,\ref{D:VOLUMEFORMS}.
For example, if $f$ is a scalar function defined on the $\gfour$-spacelike lateral boundary $\underline{\mathcal{H}}$,
$\Timefunction$ is the acoustical time function introduced at the beginning of Sect.\,\ref{S:SPACETIMEDOMAINS},
and $(\vartheta^1,\vartheta^2)$ are arbitrary local coordinates
on $\mathcal{S}_{\Timefunction}$ (which is diffeomorphic to $\mathbb{S}^2$), then
\[
\int_{\mathcal{S}_{\Timefunction}}
	f
\, d \varpi_{\gsphere}
= 
\int_{\mathbb{S}^2}
	f(\Timefunction,\vartheta^1,\vartheta^2)
\, \sqrt{\mbox{\upshape det} \gsphere(\Timefunction,\vartheta^1,\vartheta^2)} 
d \vartheta^1 d \vartheta^2,
\]
while by \eqref{E:SIDEHYPERSURFACEVOLUMEFORMEXPRESSIONWITHRESPECTTOTIMEFUNCTIONSPACELIKECASE},
we have
\[
\int_{\underline{\mathcal{H}}_{\Timefunction}}
	f
\, d \varpi_{\sidefirstfund}
= 
\int_{\Timefunction'=0}^{\Timefunction}
\int_{\mathbb{S}^2}
	\lengthofmodgen(\Timefunction',\vartheta^1,\vartheta^2) f(t,\vartheta^1,\vartheta^2)
\, \sqrt{\mbox{\upshape det} \gsphere(\Timefunction',\vartheta^1,\vartheta^2)} 
d \vartheta^1 d \vartheta^2
\, d \Timefunction'.
\]

\subsection{Differential and integral identities involving $\mathcal{S}_{\Timefunction}$}
\label{SS:DIFFERENTIATINGUNDERINTEGRAL}
The following lemma, though standard, plays a crucial rule in our proof 
of Theorem~\ref{T:MAINREMARKABLESPACETIMEINTEGRALIDENTITY}. Specifically,
we use the identity \eqref{E:KEYHYPERSURFACEINTEGRALIDENTITY} 
to show that \textbf{one of the error integrals in our
main integral identity has a good sign}, 
up to error terms that are controllable because they  
depend only on the $\underline{\mathcal{H}}$-tangential 
derivatives of various quantities.

\begin{lemma}[Differential and integral identities involving $\mathcal{S}_{\Timefunction}$]
	\label{L:DIFFERENTIATIONANDINTEGRALIDENTITEISINVOLVINGST}
	Let $\underline{\mathcal{H}}$ be $\gfour$-spacelike or $\gfour$-null (i.e. $\underline{\mathcal{H}}=\underline{\mathcal{N}}$), and let
	$f$ be a smooth function defined on $\underline{\mathcal{H}}$ (on $\underline{\mathcal{N}}$ in the null case).
	Let $\lapsemodgen > 0$ be the scalar function defined in \eqref{E:LAPSEFORMODGENCORRESPONDINGTOTIMEFUNCTION},
	and let $\modgen := \lapsemodgen \gen$ 
	be the $\underline{\mathcal{H}}$-tangent vectorfield defined in \eqref{E:NORMALIZEDAGAINSTTIMEFUNCTIONGENERATOROFHYPERSURFACE}
	(and thus $\modgen = \lapsemodgen \uLunit$ in the null case by \eqref{E:NORMALISGENERATORINNULLCASE}),
	let $\Lie_{\modgen}$ denote Lie differentiation with respect to $\modgen$,
	and let $\gsphere$ be the first fundamental form of $\mathcal{S}_{\Timefunction}$ (see Def.\,\ref{D:FIRSTFUNDAMENTALFORMSANDPROJECTIONS}).
	Let $\Timefunction$ be the acoustical time function introduced at the beginning of Subsect.\,\ref{S:SPACETIMEDOMAINS}.
	Then the following identity holds for $\Timefunction \in [0,T]$:
	\begin{align} \label{E:TIMEDERIVATIVESOFSTINTEGRAL}
		\frac{d}{d \Timefunction}
		\int_{\mathcal{S}_{\Timefunction}}
			f
		\, d \varpi_{\gsphere}
			& =
			\int_{\mathcal{S}_{\Timefunction}}
				\left\lbrace
				\modgen f
				+
				\frac{1}{2}
				f
				(\gsphere^{-1})^{\alpha \beta} \Lie_{\modgen} \gsphere_{\alpha \beta}
				\right\rbrace
				\, d \varpi_{\gsphere}.
	\end{align}
	
	In addition, 
	the following identity holds for $\Timefunction \in [0,T]$:
	\begin{align} \label{E:KEYHYPERSURFACEINTEGRALIDENTITY}
			\int_{\underline{\mathcal{H}}_{\Timefunction}}
				\modgen f
			\, d \varpi_{\gsphere} d \Timefunction'
			& =
				-
				\frac{1}{2}
				\int_{\underline{\mathcal{H}}_{\Timefunction}}
					f [(\gsphere^{-1})^{\alpha \beta} \Lie_{\modgen} \gsphere_{\alpha \beta}]
				\, d \varpi_{\gsphere} d \Timefunction'
				+
				\int_{\mathcal{S}_{\Timefunction}}
					f
				\, d \varpi_{\gsphere}
				-
				\int_{\mathcal{S}_0}
					f
				\, d \varpi_{\gsphere},
	\end{align}
	where we note that when $\underline{\mathcal{H}}$ is $\gfour$-spacelike, 
	\eqref{E:SIDEHYPERSURFACEVOLUMEFORMEXPRESSIONWITHRESPECTTOTIMEFUNCTIONSPACELIKECASE} implies that
	$d \varpi_{\gsphere} d \Timefunction' = \lengthofmodgen^{-1} \, d \varpi_{\sidefirstfund}$.
\end{lemma}

\begin{proof}
	To initiate the proof of \eqref{E:TIMEDERIVATIVESOFSTINTEGRAL},
	we note that the term $(\gsphere^{-1})^{\alpha \beta} \Lie_{\modgen} \gsphere_{\alpha \beta}$ on RHS~\eqref{E:TIMEDERIVATIVESOFSTINTEGRAL}
	is coordinate invariant and depends only on $\underline{\mathcal{H}}$-tangent tensors. Thus, the term can be evaluated using
	the local coordinates $(\Timefunction,\vartheta^1,\vartheta^2)$  
	on $\underline{\mathcal{H}}$ from the proof of Lemma~\ref{L:IDENTITIESFORVOLUMEFORMS}. 
	Specifically, from the computations carried out in the proof of Lemma~\ref{L:IDENTITIESFORVOLUMEFORMS}
	and the standard determinant differentiation identity
	$\frac{\partial}{\partial \Timefunction} \mbox{\upshape det} \gsphere 
	= \mbox{\upshape det} \gsphere 
	(\gsphere^{-1})^{\alpha \beta} \frac{\partial}{\partial \Timefunction} \gsphere_{\alpha \beta}$, 
	we see that the integrand
	$
		\left\lbrace
		\modgen f
		+
				\frac{1}{2}
				f
				(\gsphere^{-1})^{\alpha \beta} \Lie_{\modgen} \gsphere_{\alpha \beta}
				\right\rbrace
			\, d \varpi_{\gsphere}
	$
	on RHS~\eqref{E:TIMEDERIVATIVESOFSTINTEGRAL}
	can be expressed as $\frac{\partial}{\partial \Timefunction}  (f \, d \varpi_{\gsphere}) 
	= \frac{\partial}{\partial \Timefunction}  
	(f \sqrt{\mbox{\upshape det} \gsphere} d \vartheta^1 d \vartheta^2)$.
	\eqref{E:TIMEDERIVATIVESOFSTINTEGRAL} therefore follows from differentiating under the integral with respect to $\Timefunction$.
	
	\eqref{E:KEYHYPERSURFACEINTEGRALIDENTITY} follows from integrating \eqref{E:TIMEDERIVATIVESOFSTINTEGRAL}
	with respect to $\Timefunction$ and using the fundamental theorem of calculus.
\end{proof}

\section{The remarkable structure of the boundary error integrals}
\label{S:STRUCTUREOFLATERALERRORINTEGRALS}
In this section, we first prove Prop.\,\ref{P:STRUCTUREOFERRORINTEGRALS},
which yields identities for the boundary error integrals
appearing in our main integral identities (which are provided by Theorem~\ref{T:MAINREMARKABLESPACETIMEINTEGRALIDENTITY}).
Then, in Theorem~\ref{T:STRUCTUREOFERRORTERMS}, 
we closely examine the structure of the terms appearing in Prop.\,\ref{P:STRUCTUREOFERRORINTEGRALS}
and, using compact notation, reveal why they have the remarkable structures that are 
crucial for various applications.

\subsection{Key identity for the boundary error integrals}
\label{SS:STRUCTUREOFERRORINTEGRALS}
In Theorem~\ref{T:MAINREMARKABLESPACETIMEINTEGRALIDENTITY}, 
we derive our main spacetime integral identities on $\mathcal{M}_{\Timefunction}$.
The identities feature ``boundary error integrals,'' that is, 
integrals along 
$\underline{\mathcal{H}}_{\Timefunction}$
and
$\mathcal{S}_{\Timefunction}$;
the discussion surrounding equation \eqref{E:FIRSTSTEPSPACETIMEREMARKABLEIDENTITYSPECIFICVORTICITY}
shows how these error integrals emerge in the proof of the theorem.
In the next proposition, we derive identities for these boundary error integrals
which, in conjunction with Theorem~\ref{T:STRUCTUREOFERRORTERMS},
show that they have remarkable structure.

\begin{proposition}[Key identity for the boundary error integrals]
\label{P:STRUCTUREOFERRORINTEGRALS}
Let $\weight$ be an arbitrary scalar function,
and
let $J[\vortrenormalized]$
and
$J[\GradEnt]$ 
be the $\widetilde{\Sigma}_{\Timefunction}$-tangent vectorfields defined in \eqref{E:NEWELLIPTICHYPERBOLICCURRENT}.
Let $\spherenormal$ be the unit outer normal to $\mathcal{S}_{\Timefunction}$ from Def.\,\ref{D:HYPNORMANDSPHEREFORMDEFS},
let
$\lapsemodgen > 0$ be the scalar function defined in \eqref{E:LAPSEFORMODGENCORRESPONDINGTOTIMEFUNCTION} (see also \eqref{E:POSITIVITYOFLAPSEMODGEN}),
let $\uposinnerproduct > 0$ be the scalar function defined in \eqref{E:INGOINGCONDITION},
and let $\seconduposinnerproduct > 0$ be the scalar function defined in \eqref{E:SECONDINGOINGCONDITION}.
For smooth solutions (see Remark~\ref{R:SMOOTHNESSNOTNEEDED})
to the compressible Euler equations \eqref{E:TRANSPORTDENSRENORMALIZEDRELATIVECTORECTANGULAR}-\eqref{E:ENTROPYTRANSPORT},
the following integral identities hold,
where the volume and area forms are as in Def.\,\ref{D:VOLUMEFORMS}:
\begin{subequations}
\begin{align} \label{E:MAINLATERALERRORINTEGRALSFORSPECIFICVORTICITY}
		\int_{\underline{\mathcal{H}}_{\Timefunction}}
			\weight 
			\spherenormal_{\alpha} J^{\alpha}[\vortrenormalized]
		\, d \varpi_{\gsphere} d \Timefunction'
		+
		\int_{\mathcal{S}_{\Timefunction}}
			\weight 
			\frac{\uposinnerproduct}{\seconduposinnerproduct \lapsemodgen} 
			|\angvortrenormalized|_{\gsphere}^2
		\, d \varpi_{\gsphere} 
		& = 
		\int_{\mathcal{S}_0}
			\weight 
			\frac{\uposinnerproduct}{\seconduposinnerproduct \lapsemodgen}
			|\angvortrenormalized|_{\gsphere}^2
		\, d \varpi_{\gsphere} 
			\\
	& \ \
		+
		\int_{\underline{\mathcal{H}}_{\Timefunction}}
			\left\lbrace
				\underline{\mathfrak{H}}_{(\pmb{\partial} \weight)}[\vortrenormalized]
				+
				\weight
				\underline{\mathfrak{H}}[\vortrenormalized]
				+
				\weight
				\underline{\mathfrak{H}}_{(1)}[\vortrenormalized]
			\right\rbrace
		\, d \varpi_{\gsphere} d \Timefunction',
		\notag
			\\
		\int_{\underline{\mathcal{H}}_{\Timefunction}}
			\weight 
			\spherenormal_{\alpha} J^{\alpha}[\GradEnt]
		\, d \varpi_{\gsphere} d \Timefunction'
		+
		\int_{\mathcal{S}_{\Timefunction}}
			\weight 
			\frac{\uposinnerproduct}{\seconduposinnerproduct \lapsemodgen}
			|\angGradEnt|_{\gsphere}^2
		\, d \varpi_{\gsphere} 
		& = 
		\int_{\mathcal{S}_0}
			\weight 
			\frac{\uposinnerproduct}{\seconduposinnerproduct \lapsemodgen} 
			|\angGradEnt|_{\gsphere}^2
		\, d \varpi_{\gsphere} 
			\label{E:MAINLATERALERRORINTEGRALSFORENTROPYGRADIENT} 
				\\
	& \ \
		+
		\int_{\underline{\mathcal{H}}_{\Timefunction}}
			\left\lbrace
				\underline{\mathfrak{H}}_{(\pmb{\partial} \weight)}[\GradEnt]
				+
				\weight
				\underline{\mathfrak{H}}[\GradEnt]
				+
				\weight
				\underline{\mathfrak{H}}_{(2)}[\GradEnt]
			\right\rbrace
		\, d \varpi_{\gsphere} d \Timefunction'.
		\notag
	\end{align}
	\end{subequations}
		In \eqref{E:MAINLATERALERRORINTEGRALSFORSPECIFICVORTICITY}-\eqref{E:MAINLATERALERRORINTEGRALSFORENTROPYGRADIENT},
		for $\SigmatTan \in \lbrace \vortrenormalized, \GradEnt \rbrace$,
		we have
		\begin{subequations}
		\begin{align} \label{E:DERIVATIVESOFWEIGHTLATERALBOUNDARYEASYERRORINTEGRANDTERMS} 
		\underline{\mathfrak{H}}_{(\pmb{\partial} \weight)}[\SigmatTan]
		& :=
		 \frac{\uposinnerproduct}{\seconduposinnerproduct} 
		|\angV|_{\gsphere}^2
		\gen \weight
		+
		\uposinnerproduct 
		|\angV|_{\gsphere}^2
		\utang \weight
		+
		\SigmatTan_{\alpha} \spherenormal^{\alpha}
		\angV \weight,
			\\
		\underline{\mathfrak{H}}[\SigmatTan]
			&
			:=
			\frac{1}{2} 
			\frac{\uposinnerproduct}{\seconduposinnerproduct \lapsemodgen}  
			|\angV|_{\gsphere}^2
			(\gsphere^{-1})^{\alpha \beta} \Lie_{\lapsemodgen \gen} \gsphere_{\alpha \beta}
			\label{E:LATERALBOUNDARYEASYERRORINTEGRANDTERMS} \\
		& \ \
		+
		\left\lbrace
			\lapsemodgen \gen
			\left[\frac{\uposinnerproduct}{\seconduposinnerproduct \lapsemodgen} (\gsphere^{-1})^{\alpha \beta} \right]
		\right\rbrace
		\SigmatTan_{\alpha} 
	  \SigmatTan_{\beta}
		+
		\left\lbrace
		\utang
			\left[\uposinnerproduct (\gsphere^{-1})^{\alpha \beta} \right]
		\right\rbrace
		\SigmatTan_{\alpha} 
	  \SigmatTan_{\beta}
		-
		2 \SigmatTan_{\alpha}
		\angV \Transport^{\alpha}
		\notag \\
		& \ \
				+
			\uposinnerproduct 
			|\angV|_{\gsphere}^2
			\angdiv \utang 
		+
		\spherenormal_{\alpha} \SigmatTan^{\alpha}  
		\SigmatTan^{\beta}
		\angdiv \angpartialarg{\beta}
			\notag \\
		& \ \
		+
		\SigmatTan^{\alpha}
		\spherenormal^{\beta}
		\angV \gfour_{\alpha \beta}
		+
		\SigmatTan_{\alpha}
		\angV \spherenormal^{\alpha},
		\notag
	\end{align}
	\end{subequations}

	\begin{align} \label{E:SPECIFICVORTITICYMAINLATERALERRORINTEGRAND}
		\underline{\mathfrak{H}}_{(1)}[\vortrenormalized]
		&  := 
		2
		\upsigma
		\uposinnerproduct
		\Speed^{-1} 
		\exp(-2 \LogDensity) 
		\frac{p_{;\Ent}}{\bar{\varrho}}
		\frac{\seconduposinnerproduct + \lengthofgen^2}{\sqrt{\seconduposinnerproduct^2 + \lengthofgen^2 \lengthoftophypnorm^2}}
		\sqrt{\mbox{\upshape det}
		\begin{pmatrix}
		\gsphere(\angvortrenormalized,\angvortrenormalized) & 
		\gsphere(\angvortrenormalized,\keydetvectorfield)
			\\
		\gsphere(\angvortrenormalized,\keydetvectorfield) & \gsphere(\keydetvectorfield,\keydetvectorfield)
	\end{pmatrix}}
		 \\
	& 
	\ \ \ \ \ \
	\times
	\left\lbrace
			\urescalednewgenminushypnorm_a \gen v^a
				+
				\sidehypnorm_a \urescalednewgenminushypnorm^a \gen \LogDensity
				+
				\utandecompvectorfielddownarg{a} v^a
				+ 
				(g^{-1})^{ab} \sidehypnorm_a \utandecompvectorfielddownarg{b} \LogDensity
				+
			  \exp(-\LogDensity) \frac{p_{;\Ent}}{\bar{\varrho}} \sidehypnorm_a \GradEnt^a
	\right\rbrace
	\notag
					\\
			& \ \
				+
				4 
				\uposinnerproduct 
				\vortrenormalized_{\alpha} \uspecialgen^{\alpha}
				\angvortrenormalized \ln \Speed
				- 
				4 
				\uposinnerproduct 
				|\angvortrenormalized|_{\gsphere}^2
				\uspecialgen \ln \Speed 
				+
				4 
				\uposinnerproduct 
				\uspecialgen^0
				\vortrenormalized_a \angvortrenormalized v^a
				-
				4 
				\uposinnerproduct
				\angvortrenormalized^0
				\vortrenormalized_a \uspecialgen v^a
					\notag \\
		& \ \
		+
		2 
		\uposinnerproduct
		\Speed^{-4} 
		\exp(-2 \LogDensity) 
		\frac{p_{;\Ent}}{\bar{\varrho}}
		\upepsilon_{\alpha \beta \gamma \delta}
		 \uspecialgen^{\alpha} 
		 \angvortrenormalized^{\beta}
		\urescalednewgenminushypnorm^{\gamma}
		\GradEnt^{\delta}
		\gen \LogDensity
		-
		2 
		\uposinnerproduct
		\Speed^{-4} 
		\exp(-2 \LogDensity) 
		\frac{p_{;\Ent}}{\bar{\varrho}}
		\GradEnt^a \sidehypnorm_a
		\upepsilon_{\alpha \beta \gamma \delta}
		\uspecialgen^{\alpha} 
		\angvortrenormalized^{\beta} 
		\Transport^{\gamma}
		\urescalednewgenminushypnorm^{\delta} 
		\gen \LogDensity
			\notag 
			\\
	 & \ \
		+
		2 
		\uposinnerproduct
		\Speed^{-2} 
		\exp(-2 \LogDensity) 
		\frac{p_{;\Ent}}{\bar{\varrho}}
		\upepsilon_{\alpha \beta ab}
		\uspecialgen^{\alpha}  
		\angvortrenormalized^{\beta}
		(\utandecompvectorfielddownarg{a} \LogDensity)
		\GradEnt^b
		-
		2 
		\uposinnerproduct
		\Speed^{-2} 
		\exp(-2 \LogDensity) 
		\frac{p_{;\Ent}}{\bar{\varrho}}
		\GradEnt^a \sidehypnorm_a
		\upepsilon_{\alpha \beta \gamma d}
		\uspecialgen^{\alpha} 
		\angvortrenormalized^{\beta} 
		\Transport^{\gamma}
		\utandecompvectorfielddownarg{d} \LogDensity
		\notag
			\\
	& \ \
		-
		2 
		\uposinnerproduct
		\Speed^{-4} 
				\exp(-2 \LogDensity) 
				\frac{p_{;\Ent}}{\bar{\varrho}}
		\GradEnt^a \urescalednewgenminushypnorm_a
		\upepsilon_{\alpha \beta \gamma \delta}
		\uspecialgen^{\alpha} 
		\angvortrenormalized^{\beta} 
		\Transport^{\gamma}
		\gen v^{\delta}
		-
		2 
		\uposinnerproduct
		\Speed^{-4} 
				\exp(-2 \LogDensity) 
				\frac{p_{;\Ent}}{\bar{\varrho}}
		\upepsilon_{\alpha \beta \gamma \delta}
		\uspecialgen^{\alpha} 
		\angvortrenormalized^{\beta} 
		\Transport^{\gamma}
		\GradEnt^a \utandecompvectorfielddownarg{a} v^{\delta}
		\notag
			\\
	& \ \
				+
			2 
			\uposinnerproduct
			\Speed^{-2}
			\exp(\LogDensity)
			\upepsilon_{\alpha \beta \gamma \delta}
			\uspecialgen^{\alpha}
			\angvortrenormalized^{\beta}
			\Transport^{\gamma}
			\VortVort^{\delta}
				\notag
				\\
	& \ \
		-
		2 
		\uposinnerproduct
		\Speed^{-4} 
		\exp(-3 \LogDensity) 
		\left\lbrace
			\frac{p_{;\Ent}}{\bar{\varrho}}
		\right\rbrace^2
		\GradEnt^a \sidehypnorm_a
		\upepsilon_{\alpha \beta \gamma \delta}
		\uspecialgen^{\alpha} 
		\angvortrenormalized^{\beta} 
		\Transport^{\gamma}
		\GradEnt^{\delta},
				\notag
	\end{align}
and
\begin{align} \label{E:ENTROPYGRADIENTMAINLATERALERRORINTEGRAND}
		\underline{\mathfrak{H}}_{(2)}[\GradEnt]
		& := 
			4 
			\uposinnerproduct \GradEnt_{\alpha} \uspecialgen^{\alpha} \angGradEnt \ln \Speed
			- 
			4 
			\uposinnerproduct |\angGradEnt|_{\gsphere}^2  \uspecialgen \ln \Speed.
\end{align}
Above, $\VortVort$ is the $\Sigma_t$-tangent modified fluid variable from Def.\,\ref{D:RENORMALIZEDCURLOFSPECIFICVORTICITY},
$\sidehypnorm$ is the vectorfield from Def.\,\ref{D:HYPNORMANDSPHEREFORMDEFS},
$\gen$ is the $\widetilde{\Sigma}_{\Timefunction}$-tangent vectorfield from Def.\,\ref{D:HYPNORMANDSPHEREFORMDEFS},
$\uspecialgen$ is the $\underline{\mathcal{H}}_{\Timefunction}$-tangent vectorfield from Lemma~\ref{L:KEYIDBETWEENVARIOUSVECTORFIELDS},
$\utang$ is the $\mathcal{S}_{\Timefunction}$-tangent vectorfield from Lemma~\ref{L:KEYIDBETWEENVARIOUSVECTORFIELDS},
$\urescalednewgenminushypnorm$ denotes the $\Sigma_t$-tangent vectorfield from
Lemma~\ref{L:PROPERTIESOFRESCALEDGENERATORMINUSSIDEHYPNORM},
$\lbrace \utandecompvectorfielddownarg{a} \rbrace_{a=1,2,3}$ 
are the $\mathcal{S}_{\Timefunction}$-tangent vectorfields from Lemma~\ref{L:NEWDECOMPOSITIONOFCOORDINATEPARTIALDERIVATIVEVECTORFIELDS},
and the first product on RHS~\eqref{E:SPECIFICVORTITICYMAINLATERALERRORINTEGRAND}
is equal to $
		-2
		\uposinnerproduct
		\Speed^{-4} 
		\exp(-2 \LogDensity) 
		\frac{p_{;\Ent}}{\bar{\varrho}}
		\times
		\mbox{RHS~\eqref{E:KEYDETERMINANT}}$.
In particular, Prop.\,\ref{P:KEYDETERMINANT} implies that
the first product on RHS~\eqref{E:SPECIFICVORTITICYMAINLATERALERRORINTEGRAND}
vanishes when the lateral boundary $\underline{\mathcal{H}}$ is $\gfour$-null.
\end{proposition}

\begin{proof}
	Let $\SigmatTan$ be a $\Sigma_t$-tangent vectorfield.
	We integrate the identity \eqref{E:PRELIMINARYDECOMPOFBOUNDARYINTEGRAND} 
	with respect to 
		$d \varpi_{\gsphere} d \Timefunction'$ 
		(see Def.\,\ref{D:VOLUMEFORMS} for the definitions of the volume and area forms)
		over $\underline{\mathcal{H}}_{\Timefunction}$
		and observe that the integrals of the last two terms
		on RHS~\eqref{E:PRELIMINARYDECOMPOFBOUNDARYINTEGRAND} vanish since they are perfect $\angD$-divergences.
		Using \eqref{E:KEYHYPERSURFACEINTEGRALIDENTITY} with
		$
		-
		\weight
		\frac{\uposinnerproduct}{\seconduposinnerproduct \lapsemodgen} 
			|\SigmatTan|_{\gsphere}^2
		$
		in the role of $f$
		to substitute for
		the integral over $\underline{\mathcal{H}}_{\Timefunction}$ of the first product
		$
		-
		\modgen
		\left\lbrace
			\weight
			\frac{\uposinnerproduct}{\seconduposinnerproduct \lapsemodgen} 
			|\SigmatTan|_{\gsphere}^2
		\right\rbrace
		$
		on RHS~\eqref{E:PRELIMINARYDECOMPOFBOUNDARYINTEGRAND},
		we deduce the identity
		\begin{align} \label{E:INTEGRATEDPRELIMINARYDECOMPOFBOUNDARYINTEGRAND}
		\int_{\underline{\mathcal{H}}_{\Timefunction}}
			\weight \spherenormal_{\alpha} J^{\alpha}[\SigmatTan]
		\, d \varpi_{\gsphere} d \Timefunction'
		+
		\int_{\mathcal{S}_{\Timefunction}}
			\weight
			\frac{\uposinnerproduct}{\seconduposinnerproduct \lapsemodgen}
			|\SigmatTan|_{\gsphere}^2
		\, d \varpi_{\gsphere} 
		& = 
		\int_{\mathcal{S}_0}
			\weight
			\frac{\uposinnerproduct}{\seconduposinnerproduct \lapsemodgen}
			|\SigmatTan|_{\gsphere}^2
		\, d \varpi_{\gsphere} 
			\\
	& \ \
		+
		2 
		\int_{\underline{\mathcal{H}}_{\Timefunction}}
			\weight
			\uposinnerproduct
			\uspecialgen^{\alpha}
			\angV^{\beta}
			(\partial_{\alpha} \SigmatTan_{\beta} - \partial_{\beta} \SigmatTan_{\alpha})
		\, d \varpi_{\gsphere} d \Timefunction'
			\notag \\
	& \ \
		+
		\int_{\underline{\mathcal{H}}_{\Timefunction}}
			\left\lbrace
			\underline{\mathfrak{H}}_{(\pmb{\partial} \weight)}[\SigmatTan]
			+
			\weight
			\underline{\mathfrak{H}}[\SigmatTan]
			\right\rbrace
		\, d \varpi_{\gsphere} d \Timefunction',
		\notag
	\end{align}
	where 
	$\underline{\mathfrak{H}}_{(\pmb{\partial} \weight)}[\SigmatTan]$ is defined in \eqref{E:DERIVATIVESOFWEIGHTLATERALBOUNDARYEASYERRORINTEGRANDTERMS}
	and $\underline{\mathfrak{H}}[\SigmatTan]$ is defined in \eqref{E:LATERALBOUNDARYEASYERRORINTEGRANDTERMS}.
	
We then use the identity \eqref{E:INTEGRATEDPRELIMINARYDECOMPOFBOUNDARYINTEGRAND}
with $\vortrenormalized$ in the role of $\SigmatTan$
and use \eqref{E:KEYIDENTITYANTISYMMETRICPARTOFSPECIFICVORTICITYDUALGRADIENT} to substitute
for the integrand factor
$
\partial_{\alpha} \vortrenormalized_{\beta} - \partial_{\beta} \vortrenormalized_{\alpha}
$
found in the second integral on RHS~\eqref{E:INTEGRATEDPRELIMINARYDECOMPOFBOUNDARYINTEGRAND}.
The resulting identity features the integral
\begin{align} \label{E:DIFFICULTINTEGRALPROOFOFSTRUCTUREOFERRORINTEGRALS}
2 
\int_{\underline{\mathcal{H}}_{\Timefunction}}
			\weight
			\uposinnerproduct
			\Speed^{-4} 
			\exp(-2 \LogDensity) 
			\frac{p_{;\Ent}}{\bar{\varrho}}
			\upepsilon_{\alpha \beta \gamma \delta}
			\uspecialgen^{\alpha}
			\angV^{\beta}
			\left\lbrace
				-
				(\Transport v^{\gamma}) 
				\GradEnt^{\delta}
				+
				 \Transport^{\gamma}
				[\GradEnt^{\delta}
				(\partial_a v^a)
				-
				\GradEnt^a \partial_a v^{\delta}]
			\right\rbrace
		\, d \varpi_{\gsphere} d \Timefunction'
\end{align}
coming from the fifth and sixth products on RHS~\eqref{E:KEYIDENTITYANTISYMMETRICPARTOFSPECIFICVORTICITYDUALGRADIENT}.
We rewrite the 
factors
\[
\upepsilon_{\alpha \beta \gamma \delta}
		 \uspecialgen^{\alpha} 
		 \angvortrenormalized^{\beta}
		\left\lbrace
		-
		(\Transport v^{\gamma}) 
		\GradEnt^{\delta}
		+
		\Transport^{\gamma}
			[\GradEnt^{\delta}
				(\partial_a v^a)
				-
				\GradEnt^a \partial_a v^{\delta}]
		\right\rbrace
\]
in
\eqref{E:DIFFICULTINTEGRALPROOFOFSTRUCTUREOFERRORINTEGRALS} 
by first using the identity \eqref{E:PRELIMINARYDECOMPOSITIONOFSUBTLETERMS}
for substitution,
and then using the key identity \eqref{E:KEYDETERMINANT} (multiplied by $-1$)
to substitute for the first product on RHS~\eqref{E:PRELIMINARYDECOMPOSITIONOFSUBTLETERMS}
(we use \eqref{E:NULLCASEKEYDETERMINANT} in place of \eqref{E:KEYDETERMINANT} when the lateral boundary is $\gfour$-null).
We collect all of these terms (except for the common factor of $\weight$) into
the error term $\underline{\mathfrak{H}}_{(1)}[\vortrenormalized]$
defined in
\eqref{E:SPECIFICVORTITICYMAINLATERALERRORINTEGRAND}.
In total, these steps yield \eqref{E:MAINLATERALERRORINTEGRALSFORSPECIFICVORTICITY}.

The proof of \eqref{E:MAINLATERALERRORINTEGRALSFORENTROPYGRADIENT} is similar, 
but we use \eqref{E:KEYIDENTITYANTISYMMETRICPARTOFENTROPYGRADIENTDUALGRADIENT} in 
place of the identity \eqref{E:KEYIDENTITYANTISYMMETRICPARTOFSPECIFICVORTICITYDUALGRADIENT}
used in the previous paragraph.

\end{proof}

\subsection{The remarkable structure of the error terms}
\label{SS:STRUCTUREOFERRORTERMS}
Our main result in this subsection is Theorem~\ref{T:STRUCTUREOFERRORTERMS},
in which we exhibit the remarkable structure of the
terms on RHSs~\eqref{E:MAINLATERALERRORINTEGRALSFORSPECIFICVORTICITY}-\eqref{E:ENTROPYGRADIENTMAINLATERALERRORINTEGRAND}.

\subsubsection{Additional notation}
\label{SSS:MORENOTATION}
In this subsubsection, we introduce some notation that will facilitate our presentation of Theorem~\ref{T:STRUCTUREOFERRORTERMS}.
We let $\vec{\Psi} := \lbrace \LogDensity,v^1,v^2,v^3,\Ent \rbrace$ denote the array of the Cartesian components of the basic fluid variables.
If $\bf{X}$ is any vectorfield, then $\vec{\bf{X}} := \lbrace \bf{X}^0, \bf{X}^1, \bf{X}^2, \bf{X}^3 \rbrace$ denotes the array of its Cartesian components.
We omit the component $\bf{X}^0$ when $\bf{X}$ is $\Sigma_t$-tangent, e.g.
$\vec{\vortrenormalized} := \lbrace \vortrenormalized^1, \vortrenormalized^2, \vortrenormalized^3 \rbrace$.
If $\mathbf{V}$ is a vectorfield, then $\mathbf{V} \vec{\Psi} := \lbrace \mathbf{V} \LogDensity,\mathbf{V} v^1,\mathbf{V} v^2,\mathbf{V} v^3, \mathbf{V}\Ent \rbrace$
(where, for example, $\mathbf{V} \LogDensity := \mathbf{V}^{\alpha} \partial_{\alpha} \LogDensity$)
and $\mathbf{V} \vec{\bf{X}} := \lbrace \mathbf{V} \bf{X}^0, \mathbf{V} \bf{X}^1, \mathbf{V} \bf{X}^2, \mathbf{V} \bf{X}^3 \rbrace$.
Similarly, with $\lbrace \angpartialarg{\alpha} \rbrace_{\alpha=0,1,2,3}$ denoting 
the $\mathcal{S}_{\Timefunction}$-tangent vectorfields from Def.\,\ref{D:PROJECTIONSOFCARTESIANCOORDINATEVECTORFIELDS},
we set
$\angpartial \vec{\Psi} := \lbrace \angpartialarg{\alpha} \LogDensity,\angpartialarg{\alpha} v^1,\angpartialarg{\alpha} v^2,\angpartialarg{\alpha} v^3,
\angpartialarg{\alpha} \Ent \rbrace_{\alpha=0,1,2,3}$
and
$\angpartial \vec{\bf{X}} := \lbrace \angpartialarg{\beta} X^{\alpha} \rbrace_{\alpha,\beta = 0,1,2,3}$.
Moreover, with $\lbrace \sidepartialarg{\alpha} \rbrace_{\alpha=0,1,2,3}$ denoting 
the $\underline{\mathcal{H}}_{\Timefunction}$-tangent vectorfields from Def.\,\ref{D:PROJECTIONSOFCARTESIANCOORDINATEVECTORFIELDS},
we set
$\sidepartial \vec{\Psi} := \lbrace \sidepartialarg{\alpha} \LogDensity,\sidepartialarg{\alpha} v^1,\sidepartialarg{\alpha} v^2,\sidepartialarg{\alpha} v^3,
\sidepartialarg{\alpha} \Ent \rbrace_{\alpha=0,1,2,3}$
and
$\sidepartial \vec{\bf{X}} := \lbrace \sidepartialarg{\beta} X^{\alpha} \rbrace_{\alpha,\beta = 0,1,2,3}$.

If $A$ and $B$ are scalar functions or arrays of scalar functions, then 
$\mathscr{L}(A)[B]$ schematically denotes linear combinations of the elements of $B$
with coefficients that are continuous\footnote{By ``continuous,'' we mean continuous on an open set of the arguments ``$A$''; 
all of our results hold for solutions such that $A$ belongs to the open set.} functions of the elements of $A$.
For example, since $\tophypnorm_{\alpha} = \gfour_{\alpha \beta} \tophypnorm^{\beta}$,
$\gfour_{\alpha \beta} = \mathscr{L}(\vec{\Psi})$ (see \eqref{E:ACOUSTICALMETRIC}),
and $\angvortrenormalized^{\alpha} \angpartialarg{\alpha} = \vortrenormalized^{\alpha} \angpartialarg{\alpha}$,
we have
$\vortrenormalized^{\alpha} \angvortrenormalized^{\beta} \angpartialarg{\alpha} \tophypnorm_{\beta}
=
\mathscr{L}(\vec{\Psi},\vec{\vortrenormalized},\vec{\tophypnorm})[\angpartial \vec{\Psi},\angpartial \vec{\tophypnorm}]
$.

\subsubsection{The remarkable structure of the error terms}
\label{SSS:PROPSTRUCTUREOFERRORTERMS}
We now state and prove the theorem that exhibits the remarkable structure of the error terms in Prop.\,\ref{P:STRUCTUREOFERRORINTEGRALS}.

\begin{theorem}[The remarkable structure of the error terms in Prop.\,\ref{P:STRUCTUREOFERRORINTEGRALS}]
	\label{T:STRUCTUREOFERRORTERMS}
	Assume that the lateral boundary $\underline{\mathcal{H}}$ is $\gfour$-spacelike.
	Under the notation of Subsubsect.\,\ref{SSS:MORENOTATION},
	the error terms 
	$\underline{\mathfrak{H}}[\vortrenormalized]$,
	$\underline{\mathfrak{H}}[\GradEnt]$,
	$\underline{\mathfrak{H}}_{(1)}[\vortrenormalized]$,
	and
	$\underline{\mathfrak{H}}_{(2)}[\GradEnt]$
	defined in \eqref{E:LATERALBOUNDARYEASYERRORINTEGRANDTERMS}-\eqref{E:ENTROPYGRADIENTMAINLATERALERRORINTEGRAND}
	exhibit the following structure:	
	\begin{align} \label{E:ERRORTERMSSCHEMATICSTRUCTURESPACELIKECASE}
		\underline{\mathfrak{H}}[\vortrenormalized],
			\,
		\underline{\mathfrak{H}}[\GradEnt],
			\,
		\underline{\mathfrak{H}}_{(1)}[\vortrenormalized],
			\,
		\underline{\mathfrak{H}}_{(2)}[\GradEnt]
		& = 
		\mathscr{L}(\vec{\Psi},\vec{\vortrenormalized},\vec{\GradEnt},\vec{\gen},\vec{\spherenormal},\pmb{\partial} \Timefunction)
		[\sidepartial \vec{\Psi},\angpartial \vec{\gen},\sidepartial \vec{\spherenormal},\sidepartial \pmb{\partial} \Timefunction]
		\\
		& \ \
		+
		\mathscr{L}(\vec{\Psi},\vec{\vortrenormalized},\vec{\GradEnt},\vec{\gen},\vec{\spherenormal},\pmb{\partial} \Timefunction)
		[\vec{\vortrenormalized},\vec{\GradEnt},\vec{\VortVort}],
		\notag
	\end{align}
	where
	$\VortVort$ is the $\Sigma_t$-tangent modified fluid variable from Def.\,\ref{D:RENORMALIZEDCURLOFSPECIFICVORTICITY},
	$\Timefunction$ is the acoustical time function introduced at the beginning of Sect.\,\ref{S:SPACETIMEDOMAINS},
	$\gen$ is the $\widetilde{\Sigma}_{\Timefunction}$-tangent vectorfield from Def.\,\ref{D:HYPNORMANDSPHEREFORMDEFS},
	and $\spherenormal$ is the $\widetilde{\Sigma}_{\Timefunction}$-tangent vectorfield from Def.\,\ref{D:HYPNORMANDSPHEREFORMDEFS}.
	
	Moreover, in the case that the lateral boundary $\underline{\mathcal{H}} = \underline{\mathcal{N}}$ is $\gfour$-null,
	we have
	\begin{align} \label{E:ERRORTERMSSCHEMATICSTRUCTURENULLCASE}
		\underline{\mathfrak{H}}[\vortrenormalized],
			\,
		\underline{\mathfrak{H}}[\GradEnt],
			\,
		\underline{\mathfrak{H}}_{(1)}[\vortrenormalized],
			\,
		\underline{\mathfrak{H}}_{(2)}[\GradEnt]
		& = 
		\mathscr{L}(\vec{\Psi},\vec{\vortrenormalized},\vec{\GradEnt},\vec{\uLunit},\vec{N})
		[\uLunit \vec{\Psi},\angpartial \vec{\Psi},\angpartial \vec{\uLunit},\uLunit \vec{\spherenormal},\angpartial \vec{\spherenormal}]
			\\
	& \ \
		+
		\mathscr{L}(\vec{\Psi},\vec{\vortrenormalized},\vec{\GradEnt},\vec{\uLunit},\vec{N})
		[\vec{\vortrenormalized},\vec{\GradEnt},\vec{\VortVort}],
		\notag
	\end{align}
	where $\uLunit = \sidehypnorm = \gen$
	(see Convention~\ref{C:NULLCASE} and \ref{E:NORMALISGENERATORINNULLCASE})
	is the $\gfour$-null normal to $\underline{\mathcal{N}}$ normalized by $\uLunit t = 1$.
\end{theorem}

\begin{remark}[The most important feature of Theorem~\ref{T:STRUCTUREOFERRORTERMS}]
	\label{R:MOSTIMPORTANTPARTOFTHEPROP}
The most important feature of the theorem is that \emph{all derivatives on 
RHSs~\eqref{E:MAINLATERALERRORINTEGRALSFORSPECIFICVORTICITY}-\eqref{E:MAINLATERALERRORINTEGRALSFORENTROPYGRADIENT}  
are in directions
tangent to $\underline{\mathcal{H}}$}. Moreover, the most important applications occur
when the lateral boundary is $\gfour$-null, 
i.e., when $\underline{\mathcal{H}} = \underline{\mathcal{N}}$;
as we explained in Subsubsect.\,\ref{SS:APPLICATIONS} (see also Remark~\ref{R:HIGHLIGHTKEYSTRUCTURES}),
the $\underline{\mathcal{N}}$-tangential nature of the derivatives 
is crucial for the local regularity theory of solutions 
(e.g., the proof of Theorem~\ref{T:LOCALIZEDAPRIORIESTIMATES} in the null case)
as well as the study of shocks.
\end{remark}

\begin{proof}[Proof of Theorem~\ref{T:STRUCTUREOFERRORTERMS}]
Throughout we use the notation of Subsubsect.\,\ref{SSS:MORENOTATION}. 
In particular, $\vec{\Psi} = \lbrace \LogDensity,v^1,v^2,v^3,\Ent \rbrace$ denotes the array of the Cartesian components of the basic fluid variables.
We prove only \eqref{E:ERRORTERMSSCHEMATICSTRUCTURESPACELIKECASE} since \eqref{E:ERRORTERMSSCHEMATICSTRUCTURENULLCASE}
can be proved using nearly identity arguments that take into account the fact that
$\gen = \uLunit$ in the null case (see \eqref{E:NORMALISGENERATORINNULLCASE}).

To show that the terms on RHSs~\eqref{E:LATERALBOUNDARYEASYERRORINTEGRANDTERMS}-\eqref{E:ENTROPYGRADIENTMAINLATERALERRORINTEGRAND}
(where on RHS~\eqref{E:LATERALBOUNDARYEASYERRORINTEGRANDTERMS}, we have $\SigmatTan \in \lbrace \vortrenormalized, \GradEnt \rbrace$)
have the desired structure, we first note that we have the following identities for scalar functions and the Cartesian components of various tensorfields,
where throughout the proof, 
$\gensmoothfunction$ schematically denotes a smooth function that is free to vary from line to line:
\begin{itemize}
	\item $\Speed = \gensmoothfunction(\vec{\Psi})$ (see the end of Subsubsect.\,\ref{SSS:SETUPANDDEFSOFBASICFLUIDVARIABLES})
	\item $p_{;\Ent} = \gensmoothfunction(\vec{\Psi})$ (see the end of Subsubsect.\,\ref{SSS:SETUPANDDEFSOFBASICFLUIDVARIABLES})
	\item $\Transport^{\alpha} = \gensmoothfunction(\vec{\Psi})$ (see \eqref{E:MATERIALVECTORVIELDRELATIVECTORECTANGULAR})
	\item $\gfour_{\alpha \beta} = \gensmoothfunction(\vec{\Psi})$ (see \eqref{E:ACOUSTICALMETRIC})
	\item $(\gfour^{-1})^{\alpha \beta} = \gensmoothfunction(\vec{\Psi})$ (see \eqref{E:INVERSEACOUSTICALMETRIC})
	\item $g_{\alpha \beta} = \gensmoothfunction(\vec{\Psi})$ (see \eqref{E:FIRSTFUNDAMENTALFORMSIGMAT})
	\item $(g^{-1})^{\alpha \beta} = \gensmoothfunction(\vec{\Psi})$ (see \eqref{E:INVERSEFIRSTFUNDAMENTALFORMSIGMAT})
	\item $\tophypnorm^{\alpha} = \gensmoothfunction(\vec{\Psi},\pmb{\partial} \Timefunction)$	(see \eqref{E:TOPHYPNORMINTERMSOFGRADIENTOFTIMEFUNCTION})
	\item $\lengthoftophypnorm = \gensmoothfunction(\vec{\Psi},\pmb{\partial} \Timefunction)$ (see \eqref{E:LENGTHOFTOPHYPNORM})
	\item $\lengthofgen = \gensmoothfunction(\vec{\Psi},\vec{\gen})$ (see \eqref{E:LENGTHOFSIDEGEN})
	\item $\lapsemodgen = \gensmoothfunction(\vec{\gen},\pmb{\partial} \Timefunction)$ (see \eqref{E:LAPSEFORMODGENCORRESPONDINGTOTIMEFUNCTION})
	\item $\uposinnerproduct = \gensmoothfunction(\vec{\Psi},\vec{\gen},\vec{\spherenormal},\pmb{\partial} \Timefunction)$ 
		(see \eqref{E:INGOINGCONDITION} and the schematic identity for $\sidehypnorm^{\alpha}$ stated below)
	\item $\seconduposinnerproduct = \gensmoothfunction(\vec{\Psi},\vec{\gen},\pmb{\partial} \Timefunction)$ (see \eqref{E:SECONDINGOINGCONDITION})
	\item $\sidehypnorm^{\alpha} = \gensmoothfunction(\vec{\Psi},\vec{\gen},\pmb{\partial} \Timefunction)$ (see \eqref{E:SIDEHYPNORMINTERMSOFGENANDTOPHYPNORM})
	\item $\lengthofsidehypnorm = \gensmoothfunction(\vec{\Psi},\vec{\gen},\pmb{\partial} \Timefunction)$ (see \eqref{E:LENGTHOFHYPNORM})
	\item $\hat{\tophypnorm}^{\alpha} = \gensmoothfunction(\vec{\Psi},\pmb{\partial} \Timefunction)$ (see \eqref{E:TOPUNITHYPERSURFACENORMAL})
	\item $\hat{\sidehypnorm}^{\alpha} = \gensmoothfunction(\vec{\Psi},\vec{\gen},\pmb{\partial} \Timefunction)$ (see \eqref{E:UNITHYPERSURFACENORMAL})
	\item $\utang^{\alpha} = \gensmoothfunction(\vec{\Psi},\vec{\gen},\vec{\spherenormal},\pmb{\partial} \Timefunction)$ (see \eqref{E:STUTANGENTPARTOFTRANSPORT})
	\item $\urescalednewgenminushypnorm^{\alpha} = \gensmoothfunction(\vec{\Psi},\vec{\gen},\pmb{\partial} \Timefunction)$ 
		(see \eqref{E:REGULARFORMRESCALEDVERSIONOFHYPNORMMINUSNEWGEN})
	\item $\uspecialgen^{\alpha} = \gensmoothfunction(\vec{\Psi},\vec{\gen},\vec{\spherenormal},\pmb{\partial} \Timefunction)$ (see \eqref{E:SPECIALGENERATOR})
	\item $\LeftoverGradEnt^{\alpha} = \gensmoothfunction(\vec{\Psi},\vec{\GradEnt},\vec{\gen},\pmb{\partial} \Timefunction)$ 
		(see \eqref{E:ENTROPYVECTORFIELDKEYTENSORIALDECOMPOSITION})
	\item $\keydetvectorfield^{\alpha} = \gensmoothfunction(\vec{\Psi},\vec{\GradEnt},\vec{\gen},\vec{\spherenormal},\pmb{\partial} \Timefunction)$ 
		(see \eqref{E:VECTORFIELDINKEYDETERMINANT})
	\item $\utandecompvectorfieldmixedarg{\beta}{\alpha} = \gensmoothfunction(\vec{\Psi},\vec{\gen},\pmb{\partial} \Timefunction)$ 
		(see \eqref{E:NEWDECOMPOSITIONOFCOORDINATEPARTIALDERIVATIVEVECTORFIELDS})
	\item $\gsphere_{\alpha \beta} = \gensmoothfunction(\vec{\Psi},\vec{\spherenormal},\pmb{\partial} \Timefunction)$ (see \eqref{E:SPHEREFIRSTFUNDAMENTAL})
	\item $(\gsphere^{-1})^{\alpha \beta} = \gensmoothfunction(\vec{\Psi},\vec{\spherenormal},\pmb{\partial} \Timefunction)$ (see \eqref{E:INVERSESPHEREFIRSTFUNDAMENTAL})
	\item $\sphereproject_{\ \beta}^{\alpha} = \gensmoothfunction(\vec{\Psi},\vec{\spherenormal},\pmb{\partial} \Timefunction)$ (see \eqref{E:SIGMATPROJECT})
	\item $\sideproject_{\ \beta}^{\alpha} = \gensmoothfunction(\vec{\Psi},\vec{\gen},\pmb{\partial} \Timefunction)$ (see \eqref{E:HPROJECT})
	\item For any scalar function $\varphi$, since $\utang$ is $\mathcal{S}_{\Timefunction}$-tangent, we have 
		$\utang \varphi = \utang^{\alpha} \angpartialarg{\alpha} \varphi = 
			\mathscr{L}(\vec{\Psi},\vec{\gen},\vec{\spherenormal},\pmb{\partial} \Timefunction)[\angpartial \varphi		]$
\item Similarly, $\gen \varphi = \gen^{\alpha} \sidepartialarg{\alpha} \varphi = \mathscr{L}(\vec{\gen})[\sidepartial \varphi]$
\item For scalar functions $\varphi$, 
	we have $\angpartial \varphi = \mathscr{L}(\vec{\Psi},\vec{\gen})[\sidepartial \varphi]$
	(see \eqref{E:SIDEPARTIALINTERMSOFOUTERNORMALDERIVATIVEANDANGPARTIAL})
\end{itemize}
From the above facts, the desired result \eqref{E:ERRORTERMSSCHEMATICSTRUCTURESPACELIKECASE} follows easily 
by expanding the terms on RHSs~\eqref{E:LATERALBOUNDARYEASYERRORINTEGRANDTERMS}-\eqref{E:ENTROPYGRADIENTMAINLATERALERRORINTEGRAND}
(relative to the Cartesian coordinates) using the chain and product rules,
and using the following results, which we prove just below:
\begin{align} \label{E:GENLIEDERIVATIVEOFSPHEREMETRICSCHEMATIC}
	(\gsphere^{-1})^{\alpha \beta} \Lie_{\lapsemodgen \gen} \gsphere_{\alpha \beta}
	& = \mathscr{L}(\vec{\Psi},\vec{\gen},\vec{\spherenormal},\pmb{\partial} \Timefunction)
		[\sidepartial \vec{\Psi}, \angpartial \vec{\gen}, \sidepartial \vec{\spherenormal}, \sidepartial \pmb{\partial} \Timefunction],
			\\
	\angdiv \utang
	& = \mathscr{L}(\vec{\Psi},\vec{\gen},\vec{\spherenormal},\pmb{\partial} \Timefunction)[\angpartial \vec{\Psi},\angpartial \vec{\gen},\angpartial \vec{\spherenormal},\angpartial \pmb{\partial} \Timefunction],	
	\label{E:ANGDIVTANGSCHEMATIC} \\
	\angdiv \angpartialarg{\alpha}
	& = \mathscr{L}(\vec{\Psi},\vec{\spherenormal},\pmb{\partial} \Timefunction)[\angpartial \vec{\Psi}, \angpartial \vec{\spherenormal}, \angpartial \pmb{\partial} \Timefunction].
	\label{E:ANGDIVANGPARTIALSCHEMATIC}
\end{align}
The identity \eqref{E:ANGDIVTANGSCHEMATIC} follows easily from \eqref{E:EXPRESSIONFORSPHEREDIVERGENCEINCOORDINATES} 
with $\utang$ in the role of $Y$ and from using the results obtained earlier in the proof.
Similarly, \eqref{E:ANGDIVANGPARTIALSCHEMATIC}
follows easily from \eqref{E:EXPRESSIONFORSPHEREDIVERGENCEOFANGPARTIALVECTORFIELDSINCOORDINATES}
and from using the results obtained earlier in the proof.
Finally, to obtain \eqref{E:GENLIEDERIVATIVEOFSPHEREMETRICSCHEMATIC},
we compute that relative to the Cartesian coordinates, we have
\begin{align} \label{E:FIRSTSTEPGENLIEDERIVATIVEOFSPHEREMETRICSCHEMATIC}
	(\gsphere^{-1})^{\alpha \beta} \Lie_{\lapsemodgen \gen} \gsphere_{\alpha \beta}
	& = (\gsphere^{-1})^{\alpha \beta} \lapsemodgen \gen \gsphere_{\alpha \beta}
			+
			2 (\gsphere^{-1})^{\alpha \beta} \gsphere_{\gamma \beta} \partial_{\alpha} (\lapsemodgen \gen^{\gamma}).
\end{align}
Since $(\gsphere^{-1})^{\alpha \beta} \gsphere_{\gamma \beta} = \sphereproject_{\ \gamma}^{\alpha}$ 
(see the last part of Lemma~\ref{L:BASICPROPSOFFUNDAMENTALFORMSANDPROJECTIONS}),
we see that the last product on RHS~\eqref{E:FIRSTSTEPGENLIEDERIVATIVEOFSPHEREMETRICSCHEMATIC} is equal to
$
2 \gen^{\alpha} \angpartialarg{\alpha} \lapsemodgen
+
2 
\lapsemodgen
\angpartialarg{\alpha} \gen^{\alpha}
=
2 
\lapsemodgen
\angpartialarg{\alpha} \gen^{\alpha}
$,
where to obtain the last equality, we used that $\gen$ is $\gfour$-orthogonal to $\mathcal{S}_{\Timefunction}$ by construction.
Moreover, the results from earlier in the proof imply that
$
2 
\lapsemodgen
\angpartialarg{\alpha} \gen^{\alpha}
=
\mathscr{L}(\vec{\gen},\pmb{\partial} \Timefunction) [\angpartial \vec{\gen}]
$
as desired.
In addition, the results from earlier in the proof imply that the first product on RHS~\eqref{E:FIRSTSTEPGENLIEDERIVATIVEOFSPHEREMETRICSCHEMATIC}
verifies
$(\gsphere^{-1})^{\alpha \beta} \lapsemodgen \gen \gsphere_{\alpha \beta} 
= 
\mathscr{L}(\vec{\Psi},\vec{\gen},\vec{\spherenormal},\pmb{\partial} \Timefunction)
[\sidepartial \vec{\Psi}, \angpartial \vec{\gen}, \sidepartial \vec{\spherenormal}, \sidepartial \pmb{\partial} \Timefunction]$.
In total, these identities yield \eqref{E:GENLIEDERIVATIVEOFSPHEREMETRICSCHEMATIC}.

\end{proof}

\section{The main theorem: Remarkable Hodge-transport-based integral identities for $\vortrenormalized$ and $\GradEnt$}
\label{S:MAININTEGRALIDENTITIES}
In this section, we state and prove our main theorem, 
which, for compressible Euler solutions, 
provides localized, coercive integral identities 
yielding control over the first derivatives of the specific vorticity and entropy gradient.
The identities feature boundary error integrals $\int_{\underline{\mathcal{H}}_{\Timefunction}} \cdots$
and
 $\int_{\mathcal{S}_{\Timefunction}} \cdots$,
and in Theorem~\ref{T:STRUCTUREOFERRORTERMS},
we exhibited the remarkable structures of the error integrand terms in $\int_{\underline{\mathcal{H}}_{\Timefunction}} \cdots$,
with regard to their regularity and to the tangential nature of the derivatives involved.
Moreover, the integrals $\int_{\mathcal{S}_{\Timefunction}} \cdots$, are positive, 
which is crucial for the coerciveness of the identities.
Together, Theorem~\ref{T:MAINREMARKABLESPACETIMEINTEGRALIDENTITY}, Theorem~\ref{T:STRUCTUREOFERRORTERMS},
and the demonstrated positivity of the integrals $\int_{\mathcal{S}_{\Timefunction}} \cdots$
constitute the main new contributions of the paper. 
As we described in Subsect.\,\ref{SS:APPLICATIONS}, the structures revealed by 
Theorem~\ref{T:MAINREMARKABLESPACETIMEINTEGRALIDENTITY} and Theorem~\ref{T:STRUCTUREOFERRORTERMS} 
are crucial for applications. 

\begin{theorem}[The main theorem: Remarkable Hodge-transport integral identities for $\vortrenormalized$ and $\GradEnt$]
	\label{T:MAINREMARKABLESPACETIMEINTEGRALIDENTITY}
	Let $\mathcal{M} = \mathcal{M}_T$ be a spacetime region satisfying the conditions stated in 
	Subsects.\,\ref{SS:DOMAINANDTIMEFUNCTIONETC} and \ref{SS:ASSUMPTIONSONSPACETIMEREGION} for some $T > 0$;
	see Fig.\,\ref{F:SPACETIMEDOMAIN}.
	In particular, assume that the lateral boundary $\underline{\mathcal{H}} = \underline{\mathcal{H}}_T$ is $\gfour$-spacelike
	or is $\gfour$-null (in the null case, $\underline{\mathcal{H}} := \underline{\mathcal{N}} = \underline{\mathcal{N}}_T$).
	For vectorfields $\mathbf{X}$,
	let $\mathscr{Q}(\pmb{\partial} \mathbf{X},\pmb{\partial} \mathbf{X})$
	be the quadratic form defined by
	\eqref{E:MAINQUADRATICFORMFORCONTROLLINGFIRSTDERIVATIVESOFSPECIFICVORTICITYANDENTROPYGRADIENT},
	and recall that the positive definite nature of $\mathscr{Q}$ was revealed in
	Lemma~\ref{L:POSITIVITYPROPERTIESOFVARIOUSQUADRATICFORMS}.
	Let $\weight$ be an arbitrary scalar function.
	Let $\lengthofmodtophypnorm > 0$ be the scalar function defined in \eqref{E:LENGTHOFTOPHYPNORMNORMALIZEDAGAINSTTIMEFUNCTION},
	let $\lapsemodgen > 0$ be the scalar function defined in \eqref{E:LAPSEFORMODGENCORRESPONDINGTOTIMEFUNCTION} (see also \eqref{E:POSITIVITYOFLAPSEMODGEN}),
	and let $\uposinnerproduct > 0$ and $\seconduposinnerproduct > 0$
	be the scalar functions from \eqref{E:INGOINGCONDITION}-\eqref{E:SECONDINGOINGCONDITION}.
	Then for smooth solutions (see Remark~\ref{R:SMOOTHNESSNOTNEEDED})
	to the compressible Euler equations \eqref{E:TRANSPORTDENSRENORMALIZEDRELATIVECTORECTANGULAR}-\eqref{E:ENTROPYTRANSPORT},
	the following integral identities hold, where the definitions of the volume and area forms are 
	provided in Def.\,\ref{D:VOLUMEFORMS},
	and
	the remarkable structure of the 
	integrals over $\underline{\mathcal{H}}_{\Timefunction}$ on 
	RHSs~\eqref{E:SPACETIMEREMARKABLEIDENTITYSPECIFICVORTICITY}-\eqref{E:SPACETIMEREMARKABLEIDENTITYENTROPYGRADIENT}
	was revealed in Theorem~\ref{T:STRUCTUREOFERRORTERMS}
	(and we refer to Appendix~\ref{A:APPENDIXFORBULK} for a table of the notation):
	\begin{subequations}
	\begin{align}
		&
		\int_{\mathcal{M}_{\Timefunction}}	
			\weight
			\lengthofmodtophypnorm^{-1}
			\mathscr{Q}(\pmb{\partial} \vortrenormalized,\pmb{\partial} \vortrenormalized)
		\, d \varpi_{\gfour}
		+
		\int_{\mathcal{S}_{\Timefunction}}
			\weight
			\frac{\uposinnerproduct}{\seconduposinnerproduct \lapsemodgen}
			|\angvortrenormalized|_{\gsphere}^2
		\, d \varpi_{\gsphere} 
			\label{E:SPACETIMEREMARKABLEIDENTITYSPECIFICVORTICITY} \\
		& = 
		\int_{\mathcal{S}_0}
			\weight
			\frac{\uposinnerproduct}{\seconduposinnerproduct \lapsemodgen} 
			|\angvortrenormalized|_{\gsphere}^2
		\, d \varpi_{\gsphere} 
			\notag \\
	& \ \
	+	
	\int_{\mathcal{M}_{\Timefunction}}	
			\weight
			\lengthofmodtophypnorm^{-1}
			\left\lbrace
				\frac{1}{2}|\mathfrak{A}^{(\vortrenormalized)}|_{\topfirstfund}^2
				+
				|\mathfrak{B}_{(\vortrenormalized)}|_{\topfirstfund}^2
				+
				\mathfrak{C}^{(\vortrenormalized)}
				+
				\mathfrak{D}^{(\vortrenormalized)}
				+
				\mathfrak{J}_{(Coeff)}[\vortrenormalized,\pmb{\partial} \vortrenormalized]
				\right\rbrace
	\, d \varpi_{\gfour}
		\notag \\
	& \ \
		+
		\int_{\mathcal{M}_{\Timefunction}}	
			\lengthofmodtophypnorm^{-1}
			\mathfrak{J}_{(\pmb{\partial} \weight)}[\vortrenormalized,\pmb{\partial} \vortrenormalized]
		\, d \varpi_{\gfour}
		\notag \\
	& \ \
		+
		\int_{\underline{\mathcal{H}}_{\Timefunction}}
			\left\lbrace
				\underline{\mathfrak{H}}_{(\pmb{\partial} \weight)}[\vortrenormalized]
				+
				\weight
				\underline{\mathfrak{H}}[\vortrenormalized]
				+
				\weight
				\underline{\mathfrak{H}}_{(1)}[\vortrenormalized]
			\right\rbrace
		\, d \varpi_{\gsphere} d \Timefunction',
			\notag
				\\
	&
		\int_{\mathcal{M}_{\Timefunction}}	
			\weight
			\lengthofmodtophypnorm^{-1}
			\mathscr{Q}(\pmb{\partial} \GradEnt,\pmb{\partial} \GradEnt)
		\, d \varpi_{\gfour}
		+
		\int_{\mathcal{S}_{\Timefunction}}
			\weight
			\frac{\uposinnerproduct}{\seconduposinnerproduct \lapsemodgen} 
			|\angvortrenormalized|_{\gsphere}^2
		\, d \varpi_{\gsphere} 
			\label{E:SPACETIMEREMARKABLEIDENTITYENTROPYGRADIENT} \\
		& = 
		\int_{\mathcal{S}_0}
			\weight
			\frac{\uposinnerproduct}{\seconduposinnerproduct \lapsemodgen} 
			|\angvortrenormalized|_{\gsphere}^2
		\, d \varpi_{\gsphere} 
			\notag \\
	& \ \
	+	
	\int_{\mathcal{M}_{\Timefunction}}	
			\weight
			\lengthofmodtophypnorm^{-1}
			\left\lbrace
				\frac{1}{2}|\mathfrak{A}^{(\GradEnt)}|_{\topfirstfund}^2
				+
				|\mathfrak{B}_{(\GradEnt)}|_{\topfirstfund}^2
				+
				\mathfrak{C}^{(\GradEnt)}
				+
				\mathfrak{D}^{(\GradEnt)}
				+
				\mathfrak{J}_{(Coeff)}[\GradEnt,\pmb{\partial} \GradEnt]
			\right\rbrace
	\, d \varpi_{\gfour}
		\notag \\
	& \ \
	+	
	\int_{\mathcal{M}_{\Timefunction}}	
			\lengthofmodtophypnorm^{-1}
			\mathfrak{J}_{(\pmb{\partial} \weight)}[\GradEnt,\pmb{\partial} \GradEnt]
	\, d \varpi_{\gfour}
		\notag \\
	& \ \
		+
		\int_{\underline{\mathcal{H}}_{\Timefunction}}
			\left\lbrace
				\underline{\mathfrak{H}}_{(\pmb{\partial} \weight)}[\GradEnt]
				+
				\weight
				\underline{\mathfrak{H}}[\GradEnt]
				+
				\weight
				\underline{\mathfrak{H}}_{(2)}[\GradEnt]
			\right\rbrace
		\, d \varpi_{\gsphere} d \Timefunction'.
			\notag
	\end{align}
	\end{subequations}
	On RHSs~\eqref{E:SPACETIMEREMARKABLEIDENTITYSPECIFICVORTICITY}-\eqref{E:SPACETIMEREMARKABLEIDENTITYENTROPYGRADIENT},
	$\mathfrak{A}^{(\vortrenormalized)}$ 
	and
	$\mathfrak{A}^{(\GradEnt)}$
	are two-forms with the Cartesian components
	\begin{subequations}
	\begin{align}
		\mathfrak{A}_{\alpha \beta}^{(\vortrenormalized)}
		&  :=
				2 (\partial_{\beta} \ln \Speed) \vortrenormalized_{\alpha} 
				- 
				2 (\partial_{\alpha} \ln \Speed) \vortrenormalized_{\beta}
				+
				2 \updelta_{\alpha}^0 \vortrenormalized_a \partial_{\beta} v^a
				-
				2 \updelta_{\beta}^0 \vortrenormalized_a \partial_{\alpha} v^a
				\label{E:SPECIFICVORTICITYSPACETIMEERRORTERMANTISYMMETRIC} \\
			& \ \
				-
				\Speed^{-4} 
				\exp(-2 \LogDensity) 
				\frac{p_{;\Ent}}{\bar{\varrho}}
				\upepsilon_{\alpha \beta \gamma \delta}
				(\Transport v^{\gamma}) 
				\GradEnt^{\delta}
					\notag
					\\
			& \ \
				+
			\Speed^{-4} 
			\exp(-2 \LogDensity) 
			\frac{p_{;\Ent}}{\bar{\varrho}}
			\upepsilon_{\alpha \beta \gamma \delta}
			 \Transport^{\gamma}
			[\GradEnt^{\delta}
				(\partial_a v^a)
				-
				\GradEnt^a \partial_a v^{\delta}]
			\notag
				\\
		& \ \
			+
			\Speed^{-2}
			\exp(\LogDensity)
			\upepsilon_{\alpha \beta \gamma \delta}
			\Transport^{\gamma}
			\VortVort^{\delta},
			\notag
				\\
			\mathfrak{A}_{\alpha \beta}^{(\GradEnt)}
			& :
				= 
				2 (\partial_{\beta} \ln \Speed) \GradEnt_{\alpha} 
				- 
				2 (\partial_{\alpha} \ln \Speed) \GradEnt_{\beta},
	\end{align}
	\end{subequations}
	$\mathfrak{B}_{(\vortrenormalized)}$ and $\mathfrak{B}_{(\GradEnt)}$
	are $\Sigma_t$-tangent vectorfields with the Cartesian spatial components
	\begin{subequations}
	\begin{align}
			\mathfrak{B}_{(\vortrenormalized)}^i
			& 
			:=
			\vortrenormalized^a \partial_a v^i
			-
			\exp(-2 \LogDensity) \Speed^{-2} \frac{p_{;\Ent}}{\bar{\varrho}} \upepsilon_{iab} (\Transport v^a) \GradEnt^b,	
				\label{E:SPECIFICVORTICITYSPACETIMEERRORTERMTRANSPORTDERIVATIVES} \\
			\mathfrak{B}_{(\GradEnt)}^i
			& :=
				- 
				\GradEnt^a \partial_a v^i
				+ 
				\upepsilon_{iab} \exp(\LogDensity) \vortrenormalized^a \GradEnt^b,
				\label{E:ENTROPYGRADIENTSPACETIMEERRORTERMTRANSPORTDERIVATIVES}
	\end{align}
	\end{subequations}
	$\mathfrak{C}^{(\vortrenormalized)}$ and $\mathfrak{C}^{(\GradEnt)}$ are scalar functions
	defined relative to the Cartesian coordinates by
	\begin{subequations}
	\begin{align}
		\mathfrak{C}^{(\vortrenormalized)}
		& :=
			-
			2 
			(\projectedtransport_a \projectedtransport \vortrenormalized^a)
			\vortrenormalized^b \partial_b \LogDensity
			-
			2
			\lengthoftophypnorm 
			(\projectedtransport_a \projectedtransport \vortrenormalized^a)
			\projectedtransport_b \mathfrak{B}_{(\vortrenormalized)}^b,
				\label{E:SPECIFICVORTICITYSPACETIMEERRORTERMNEEDSTOBEABSORBED} \\
	\mathfrak{C}^{(\GradEnt)}
		& :=
			2 
			(\projectedtransport_a \projectedtransport \GradEnt^a)
			\left\lbrace
			\exp(2 \LogDensity) \DivGradEnt 
				+ 
			\GradEnt^b \partial_b \LogDensity
			\right\rbrace
			-
			2
			\lengthoftophypnorm 
			(\projectedtransport_a \projectedtransport \GradEnt^a)
			\projectedtransport_b \mathfrak{B}_{(\GradEnt)}^b,
			\label{E:ENTROPYGRADIENTSPACETIMEERRORTERMNEEDSTOBEABSORBED}
\end{align}
\end{subequations}
$\mathfrak{D}^{(\vortrenormalized)}$ and $\mathfrak{D}^{(\GradEnt)}$
are scalar functions defined relative to the Cartesian coordinates by
\begin{subequations}
\begin{align}
		\mathfrak{D}^{(\vortrenormalized)}
		&
		:=
			(\vortrenormalized^a \partial_a \LogDensity)^2
			+
			(
				\lengthoftophypnorm
				\projectedtransport_a \mathfrak{B}_{(\vortrenormalized)}^a
			)^2
			+
			2
			\lengthoftophypnorm 
			(\vortrenormalized^a \partial_a \LogDensity)
			\projectedtransport_b \mathfrak{B}_{(\vortrenormalized)}^b,
				\label{E:SPECIFICVORTICITYSPACETIMEERRORTERMDIVERGENCEERRORS} \\
		\mathfrak{D}^{(\GradEnt)}
		&
		:=
			\left\lbrace
			\exp(2 \LogDensity) \DivGradEnt 
				+ 
			\GradEnt^a \partial_a \LogDensity
			\right\rbrace^2
			+
			(
			\lengthoftophypnorm
			\projectedtransport_a \mathfrak{B}_{(\GradEnt)}^a
			)^2
			-
			2
			\lengthoftophypnorm 
			\left\lbrace
			\exp(2 \LogDensity) \DivGradEnt 
				+ 
			\GradEnt^a \partial_a \LogDensity
			\right\rbrace
			\projectedtransport_b
			\mathfrak{B}_{(\GradEnt)}^b,
				\label{E:ENTROPYGRADIENTSPACETIMEERRORTERMDIVERGENCEERRORS}
	\end{align}
	\end{subequations}
	for $\SigmatTan \in \lbrace \vortrenormalized, \GradEnt \rbrace$,
	the scalar function $\mathfrak{J}_{(Coeff)}[\SigmatTan,\pmb{\partial} \SigmatTan]$ 
	is defined relative to the Cartesian coordinates by
	\begin{align} \label{E:ELLIPTICHYPERBOLICCURRENTCOEFFICIENTERRORTERM}
		\mathfrak{J}_{(Coeff)}[\SigmatTan,\pmb{\partial} \SigmatTan]
		& =
			\SigmatTan^{\alpha} 
			\gfour_{\beta \gamma}
			(\toppartialarg{\alpha} \hat{\tophypnorm}^{\beta})
			\hat{\tophypnorm} \SigmatTan^{\gamma}
			-
			\SigmatTan_{\alpha} 
			(\toppartialarg{\beta} \hat{\tophypnorm}^{\alpha})
			\hat{\tophypnorm} \SigmatTan^{\beta}
				\\
		& \ \
				+
		\SigmatTan^{\alpha} 
		\hat{\tophypnorm}_{\alpha} 
		(\toppartialarg{\beta} \hat{\tophypnorm}^{\beta}) 
		\toppartialarg{\gamma} \SigmatTan^{\gamma}
			-
			\SigmatTan^{\alpha} \hat{\tophypnorm}_{\alpha} 
			(\toppartialarg{\beta} \hat{\tophypnorm}^{\gamma}) 
			\toppartialarg{\gamma} 
			\SigmatTan^{\beta} 
				\notag \\
		& \ \
		+
		\SigmatTan^{\alpha} 
		\hat{\tophypnorm}_{\beta}
		(\toppartialarg{\alpha} \hat{\tophypnorm}^{\gamma})
		\toppartialarg{\gamma} \SigmatTan^{\beta}
			-
			\SigmatTan^{\alpha} 
			\hat{\tophypnorm}_{\beta} 
			(\toppartialarg{\gamma} \hat{\tophypnorm}^{\gamma})  
			\toppartialarg{\alpha} \SigmatTan^{\beta}
		\notag
			\\
		&  \ \
			+
			\SigmatTan^{\alpha} 
			\hat{\tophypnorm}^{\beta}
			(\toppartialarg{\alpha} \gfour_{\beta \gamma})
			\hat{\tophypnorm} \SigmatTan^{\gamma}
			-
			\SigmatTan^{\alpha} 
			\hat{\tophypnorm}^{\beta}
			(\toppartialarg{\gamma} \gfour_{\alpha \beta})
			\hat{\tophypnorm} \SigmatTan^{\gamma}
				\notag \\
		& \ \
			+
			\frac{1}{2}
			\SigmatTan^{\alpha} 
			\hat{\tophypnorm}_{\beta}
			\hat{\tophypnorm}^{\gamma}
			\hat{\tophypnorm}^{\delta}
			(\toppartialarg{\alpha} \gfour_{\gamma \delta})			
			\hat{\tophypnorm} \SigmatTan^{\beta}
			-
			\frac{1}{2}
			\SigmatTan^{\alpha} \hat{\tophypnorm}_{\alpha} 
			\hat{\tophypnorm}^{\beta}
			\hat{\tophypnorm}^{\gamma}
			(\toppartialarg{\delta} \gfour_{\beta \gamma}) 
			\hat{\tophypnorm} \SigmatTan^{\delta} 
				\notag \\
	& \ \
		+
		2
		\SigmatTan^{\alpha}
		(\toppartialarg{\beta} \gfour_{\alpha \gamma})
		\toppartialuparg{\gamma} \SigmatTan^{\beta}
		-
		2
		\topproject_{\ \beta}^{\alpha}
		\SigmatTan^{\gamma}
		(\toppartialuparg{\delta} \gfour_{\alpha \gamma}) 
		\toppartialarg{\delta} \SigmatTan^{\beta}
				\notag \\
	& \ \
		+
		\frac{1}{2}
		(\topfirstfund^{-1})^{\alpha \beta}
		\SigmatTan^{\gamma} 
		(\toppartialarg{\gamma} \gfour_{\alpha \beta})
		\toppartialarg{\delta} \SigmatTan^{\delta}
		-
		\frac{1}{2}
		(\topfirstfund^{-1})^{\alpha \beta}
		\SigmatTan^{\gamma} 
		(\toppartialarg{\delta} \gfour_{\alpha \beta})
		\toppartialarg{\gamma} \SigmatTan^{\delta}
		\notag
			\\
		& \ \
			+
			\SigmatTan^{\alpha} 
			\SigmatTan^{\beta}
			(\toppartialuparg{\gamma} \gfour_{\alpha \delta})
			(\toppartialuparg{\delta} \gfour_{\beta \gamma})
		-
		\SigmatTan^{\alpha} 
		\SigmatTan^{\beta}
		(\topfirstfund^{-1})^{\gamma \delta} 
		(\toppartialuparg{\kappa} \gfour_{\alpha \gamma}) 
		\toppartialarg{\kappa} \gfour_{\beta \delta},
				\notag
	\end{align}
	for $\SigmatTan \in \lbrace \vortrenormalized, \GradEnt \rbrace$,
	the scalar function $\mathfrak{J}_{(\pmb{\partial} \weight)}[\SigmatTan,\pmb{\partial} \SigmatTan]$ 
	is defined relative to the Cartesian coordinates by
	\begin{align} \label{E:ELLIPTICHYPERBOLICCURRENTBULKERRORTERMWITHWEIGHTDERIVATIVES}
		\mathfrak{J}_{(\pmb{\partial} \weight)}[\SigmatTan,\pmb{\partial} \SigmatTan]
		:= - J[\SigmatTan] \weight 
		=
		\SigmatTan^{\kappa}  
		(\toppartialarg{\kappa} \weight)
		\toppartialarg{\lambda} \SigmatTan^{\lambda}
		-
		\SigmatTan^{\kappa} 
		(\toppartialarg{\lambda} \weight)
		\toppartialarg{\kappa} \SigmatTan^{\lambda},
	\end{align}
	for $\SigmatTan \in \lbrace \vortrenormalized, \GradEnt \rbrace$,
	the scalar functions $\underline{\mathfrak{H}}_{(\pmb{\partial} \weight)}[\SigmatTan]$ 
	and
	$\underline{\mathfrak{H}}[\SigmatTan]$ 
	are defined relative to the Cartesian coordinates by
	\begin{subequations}
		\begin{align} \label{E:MAINTHMWEIGHTDERIVATVELATERALBOUNDARYEASYERRORINTEGRANDTERMS}
		\underline{\mathfrak{H}}_{(\pmb{\partial} \weight)}[\SigmatTan]
		& :=
		 \frac{\uposinnerproduct}{\seconduposinnerproduct} 
		|\angV|_{\gsphere}^2
		\gen \weight
		+
		\uposinnerproduct 
		|\angV|_{\gsphere}^2
		\utang \weight
		+
		\SigmatTan_{\alpha} \spherenormal^{\alpha}
		\angV \weight,
			\\
		\underline{\mathfrak{H}}[\SigmatTan]
			&
			:=
			\frac{1}{2} 
			\frac{\uposinnerproduct}{\seconduposinnerproduct \lapsemodgen}  
			|\angV|_{\gsphere}^2
			(\gsphere^{-1})^{\alpha \beta} \Lie_{\lapsemodgen \gen} \gsphere_{\alpha \beta}
			\label{E:MAINTHMLATERALBOUNDARYEASYERRORINTEGRANDTERMS} 
				\\
		& \ \
		+
		\left\lbrace
			\lapsemodgen \gen
			\left[\frac{\uposinnerproduct}{\seconduposinnerproduct \lapsemodgen} (\gsphere^{-1})^{\alpha \beta} \right]
		\right\rbrace
		\SigmatTan_{\alpha} 
	  \SigmatTan_{\beta}
		+
		\left\lbrace
		\utang
			\left[\uposinnerproduct (\gsphere^{-1})^{\alpha \beta} \right]
		\right\rbrace
		\SigmatTan_{\alpha} 
	  \SigmatTan_{\beta}
		-
		2 \SigmatTan_{\alpha}
		\angV \Transport^{\alpha}
		\notag \\
		& \ \
				+
			\uposinnerproduct 
			|\angV|_{\gsphere}^2
			\angdiv \utang 
		+
		\spherenormal_{\alpha} \SigmatTan^{\alpha}  
		\SigmatTan^{\beta}
		\angdiv \angpartialarg{\beta}
			\notag \\
		& \ \
		+
		\SigmatTan^{\alpha}
		\spherenormal^{\beta}
		\angV \gfour_{\alpha \beta}
		+
		\SigmatTan_{\alpha}
		\angV \spherenormal^{\alpha},
		\notag
	\end{align}
	\end{subequations}
	the scalar function
	$\underline{\mathfrak{H}}_{(1)}[\vortrenormalized]$
	is defined relative to the Cartesian coordinates by
	\begin{align} \label{E:MAINTHMSPECIFICVORTITICYMAINLATERALERRORINTEGRAND}
		\underline{\mathfrak{H}}_{(1)}[\vortrenormalized]
		&  
		: = 
		 2
		\upsigma
		\uposinnerproduct
		\Speed^{-1} 
		\exp(-2 \LogDensity) 
		\frac{p_{;\Ent}}{\bar{\varrho}}
		\frac{\seconduposinnerproduct + \lengthofgen^2}{\sqrt{\seconduposinnerproduct^2 + \lengthofgen^2 \lengthoftophypnorm^2}}
		\sqrt{\mbox{\upshape det}
		\begin{pmatrix}
		\gsphere(\angvortrenormalized,\angvortrenormalized) & 
		\gsphere(\angvortrenormalized,\keydetvectorfield)
			\\
		\gsphere(\angvortrenormalized,\keydetvectorfield) & \gsphere(\keydetvectorfield,\keydetvectorfield)
	\end{pmatrix}}
		\\
	& \ \ \ \ \ \
	\times
	\left\lbrace
			\urescalednewgenminushypnorm_a \gen v^a
				+
				\sidehypnorm_a \urescalednewgenminushypnorm^a \gen \LogDensity
				+
				\utandecompvectorfielddownarg{a} v^a
				+ 
				(g^{-1})^{ab} \sidehypnorm_a \utandecompvectorfielddownarg{b} \LogDensity
				+
			  \exp(-\LogDensity) \frac{p_{;\Ent}}{\bar{\varrho}} \sidehypnorm_a \GradEnt^a
	\right\rbrace	
			\notag		\\
			& \ \
				+
				4 
				\uposinnerproduct 
				\vortrenormalized_{\alpha} \uspecialgen^{\alpha}
				\angvortrenormalized \ln \Speed 
				- 
				4 
				\uposinnerproduct 
				|\angvortrenormalized|_{\gsphere}^2
				\uspecialgen \ln \Speed 
				+
				4 
				\uposinnerproduct 
				\uspecialgen^0
				\vortrenormalized_a \angvortrenormalized v^a
				-
				4 
				\uposinnerproduct
				\angvortrenormalized^0
				\vortrenormalized_a \uspecialgen v^a
					\notag \\
		& \ \
		+
		2 
		\uposinnerproduct
		\Speed^{-4} 
		\exp(-2 \LogDensity) 
		\frac{p_{;\Ent}}{\bar{\varrho}}
		\upepsilon_{\alpha \beta \gamma \delta}
		 \uspecialgen^{\alpha} 
		 \angvortrenormalized^{\beta}
		\urescalednewgenminushypnorm^{\gamma}
		\GradEnt^{\delta}
		\gen \LogDensity
		-
		2 
		\uposinnerproduct
		\Speed^{-4} 
		\exp(-2 \LogDensity) 
		\frac{p_{;\Ent}}{\bar{\varrho}}
		\GradEnt^a \sidehypnorm_a
		\upepsilon_{\alpha \beta \gamma \delta}
		\uspecialgen^{\alpha} 
		\angvortrenormalized^{\beta} 
		\Transport^{\gamma}
		\urescalednewgenminushypnorm^{\delta} 
		\gen \LogDensity
			\notag 
			\\
	 & \ \
		+
		2 
		\uposinnerproduct
		\Speed^{-2} 
		\exp(-2 \LogDensity) 
		\frac{p_{;\Ent}}{\bar{\varrho}}
		\upepsilon_{\alpha \beta ab}
		\uspecialgen^{\alpha}  
		\angvortrenormalized^{\beta}
		(\utandecompvectorfielddownarg{a} \LogDensity)
		\GradEnt^b
		-
		2 
		\uposinnerproduct
		\Speed^{-2} 
		\exp(-2 \LogDensity) 
		\frac{p_{;\Ent}}{\bar{\varrho}}
		\GradEnt^a \sidehypnorm_a
		\upepsilon_{\alpha \beta \gamma d}
		\uspecialgen^{\alpha} 
		\angvortrenormalized^{\beta} 
		\Transport^{\gamma}
		\utandecompvectorfielddownarg{d} \LogDensity
		\notag
			\\
	& \ \
		-
		2 
		\uposinnerproduct
		\Speed^{-4} 
				\exp(-2 \LogDensity) 
				\frac{p_{;\Ent}}{\bar{\varrho}}
		\GradEnt^a \urescalednewgenminushypnorm_a
		\upepsilon_{\alpha \beta \gamma \delta}
		\uspecialgen^{\alpha} 
		\angvortrenormalized^{\beta} 
		\Transport^{\gamma}
		\gen v^{\delta}
		-
		2 
		\uposinnerproduct
		\Speed^{-4} 
				\exp(-2 \LogDensity) 
				\frac{p_{;\Ent}}{\bar{\varrho}}
		\upepsilon_{\alpha \beta \gamma \delta}
		\uspecialgen^{\alpha} 
		\angvortrenormalized^{\beta} 
		\Transport^{\gamma}
		\GradEnt^a \utandecompvectorfielddownarg{a} v^{\delta}
		\notag
			\\
	& \ \
				+
			2 
			\uposinnerproduct
			\Speed^{-2}
			\exp(\LogDensity)
			\upepsilon_{\alpha \beta \gamma \delta}
			\uspecialgen^{\alpha}
			\angvortrenormalized^{\beta}
			\Transport^{\gamma}
			\VortVort^{\delta}
				\notag
				\\
	& \ \
		-
		2 
		\uposinnerproduct
		\Speed^{-4} 
		\exp(-3 \LogDensity) 
		\left\lbrace
			\frac{p_{;\Ent}}{\bar{\varrho}}
		\right\rbrace^2
		\GradEnt^a \sidehypnorm_a
		\upepsilon_{\alpha \beta \gamma \delta}
		\uspecialgen^{\alpha} 
		\angvortrenormalized^{\beta} 
		\Transport^{\gamma}
		\GradEnt^{\delta},
				\notag
	\end{align}
and the scalar function
$\underline{\mathfrak{H}}_{(2)}[\GradEnt]$
is defined relative to the Cartesian coordinates by
\begin{align} \label{E:MAINTHMENTROPYGRADIENTMAINLATERALERRORINTEGRAND}
		\underline{\mathfrak{H}}_{(2)}[\GradEnt]
		& := 
			4 
			\uposinnerproduct \GradEnt_{\alpha} \uspecialgen^{\alpha} \angGradEnt \ln \Speed
			- 
			4 
			\uposinnerproduct |\angGradEnt|_{\gsphere}^2  \uspecialgen \ln \Speed.
\end{align}
Note also that Prop.\,\ref{P:KEYDETERMINANT} implies that
the first product on RHS~\eqref{E:MAINTHMSPECIFICVORTITICYMAINLATERALERRORINTEGRAND}
vanishes when $\underline{\mathcal{H}}$ is $\gfour$-null.	
	
\end{theorem}

\begin{remark}[Highlighting some of the key structures in Theorem~\ref{T:MAINREMARKABLESPACETIMEINTEGRALIDENTITY}]
	\label{R:HIGHLIGHTKEYSTRUCTURES}
		Here we emphasize some of the key structures in the equations of Theorem~\ref{T:MAINREMARKABLESPACETIMEINTEGRALIDENTITY}.
		\begin{itemize}
			\item All derivatives of $(\LogDensity,v,\Ent)$ and $(\gen,\spherenormal,\pmb{\partial} \Timefunction)$
				on RHSs~\eqref{E:MAINTHMLATERALBOUNDARYEASYERRORINTEGRANDTERMS}-\eqref{E:MAINTHMENTROPYGRADIENTMAINLATERALERRORINTEGRAND}
				are in directions \emph{tangent} to $\underline{\mathcal{H}}$.
				In Theorem~\ref{T:STRUCTUREOFERRORTERMS}, we provided a precise statement capturing this structure;
				see also Remark~\ref{R:MOSTIMPORTANTPARTOFTHEPROP}.
				We now point out two reasons why this structure is crucial for applications in the $\gfour$-null case $\underline{\mathcal{H}} = \underline{\mathcal{N}}$:
				\textbf{i)} The wave equation fluxes degenerate along null hypersurfaces, and they control only derivatives 
				in directions tangent to $\underline{\mathcal{N}}$
				(see, for example, \eqref{E:NULLCASEWAVEFLUXESSEMICOERCIVE} in a model case in which $\Timefunction = t$). 
				Thus, if $ \underline{\mathcal{N}}_{\Timefunction}$-transversal derivatives were present
				on RHSs~\eqref{E:MAINTHMLATERALBOUNDARYEASYERRORINTEGRANDTERMS}-\eqref{E:MAINTHMENTROPYGRADIENTMAINLATERALERRORINTEGRAND}, 
				then 
				\emph{the integrals 
				$\int_{\underline{\mathcal{N}_{\Timefunction}}}
				\cdots
				$
				on RHSs~\eqref{E:SPACETIMEREMARKABLEIDENTITYSPECIFICVORTICITY}-\eqref{E:SPACETIMEREMARKABLEIDENTITYENTROPYGRADIENT}
				would be uncontrollable from the point of view of regularity}.
				 \textbf{ii)} In applications to shock waves, the $\underline{\mathcal{N}}_{\Timefunction}$-tangential derivatives
					of the solution are less singular than the $\underline{\mathcal{N}}_{\Timefunction}$-transversal derivatives.
					Thus, in the $\gfour$-null case, the identity \eqref{E:ERRORTERMSSCHEMATICSTRUCTURENULLCASE}
					\emph{signifies the absence of the most singular terms}. 
					This is a manifestation of the good ``remarkable quasilinear null structures''
					highlighted in the indented paragraph near the beginning of the article.
					See Subsubsect.\,\ref{SSS:STANDARDNULLFORMS} for further discussion on the importance of
					$\underline{\mathcal{N}}_{\Timefunction}$-tangential derivatives in the context of shock formation.
			\item All terms on RHSs~\eqref{E:SPACETIMEREMARKABLEIDENTITYSPECIFICVORTICITY}-\eqref{E:SPACETIMEREMARKABLEIDENTITYENTROPYGRADIENT}
				are controllable from the point of view of regularity. 
				We make this statement precise in Theorem~\ref{T:LOCALIZEDAPRIORIESTIMATES}
				in a model case in which $\Timefunction = t$.
		\end{itemize}
\end{remark}

\begin{proof}[Proof of Theorem~\ref{T:MAINREMARKABLESPACETIMEINTEGRALIDENTITY}]
	We first prove \eqref{E:SPACETIMEREMARKABLEIDENTITYSPECIFICVORTICITY}.
	We start by considering the $\widetilde{\Sigma}_{\Timefunction}$-divergence identity \eqref{E:NEWSTANDARDDIVERGENCEIDENTITYFORELLIPTICHYPERBOLICCURRENT}
	with $\vortrenormalized$ in the role of $\SigmatTan$.
	We add the term
	$\weight \topfirstfund_{\alpha \beta} (\Transport \vortrenormalized^{\alpha}) (\Transport \vortrenormalized^{\beta})$
	to each side of this identity.
	Considering \eqref{E:IDENTITYMAINQUADRATICFORMFORCONTROLLINGFIRSTDERIVATIVESOFSPECIFICVORTICITYANDENTROPYGRADIENT}, 
	we see that after this addition, 
	the left-hand side of the resulting identity is equal to
	$\weight \mathscr{Q}(\pmb{\partial} \vortrenormalized,\pmb{\partial} \vortrenormalized)$.
	We then integrate the identity over $\widetilde{\Sigma}_{\Timefunction}$ with respect to the volume form
	$d \varpi_{\topfirstfund}$ of $\topfirstfund$ 
	and use the divergence theorem
	to obtain an integral identity, 
	which features the
	main term
	$
	\int_{\widetilde{\Sigma}_{\Timefunction}}	
			\weight
			\mathscr{Q}(\pmb{\partial} \vortrenormalized,\pmb{\partial} \vortrenormalized)
		\, d \varpi_{\topfirstfund}
	$
	on the left-hand side and, on the right-hand side,
	the boundary term
	$
	\int_{\mathcal{S}_{\Timefunction}}
				\weight \spherenormal_{\alpha} J^{\alpha}[\vortrenormalized]
	\,  d \varpi_{\gsphere}
	$,
	which comes from the term $\widetilde{\nabla}_{\alpha} \left(\weight J^{\alpha}[\SigmatTan] \right)$ 
	on RHS~\eqref{E:NEWSTANDARDDIVERGENCEIDENTITYFORELLIPTICHYPERBOLICCURRENT}.
	We then integrate the integral identity with respect to $\Timefunction$ 
	and use Lemma~\ref{L:IDENTITIESFORVOLUMEFORMS} to obtain,
	in view of \eqref{E:IDENTITYMAINQUADRATICFORMFORCONTROLLINGFIRSTDERIVATIVESOFSPECIFICVORTICITYANDENTROPYGRADIENT},
	the identity
	\begin{align} \label{E:FIRSTSTEPSPACETIMEREMARKABLEIDENTITYSPECIFICVORTICITY}
		&
		\int_{\mathcal{M}_{\Timefunction}}	
			\weight
			\lengthofmodtophypnorm^{-1}
			\mathscr{Q}(\pmb{\partial} \vortrenormalized,\pmb{\partial} \vortrenormalized)
		\, d \varpi_{\gfour}
		=
			\int_{\underline{\mathcal{H}}_{\Timefunction}}
				\weight
				\spherenormal_{\alpha} J^{\alpha}[\vortrenormalized]
			\, d \varpi_{\gsphere} d \Timefunction'
				\\
		& \ \
			+
			\int_{\mathcal{M}_{\Timefunction}}		
				\weight
				\lengthofmodtophypnorm^{-1}
				\left\lbrace
					\topfirstfund_{\alpha \beta} (\Transport \vortrenormalized^{\alpha}) \Transport \vortrenormalized^{\beta}
					+
					\mathfrak{J}_{(\widetilde{Antisym})}[\pmb{\partial} \vortrenormalized,\pmb{\partial} \vortrenormalized]
					+
					\mathfrak{J}_{(Div)}[\pmb{\partial} \vortrenormalized,\pmb{\partial} \vortrenormalized]
					+
					\mathfrak{J}_{(Coeff)}[\vortrenormalized,\pmb{\partial} \vortrenormalized]
			\right\rbrace
			\, d \varpi_{\gfour}
			\notag
				\\
		& \ \
			+
			\int_{\mathcal{M}_{\Timefunction}}		
				\lengthofmodtophypnorm^{-1}
				\mathfrak{J}_{(\pmb{\partial} \weight)}[\vortrenormalized,\pmb{\partial} \vortrenormalized]
			\, d \varpi_{\gfour}.
			\notag
	\end{align}
	To handle the first integral 
	$
	\int_{\underline{\mathcal{H}}_{\Timefunction}}
				\weight \spherenormal_{\alpha} J^{\alpha}[\vortrenormalized]
			\, d \varpi_{\gsphere} d \Timefunction'
	$
	on RHS~\eqref{E:FIRSTSTEPSPACETIMEREMARKABLEIDENTITYSPECIFICVORTICITY},
	we simply use the identity \eqref{E:MAINLATERALERRORINTEGRALSFORSPECIFICVORTICITY} 
	to substitute for
	$
	\int_{\underline{\mathcal{H}}_{\Timefunction}}
		\weight \spherenormal_{\alpha} J^{\alpha}[\vortrenormalized]
	\, d \varpi_{\gsphere} d \Timefunction'
	$.
	To handle the integral
	$
		\int_{\mathcal{M}_{\Timefunction}}		
				\weight
				\lengthofmodtophypnorm^{-1}
				\topfirstfund_{\alpha \beta} (\Transport \vortrenormalized^{\alpha}) \Transport \vortrenormalized^{\beta}
		\, d \varpi_{\gfour}
	$,
	we note that 
	\eqref{E:RENORMALIZEDVORTICTITYTRANSPORTEQUATION},
	\eqref{E:SPECIFICVORTICITYLINEARORBETTER},
	\eqref{E:SPECIFICVORTICITYSPACETIMEERRORTERMTRANSPORTDERIVATIVES},
	and the identity $\vortrenormalized^0 = 0$
	imply
	$
	\weight
	\lengthofmodtophypnorm^{-1}
	\topfirstfund_{\alpha \beta} (\Transport \vortrenormalized^{\alpha}) \Transport \vortrenormalized^{\beta}
	=
	\weight
	\lengthofmodtophypnorm^{-1}
	|\mathfrak{B}_{(\vortrenormalized)}|_{\topfirstfund}^2
	$,
	which is explicitly featured on 
	RHS~\eqref{E:SPACETIMEREMARKABLEIDENTITYSPECIFICVORTICITY}.
	To handle 
	$
		\int_{\mathcal{M}_{\Timefunction}}		
				\weight
				\lengthofmodtophypnorm^{-1}
				\mathfrak{J}_{(\widetilde{Antisym})}[\pmb{\partial} \vortrenormalized,\pmb{\partial} \vortrenormalized]
		\, d \varpi_{\gfour}
	$,
	we simply use
	\eqref{E:ANTISYMMETRICERRORTERMDIVERGENCEOFSIGMATILDECURRENT},
	\eqref{E:KEYIDENTITYANTISYMMETRICPARTOFSPECIFICVORTICITYDUALGRADIENT},
	and
	\eqref{E:SPECIFICVORTICITYSPACETIMEERRORTERMANTISYMMETRIC}
	to deduce that this integral is equal to the integral
	$
		\frac{1}{2}
		\int_{\mathcal{M}_{\Timefunction}}		
				\weight
				\lengthofmodtophypnorm^{-1}
				|\mathfrak{A}^{(\vortrenormalized)}|_{\topfirstfund}^2
		\, d \varpi_{\gfour}
	$
	featured on RHS~\eqref{E:SPACETIMEREMARKABLEIDENTITYSPECIFICVORTICITY}.
	To handle 
	$
		\int_{\mathcal{M}_{\Timefunction}}		
			\weight
			\lengthofmodtophypnorm^{-1}
				\mathfrak{J}_{(Div)}[\pmb{\partial} \vortrenormalized,\pmb{\partial} \vortrenormalized]
		\, d \varpi_{\gfour}
	$,
	we consider equation \eqref{E:DIVSYMMETRICERRORTERMDIVERGENCEOFSIGMATILDECURRENT} with $\vortrenormalized$
	in the role of $\SigmatTan$. Using \eqref{E:FLATDIVOFRENORMALIZEDVORTICITY} and \eqref{E:RENORMALIZEDVORTICITYDIVLINEARORBETTER},
	we rewrite all factors of $\partial_a \vortrenormalized^a$
	on RHS~\eqref{E:DIVSYMMETRICERRORTERMDIVERGENCEOFSIGMATILDECURRENT}
	as $- \vortrenormalized^a \partial_a \LogDensity$.
	We place those resulting products involving a factor of $\projectedtransport_a \projectedtransport \vortrenormalized^a$
	on RHS~\eqref{E:SPECIFICVORTICITYSPACETIMEERRORTERMNEEDSTOBEABSORBED},
	and we place the remaining products on RHS~\eqref{E:SPECIFICVORTICITYSPACETIMEERRORTERMDIVERGENCEERRORS}.
	In total, we have proved \eqref{E:SPACETIMEREMARKABLEIDENTITYSPECIFICVORTICITY}.
	
	The identity \eqref{E:SPACETIMEREMARKABLEIDENTITYENTROPYGRADIENT} can be proved using nearly identical arguments
	where, compared to the previous paragraph,
	we use \eqref{E:MAINLATERALERRORINTEGRALSFORENTROPYGRADIENT} in the role of \eqref{E:MAINLATERALERRORINTEGRALSFORSPECIFICVORTICITY},
	\eqref{E:GRADENTROPYTRANSPORT} in the role of \eqref{E:RENORMALIZEDVORTICTITYTRANSPORTEQUATION},
	\eqref{E:KEYIDENTITYANTISYMMETRICPARTOFENTROPYGRADIENTDUALGRADIENT}
	in the role of \eqref{E:KEYIDENTITYANTISYMMETRICPARTOFSPECIFICVORTICITYDUALGRADIENT},
	and the identity
	$
	\partial_a \GradEnt^a
	=
	\exp(2 \LogDensity) \DivGradEnt
	+
	\GradEnt^a \partial_a \LogDensity
	$
	(which follows from \eqref{E:RENORMALIZEDDIVOFENTROPY})
	in the role of \eqref{E:FLATDIVOFRENORMALIZEDVORTICITY}.

\end{proof}

\section{An application: Localized a priori estimates via the integral identities}
\label{S:APRIORI}
In this section, we provide a basic application of our main results: 
the derivation of a priori estimates for solutions
that exhibit the localized gain of a derivative for the specific vorticity and entropy gradient,
as we described in Subsect.\,\ref{SS:APPLICATIONS}.
\begin{quote}
	To streamline the presentation, 
	throughout this section,
	we assume that the acoustical time function $\Timefunction$ from the beginning of Sect.\,\ref{S:SPACETIMEDOMAINS}
	is equal to the Cartesian time function $t$.
\end{quote}
Our main goal is to derive localized energy-flux-elliptic estimates for solutions to the
compressible Euler equations, more precisely the formulation provided by Theorem~\ref{T:GEOMETRICWAVETRANSPORTSYSTEM}.
See Theorem~\ref{T:LOCALIZEDAPRIORIESTIMATES} for a precise statement of these estimates,
which rely on standard $C^1$-type boundedness assumptions that we state in Subsect.\,\ref{SS:ASSUMPTIONSONTHESOLUTION}.
The compressible Euler formulation provided by Theorem~\ref{T:GEOMETRICWAVETRANSPORTSYSTEM}
has an ``evolution-part'' and an ``elliptic-part.''
The main ingredients needed to control the elliptic-part
are the integral identities on the spacetime region $\mathcal{M} = \mathcal{M}_T$ 
provided by Theorem~\ref{T:MAINREMARKABLESPACETIMEINTEGRALIDENTITY}
and the structural features of the lateral error integrals revealed by Theorem~\ref{T:STRUCTUREOFERRORTERMS}
(which are important when the lateral boundary $\underline{\mathcal{H}}$ of $\mathcal{M}$ is $\gfour$-null).
In this section, we complement these results with similar, but much simpler results 
for the evolution-part of the system.
Compared to Theorem~\ref{T:MAINREMARKABLESPACETIMEINTEGRALIDENTITY},
the results of this section are standard,
though some aspects of our analysis 
rely on the detailed structure of the acoustical metric $\gfour$
of Def.\,\ref{D:ACOUSTICALMETRIC} and
the geometry of $\mathcal{M}$,
which we derived in Sect.\,\ref{S:SPACETIMEDOMAINS}.

\subsection{Various identities specialized to the case $\Timefunction = t$}
\label{SS:VARIOUSIDENTITIESWHENTIMEFUNCTIONISCARTESIAN}
In the next proposition, we provide some identities that hold
when $\Timefunction = t$.

\begin{proposition}[Various identities that hold when $\Timefunction \equiv t$]
	\label{P:VARIOUSIDENTITIESWHENTIMEFUNCTIONISCARTESIAN}
	Assume that the acoustical time function $\Timefunction$ from the beginning of Sect.\,\ref{S:SPACETIMEDOMAINS}
	is equal to the Cartesian time function $t$.
	Then the following identities hold
	(and we refer to Appendix~\ref{A:APPENDIXFORBULK} for a table of the notation):
	\begin{align} \label{E:WHENTIMEFUNCTIONISCARTESIANTRANSOPORTCOINCIDESWITHTOPNORMAL}
		\tophypnorm
		& 
		= 
		\modtophypnorm
		=
		\Transport,
			\\
		\lengthoftophypnorm
		& =  
		\lengthofmodtophypnorm
		= 1,
		\label{E:LENGTHOFMODTOPHYPNORMISUNITY}
	\end{align}	
	
	\begin{align}
		\modgen
		& = \gen,
			\label{E:MODGENISEQUALTOGEN} \\
		\lengthofmodgen
		& = \lengthofgen,
		\label{E:LENGTHOFMODGENISLENGTHOFGEN}
			\\
		\lapsemodgen
		& = 1,
			\label{E:LAPSEMODGENISUNITY} 
\end{align}
	
	\begin{align} \label{E:GENCOEFFICIENTISUNITY}
		\seconduposinnerproduct
		& = 1,
	\end{align}
	
	\begin{align} \label{E:SIGMATFIRSTFUNDAGREESWITHSIGMATILDEFIRSTFUND}
		\topfirstfund 
		& = g,
		&
		\topfirstfund^{-1} 
		& = g^{-1},
		&
	 \Sigmatproject 
	& = \topproject,
	\end{align}
	
	\begin{align} \label{E:PROJECTEDTRANSPORTVANISHES}
		\projectedtransport
		& = 0.
	\end{align}
	
	Moreover, when $\underline{\mathcal{H}}$ is $\gfour$-null, the following identities hold:
	\begin{align} \label{E:NULLCASESPHERENORMALCOEFFICIENTISUNITY}
		\uposinnerproduct
		& = 1,
			\\
		\utang 
		& = 0,
		\label{E:GNULLCASEANGULARVECTORFIELDVANISHES}
			\\
		\Transport
		& = \gen + \spherenormal
			=
			\uLunit + \spherenormal,
		\label{E:NULLCASETRANSPORTINTERMSOFGENANDSPHERENORMAL}
	\end{align}
	
	\begin{align} \label{E:INVERSEACOUSTICALMETRICNULLCASE}
	(\gfour^{-1})^{\alpha \beta}
	& = - 
			\frac{1}{2} \Lunit^{\alpha} \uLunit^{\beta}
			- 
			\frac{1}{2} \uLunit^{\alpha} \Lunit^{\beta}
			+
			(\gsphere^{-1})^{\alpha \beta},
\end{align}
where
\begin{align} \label{E:LUNITINNULLCASEWHENTIMEFUNCTIONISCARTESIAN}
	\Lunit
	 & := 
		\uLunit 
		+ 
		2 \spherenormal
		=
		\Transport 
		+ 
		\spherenormal
\end{align}
is an outgoing $\gfour$-null vectorfield that is $\gfour$-orthogonal to $\mathcal{S}_t$
and that verifies 
\begin{align} \label{E:GINNERPRODUCTOFLUNITANDTRANSPORT}
	\gfour(\Lunit,\Transport)
	& = - 1.
\end{align}

\end{proposition}

\begin{proof}
	To prove \eqref{E:WHENTIMEFUNCTIONISCARTESIANTRANSOPORTCOINCIDESWITHTOPNORMAL},
	we first note that $\widetilde{\Sigma}_{\Timefunction}$ is equal to a portion of $\Sigma_t$
	since $\Timefunction = t$.
	Thus, from 
	\eqref{E:MATERIALVECTORVIELDRELATIVECTORECTANGULAR},
	\eqref{E:TRANSPORTONEFORMIDENTITY},
	and
	Def.\,\ref{D:HYPNORMANDSPHEREFORMDEFS}, 
	it follows that
	$\tophypnorm$,
	$\modtophypnorm$,
	and $\Transport$ are parallel
	and verify $\tophypnorm t = \modtophypnorm t = \Transport t = 1$.
	That is, these three vectorfields are equal, as desired.
	From this fact,
	\eqref{E:TRANSPORTISUNITLENGTHANDTIMELIKE},
	and Def.\,\ref{D:LENGTHOFVARIOUSVECTORFIELDSETC},
	we conclude
	\eqref{E:LENGTHOFMODTOPHYPNORMISUNITY}.
	\eqref{E:SIGMATFIRSTFUNDAGREESWITHSIGMATILDEFIRSTFUND}
	then follows from these results
	and Def.\,\ref{D:FIRSTFUNDAMENTALFORMSANDPROJECTIONS}.
	
	\eqref{E:MODGENISEQUALTOGEN}-\eqref{E:LAPSEMODGENISUNITY}
	then follow as simple consequences of
	Def.\,\ref{D:LENGTHOFVARIOUSVECTORFIELDSETC},
	\eqref{E:GENERATOROFHYPERSURFACE}-\eqref{E:NORMALIZEDAGAINSTTIMEFUNCTIONGENERATOROFHYPERSURFACE},
	and our assumption $\Timefunction = t$.
	
	\eqref{E:GENCOEFFICIENTISUNITY} follows from 
	\eqref{E:EQUIVALENTGENERATOROFHYPERSURFACE},
	\eqref{E:SECONDINGOINGCONDITION},
	and
	\eqref{E:WHENTIMEFUNCTIONISCARTESIANTRANSOPORTCOINCIDESWITHTOPNORMAL}.
	
	\eqref{E:PROJECTEDTRANSPORTVANISHES} then follows easily 
	from Def.\,\ref{D:PROJECTIONOFPONTTOTILDESIGMA},
	\eqref{E:WHENTIMEFUNCTIONISCARTESIANTRANSOPORTCOINCIDESWITHTOPNORMAL},
	\eqref{E:SIGMATFIRSTFUNDAGREESWITHSIGMATILDEFIRSTFUND},
	and the fact that $\topproject \tophypnorm = 0$.
	
	Next, from the proof of Lemma~\ref{L:SOMECONVENIENTIDENTITIES},
	\eqref{E:NORMALISGENERATORINNULLCASE},
	and \eqref{E:WHENTIMEFUNCTIONISCARTESIANTRANSOPORTCOINCIDESWITHTOPNORMAL},
	it follows that 
	$\mbox{\upshape span} \lbrace \uLunit, \Transport \rbrace$
	is the $\gfour$-orthogonal complement of $\mathcal{S}_t$.
	Since $\spherenormal$is $\gfour$-orthogonal to $\mathcal{S}_t$, 
	there exist scalar functions $a_1$ and $a_2$
	such that $\spherenormal = a_1 \uLunit + a_2 \Transport$. Taking the $\gfour$-inner product
	of this identity with respect to $\Transport$ and using \eqref{E:TRANSPORTISUNITLENGTHANDTIMELIKE},
	\eqref{E:EQUIVALENTFUTURENORMALTOHYPERSURFACE},
	and Convention~\ref{C:NULLCASE},
	we find that $a_1 = - a_2$, that is, $\spherenormal = a_1 (\uLunit - \Transport)$.
	Taking the $\gfour$-inner product of this identity with respect to itself and using that $\uLunit$ is $\gfour$-null,
	\eqref{E:SPHERENORMALISUNITLENGTH},
	\eqref{E:TRANSPORTISUNITLENGTHANDTIMELIKE},
	\eqref{E:EQUIVALENTFUTURENORMALTOHYPERSURFACE},
	and Convention~\ref{C:NULLCASE},
	we find that $1 = a_1^2$. Also using \eqref{E:INGOINGCONDITION}, we find that $a_1 = - 1$,
	that is, $\spherenormal = \Transport - \uLunit$.
	Considering also \eqref{E:NORMALISGENERATORINNULLCASE},
	we conclude \eqref{E:NULLCASESPHERENORMALCOEFFICIENTISUNITY},
	\eqref{E:NULLCASETRANSPORTINTERMSOFGENANDSPHERENORMAL},
	and the last equality in \eqref{E:LUNITINNULLCASEWHENTIMEFUNCTIONISCARTESIAN}.

	\eqref{E:GINNERPRODUCTOFLUNITANDTRANSPORT} follows easily from
	\eqref{E:LUNITINNULLCASEWHENTIMEFUNCTIONISCARTESIAN},
	\eqref{E:TRANSPORTISUNITLENGTHANDTIMELIKE},
	\eqref{E:TRANSPORTONEFORMIDENTITY},
	and the fact that $\spherenormal$ is $\Sigma_t$-tangent since $\Timefunction = t$
	(in particular, $\gfour(\Transport,\spherenormal) = 0$ in the present context).
	Considering also \eqref{E:SPHERENORMALISUNITLENGTH}, we find that $\gfour(\Lunit,\Lunit) = 0$,
	as desired.
	
	\eqref{E:GNULLCASEANGULARVECTORFIELDVANISHES} follows from the first equality in \eqref{E:TRANSPORTDECOMPOSITION},
	\eqref{E:GENCOEFFICIENTISUNITY},
	\eqref{E:NULLCASESPHERENORMALCOEFFICIENTISUNITY},
	and the first equality in \eqref{E:NULLCASETRANSPORTINTERMSOFGENANDSPHERENORMAL}.
	
	The prove \eqref{E:INVERSEACOUSTICALMETRICNULLCASE},
	we first note that
	Convention~\ref{C:NULLCASE}
	and the results from earlier in the proof
	imply that 
	$\gfour(\Lunit,\Lunit) = \gfour(\uLunit,\uLunit) = 0$,
	and
	$\gfour(\Lunit,\uLunit) 
	=
	\gfour(\Transport + \spherenormal,\Transport - \spherenormal) 
	=
	- 2$,
	that
	$\lbrace \uLunit, \Lunit \rbrace$
	spans the $\gfour$-orthogonal complement of $\mathcal{S}_t$,
	and that for any local $\gfour$-orthogonal frame $\lbrace e_{(1)}, e_{(2)} \rbrace$ on $\mathcal{S}_t$,
	the set $\lbrace \uLunit, \Lunit, e_{(1)}, e_{(2)} \rbrace$ spans the tangent space
	of $\mathcal{M}$ at each point where it is defined.
	Using these facts and 
	\eqref{E:GSPHEREAGREESWITHGONSTTANGENTVECTORFIELDS}-\eqref{E:GSPHEREVANISHESONSPANOFTOPHYPNORMANDSPHERENORMAL},
	the identity \eqref{E:INVERSEACOUSTICALMETRICNULLCASE} is straightforward to verify by
	contracting each side of it against pairs of elements of the frame $\lbrace \uLunit, \Lunit, e_{(1)}, e_{(2)} \rbrace$.

\end{proof}

\subsection{The geometric energy method for wave equations}
\label{SS:VECTORFIELDMULTIPLIERMETHOD}
To derive energy-flux identities for solutions to the covariant wave equations 
\eqref{E:VELOCITYWAVEEQUATION}-\eqref{E:ENTROPYWAVEEQUATION}
we will use the standard vectorfield multiplier method,
which we review in this subsection.
To obtain coercive energies and fluxes, we will use the ``quasilinear multiplier''
$\Transport = \partial_t + v^a \partial_a$, which is adapted to the solution.
By \eqref{E:TRANSPORTISUNITLENGTHANDTIMELIKE}, $\Transport$ is always $\gfour$-timelike,
which is the key property that leads to coercive energies and fluxes.

\subsubsection{Energy-momentum tensor, energy current, deformation tensor, and dominant energy condition}
\label{SSS:ENERGYMOMENTUM}
We start by recalling that $\Dfour$ denotes the Levi--Civita connection of $\gfour$ (see Subsect.\,\ref{SS:LEVICIVITACONNECTIONS}).
Let $\varphi$ be a scalar function 
(in practice, $\varphi$ will be a solution to one of the wave equations \eqref{E:VELOCITYWAVEEQUATION}-\eqref{E:ENTROPYWAVEEQUATION}).
We define the energy-momentum tensor associated to $\varphi$
to be the following symmetric type $\binom{0}{2}$ tensorfield:
\begin{align} \label{E:ENMOMENTUMTENSOR}
	\enmomem_{\alpha \beta}[\varphi]
	& := \partial_{\alpha} \varphi \partial_{\beta} \varphi
		- 
		\frac{1}{2} \gfour_{\alpha \beta} (\gfour^{-1})^{\kappa \lambda} \partial_{\kappa} \varphi \partial_{\lambda} \varphi.
\end{align}
Given $\varphi$ and any ``multiplier'' vectorfield $\mathbf{X}$,
we define the corresponding energy current
$\Jenarg{\mathbf{X}}{\alpha}[\varphi]$ to be the following\footnote{We remind the reader that we use the conventions of Subsubsect.\,\ref{SSS:ACOUSTICALMETRIC}
for lowering and raising Greek indices.} 
vectorfield:
\begin{align} \label{E:MULTIPLIERVECTORFIELD}
	\Jenarg{\mathbf{X}}{\alpha}[\varphi]
	& := \enmomem^{\alpha \beta}[\varphi] 
		\mathbf{X}_{\beta}.
\end{align}
We define the deformation tensor of $\mathbf{X}$ to be the following
symmetric type $\binom{0}{2}$ tensorfield:
\begin{align} \label{E:DEFORMATIONTENSOR}
\deformarg{\mathbf{X}}{\alpha}{\beta}
& := \Dfour_{\alpha} \mathbf{X}_{\beta}	
		+
		\Dfour_{\beta} \mathbf{X}_{\alpha}.
\end{align}

The \emph{dominant energy condition} is the following well-known result: $\enmomem_{\alpha \beta}[\varphi] \mathbf{X}^{\alpha} \mathbf{Y}^{\beta}$
is a positive definite quadratic form in $\pmb{\partial} \varphi$ when $\mathbf{X}$ and $\mathbf{Y}$ are both future-directed 
(see Footnote~\ref{FN:FUTUREDIRECTED})
and $\gfour$-timelike,
and $\enmomem_{\alpha \beta}[\varphi] \mathbf{X}^{\alpha} \mathbf{Y}^{\beta}$ is positive semi-definite if $\mathbf{X}$ and $\mathbf{Y}$ are both future-directed, $\mathbf{X}$ is $\gfour$-timelike, and $\mathbf{Y}$ is $\gfour$-null.
These properties are what allow one to construct \emph{coercive} energies and fluxes for wave equation solutions.
For these reasons, we will be particularly interested in the case $\mathbf{X} := \Transport$,
which is future-directed and, by \eqref{E:TRANSPORTISUNITLENGTHANDTIMELIKE}, always $\gfour$-timelike.
In this case, relative to the Cartesian coordinates, we have
$\Transport_{\alpha} = - \updelta_{\alpha}^0$ (see \eqref{E:TRANSPORTONEFORMIDENTITY}),
and it is straightforward to verify the following identity
(where $\Chfour_{\alpha \ \beta}^{\ 0}$ are Christoffel symbols of $\gfour$ relative to the Cartesian coordinates, 
as in Subsect.\,\ref{SS:LEVICIVITACONNECTIONS}):
\begin{align} \label{E:DEFORMATIONTENSOROFTRANSPORTRELATIVETOCARTESIAN}
	\deformarg{\Transport}{\alpha}{\beta}
	= 2 \Chfour_{\alpha \ \beta}^{\ 0}.
\end{align}

A straightforward computation yields the following
identity, which will form the starting point for our energy-flux identities for 
the wave equations \eqref{E:VELOCITYWAVEEQUATION}-\eqref{E:ENTROPYWAVEEQUATION}:
\begin{align} \label{E:DIVERGENCEOFENERGYCURRENT}
	\Dfour_{\kappa} \Jenarg{\mathbf{X}}{\kappa}[\varphi]
	& = (\square_{\gfour} \varphi) \mathbf{X} \varphi
			+
			\frac{1}{2} 
			\enmomem^{\kappa \lambda} \deformarg{\mathbf{X}}{\kappa}{\lambda}.
\end{align}

\subsection{Definitions of the geometric energies and fluxes and energy-flux identities}
\label{SS:DEFINITIONSOFENERGIESANDFLUXES}
In this subsection, we define the geometric energies and fluxes that we will use to analyze solutions
to the equations of Theorem~\ref{T:GEOMETRICWAVETRANSPORTSYSTEM}. We then derive energy-flux identities
for these quantities.

\subsubsection{Definitions of the energies and fluxes}
We now define the energies and fluxes. See Lemma~\ref{L:COERCIVENESS} for quantified statements regarding their coerciveness properties.

\begin{definition}[Energies and fluxes]
\label{D:ENERGIESANDFLUXES}
Assume that the acoustical time function $\Timefunction$ from the beginning of Sect.\,\ref{S:SPACETIMEDOMAINS} is equal to
the Cartesian time function $t$.
If the lateral boundary $\underline{\mathcal{H}}$ is $\gfour$-spacelike, then let
$\widetilde{\Sigma}_t$
and
$\underline{\mathcal{H}}_t$
be the hypersurface portions defined in Subsect.\,\ref{SS:DOMAINANDTIMEFUNCTIONETC},
and let $\Transport$
and
$\hat{\sidehypnorm}$
be, respectively, their future-directed (see Footnote~\ref{FN:FUTUREDIRECTED}) 
unit normals
(see 
\eqref{E:MATERIALVECTORVIELDRELATIVECTORECTANGULAR},
\eqref{E:TRANSPORTISUNITLENGTHANDTIMELIKE}, 
\eqref{E:FUTURENORMALTOHYPERSURFACE},
\eqref{E:UNITHYPERSURFACENORMAL},
and \eqref{E:WHENTIMEFUNCTIONISCARTESIANTRANSOPORTCOINCIDESWITHTOPNORMAL}).
Similarly, if the lateral boundary $\underline{\mathcal{H}}$ is $\gfour$-null,
then let $\uLunit$ be the future-directed null normal normalized by $\uLunit t = 1$
(see \eqref{E:FUTURENORMALTOHYPERSURFACE} and Convention~\ref{C:NULLCASE}).
Let $\varphi$ be a scalar function, let $\Jenarg{\Transport}{\alpha}[\varphi]$
be the energy current defined by \eqref{E:MULTIPLIERVECTORFIELD},
and recall that our geometric volume and forms are defined in Def.\,\ref{D:VOLUMEFORMS}.

For $t \in [0,T]$, we define the following ``wave'' and ``transport'' energies along $\widetilde{\Sigma}_t$:
\begin{align} \label{E:SIGMATENERGYDEF}
	\mathbb{E}_{(Wave)}[\varphi](t)
	& := \int_{\widetilde{\Sigma}_t}
					\left\lbrace
						\gfour_{\alpha \beta} \Jenarg{\Transport}{\alpha}[\varphi] \Transport^{\beta}
						+ 
						\varphi^2
					\right\rbrace
			 \, d \varpi_g,
	&
	\mathbb{E}_{(Transport)}[\varphi](t)
	& := \int_{\widetilde{\Sigma}_t}
				\varphi^2
			 \, d \varpi_g.
	\end{align}
	
	When the lateral boundary $\underline{\mathcal{H}}$
	is $\gfour$-spacelike, we define the following ``wave'' and ``transport'' $\underline{\mathcal{H}}$-fluxes,
	where the scalar function $\lengthofsidehypnorm > 0$ is defined by \eqref{E:LENGTHOFHYPNORM}:
	\begin{subequations}
	\begin{align}
	\mathbb{F}_{(Wave)}[\varphi](t)
	& := \int_{\underline{\mathcal{H}}_t}
					\left\lbrace
						\gfour_{\alpha \beta} \Jenarg{\Transport}{\alpha}[\varphi] \hat{\sidehypnorm}^{\beta}
						+ 
						\frac{1}{\lengthofsidehypnorm}
						\varphi^2
					\right\rbrace
			 \, d \varpi_{\sidefirstfund},
	& 
	\mathbb{F}_{(Transport)}[\varphi](t)
	& := \int_{\underline{\mathcal{H}}_t}
				\frac{1}{\lengthofsidehypnorm}
				\varphi^2
			 \, d \varpi_{\sidefirstfund}.
				\label{E:FLUXDEF} 
	\end{align}
	
	Finally, when the lateral boundary $\underline{\mathcal{H}} := \underline{\mathcal{N}}$
	is $\gfour$-null, we define the following ``wave'' and ``transport'' $\underline{\mathcal{N}}$-fluxes:
	\begin{align}
	\mathbb{F}_{(Wave)}[\varphi](t)
	& := \int_{\underline{\mathcal{N}}_t}
					\left\lbrace
						\gfour_{\alpha \beta} \Jenarg{\Transport}{\alpha}[\varphi]  \uLunit^{\beta}
						+ 
						\varphi^2
					\right\rbrace
			 \, d \varpi_{\gsphere} \, dt',
	&
	\mathbb{F}_{(Transport)}[\varphi](t)
	& := \int_{\underline{\mathcal{N}}_t}
				\varphi^2
			 \, d \varpi_{\gsphere} \, dt'.
	\label{E:NULLFLUXDEF}
	\end{align}
	\end{subequations}
\end{definition}

\subsubsection{Energy-flux identities}
\label{SSS:ENERGYFLUXIDENTITIES}
We now derive energy-flux identities for the quantities from Def.\,\ref{D:ENERGIESANDFLUXES}.

\begin{proposition}[Energy-flux identities]
	\label{P:ENERGYFLUXID}
	Under the assumptions stated in Def.\,\ref{D:ENERGIESANDFLUXES},
	the following ``wave'' energy-flux identity holds for $t \in [0,T]$,
	where the volume forms are defined in Def.\,\ref{D:VOLUMEFORMS}
	and $\deformarg{\Transport}{\alpha}{\beta}$ is defined by \eqref{E:DEFORMATIONTENSOR}:
	\begin{align}
		\mathbb{E}_{(Wave)}[\varphi](t)
		+
		\mathbb{F}_{(Wave)}[\varphi](t)
		& = \mathbb{E}_{(Wave)}[\varphi](0)
				-
				\int_{\mathcal{M}_t}
					(\square_{\gfour} \varphi) \Transport \varphi
				\, d \varpi_{\gfour}
				+
				2
				\int_{\mathcal{M}_t}
					(\Transport \varphi) \varphi 
				\, d \varpi_{\gfour}
					\label{E:WAVEENERGYIDENTITY} \\
		& \ \
				-
				\frac{1}{2}
				\int_{\mathcal{M}_t}
					\enmomem^{\alpha \beta}[\varphi] \deformarg{\Transport}{\alpha}{\beta}
				\, d \varpi_{\gfour}
				+
				\frac{1}{2}
				\int_{\mathcal{M}_t}
					\varphi^2 \deformmixedarg{\Transport}{\alpha}{\alpha}
				\, d \varpi_{\gfour}.
				\notag
	\end{align}
	
	Moreover, the following ``transport'' energy-flux identity holds:
		\begin{align}
		\mathbb{E}_{(Transport)}[\varphi](t)
		+
		\mathbb{F}_{(Transport)}[\varphi](t)
		& = \mathbb{E}_{(Transport)}[\varphi](0)
				+
				2
				\int_{\mathcal{M}_t}
					(\Transport \varphi) \varphi
				\, d \varpi_{\gfour}
				+
				\frac{1}{2}
				\int_{\mathcal{M}_t}
					\varphi^2 \deformmixedarg{\Transport}{\alpha}{\alpha}
				\, d \varpi_{\gfour}. 
			\label{E:TRANSPORTENERGYIDENTITY}
	\end{align}
	
\end{proposition}

\begin{proof}
		We start by reminding the reader that in this section, the acoustical time function $\Timefunction$ is equal to the Cartesian time function $t$.
		To prove \eqref{E:WAVEENERGYIDENTITY},
		we consider the energy current 
		$\Jenwithlowerarg{\Transport}{\alpha} 
		:= 
		\Jenarg{\Transport}{\alpha}[\varphi]
		-
		\varphi^2
		\Transport^{\alpha}
		$, 
		(where $\Jenarg{\Transport}{\alpha}[\varphi]$ is defined by \eqref{E:MULTIPLIERVECTORFIELD} with $\Transport$ in the role of $\mathbf{X}$).
		We integrate $\Dfour_{\alpha} \Jenwithlowerarg{\Transport}{\alpha}$
		with respect to $d \varpi_{\gfour}$ (see Def.\,\ref{D:VOLUMEFORMS} for the definitions of the volume and area forms)
		over the spacetime region $\mathcal{M}_t$ (see \eqref{E:TIMETRUNCATEDDOMAIN}) with respect to $d \varpi_{\gfour}$
		and apply the divergence theorem.
		The relevant unit normals to the boundary surfaces 
		$\widetilde{\Sigma}_0$,
		$\widetilde{\Sigma}_t$,
		and $\underline{\mathcal{H}}_t$
		($\underline{\mathcal{N}}_t$ in the null case)
		are stated in Def.\,\ref{D:ENERGIESANDFLUXES}.
		The relevant volume form on $\widetilde{\Sigma}_0$ and $\widetilde{\Sigma}_t$
		is $d \varpi_g$,
		while when $\underline{\mathcal{H}}_t$ is $\gfour$-spacelike, the relevant volume form on
		$\underline{\mathcal{H}}_t$ is $d \varpi_{\sidefirstfund}$; we will comment on the null case
		$\underline{\mathcal{H}}_t = \underline{\mathcal{N}}_t$ later in the proof.
		With the help of \eqref{E:TRANSPORTISUNITLENGTHANDTIMELIKE}, 
		\eqref{E:EQUIVALENTFUTURENORMALTOHYPERSURFACE},
		and
		\eqref{E:INNERPRODUCTOFTRANPORTANDFUTUREUNITNORMALTOHYPERSURFACE},
		we see that 
		the boundary integrals that arise in the divergence theorem
		are precisely the wave energies and fluxes
		from Def.\,\ref{D:ENERGIESANDFLUXES}.
		Moreover, we re-express the ``bulk term''
		$
			\int_{\mathcal{M}_t}
					\Dfour_{\alpha} \Jenwithlowerarg{\Transport}{\alpha}
				\, d \varpi_{\gfour}
		$
		using the identity 
		$\Dfour_{\alpha} \Jenwithlowerarg{\Transport}{\alpha} = 
		(\square_{\gfour} \varphi) \Transport \varphi
			+
			\frac{1}{2} 
			\enmomem^{\kappa \lambda} \deformarg{\mathbf{\Transport}}{\kappa}{\lambda}
			-
			2 (\Transport \varphi) \varphi 
			-
			\frac{1}{2}
			\varphi^2 \deformmixedarg{\Transport}{\alpha}{\alpha}
		$,
		which follows from 
		\eqref{E:DEFORMATIONTENSOR},
		\eqref{E:DIVERGENCEOFENERGYCURRENT},
		and straightforward computations.
		We clarify that when $\underline{\mathcal{H}}$ is $\gfour$-spacelike,
		$\Transport|_{\widetilde{\Sigma}_0}$ points inwards to $\mathcal{M}_t$,
		$\Transport|_{\widetilde{\Sigma}_t}$ points outwards to $\mathcal{M}_t$,
		and $\hat{\sidehypnorm}$ points outwards to $\mathcal{M}_t$
		(see \eqref{E:UNITHYPERSURFACENORMAL} and Subsect.\,\ref{SS:ASSUMPTIONSONSPACETIMEREGION}),
		and that due to the Lorentzian nature of $\gfour$,
		in the divergence theorem, the bulk integral
		$
			\int_{\mathcal{M}_t}
					\Dfour_{\alpha} \Jenwithlowerarg{\Transport}{\alpha}
				\, d \varpi_{\gfour}
		$
		is equal to boundary integrals involving \emph{inward} pointing normals.
		This yields \eqref{E:WAVEENERGYIDENTITY} when $\underline{\mathcal{H}}$
		is $\gfour$-spacelike, and in particular explains the signs
		in \eqref{E:WAVEENERGYIDENTITY}.
		Next, we note that the identity \eqref{E:WAVEENERGYIDENTITY} in the $\gfour$-null case 
		can be obtained as an appropriate limit of the $\gfour$-spacelike case.
		Specifically, one can use the relations
		$
		\hat{\sidehypnorm}^{\alpha}
		d \varpi_{\sidefirstfund}
		=
		\frac{\lengthofgen}{\lengthofsidehypnorm}
		\sidehypnorm^{\alpha}
		d \varpi_{\gsphere} dt'
		$
		(see \eqref{E:UNITHYPERSURFACENORMAL}, 
		\eqref{E:SIDEHYPERSURFACEVOLUMEFORMEXPRESSIONWITHRESPECTTOTIMEFUNCTIONSPACELIKECASE}, and \eqref{E:LENGTHOFMODGENISLENGTHOFGEN})
		and the fact that if we take a limit as $\underline{\mathcal{H}}$ becomes
		$\gfour$-null (i.e., as $\lengthofsidehypnorm \downarrow 0$), 
		then with $\uLunit$ denoting the $\gfour$-normal (i.e., the null generator) of the limiting null hypersurface, 
		we have $\sidehypnorm \rightarrow \uLunit$
		(since $\sidehypnorm t = \uLunit t = 1$ by \eqref{E:FUTURENORMALTOHYPERSURFACE} and Convention~\ref{C:NULLCASE})
		and $\frac{\lengthofgen}{\lengthofsidehypnorm} \rightarrow 1$,
		where the latter limit follows from \eqref{E:RATIOLENGTHOFSIDEHYPNORMLENGTHOFGEN},
		the identity \eqref{E:LENGTHOFMODTOPHYPNORMISUNITY} for $\lengthoftophypnorm$,
		and \eqref{E:GENCOEFFICIENTISUNITY}.
		
		The identity \eqref{E:TRANSPORTENERGYIDENTITY}
		can be proved in a similar but simpler fashion by applying the divergence theorem
		to the vectorfield 
		$
		-
		\varphi^2
		\Transport^{\alpha}
		$
		on the spacetime region $\mathcal{M}_t$; we omit the details.
	\end{proof}

\subsection{$L^2$-type Controlling quantities}
\label{SS:CONTROLLINGQUANTITIES}
In this section, we combine the previously derived integral identities and use them
to derive localized a priori estimates for solutions. Compared to standard results,
our estimates yield a gain of one derivative for the specific vorticity and entropy
(assuming that the initial data enjoy the same gain),
i.e., we exhibit application \textbf{I} described in Subsect.\,\ref{SS:APPLICATIONS}.
The main result of this section is Theorem~\ref{T:LOCALIZEDAPRIORIESTIMATES}.
The theorem is of particular interest in the case that the lateral boundary 
$\underline{\mathcal{N}}$ is $\gfour$-null;
as we discussed in Subsect.\,\ref{SS:APPLICATIONS}, 
the null case is important for applications to shock waves, 
and to handle the degeneracy of wave energies along $\underline{\mathcal{H}} = \underline{\mathcal{N}}$,
we must exploit the special structures in the lateral boundary integrals of Prop.\,\ref{P:STRUCTUREOFERRORINTEGRALS},
which we derived in Theorem~\ref{T:STRUCTUREOFERRORTERMS}.
Specifically, we exploit that the integrands involve derivatives only in directions that are tangent to $\underline{\mathcal{N}}$.

We state our a priori estimates in terms of the $L^2$-type controlling quantities provided by the following definition.

\begin{definition}[The controlling quantities]
	\label{D:CONTROLLINGQUANTITIES}
		Let $\mathcal{M} = \mathcal{M}_T$ be a spacetime region satisfying the conditions stated in 
		Subsects.\,\ref{SS:DOMAINANDTIMEFUNCTIONETC} and \ref{SS:ASSUMPTIONSONSPACETIMEREGION} for some $T > 0$;
		see Fig.\,\ref{F:SPACETIMEDOMAIN}.
		In particular, assume that the lateral boundary $\underline{\mathcal{H}} = \underline{\mathcal{H}}_T$ is $\gfour$-spacelike
		or is $\gfour$-null (in the null case, $\underline{\mathcal{H}} := \underline{\mathcal{N}} = \underline{\mathcal{N}}_T$).
		Assume further that the acoustical time function $\Timefunction$ from the beginning of Sect.\,\ref{S:SPACETIMEDOMAINS} is equal to
		the Cartesian time function $t$.
		We define the following controlling quantities,
		where the volume forms are defined in Def.\,\ref{D:VOLUMEFORMS},
		$\uposinnerproduct > 0$ is as in \eqref{E:INGOINGCONDITION},
		the quadratic form $\mathscr{Q}$ on RHS~\eqref{E:SPACETIMECONTROLLINGQUANTITY} is as in 
		Def.\,\ref{D:QUADRATICFORMSFORCONTROLLINGFIRSTDERIVATIVESOFSPECIFICVORTICITYANDENTROPYGRADIENT}
		and Lemma~\ref{L:POSITIVITYPROPERTIESOFVARIOUSQUADRATICFORMS},
		and the energies $\mathbb{E}$ and fluxes $\mathbb{F}$ 
		on RHS~\eqref{E:ENERGYFLUXCONTROLLINGQUANTITY} are as in Def.\,\ref{D:ENERGIESANDFLUXES}:
	\begin{subequations}
	\begin{align}
		\spacetimeen(t)
		& :=
		\int_{\mathcal{M}_t}	
			\mathscr{Q}(\pmb{\partial} \vortrenormalized,\pmb{\partial} \vortrenormalized)
		\, d \varpi_{\gfour}
		+
			\int_{\mathcal{M}_t}	
			\mathscr{Q}(\pmb{\partial} \GradEnt,\pmb{\partial} \GradEnt)
		\, d \varpi_{\gfour}
		+
		\int_{\mathcal{S}_t}
			\uposinnerproduct
			|\angvortrenormalized|_{\gsphere}^2
		\, d \varpi_{\gsphere} 
		+
		\int_{\mathcal{S}_t}
			\uposinnerproduct
			|\angGradEnt|_{\gsphere}^2
		\, d \varpi_{\gsphere}, 
					\label{E:SPACETIMECONTROLLINGQUANTITY} \\
		\toten(t)
		& := 
				\sum_{\varphi \in \lbrace \LogDensity,v^i,\Ent \rbrace_{i=1,2,3}} \mathbb{E}_{(Wave)}[\varphi](t)
					+
					\sum_{\varphi \in \lbrace \LogDensity,v^i,\Ent \rbrace_{i=1,2,3}} \mathbb{F}_{(Wave)}[\varphi](t)
					\label{E:ENERGYFLUXCONTROLLINGQUANTITY} \\
		& \ \
					+
					\sum_{\varphi \in \lbrace \vortrenormalized^i,\GradEnt^i,\VortVort^i,\DivGradEnt \rbrace_{i=1,2,3}} \mathbb{E}_{(Transport)}[\varphi](t)
					+
					\sum_{\varphi \in \lbrace \vortrenormalized^i,\GradEnt^i,\VortVort^i,\DivGradEnt \rbrace_{i=1,2,3}} \mathbb{F}_{(Transport)}[\varphi](t).
					\notag
\end{align}
\end{subequations}
\end{definition}

\subsection{Combining the integral identities}
In the next proposition, we set up the derivation of the a priori estimates by
combining the integral identities of
Theorem~\ref{T:MAINREMARKABLESPACETIMEINTEGRALIDENTITY} 
and Prop.\,\ref{P:ENERGYFLUXID} 
and restating them in terms of the controlling quantities
of Def.\,\ref{D:CONTROLLINGQUANTITIES}.

\begin{proposition}[Combining the integral identities]
	\label{P:COMBINEDINTEGRALIDENTITIES}
	Let $\mathcal{M} = \mathcal{M}_T$ be a spacetime region satisfying the conditions stated in 
	Subsects.\,\ref{SS:DOMAINANDTIMEFUNCTIONETC} and \ref{SS:ASSUMPTIONSONSPACETIMEREGION} for some $T > 0$;
	see Fig.\,\ref{F:SPACETIMEDOMAIN}.
	Assume further that the acoustical time function $\Timefunction$ from the beginning of Sect.\,\ref{S:SPACETIMEDOMAINS} is equal to
	the Cartesian time function $t$ (see Footnote~\ref{FN:GENERALIZETOOTHERTIMEFUNCTIONS}).
	Then for smooth solutions (see Remark~\ref{R:SMOOTHNESSNOTNEEDED})
	to the compressible Euler equations \eqref{E:TRANSPORTDENSRENORMALIZEDRELATIVECTORECTANGULAR}-\eqref{E:ENTROPYTRANSPORT}
	on $\mathcal{M}_T$,
	the controlling quantities $\spacetimeen(t)$ and $\toten(t)$
	from Def.\,\ref{D:CONTROLLINGQUANTITIES} verify the following
	identities for $t \in [0,T]$, 
	where the volume forms are defined in Def.\,\ref{D:VOLUMEFORMS},
	and the terms $\mathfrak{A}^{(\vortrenormalized)}, \cdots, \underline{\mathfrak{H}}_{(2)}[\GradEnt]$
	on RHS~\eqref{E:SPACETIMECOMBINEDINTEGRALID}
	are defined in \eqref{E:SPECIFICVORTICITYSPACETIMEERRORTERMANTISYMMETRIC}-\eqref{E:MAINTHMENTROPYGRADIENTMAINLATERALERRORINTEGRAND}:
	\begin{subequations}
	\begin{align}
		\spacetimeen(t)
		 & = 
			\spacetimeen(0)
				\label{E:SPACETIMECOMBINEDINTEGRALID} \\
		& \ \
			+	
			\int_{\mathcal{M}_t}	
			\left\lbrace
				\frac{1}{2}|\mathfrak{A}^{(\vortrenormalized)}|_{\topfirstfund}^2
				+
				|\mathfrak{B}_{(\vortrenormalized)}|_{\topfirstfund}^2
				+
				\mathfrak{C}^{(\vortrenormalized)}
				+
				\mathfrak{D}^{(\vortrenormalized)}
				+
				\mathfrak{J}_{(Coeff)}[\vortrenormalized,\pmb{\partial} \vortrenormalized]
			\right\rbrace
	\, d \varpi_{\gfour}
				\notag \\
		& \ \
			+
		\int_{\mathcal{M}_t}	
			\left\lbrace
				\frac{1}{2}|\mathfrak{A}^{(\GradEnt)}|_{\topfirstfund}^2
				+
				|\mathfrak{B}_{(\GradEnt)}|_{\topfirstfund}^2
				+
				\mathfrak{C}^{(\GradEnt)}
				+
				\mathfrak{D}^{(\GradEnt)}
				+
				\mathfrak{J}_{(Coeff)}[\GradEnt,\pmb{\partial} \GradEnt]
			\right\rbrace
	\, d \varpi_{\gfour}
	\notag
				\\
		& \ \
			+
			\int_{\underline{\mathcal{H}}_t}
			\left\lbrace
				\underline{\mathfrak{H}}[\vortrenormalized]
				+
				\underline{\mathfrak{H}}_{(1)}[\vortrenormalized]
			\right\rbrace
		\, d \varpi_{\gsphere} d t'
		+
		\int_{\underline{\mathcal{H}}_t}
			\left\lbrace
				\underline{\mathfrak{H}}[\GradEnt]
				+
				\underline{\mathfrak{H}}_{(2)}[\GradEnt]
			\right\rbrace
		\, d \varpi_{\gsphere} d t',
			\notag
				\\
		\toten(t)
		 & = \toten(0)
				-
				\sum_{\varphi \in \lbrace \LogDensity,v^i,\Ent \rbrace_{i=1,2,3}}
				\int_{\mathcal{M}_t}
					(\square_{\gfour} \varphi) \Transport \varphi
				\, d \varpi_{\gfour}
				-
				\frac{1}{2}
				\sum_{\varphi \in \lbrace \LogDensity,v^i,\Ent \rbrace_{i=1,2,3}}
				\int_{\mathcal{M}_t}
					\enmomem^{\alpha \beta}[\varphi] \deformarg{\Transport}{\alpha}{\beta}
				\, d \varpi_{\gfour}
				\label{E:ENERGYFLUXCOMBINEDINTEGRALID} \\
		& \ \
				+
				2
				\sum_{\varphi \in \lbrace \vortrenormalized^i,\GradEnt^i,\VortVort^i,\DivGradEnt \rbrace_{i=1,2,3}}
				\int_{\mathcal{M}_t}
					(\Transport \varphi) \varphi
				\, d \varpi_{\gfour}
				+
				\frac{1}{2}
				\sum_{\varphi \in \lbrace \GradEnt^i,\VortVort^i,\DivGradEnt \rbrace_{i=1,2,3}}
				\int_{\mathcal{M}_t}
					\varphi^2 \deformmixedarg{\Transport}{\alpha}{\alpha}
				\, d \varpi_{\gfour}.
				\notag
	\end{align}
	\end{subequations}
	We also note that \eqref{E:SIDEHYPERSURFACEVOLUMEFORMEXPRESSIONWITHRESPECTTOTIMEFUNCTIONSPACELIKECASE} and \eqref{E:LENGTHOFMODGENISLENGTHOFGEN}
	imply that when $\underline{\mathcal{H}}_t$ is $\gfour$-spacelike, 
	we have
	$d \varpi_{\gsphere} dt' = \lengthofgen^{-1} \, d \varpi_{\sidefirstfund}$,
	where $\lengthofgen > 0$ is the scalar function defined in \eqref{E:LENGTHOFSIDEGEN}.
\end{proposition}	

\begin{proof}
	\eqref{E:SPACETIMECOMBINEDINTEGRALID} follows from definition \eqref{E:SPACETIMECONTROLLINGQUANTITY},
	Theorem~\ref{T:MAINREMARKABLESPACETIMEINTEGRALIDENTITY} with $\weight := 1$,
	and the last equality in \eqref{E:LENGTHOFMODTOPHYPNORMISUNITY}.
	\eqref{E:ENERGYFLUXCOMBINEDINTEGRALID} follows from definition \eqref{E:ENERGYFLUXCONTROLLINGQUANTITY} 
	and Prop.\,\ref{P:ENERGYFLUXID}.
\end{proof}

\subsection{Notation regarding constants and the norm $\| \cdot \|_{C(\mathcal{R})}$}
\label{SS:NOTATIONFORCONSTANTS}
In the rest of Sect.\,\ref{S:APRIORI}, $C > 0$ denotes a uniform constant that is free to vary from line to line.
$C$ is allowed to depend on the set $\mathfrak{K}$ of fluid variable space
featured below in equation \eqref{E:KEYCONTAINMENTOFSOLUTIONINCOMPACTSUBSETOFSTATESPACE}.

For given quantities $A,B \geq 0$,
we write $A \lesssim B$ to mean that there exists a $C > 0$ such that $A \leq C B$.
We write $A \approx B$ to mean that $A \lesssim B$ and $B \lesssim A$.

Moreover, if $\varphi$ is a continuous scalar-valued function and $\mathcal{R} \subset \mathcal{M}_T$,
is any subset
then 
\begin{align}
\| \varphi \|_{C(\mathcal{R})}
:= 
\sup_{q \in \mathcal{R}} |\varphi(q)|.
\end{align}

\subsection{Assumptions on the solution and coerciveness of the controlling quantities}
\label{SS:ASSUMPTIONSONTHESOLUTION}
In this subsection, we state some standard $C^1$-type boundedness assumptions on the solution that we 
will use in our derivation of a priori estimates, i.e., in our proof of Theorem~\ref{T:LOCALIZEDAPRIORIESTIMATES}.
Moreover, in Lemma~\ref{L:COERCIVENESS}, we use the assumptions
to quantify the coerciveness of the energies and fluxes from Def.\,\ref{D:ENERGIESANDFLUXES}.

\subsubsection{Assumptions on the solution}
\label{SSS:ASSUMPTIONSONTHESOLUTION}
The following definition describes the subset of solution space
on which the compressible Euler equations  
(specifically, the equations of Theorem~\ref{T:GEOMETRICWAVETRANSPORTSYSTEM})
are hyperbolic in a non-degenerate sense.

\begin{definition}[Regime of hyperbolicity] \label{D:REGIMEOFHYPERBOLICITY}
	We define $\mathscr{H}$ as follows:
	\begin{align} \label{E:REGIMEOFHYPERBOLICITY}
		\mathscr{H}
		& := 
		\left\lbrace
			(\LogDensity,\Ent,v,\vortrenormalized,\GradEnt)
			\in
			\mathbb{R} \times \mathbb{R} \times \mathbb{R}^3 \times \mathbb{R}^3 \times \mathbb{R}^3
			\ | \
			0 < \Speed(\LogDensity,\Ent) < \infty
		\right\rbrace.
	\end{align}
\end{definition}

We now state our assumptions on the solution.

\begin{remark}[The pointwise norms $|\cdot|_{\euct}$ and $|\cdot|_{\euc}$]
	We refer to Subsubsect.\,\ref{SSS:GRADIENTSANDPOINTWISENORMS}
	for the definitions of the pointwise 
	Euclidean norms
	$|\cdot|_{\euct}$
	and
	$|\cdot|_{\euc}$.
\end{remark}

\begin{center}
	\underline{\large \textbf{Assumptions on the solution}}
\end{center}

\begin{enumerate}
	\item We assume that for some $T > 0$,
		$(\LogDensity,\Ent,v)$
		is a smooth solution 
		(see Remark~\ref{R:SMOOTHNESSNOTNEEDED})
		to the compressible Euler equations \eqref{E:TRANSPORTDENSRENORMALIZEDRELATIVECTORECTANGULAR}-\eqref{E:ENTROPYTRANSPORT}
		(and thus $(\LogDensity,\Ent,v,\vortrenormalized,\GradEnt,\VortVort,\DivGradEnt)$ is a solution 
		to the equations of Theorem~\ref{T:GEOMETRICWAVETRANSPORTSYSTEM})
		on a compact subset $\mathcal{M} = \mathcal{M}_T$ of spacetime
		satisfying the conditions stated in 
		Subsects.\,\ref{SS:DOMAINANDTIMEFUNCTIONETC} and \ref{SS:ASSUMPTIONSONSPACETIMEREGION}.
		We also assume that the acoustical time function $\Timefunction$ from the beginning of Sect.\,\ref{S:SPACETIMEDOMAINS} is equal to
		the Cartesian time function $t$ (see Footnote~\ref{FN:GENERALIZETOOTHERTIMEFUNCTIONS}).
	\item In particular, as is stated in \eqref{E:SPLITTINGOFBOUNDARYOFREGION},
		we assume that the boundary of $\mathcal{M}_T$ 
		is the union of a flat portion consisting of a compact subset of $\Sigma_T$ (denoted by $\widetilde{\Sigma}_T$),
		of a flat portion consisting of a compact subset of $\Sigma_0$ (denoted by $\widetilde{\Sigma}_0$),
		and of lateral boundary consisting of a $\gfour$-spacelike or $\gfour$-null hypersurface 
		$\underline{\mathcal{H}}_T$,
		where we use the alternate notation $\underline{\mathcal{H}}_T = \underline{\mathcal{N}}_T$
		in the null case.
	\item Let $\mathscr{H}$ be as in \eqref{E:REGIMEOFHYPERBOLICITY}.
		We assume that there is a compact subset $\mathfrak{K}$ of $\mathscr{H}$ 
		such that 
		\begin{align} \label{E:KEYCONTAINMENTOFSOLUTIONINCOMPACTSUBSETOFSTATESPACE}
			(\LogDensity,\Ent,v,\vortrenormalized,\GradEnt)(\mathcal{M}_T) \subset \mathfrak{K}.
		\end{align}
		Note that the assumption \eqref{E:KEYCONTAINMENTOFSOLUTIONINCOMPACTSUBSETOFSTATESPACE} 
		implies a uniform $L^{\infty}(\mathcal{M}_T)$ bound for
		$|(\LogDensity,\Ent,v,\vortrenormalized,\GradEnt)|_{\euct}$, 
		a fact which we will silently use throughout the rest of Sect.\,\ref{S:APRIORI}.
	\item Under the notation of Subsubsects.\,\ref{SSS:GRADIENTS}-\ref{SSS:GRADIENTSANDPOINTWISENORMS},
		we assume that the constant $A = A(\mathcal{M}_T) > 0$ is such that
		\begin{align} \label{E:ASSUMEDC1BOUND}
			\| |(\LogDensity,\Ent,v)|_{\euct} \|_{C(\mathcal{M}_T)}
			+
			\| |\pmb{\partial} (\LogDensity,\Ent,v)|_{\euc} \|_{C(\mathcal{M}_T)}
			& \leq A.
		\end{align}
	\item	In the case that $\underline{\mathcal{H}}_T$ is $\gfour$-spacelike, 
		let $\gen$ and $\spherenormal$ be the vectorfields from
		Def.\,\ref{D:HYPNORMANDSPHEREFORMDEFS},
		and let $\uposinnerproduct > 0$ be the scalar function defined in \eqref{E:INGOINGCONDITION}.
		Under the notation of Subsubsects.\,\ref{SSS:GRADIENTS}-\ref{SSS:GRADIENTSANDPOINTWISENORMS} and \ref{SSS:MORENOTATION},
		we assume that the constant $B = B(\underline{\mathcal{H}}_T) > 0$ is such that
		\begin{align} \label{E:VECTORFIELDSTENSORIALC1BOUND}
			\| |(\gen,\spherenormal)|_{\euc} \|_{C(\underline{\mathcal{H}}_T)}
			+
			\left\| \frac{1}{\uposinnerproduct} \right\|_{C(\underline{\mathcal{H}}_T)}
			+
			\| |(\sidepartial \vec{\gen},\sidepartial \vec{\spherenormal})|_{\euc} \|_{C(\underline{\mathcal{H}}_T)}
			& \leq B.
		\end{align}
		\item In the case that $\underline{\mathcal{N}}_T$ is $\gfour$-null, 
		let $\spherenormal$ be the vectorfield from
		Def.\,\ref{D:HYPNORMANDSPHEREFORMDEFS},
		and let $\uLunit$ be as in 
		Convention~\ref{C:NULLCASE},
		Def.\,\ref{D:LENGTHOFVARIOUSVECTORFIELDSETC}, 
		and \eqref{E:NORMALISGENERATORINNULLCASE}
		(and note that Prop.\,\ref{P:VARIOUSIDENTITIESWHENTIMEFUNCTIONISCARTESIAN} 
		implies that $\utang = 0$ and $\uposinnerproduct = 1$ in the present context).
		Under the notation of Subsubsects.\,\ref{SSS:GRADIENTS}-\ref{SSS:GRADIENTSANDPOINTWISENORMS} and \ref{SSS:MORENOTATION},
		we assume that the constant $B = B(\underline{\mathcal{N}}_T) > 0$ is such that
		\begin{align} \label{E:VECTORFIELDSTENSORIALC1BOUNDNULLCASE}
			\| |\uLunit|_{\euc} \|_{C(\underline{\mathcal{N}}_T)}
			+
			\| |(\uLunit \vec{\uLunit},\angpartial \vec{\uLunit})|_{\euc} \|_{C(\underline{\mathcal{N}}_T)}
			& \leq B.
		\end{align}
\end{enumerate}

\begin{remark}[Additional assumptions are needed to control the solution's higher-order derivatives, and the sub-optimality of the assumptions]
	\label{R:ADDITIONALASSUMPTIONSFORHIGERDERIVATIVES}
	The assumptions we have stated above are sufficient for deriving a priori energy estimates
	for solutions to the equations of Theorem~\ref{T:GEOMETRICWAVETRANSPORTSYSTEM}.
	To obtain $L^2$-type energy estimates for the solutions' 
	higher-order derivatives, one would need additional norm-boundedness-type 
	assumptions on the derivatives of some of the solution variables.
	For example, to control the higher-order derivatives of some of the derivative-quadratic
	terms on RHS~\eqref{E:EVOLUTIONEQUATIONFLATCURLRENORMALIZEDVORTICITY},
	one could supplement the assumed bound \eqref{E:ASSUMEDC1BOUND}
	with an assumed bound for
	$\| |\pmb{\partial} (\vortrenormalized,\GradEnt)|_{\euc} \|_{C(\mathcal{M}_T)}$.
	These assumptions are far from optimal;
	see \cites{mDcLgMjS2019,qW2019} for recent results on low regularity solutions.
\end{remark}

\subsubsection{Coerciveness of the energies and fluxes}
\label{SSS:COERCIVENESSOFWAVEENERGIESANDFLUXES}
In this subsubsection, we use the assumption $\Timefunction \equiv t$
and the assumptions of Subsubsect.\,\ref{SSS:ASSUMPTIONSONTHESOLUTION}
to exhibit the coerciveness properties of the energies and fluxes from Def.\,\ref{D:ENERGIESANDFLUXES}. 
We highlight the coerciveness result \eqref{E:NULLCASEWAVEFLUXESSEMICOERCIVE}
for the wave fluxes, which shows that they degenerate along $\gfour$-null hypersurfaces,
controlling tangential but \emph{not transversal} derivatives.

In the rest of Sect.\,\ref{S:APRIORI}, if $\varphi$ is a scalar function and $\mathcal{R} \subset \mathcal{M}_T$, 
then $\| \varphi \|_{L^2(\mathcal{R})}$
denotes the $L^2$ norm of $\varphi$ over $\mathcal{R}$,
where the volume forms used in computing the $L^2$ norms are the ones
from Def.\,\ref{D:VOLUMEFORMS}.
For example, 
$\| \varphi \|_{L^2(\mathcal{M}_T)}^2 
:=
\int_{\mathcal{M}_T}
	\varphi^2
\, d \varpi_{\gfour} 
=
\int_0^T
\int_{\widetilde{\Sigma}_{t'}}
	\varphi^2
\, d \varpi_g
\, dt' 
$,
where the second equality follows from \eqref{E:SPACTIMEVOLUMEFORMANDSIGMATVOLUMEFORMRELATIVETOCARTESIANCOORDINATES}.
Moreover, for $0 \leq t \leq T$,
$\| \varphi \|_{L^2(\widetilde{\Sigma}_t)}^2 
:=
\int_{\widetilde{\Sigma}_t}
	\varphi^2
\, d \varpi_g
$
and (recall that $\underline{\mathcal{N}}_t = \cup_{t' \in [0,t]} \mathcal{S}_{t'}$)
$\| \varphi \|_{L^2(\underline{\mathcal{N}}_t)}^2 
:=
\int_0^t
\int_{\mathcal{S}_{t'}}
	\varphi^2
\, d \varpi_{\gsphere} 
dt'
$. 
If $\vec{\varphi}$ is a tensorfield or an array of scalar-valued functions,
then $\| \vec{\varphi} \|_{L^2(\mathcal{R})}^2$
is defined to be the sum of the squares of the $L^2$ norms of the
Cartesian components of the elements of $\vec{\varphi}$.
For example, if $\varphi$ is a scalar function, then
$\| \pmb{\partial} \varphi \|_{L^2(\underline{\mathcal{H}}_t)}^2
:=
\int_{\underline{\mathcal{H}}_t}
	|\pmb{\partial} \varphi|_{\euc}^2
\, d \varpi_{\sidefirstfund} 
$,
where $|\pmb{\partial} \varphi|_{\euc}^2$ is defined in Subsect.\,\ref{SS:PROPERTIESOFMETRICS}.

\begin{lemma}[Coerciveness of the energies and fluxes]
	\label{L:COERCIVENESS}
	Under the assumption $\Timefunction \equiv t$
	and the assumptions of Subsubsect.\,\ref{SSS:ASSUMPTIONSONTHESOLUTION},
	the energies and fluxes from Def.\,\ref{D:ENERGIESANDFLUXES}
	verify the following inequalities for $t \in [0,T]$, 
	where the implicit constants depend on the compact subset $\mathfrak{K}$ of solution-variable space
	from Point 3 of Subsubsect.\,\ref{SSS:ASSUMPTIONSONTHESOLUTION}:
	\begin{subequations}
	\begin{align}
		\mathbb{E}_{(Wave)}[\varphi](t)
		& \approx 
			\| \pmb{\partial} \varphi \|_{L^2(\widetilde{\Sigma}_t)}^2
			+
			\| \varphi \|_{L^2(\widetilde{\Sigma}_t)}^2,
				\label{E:WAVEENERGIESCOERCIVE} \\
			\mathbb{E}_{(Transport)}[\varphi](t)
		& = 
			\| \varphi \|_{L^2(\widetilde{\Sigma}_t)}^2.
			\label{E:TRANSPORTENERGIESCOERCIVE}
	\end{align}
	\end{subequations}
	
	Moreover, if the lateral boundary $\underline{\mathcal{H}}$ is $\gfour$-spacelike, 
	then the following estimates hold for $t \in [0,T]$, where 
	the implicit constants depend on the compact subset $\mathfrak{K}$ of solution-variable space
	from Point 3 of Subsubsect.\,\ref{SSS:ASSUMPTIONSONTHESOLUTION}
	as well as the reciprocal of the scalar function $\lengthofsidehypnorm$ from \eqref{E:LENGTHOFHYPNORM}
	(which is positive when $\underline{\mathcal{H}}$ is $\gfour$-spacelike):
	\begin{subequations}
	\begin{align}
		\mathbb{F}_{(Wave)}[\varphi](t)
		& \approx 
			\| \pmb{\partial} \varphi \|_{L^2(\underline{\mathcal{H}}_t)}^2
			+
			\| \varphi \|_{L^2(\underline{\mathcal{H}}_t)}^2,
				\label{E:WAVEFLUXESCOERCIVE} \\
		\mathbb{F}_{(Transport)}[\varphi](t)
		&	
			\approx
			\| \varphi \|_{L^2(\underline{\mathcal{H}}_t)}^2.
			\label{E:TRANSPORTFLUXESCOERCIVE}
	\end{align}
	\end{subequations}
	
	In addition, if the lateral boundary $\underline{\mathcal{N}}$ is $\gfour$-null, 
	then the following identities hold for $t \in [0,T]$,
	where the $\gfour$-null vectorfield $\uLunit$ on RHS~\eqref{E:NULLCASEWAVEFLUXESSEMICOERCIVE} 
	is as in Convention~\ref{C:NULLCASE} and \eqref{E:NORMALISGENERATORINNULLCASE}:
	\begin{subequations}
	\begin{align}
		\mathbb{F}_{(Wave)}[\varphi](t)
		& =
			\frac{1}{2}
			\| (\uLunit \varphi,|\angD \varphi|_{\gsphere}) \|_{L^2(\underline{\mathcal{N}}_t)}^2
			+
			\| \varphi \|_{L^2(\underline{\mathcal{N}}_t)}^2,
				\label{E:NULLCASEWAVEFLUXESSEMICOERCIVE} \\
		\mathbb{F}_{(Transport)}[\varphi](t)
		&	=
			\| \varphi \|_{L^2(\underline{\mathcal{N}}_t)}^2.
			\label{E:NULLCASETRANSPORTFLUXESSEMICOERCIVE}
	\end{align}
	\end{subequations}
	
	Finally, if the lateral boundary is either $\gfour$-spacelike at each of its points or $\gfour$-null at each of its points,
	then the following estimates hold for $t \in [0,T]$,
	where the implicit constants depend on the compact subset $\mathfrak{K}$ of solution-variable space
	from Point 3 of Subsubsect.\,\ref{SSS:ASSUMPTIONSONTHESOLUTION}
	and, when the lateral boundary is $\gfour$-spacelike, 
	the reciprocal of the scalar function $\lengthofsidehypnorm$ from \eqref{E:LENGTHOFHYPNORM}:
	\begin{align} \label{E:COERCEIVENESSOFSPACETIMECONTROLLINGQUANTITIY}
		\spacetimeen(t)
		& \approx
		\int_{\mathcal{M}_t}
			|\pmb{\partial} \vortrenormalized|_{\euc}^2
		\, d \varpi_{\gfour}
		+
		\int_{\mathcal{M}_t}
			|\pmb{\partial} \GradEnt|_{\euc}^2
		\, d \varpi_{\gfour}
		+
		\int_{\mathcal{S}_t}
			|\angvortrenormalized|_{\gsphere}^2
		\, d \varpi_{\gsphere}
		+
		\int_{\mathcal{S}_t}
			|\angGradEnt|_{\gsphere}^2
		\, d \varpi_{\gsphere}.
	\end{align}
\end{lemma}

\begin{proof}
The equations \eqref{E:TRANSPORTENERGIESCOERCIVE} and \eqref{E:NULLCASETRANSPORTFLUXESSEMICOERCIVE}
follow directly from the definitions.
\eqref{E:TRANSPORTFLUXESCOERCIVE} follows from definition \eqref{E:FLUXDEF}
and the fact that $\lengthofsidehypnorm \approx 1$ when $\underline{\mathcal{H}}$ is $\gfour$-spacelike
(see \eqref{E:LENGTHOFHYPNORM}).

To prove \eqref{E:WAVEENERGIESCOERCIVE},
we first use 
\eqref{E:MATERIALVECTORVIELDRELATIVECTORECTANGULAR},
Def.\,\ref{D:ACOUSTICALMETRIC},
\eqref{E:TRANSPORTISUNITLENGTHANDTIMELIKE}, 
and \eqref{E:ENMOMENTUMTENSOR}
to compute that
$2 \enmomem_{\alpha \beta} \Transport^{\alpha} \Transport^{\beta}
= (\Transport \varphi)^2 + |\nabla \varphi|_g^2
= (\partial_t \varphi)^2 
	+ 
	(v^a \partial_a \varphi)^2 
	+ 2 (\partial_t \varphi) (v^a \partial_a \varphi)
	+ 
\Speed^2 \updelta^{ab} \partial_a \varphi \partial_b \varphi
$.
The assumptions of Subsubsect.\,\ref{SSS:ASSUMPTIONSONTHESOLUTION} guarantee that there are constants $0< c_1 \leq c_2$ 
such that the speed of sound $\Speed$ verifies
$c_1 \leq \inf_{\mathcal{M}_T} \Speed \leq \sup_{\mathcal{M}_T} \Speed \leq c_2$ 
and a constant $C' > 0$ such that $\| |v|_g \|_{L^{\infty}(\mathcal{M}_T)} \leq C'$.
Hence, by the Cauchy--Schwarz and Young's inequalities,
for any $\varepsilon > 0$, we have $2 |(\partial_t \varphi) (v^a \partial_a \varphi)| 
\leq 
\frac{1}{1 + \varepsilon} (\partial_t \varphi)^2 
+ 
(v^a \partial_a \varphi)^2
+
\varepsilon (C')^2 |\nabla \varphi|_g^2
$,
where $|\nabla \varphi|_g^2 = \Speed^2 \updelta^{ab} \partial_a \varphi \partial_b \varphi$
and $\updelta^{ab}$ is the Kronecker delta.
Choosing $\varepsilon$ such that
$\varepsilon (C')^2 = \frac{1}{2}$, 
we conclude that
$2 \enmomem_{\alpha \beta} \Transport^{\alpha} \Transport^{\beta}
\geq
\frac{\varepsilon}{1 + \varepsilon}
(\partial_t \varphi)^2 
+
\frac{1}{2}
\Speed^2 \updelta^{ab} \partial_a \varphi \partial_b \varphi 
\gtrsim |\pmb{\partial} \varphi|_{\euc}^2
$. A similar but simpler argument yields that 
$|\enmomem_{\alpha \beta} \Transport^{\alpha} \Transport^{\beta}| \lesssim |\pmb{\partial} \varphi|_{\euc}^2$.
In view of definitions \eqref{E:MULTIPLIERVECTORFIELD} and \eqref{E:SIGMATENERGYDEF},
we conclude \eqref{E:WAVEENERGIESCOERCIVE}.

To prove \eqref{E:WAVEFLUXESCOERCIVE},
we first decompose $\hat{\sidehypnorm}$ into a vectorfield that is parallel to $\Transport$ and 
a vectorfield that is $\gfour$-orthogonal to $\Transport$. More precisely, we decompose
$\hat{\sidehypnorm} = \upalpha \Transport + \upbeta \hat{P}$,
where $\upalpha, \upbeta$ are real-valued functions, and the vectorfield $\hat{P}$ verifies 
$\gfour(\Transport,\hat{P}) = 0$ and $\gfour(\hat{P},\hat{P}) = g(\hat{P},\hat{P}) = 1$. 
From the fact that $\gfour(\Transport,\hat{\sidehypnorm}) < 0$
(see \eqref{E:INNERPRODUCTOFTRANPORTANDFUTUREUNITNORMALTOHYPERSURFACE})
and the relation
$\gfour(\Transport,\hat{\sidehypnorm}) = - \upalpha$ (which follows easily from taking the $\gfour$-inner product 
of both sides of the decomposition with respect to $\Transport$ and using \eqref{E:TRANSPORTISUNITLENGTHANDTIMELIKE}),
we find that $\upalpha > 0$.
Thus, taking the $\gfour$-inner product of each side of the decomposition with respect to itself and
using the fact that $\gfour(\hat{\sidehypnorm},\hat{\sidehypnorm}) = -1$ (see Def.\,\ref{D:LENGTHOFVARIOUSVECTORFIELDSETC}),
we compute that $\upalpha = \sqrt{1 + \upbeta^2}$.
Using these facts, \eqref{E:INVERSEACOUSTICALMETRIC}, and \eqref{E:ENMOMENTUMTENSOR}, 
we compute that
$\enmomem_{\alpha \beta} \Transport^{\alpha} \hat{\sidehypnorm}^{\beta}
 = 
\frac{1}{2} \sqrt{1 + \upbeta^2} (\Transport \varphi)^2
+
\upbeta (\Transport \varphi) \hat{P} \varphi
+ 
\frac{1}{2} \sqrt{1 + \upbeta^2} |\nabla \varphi|_g^2
$. Next, using the Cauchy--Schwarz and Young's inequalities, we bound the magnitude of the cross term as follows:
$|\upbeta (\Transport \varphi) \hat{P} \varphi| \leq \frac{|\upbeta|}{2} (\Transport \varphi)^2 +  \frac{|\upbeta|}{2} |\nabla \varphi|_g^2$.
It follows that
$
\enmomem_{\alpha \beta} \Transport^{\alpha} \hat{\sidehypnorm}^{\beta}
\geq 
\frac{1}{2}
(\sqrt{1 + \upbeta^2} - |\upbeta|) (\partial_t \varphi)^2
+
\frac{1}{2}
(\sqrt{1 + \upbeta^2} - |\upbeta|) |\nabla \varphi|_g^2
$.
Moreover, since $\hat{\sidehypnorm}$ is $\gfour$-timelike by assumption,
it follows that there exists a constant $C_1 > 0$
such that $\sup_{\underline{\mathcal{H}}_T} |\upbeta| \leq C_1$.
It follows that 
on $\underline{\mathcal{H}}_T$,
the two factors of
$
(\sqrt{1 + \upbeta^2} - |\upbeta|)
$
are uniformly bounded from above and below by positive constants
depending on $C_1$.
Using this fact and the bounds on $\Speed$ noted in the previous paragraph,
we conclude that
$
\enmomem_{\alpha \beta} \Transport^{\alpha} \hat{\sidehypnorm}^{\beta}
\gtrsim
|\pmb{\partial} \varphi|_{\euc}^2
$
as desired.
A similar but simpler argument yields that 
$|\enmomem_{\alpha \beta} \Transport^{\alpha} \hat{\sidehypnorm}^{\beta}| \lesssim |\pmb{\partial} \varphi|_{\euc}^2$.
In view of definitions \eqref{E:MULTIPLIERVECTORFIELD} and \eqref{E:FLUXDEF}, we conclude \eqref{E:WAVEFLUXESCOERCIVE}.

We now prove \eqref{E:NULLCASEWAVEFLUXESSEMICOERCIVE}.
Recalling that $\uLunit$ is alternate notation for $\sidehypnorm$ in the $\gfour$-null case
and using 
\eqref{E:EQUIVALENTFUTURENORMALTOHYPERSURFACE},
\eqref{E:INVERSEACOUSTICALMETRICNULLCASE}, 
and
\eqref{E:LUNITINNULLCASEWHENTIMEFUNCTIONISCARTESIAN}
(which implies that $\Transport = \frac{1}{2}(\uLunit + \Lunit)$),
we compute that 
$2 \enmomem_{\alpha \beta} \Transport^{\alpha} \uLunit^{\beta}
= 
(\uLunit \varphi)^2 
+
|\angD \varphi|_{\gsphere}^2
$.
In view of definitions \eqref{E:MULTIPLIERVECTORFIELD} and \eqref{E:NULLFLUXDEF},
we conclude \eqref{E:NULLCASEWAVEFLUXESSEMICOERCIVE}.

Finally, we prove \eqref{E:COERCEIVENESSOFSPACETIMECONTROLLINGQUANTITIY}.
From definition \eqref{E:SPACETIMECONTROLLINGQUANTITY},
\eqref{E:IDENTITYMAINQUADRATICFORMFORCONTROLLINGFIRSTDERIVATIVESOFSPECIFICVORTICITYANDENTROPYGRADIENT},
\eqref{E:SIGMATFIRSTFUNDAGREESWITHSIGMATILDEFIRSTFUND}-\eqref{E:PROJECTEDTRANSPORTVANISHES},
the fact that $g_{ab} = \Speed^{-2} \updelta_{ab}$ (see \eqref{E:ACOUSTICALMETRIC} and \eqref{E:FIRSTFUNDAMENTALFORMSIGMAT}),
and the bounds on $\Speed$ noted two paragraphs above,
we find that
\begin{align} \label{E:PROOFCOERCEIVENESSOFSPACETIMECONTROLLINGQUANTITIY}
		\spacetimeen(t)
		& \approx
		\int_{\mathcal{M}_t}
			\left\lbrace
				|\partial \vortrenormalized|_{\euc}^2
				+
				\sum_{a=1}^3 (\Transport \vortrenormalized)^2
				+
				|\partial \GradEnt|_{\euc}^2
				+
				\sum_{a=1}^3 (\Transport \GradEnt^a)^2
			\right\rbrace
		\, d \varpi_{\gfour}
		+
		\int_{\mathcal{S}_t}
			\uposinnerproduct
			\left\lbrace
				|\angvortrenormalized|_{\gsphere}^2
				+
				|\angGradEnt|_{\gsphere}^2
			\right\rbrace
		\, d \varpi_{\gsphere}.
	\end{align}
	Using arguments similar to the ones we used in proving
	\eqref{E:WAVEENERGIESCOERCIVE} (based on Cauchy--Schwarz and Young's inequality),
	we find that
	$
	|\pmb{\partial} \vortrenormalized|_{\euc}^2
	\approx
	|\partial \vortrenormalized|_{\euc}^2
		+
	\sum_{a=1}^3 (\Transport \vortrenormalized)^2
	$
	and
	$
		|\pmb{\partial} \GradEnt|_{\euc}^2
			\approx
			|\partial \GradEnt|_{\euc}^2
				+
				\sum_{a=1}^3 (\Transport \GradEnt^a)^2
	$.
	From these estimates,
	\eqref{E:INGOINGCONDITION},
	and
	\eqref{E:PROOFCOERCEIVENESSOFSPACETIMECONTROLLINGQUANTITIY},
	the desired result \eqref{E:COERCEIVENESSOFSPACETIMECONTROLLINGQUANTITIY} readily follows.
\end{proof}

\subsection{Localized a priori estimates}
\label{SS:LOCALIZEDAPRIORI}
We now prove the main result of Sect.\,\ref{S:APRIORI}, namely Theorem~\ref{T:LOCALIZEDAPRIORIESTIMATES},
which yields a priori estimates exhibiting the gain of regularity for the specific vorticity and the entropy gradient
(as is manifested by \eqref{E:COMBINEDSPACETIMEINTEGRALGRONWALLED}),
as we described in the introduction (see in particular Point \textbf{I} of Subsect.\,\ref{SS:APPLICATIONS}).
We again stress that when the lateral boundary is $\gfour$-null (i.e., $\underline{\mathcal{H}} = \underline{\mathcal{N}}$),
the theorem crucially relies 
on the precise structures shown in 
Theorem~\ref{T:STRUCTUREOFERRORTERMS}
and
Theorem~\ref{T:MAINREMARKABLESPACETIMEINTEGRALIDENTITY}.
In particular, in the $\gfour$-null case, 
these structures are needed to control the error integrals 
$
\int_{\underline{\mathcal{N}}_t}
		\cdots
$
on RHSs~\eqref{E:SPACETIMEREMARKABLEIDENTITYSPECIFICVORTICITY}-\eqref{E:SPACETIMEREMARKABLEIDENTITYENTROPYGRADIENT}
(recall that $\underline{\mathcal{N}}_t = \underline{\mathcal{H}}_{\Timefunction}$ in the present context),
since \eqref{E:NULLCASEWAVEFLUXESSEMICOERCIVE} shows that the wave fluxes on $\underline{\mathcal{N}}_t$
control only $\underline{\mathcal{N}}_t$-tangential derivatives
in the $\gfour$-null case.

\begin{theorem}[Localized a priori estimates exhibiting the gain in regularity for $\vortrenormalized$ and $\GradEnt$]
	\label{T:LOCALIZEDAPRIORIESTIMATES}
	Let $\mathcal{M} = \mathcal{M}_T$ be a spacetime region satisfying the conditions stated in 
	Subsects.\,\ref{SS:DOMAINANDTIMEFUNCTIONETC} and \ref{SS:ASSUMPTIONSONSPACETIMEREGION} for some $T > 0$;
	see Fig.\,\ref{F:SPACETIMEDOMAIN}.
	In particular, assume that the lateral boundary $\underline{\mathcal{H}} = \underline{\mathcal{H}}_T$ is $\gfour$-spacelike
	or is $\gfour$-null (in the null case, $\underline{\mathcal{H}} := \underline{\mathcal{N}} = \underline{\mathcal{N}}_T$).
	Assume that the acoustical time function $\Timefunction$ from the beginning of Sect.\,\ref{S:SPACETIMEDOMAINS} is equal to\footnote{As we mentioned at 
	the beginning of Sect.\,\ref{S:APRIORI}, 
	we make this assumption only to shorten the presentation; the results of Theorem~\ref{T:LOCALIZEDAPRIORIESTIMATES}
	generalize in a straightforward fashion to the case of general smooth acoustical time functions. \label{FN:GENERALIZETOOTHERTIMEFUNCTIONS}} 
	the Cartesian time function $t$.
	Consider a smooth solution
	(see Remark~\ref{R:SMOOTHNESSNOTNEEDED})
	to the compressible Euler equations \eqref{E:TRANSPORTDENSRENORMALIZEDRELATIVECTORECTANGULAR}-\eqref{E:ENTROPYTRANSPORT}
	on $\mathcal{M}_T$
	that satisfies the assumptions stated in Subsubsect.\,\ref{SSS:ASSUMPTIONSONTHESOLUTION}.
	Let $\mathfrak{K}$ be the set from Point 3 of Subsubsect.\,\ref{SSS:ASSUMPTIONSONTHESOLUTION},
	let $A$ be the assumed bound on the $C^1$ norm of $(\LogDensity,\Ent,v)$ on $\mathcal{M}_T$ stated in \eqref{E:ASSUMEDC1BOUND},
	and let $B$ be the assumed bound on $(\gen,\spherenormal)$ and some of their $\underline{\mathcal{H}}_T$-tangential first derivatives
	stated in \eqref{E:VECTORFIELDSTENSORIALC1BOUND} 
	(see \eqref{E:VECTORFIELDSTENSORIALC1BOUNDNULLCASE} for the assumed $C^1$ norm bound in the case of a lateral null hypersurface).
	Then there exists a constant $C = C(\mathfrak{K},A,B)$ (which we allow to vary from line to line)
	such that the controlling quantities $\spacetimeen(t)$ and $\toten(t)$
	from Def.\,\ref{D:CONTROLLINGQUANTITIES}
	verify the following inequalities for $t \in [0,T]$:
	\begin{subequations}
	\begin{align}
		\spacetimeen(t)
		& \leq
			2 \spacetimeen(0)
			+
			C
			\int_0^t
				\toten(t')
			\, d t'
			+
			C \toten(t), 
			\label{E:COMBINEDSPACETIMEINTEGRALGRONWALLREADY} \\
		\toten(t)
		& \leq 
				\toten(0)
				+
				\spacetimeen(0)
			+
			C
			\int_0^t
				\toten(t')
			\, d t'.
			\label{E:COMBINEDENERGYFLUXGRONWALLREADY}
		\end{align}
		\end{subequations}
		
		Moreover, the following inequalities hold for $0 \leq t \leq T$:
		\begin{subequations}
		\begin{align} 
		\spacetimeen(t)
		& \leq  
				C
				\left\lbrace
					\toten(0)
					+
					\spacetimeen(0)
				\right\rbrace
				\exp\left(C t\right),
			\label{E:COMBINEDSPACETIMEINTEGRALGRONWALLED}
				\\
		\toten(t)
			& \leq 
			\left\lbrace
					\toten(0)
					+
					\spacetimeen(0)
				\right\rbrace
				\exp\left(C t\right).
			\label{E:COMBINEDENERGYFLUXGRONWALLED}
		\end{align}
		\end{subequations}
\end{theorem}	

\begin{proof}
		First, we use the wave equations \eqref{E:VELOCITYWAVEEQUATION}-\eqref{E:ENTROPYWAVEEQUATION}
		to algebraically substitute for the terms $\square_{\gfour} \varphi$ on the first line of RHS~\eqref{E:ENERGYFLUXCOMBINEDINTEGRALID},
		and we use the transport equations 
		\eqref{E:RENORMALIZEDVORTICTITYTRANSPORTEQUATION}, 
		\eqref{E:GRADENTROPYTRANSPORT},
		\eqref{E:EVOLUTIONEQUATIONFLATCURLRENORMALIZEDVORTICITY},
		and \eqref{E:TRANSPORTFLATDIVGRADENT}
		to algebraically substitute for the terms $\Transport \varphi$
		on the second line RHS~\eqref{E:ENERGYFLUXCOMBINEDINTEGRALID}.
		We then use \eqref{E:RENORMALIZEDVORTICTITYTRANSPORTEQUATION} and \eqref{E:GRADENTROPYTRANSPORT}
		to algebraically substitute for all factors of
		$\Transport \vortrenormalized^i$ and $\Transport \GradEnt^i$ in all of the resulting expressions.
		After these substitutions, 
		we use the volume form identities of Lemma~\ref{L:IDENTITIESFORVOLUMEFORMS},
		the assumptions stated in Subsubsect.\,\ref{SSS:ASSUMPTIONSONTHESOLUTION},
		Theorem~\ref{T:STRUCTUREOFERRORTERMS},
		and the coerciveness estimates of Lemma~\ref{L:COERCIVENESS},
		to bound all integrand factors on RHS~\eqref{E:ENERGYFLUXCOMBINEDINTEGRALID}
		by $C$ times a quadratic term that is controlled by the
		controlling quantities of Def.\,\ref{D:CONTROLLINGQUANTITIES}.
		Note that we are using the fact that the right-hand side of the identity \eqref{E:DEFORMATIONTENSOROFTRANSPORTRELATIVETOCARTESIAN}
		for the deformation tensor components $\deformarg{\Transport}{\alpha}{\beta}$
		(which also appear on RHS~\eqref{E:ENERGYFLUXCOMBINEDINTEGRALID})
		can be expressed as smooth functions of $(\LogDensity,v^1,v^2,v^3,\Ent)$
		times a factor that is linear in $\pmb{\partial}(\LogDensity,v^1,v^2,v^3,\Ent)$.
		Also using the Cauchy--Schwarz inequality for integrals,
		we deduce that
		\begin{align} 	\label{E:COMBINEDENERGYFLUXALMOSTGRONWALLREADY} 
		\toten(t)
		& \leq 
			\toten(0)
			+
			C
			\int_0^t
				\toten(t')
			\, d t'
			+
			C
			\sqrt{
			\int_0^t
				\toten(t')
			\, d t'}
			\sqrt{\spacetimeen(t)}.
	\end{align}
	We clarify that the last product on RHS~\eqref{E:COMBINEDENERGYFLUXALMOSTGRONWALLREADY}
	comes from the integral
	\[
	2
				\sum_{\varphi \in \lbrace \vortrenormalized^i,\GradEnt^i,\VortVort^i,\DivGradEnt \rbrace_{i=1,2,3}}
				\int_{\mathcal{M}_t}
					(\Transport \varphi) \varphi
				\, d \varpi_{\gfour}
	\]
	on RHS~\eqref	{E:ENERGYFLUXCOMBINEDINTEGRALID} in the cases $\varphi \in \lbrace \VortVort^i,\DivGradEnt \rbrace_{i=1,2,3}$,
	specifically from the terms on RHSs \eqref{E:EVOLUTIONEQUATIONFLATCURLRENORMALIZEDVORTICITY} and \eqref{E:TRANSPORTFLATDIVGRADENT}
	that depend on the terms $\partial \vortrenormalized$ and $\partial \GradEnt$; by Cauchy--Schwarz, the corresponding integrals
	are bounded by 
	$
	C
	\sqrt{
	\int_0^t
			\toten(t')
	\, d t'}
	\sqrt{
	\int_{\mathcal{M}_t}
			\left\lbrace
				|\partial \vortrenormalized|_{\euct}^2
				+
				|\partial \GradEnt|_{\euct}^2
			\right\rbrace
	\, d \varpi_{\gfour}
	}
	$,
	which in turn is bounded by the last product on RHS~\eqref{E:COMBINEDENERGYFLUXALMOSTGRONWALLREADY} as desired.
	
	Similarly, using \eqref{E:SPACETIMECOMBINEDINTEGRALID} 
	(without further need to use the equations of Theorem~\ref{T:GEOMETRICWAVETRANSPORTSYSTEM}),
	and exploiting the fact (highlighted in Remark~\ref{R:HIGHLIGHTKEYSTRUCTURES})
	that \eqref{E:ERRORTERMSSCHEMATICSTRUCTURESPACELIKECASE}-\eqref{E:ERRORTERMSSCHEMATICSTRUCTURENULLCASE}
		show that RHS~\eqref{E:MAINTHMLATERALBOUNDARYEASYERRORINTEGRANDTERMS} (with $\vortrenormalized$ and $\GradEnt$ in the role of $\SigmatTan$)
		does not involve any $\underline{\mathcal{H}}$-transversal derivatives of $(\LogDensity,v,\Ent)$
		or $(\gen,\utang,\spherenormal)$
		and that the same statement holds for 
		RHS~\eqref{E:MAINTHMWEIGHTDERIVATVELATERALBOUNDARYEASYERRORINTEGRANDTERMS} 
		(this remark is trivial since $\weight := 1$ in the present context)
		and
		RHSs~\eqref{E:MAINTHMSPECIFICVORTITICYMAINLATERALERRORINTEGRAND}-\eqref{E:MAINTHMENTROPYGRADIENTMAINLATERALERRORINTEGRAND},
	we find that
	\begin{align}  \label{E:COMBINEDSPACETIMEINTEGRALALMOSTGRONWALLREADY}
	\spacetimeen(t)
		& \leq
			\spacetimeen(0)
			+
			C
			\int_0^t
				\toten(t')
			\, d t'
			+
			C
			\sqrt{
			\int_0^t
				\toten(t')
			\, d t'}
			\sqrt{\spacetimeen(t)}
			+
			C \toten(t).
	\end{align}
	We stress that when the lateral hypersurface is $\gfour$-null (i.e., $\underline{\mathcal{H}} = \underline{\mathcal{N}}$,
	the coerciveness result \eqref{E:NULLCASEWAVEFLUXESSEMICOERCIVE} shows that $\toten(t)$
	controls (in $L^2(\underline{\mathcal{N}}_t)$) 
	the derivatives of $(\LogDensity,v,\Ent)$ \emph{only} in the $\underline{\mathcal{N}}_t$-tangential directions;
	this is the reason that
	\emph{the absence of the $\underline{\mathcal{N}}$-transversal derivatives of $(\LogDensity,v,\Ent)$ 
	on RHSs~\eqref{E:MAINTHMLATERALBOUNDARYEASYERRORINTEGRANDTERMS}-\eqref{E:MAINTHMENTROPYGRADIENTMAINLATERALERRORINTEGRAND}
	is critically important
	in the $\gfour$-null case.}
	
	\eqref{E:COMBINEDSPACETIMEINTEGRALGRONWALLREADY} then follows from \eqref{E:COMBINEDSPACETIMEINTEGRALALMOSTGRONWALLREADY}
	and Young's inequality.
	
	\eqref{E:COMBINEDENERGYFLUXGRONWALLREADY} then follows from
	\eqref{E:COMBINEDSPACETIMEINTEGRALGRONWALLREADY},
	\eqref{E:COMBINEDENERGYFLUXALMOSTGRONWALLREADY}, 
	and Young's inequality.
	
	\eqref{E:COMBINEDENERGYFLUXGRONWALLED} follows from \eqref{E:COMBINEDENERGYFLUXGRONWALLREADY}
	and Gronwall's inequality.
	\eqref{E:COMBINEDSPACETIMEINTEGRALGRONWALLED}
	follows from \eqref{E:COMBINEDSPACETIMEINTEGRALGRONWALLREADY} and \eqref{E:COMBINEDENERGYFLUXGRONWALLED}.
\end{proof}

\section{Remarkable Hodge-transport-based integral identities relative to double-null foliations} 
\label{S:DOUBLENULL}
In this section, we extend the integral identities of Theorem~\ref{T:MAINREMARKABLESPACETIMEINTEGRALIDENTITY}
so that they apply to spacetime regions that are double-null foliated, that is, foliated by a pair $u,\underline{u}$ 
of acoustical eikonal functions. We provide the main integral identities in Theorem~\ref{T:DOUBLENULLMAINTHEOREM}.
We highlight that an analog of Theorem~\ref{T:STRUCTUREOFERRORTERMS} also holds in the present context, that is, that
the error terms in Theorem~\ref{T:DOUBLENULLMAINTHEOREM} along $\gfour$-null hypersurfaces involve only tangential derivatives;
see Remark~\ref{R:DOUBLENULLNULLHYPERSURFACEERRORTERMSTANGENTIALDERIVATIVES}.
Before proving the theorem, we set up the double-null foliation and
provide analogs of results from the previous sections,
modified so as to apply in the present context.

\begin{center}
\begin{overpic}[scale=.5,grid=false]{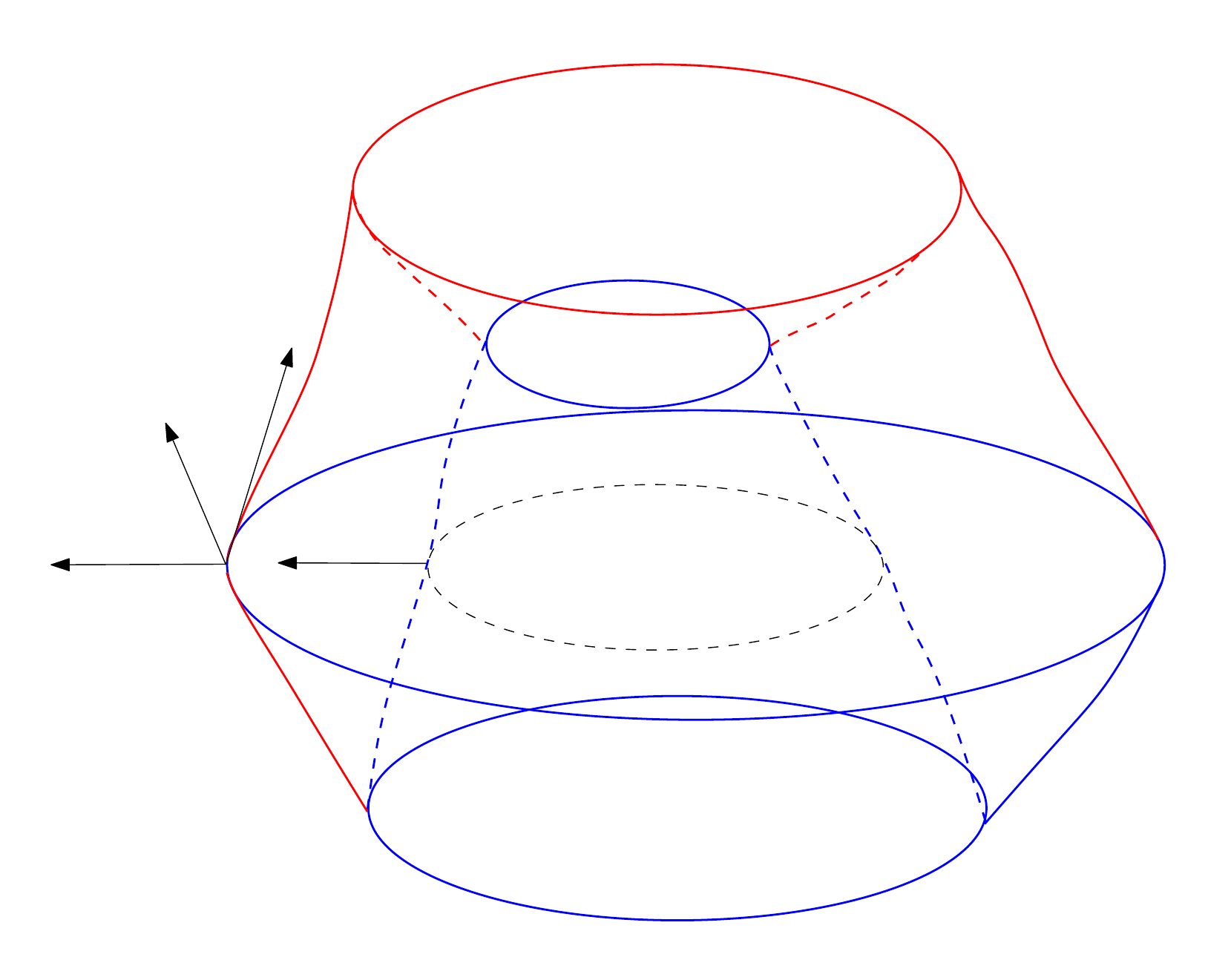} 
\put (33,2.5) {\large$\displaystyle \mathcal{S}_{1,0}$}
\put (22,28.3) {\large$\displaystyle \mathcal{S}_{1,\underline{U}}$}
\put (48,30) {\large$\displaystyle \mathcal{S}_{u = 1 + \underline{U},0}$}
\put (47,60) {\large$\displaystyle \mathcal{S}_{U,0}$}
\put (21,71) {\large$\displaystyle \mathcal{S}_{U,\underline{U}}$}
\put (87,15) {\large$\displaystyle \mathcal{H}_1$}
\put (70,50) {\large$\displaystyle \mathcal{H}_U$}
\put (75.5,29) {\large$\displaystyle \underline{\mathcal{H}}_0$}
\put (86,55) {\large$\displaystyle \underline{\mathcal{H}}_{\underline{U}}$}
\put (4,35) {\large$\displaystyle \spherenormal$}
\put (22,35) {\large$\displaystyle \spherenormal$}
\put (12,46.5) {\large$\displaystyle \newL$}
\put (21,54) {\large$\displaystyle \newuL$}
\end{overpic}
\captionof{figure}{A spacetime region $\mathcal{M}$ that can be covered by a double-null foliation}
\label{F:DOUBLENULL}
\end{center}

\subsection{Setup of the double-null foliations}
\label{SS:DOUBLENULL}
We will derive integral identities for compressible Euler solutions
on spacetime regions $\mathcal{M}$ that are
foliated by the level sets of a pair $u,\underline{u}$ of acoustic eikonal functions,
which we assume to be given solutions of the acoustical eikonal equation:
\begin{align} \label{E:EIKONALFUNCTION}
	(\gfour^{-1})^{\alpha \beta}
	\partial_{\alpha} u 
	\partial_{\beta} u
	& = 0,
	&
	(\gfour^{-1})^{\alpha \beta}
	\partial_{\alpha} \underline{u} 
	\partial_{\beta} \underline{u}
	& = 0.
\end{align}
As before, in \eqref{E:EIKONALFUNCTION}, $\gfour$ is the acoustical metric of Def.\,\ref{D:ACOUSTICALMETRIC}.
We let $\mathcal{H}_u$ and  $\underline{\mathcal{H}}_{\underline{u}}$ respectively denote the level
sets\footnote{Throughout, we abuse notation by using the symbols ``$u$'' and ``$\underline{u}$'' 
to denote the acoustical eikonal functions and the values that they take on;
the precise meaning of the symbols will be clear from context.} of $u$ and $\underline{u}$. 
We assume that there are constants 
$U$ and $\underline{U}$ satisfying
$1 + \underline{U} < U$ and $0 < \underline{U}$ such that
$u$ and $\underline{u}$ are smooth with non-vanishing, transversal,
past-directed\footnote{Equivalently, we assume that $\Transport u > 0$ and $\Transport \underline{u} > 0$.} gradients 
on a spacetime region $\mathcal{M}$ corresponding to $1 \leq u \leq U$ and $0 \leq \underline{u} \leq \underline{U}$.
Then by \eqref{E:EIKONALFUNCTION},
$\mathcal{H}_u$ and  $\underline{\mathcal{H}}_{\underline{u}}$
are three-dimensional $\gfour$-null\footnote{The vectorfield 
$(\gfour^{-1})^{\alpha \beta} \partial_{\beta} u$ is $\gfour$-null and $\gfour$-orthogonal to the level sets of $u$,
while the vectorfield $(\gfour^{-1})^{\alpha \beta} \partial_{\beta} \underline{u}$ is $\gfour$-null and $\gfour$-orthogonal to the level sets of 
$\underline{u}$.} 
hypersurfaces on $\mathcal{M}$ that intersect transversally
in two-dimensional $\gfour$-spacelike submanifolds
\begin{align} \label{E:SPHERESDOUBLENULL}
	\mathcal{S}_{u,\underline{u}}
	& := \mathcal{H}_u \cap \underline{\mathcal{H}}_{\underline{u}}.
\end{align}
We assume that all of the $\mathcal{S}_{u,\underline{u}}$ are
diffeomorphic to $\mathbb{S}^2$.

\begin{remark}[On the width of the regions]	
	\label{R:WIDTH}
	Our assumptions on $U$ and $\underline{U}$ imply that the $u$-width of $\mathcal{M}$ is larger than its $\underline{u}$-width.
	However, this is only for convenience of exposition; our results could readily be generalized to handle the case that the
	 $u$-width of $\mathcal{M}$ is less than or equal to the $\underline{u}$-width.
\end{remark}

We define
\begin{align} \label{E:DOUBLENULLTIMEFUNCTION}
	\Timefunction
	& := u + \underline{u}.
\end{align}
From \eqref{E:EIKONALFUNCTION} and our assumption that the gradients of $u$ and $\underline{u}$ are past-directed,
it follows that
\begin{align} \label{E:DOUBLENULLTIMEFUNCTIONHASTIMELIKENORMAL}
	(\gfour^{-1})^{\alpha \beta}
	\partial_{\alpha} \Timefunction 
	\partial_{\beta} \Timefunction
	& < 0.
\end{align}
In particular, the $\gfour$-normal to the level sets of $\Timefunction$ are $\gfour$-timelike, and thus these level
sets are $\gfour$-spacelike. That is, $\Timefunction$ is an acoustical time function on the region under study.
For $(u',\underline{u}') \in [1,U] \times [0,\underline{U}]$
and $\Timefunction' \in [1,U + \underline{U}]$,
we define
\begin{subequations}
\begin{align}
	\widetilde{\Sigma}_{\Timefunction'}
	& := \mathcal{M}_{U,\underline{U}}
			\cap
			\lbrace \Timefunction = \Timefunction' \rbrace,
			\\
	\mathcal{H}_{u'}(0,\underline{u}')
	& := \mathcal{H}_{u'} \cap \lbrace 0 \leq \underline{u} \leq \underline{u}' \rbrace,
		\\
	\underline{\mathcal{H}}_{\underline{u}}(1,u')
	& := \underline{\mathcal{H}}_{\underline{u}} \cap \lbrace 1 \leq u \leq u' \rbrace,
		\\
	\mathcal{M}_{u',\underline{u}'}
	& :=
			\lbrace 1 \leq u \leq u' \rbrace
			\cap
			\lbrace 0 \leq \underline{u} \leq \underline{u}' \rbrace 
			= \cup_{u'' \in [1,u']} \mathcal{H}_{u''}(0,\underline{u}')
			= \cup_{\underline{u}'' \in [0,\underline{u}']} \underline{\mathcal{H}}_{\underline{u}''}(1,u')
				\\
	& = \cup_{(u'',\underline{u}'') \in [1,u'] \times [0,\underline{u}']} \mathcal{S}_{u'',\underline{u}''}.
	\notag
\end{align}
\end{subequations}
Note that $\mathcal{M} = \mathcal{M}_{U,\underline{U}}$ and that on $\mathcal{M}_{U,\underline{U}}$, we have
\begin{align} \label{E:DOUBLENULLTIMEFUNCTIONRANGE}
	1 \leq \Timefunction \leq U + \underline{U}.
\end{align}
Note also that $\widetilde{\Sigma}_1$ and $\widetilde{\Sigma}_{U + \underline{U}}$ are degenerate
in the sense that they are not three-dimensional submanifolds with boundary, 
but rather are two-dimensional submanifolds:
$\widetilde{\Sigma}_1 = \mathcal{S}_{1,0}$,
$\widetilde{\Sigma}_{U + \underline{U}} = \mathcal{S}_{U,\underline{U}}$.
Moreover, we note that for $\Timefunction' \in (1,U + \underline{U})$,
the boundary of $\widetilde{\Sigma}_{\Timefunction'}$ (in the sense of a manifold-with-boundary),
which we denote by $\partial \widetilde{\Sigma}_{\Timefunction'}$,
satisfies:
\begin{subequations}
\begin{align} \label{E:DOUBLENULLLOWERREGIONBOUNDARYOFSPACELIKEHYPERSURFACEPORTION}
	\partial \widetilde{\Sigma}_{\Timefunction'}
	& = \mathcal{S}_{\Timefunction',0} \cup \mathcal{S}_{1,\Timefunction'-1},
	&
	\Timefunction' \in (1,1 + \underline{U}],
		\\
	\partial \widetilde{\Sigma}_{\Timefunction'}
	& = \mathcal{S}_{\Timefunction',0} \cup \mathcal{S}_{\Timefunction'-\underline{U},\underline{U}},
	&
	\Timefunction' \in [1 + \underline{U},U],
		\label{E:DOUBLENULLMIDDLEREGIONBOUNDARYOFSPACELIKEHYPERSURFACEPORTION} 
		\\
	\partial \widetilde{\Sigma}_{\Timefunction'}
	& = \mathcal{S}_{U,\Timefunction'-U} \cup \mathcal{S}_{\Timefunction'-\underline{U},\underline{U}},
	&
	\Timefunction' \in [U,U + \underline{U}).
		\label{E:DOUBLENULLUPPERREGIONBOUNDARYOFSPACELIKEHYPERSURFACEPORTION} 
\end{align}
\end{subequations}
In each of the disjoint unions on
RHSs~\eqref{E:DOUBLENULLLOWERREGIONBOUNDARYOFSPACELIKEHYPERSURFACEPORTION}-\eqref{E:DOUBLENULLUPPERREGIONBOUNDARYOFSPACELIKEHYPERSURFACEPORTION},
we refer to the first set as the ``inner boundary'' of $\widetilde{\Sigma}_{\Timefunction'}$ and the second set as
the ``outer boundary'' of $\widetilde{\Sigma}_{\Timefunction'}$; see Fig.\,\ref{F:DOUBLENULL}.

For the purpose of deriving the integral identities,
we assume that the fluid solution is smooth on $\mathcal{M}_{U,\underline{U}}$.
Moreover, for the purpose of \emph{interpreting} the integral identities,
we imagine that the ``state of the fluid solution'' is prescribed on $\mathcal{H}_1(0,\underline{U})$ and $\underline{\mathcal{H}}_0(1,U)$
(i.e., we view these as null hypersurfaces where ``initial data'' are posed);
as we mentioned in Subsect.\,\ref{SS:APPLICATIONS},
a full treatment of the characteristic initial value problem will be the subject of a future work.

\subsection{Geometric quantities adapted to the double-null foliation}
\label{SS:DOUBLENULLGEOMETRICQUANTITIES}

\subsubsection{$\gfour$-null vectorfields and related scalar functions}
\label{SSS:DOUBLENULLGFOURNULLVECTORFIELDS}
Associated to $u$ and $\underline{u}$, we define the geodesic vectorfields\footnote{More precisely, 
using \eqref{E:EIKONALFUNCTION}, is straightforward to show
that $\Dfour_{\Lgeo} \Lgeo = \Dfour_{\uLgeo} \uLgeo = 0$, where $\Dfour$ is the Levi--Civita connection of $\gfour$.} 
\begin{align} \label{E:NULLGEODESIVECTORFIELDS}
	\Lgeo^{\alpha}
	& := - (\gfour^{-1})^{\alpha \beta} \partial_{\beta} u,
	&
	\uLgeo^{\alpha}
	& := - 	
			(\gfour^{-1})^{\alpha \beta} \partial_{\beta} \underline{u}.
	\end{align}
From \eqref{E:NULLGEODESIVECTORFIELDS}, it follows that $\Lgeo$ is $\gfour$-orthogonal to $\mathcal{H}_u$,
while $\uLgeo$ is $\gfour$-orthogonal to $\underline{\mathcal{H}}_{\underline{u}}$.
The equations in \eqref{E:EIKONALFUNCTION} imply that $\Lgeo$ and $\uLgeo$ are $\gfour$-null:
\begin{align} \label{E:NULLGEODESIVECTORFIELDSAREINFACTNULL}
	\gfour(\Lgeo,\Lgeo)
	&
	= \gfour(\uLgeo,\uLgeo)
	= 0.
\end{align}

Next, we define the following scalar functions on $\mathcal{M}$:
\begin{subequations}
\begin{align} \label{E:LFOLIATIONDENSITY}
	\upmu
	& := 
	\frac{- 1}{
	(\gfour^{-1})^{\alpha \beta}
	\partial_{\alpha} t 
	\partial_{\beta} u}
	=
	\frac{1}{\Lgeo^0},
		\\
	\underline{\upmu}
	& := 
	\frac{-1}{
	(\gfour^{-1})^{\alpha \beta}
	\partial_{\alpha} t 
	\partial_{\beta} \underline{u}}
	=
	\frac{1}{\uLgeo^0},
		\label{E:ULFOLIATIONDENSITY} \\
	\MagnitueofinnerproductofnewLandnewuL
	& := 
	\frac{-1}{
	(\gfour^{-1})^{\alpha \beta}
	\partial_{\alpha} u 
	\partial_{\beta} \underline{u}}
	=
	\frac{-1}{\gfour(\Lgeo,\uLgeo)},
	\label{E:RECIPROCALOFNEGATIVENULLGEODESICVECTORFIELDINNERPRODUCT}
		\\
	\ReciprocalLunitAppliedtoTimeFunction 
	& := \frac{\MagnitueofinnerproductofnewLandnewuL}{\upmu},
	&
	\ReciprocaluLunitAppliedtoTimeFunction
	& := \frac{\MagnitueofinnerproductofnewLandnewuL}{\underline{\upmu}}.
	\label{E:RATIOOFNULLGEOSICINNERPRODUCTANDFOLIATIONDENSITY}
\end{align}
\end{subequations}
The assumptions of Subsect.\,\ref{SS:DOUBLENULL} imply that
\begin{align} \label{E:POSITIVITYOFFOLIATIONDENSITYANDNULLGEODESICINNERPRODUCTETC}
	\upmu 
	& > 0,
	&
	\underline{\upmu}
	& > 0,
	&
	\MagnitueofinnerproductofnewLandnewuL
	& > 0,
	&
	\ReciprocalLunitAppliedtoTimeFunction
	& > 0,
	&
	\ReciprocaluLunitAppliedtoTimeFunction
	& > 0.
\end{align}

Next, we define the following vectorfields, which are rescaled versions of $\Lgeo$ and $\uLgeo$:
\begin{subequations}
\begin{align}
	\Lunit^{\alpha}
	& := - \upmu (\gfour^{-1})^{\alpha \beta} \partial_{\beta} u,
	&
	\uLunit^{\alpha}
	& := - 	
			\underline{\upmu} (\gfour^{-1})^{\alpha \beta} \partial_{\beta} \underline{u}
		 \label{E:DOUBLENULLCARTESIANTIMENORMALIZEDNULLVECTORFIELDS} \\
	\newL^{\alpha}
	& := - \MagnitueofinnerproductofnewLandnewuL (\gfour^{-1})^{\alpha \beta} \partial_{\beta} u,
	&
	\newuL^{\alpha}
	& := - \MagnitueofinnerproductofnewLandnewuL (\gfour^{-1})^{\alpha \beta} \partial_{\beta} \underline{u}.
	 \label{E:DOUBLENULLEIKONALFUNCTIONNORMALIZEDNULLVECTORFIELDS}
\end{align}
\end{subequations}

The following identities easily follow from the above definitions:
\begin{align} \label{E:DOUBLENULLALLTHENEWVECTORFIELDSARENULL}
	\gfour(\Lunit,\Lunit)
	& = \gfour(\uLunit,\uLunit)
	= \gfour(\newL,\newL)
	= \gfour(\newuL,\newuL)
	= 0,
\end{align}

\begin{subequations}
\begin{align} \label{E:INNERPRODUCTOFCARTESIANNORMALIZEDNULLVECTORFIELDS}
	\gfour(\Lunit,\uLunit)
	& = \frac{- \upmu \underline{\upmu}}{\MagnitueofinnerproductofnewLandnewuL}
		= \frac{- \MagnitueofinnerproductofnewLandnewuL}{\ReciprocalLunitAppliedtoTimeFunction \ReciprocaluLunitAppliedtoTimeFunction},
		\\
	\gfour(\newL,\newuL)
	& = - \MagnitueofinnerproductofnewLandnewuL,
	\label{E:INNERPRODUCTOFEIKONALFUNCTIONNORMALIZEDNULLVECTORFIELDS}
\end{align}
\end{subequations}

\begin{subequations}
\begin{align}
	\Lunit u 
	& = \uLunit \underline{u}
	= 0,
	&
	\Lunit \underline{u}
	& = \frac{1}{\ReciprocalLunitAppliedtoTimeFunction},
	\,
	\uLunit u
	= \frac{1}{\ReciprocaluLunitAppliedtoTimeFunction},
	&
	\Lunit t 
	& = \uLunit t 
	= 1,
		\label{E:LUNITANDULUNITAPPLIEDTOEIKONALANDCARTESIANTIME} \\
	\newL u 
	& = \newuL \underline{u}
	= 0,
	&
	\newL \underline{u}
	& =
	\newuL u
	= 1,
	&
	\newL \Timefunction
	& =
	\newuL \Timefunction
	= 1,
	\label{E:GEOMETRICCOORDINATENULLVECTORFIELDSAPPLIEDTOGEOMETRICCOORDINATES}
\end{align}
\end{subequations}

\begin{align} \label{E:RELATIONBETWEENCARTESIANNORMALIZEDNULLVECTORFIELDSANDEIKONALFUNCTIONORMALIZEDNULLVECTORFIELDS}
	\newL
	& = \ReciprocalLunitAppliedtoTimeFunction \Lunit,
	&
	\newuL
	& = \ReciprocaluLunitAppliedtoTimeFunction \uLunit.
\end{align}

Note that by \eqref{E:TRANSPORTONEFORMIDENTITY}, the last two equalities in \eqref{E:LUNITANDULUNITAPPLIEDTOEIKONALANDCARTESIANTIME} are equivalent to
\begin{align} \label{E:EQUIVALENTLUNITANDULUNITAPPLIEDTOEIKONALANDCARTESIANTIME}
	\gfour(\Lunit,\Transport)
	& 
	= \gfour(\uLunit,\Transport)
	= - 1.
\end{align}

From \eqref{E:FUTURENORMALTOHYPERSURFACE},
Convention~\ref{C:NULLCASE},
and the last equality in
\eqref{E:LUNITANDULUNITAPPLIEDTOEIKONALANDCARTESIANTIME}, 
it follows that the vectorfield denoted by ``$\uLunit$'' in this section has the same properties
as the vectorfield denoted by the same symbol in Sects.\,\ref{S:SPACETIMEDOMAINS}-\ref{S:APRIORI} .

\subsubsection{Additional geometric vectorfields and scalar functions}
\label{SSS:DOUBLENULLOTHERVECTORFIELDSANDSCALARFUNCTIONS}
In this subsubsection, we define some additional vectorfields and scalar functions that play a role in the ensuing analysis.

\begin{definition}[The vectorfield $\spherenormal$]
\label{D:DOUBLENULLSPHERENORMAL}
We define $\spherenormal$ to be the following vectorfield:
\begin{align} \label{E:DOUBLENULLSPHERENORMAL}
	\spherenormal
	& : = 
		\frac{
		\newL
		-
		\newuL}{\sqrt{2 \MagnitueofinnerproductofnewLandnewuL}}
		= 
		\frac{
		\ReciprocalLunitAppliedtoTimeFunction \Lunit
		-
		\ReciprocaluLunitAppliedtoTimeFunction \uLunit}{\sqrt{2 \MagnitueofinnerproductofnewLandnewuL}},
\end{align}
where the second equality follows from \eqref{E:RELATIONBETWEENCARTESIANNORMALIZEDNULLVECTORFIELDSANDEIKONALFUNCTIONORMALIZEDNULLVECTORFIELDS}.
\end{definition}

From 
\eqref{E:DOUBLENULLTIMEFUNCTION},
\eqref{E:DOUBLENULLSPHERENORMAL},
 and \eqref{E:GEOMETRICCOORDINATENULLVECTORFIELDSAPPLIEDTOGEOMETRICCOORDINATES}, 
it follows that $\spherenormal \Timefunction = 0$, that is, that $\spherenormal$ is $\widetilde{\Sigma}_{\Timefunction}$-tangent.
Since $\newL$ and $\newuL$ are $\gfour$-orthogonal to $\mathcal{S}_{u,\underline{u}}$, it follows from \eqref{E:DOUBLENULLSPHERENORMAL}
that $\spherenormal$ is also $\gfour$-orthogonal to $\mathcal{S}_{u,\underline{u}}$.
We also note that $\mbox{\upshape span} \lbrace \newuL, \newL \rbrace$ is equal to the $\gfour$-orthogonal complement of $\mathcal{S}_{u,\underline{u}}$.
Moreover, from the first equality in \eqref{E:DOUBLENULLSPHERENORMAL},
\eqref{E:DOUBLENULLALLTHENEWVECTORFIELDSARENULL},
and \eqref{E:INNERPRODUCTOFEIKONALFUNCTIONNORMALIZEDNULLVECTORFIELDS},
we compute that
\begin{align} \label{E:DOUBLENULLSPHERENORMALISUNITLENGTH}
	\gfour(\spherenormal,\spherenormal)
	& = 1.
\end{align}

\begin{remark}[The orientation of $\spherenormal$ and the relevance for the divergence theorem]
	\label{R:DOUBLENULLORIENTATIONOFSPHERENORMAL}
	From the above discussion, it follows that the vectorfield denoted by ``$\spherenormal$''
	in \eqref{E:DOUBLENULLSPHERENORMAL} has the same properties
	as the vectorfield denoted by the same symbol in Sects.\,\ref{S:SPACETIMEDOMAINS}-\ref{S:APRIORI} 
	(see Def.\,\ref{D:HYPNORMANDSPHEREFORMDEFS}).
	We highlight that $\spherenormal$ points outwards to $\widetilde{\Sigma}_{\Timefunction}$ at its outer boundary
	while $\spherenormal$ points inwards to $\widetilde{\Sigma}_{\Timefunction}$ at its inner boundary;
	see just below 
	\eqref{E:DOUBLENULLLOWERREGIONBOUNDARYOFSPACELIKEHYPERSURFACEPORTION}-\eqref{E:DOUBLENULLUPPERREGIONBOUNDARYOFSPACELIKEHYPERSURFACEPORTION}
	for the definitions of the inner and outer boundaries of $\widetilde{\Sigma}_{\Timefunction}$.
	The precise orientation of $\spherenormal$ will be important for the sign of various terms when we apply the divergence theorem
	on $\widetilde{\Sigma}_{\Timefunction}$;
	see Fig.\,\ref{F:DOUBLENULL}.
\end{remark}

Next, we note that straightforward calculations imply
that the vectorfields
$\tophypnorm$ 
and
$\modtophypnorm$
from Def.\,\ref{D:HYPNORMANDSPHEREFORMDEFS} can be expressed as follows in the present context of double-null foliations:
\begin{align} \label{E:DOUBLENULLTOPYHYPNORMANDMODTOPHYPNORM}
	\tophypnorm
	& = \frac{\newL + \newuL}{\ReciprocalLunitAppliedtoTimeFunction + \ReciprocaluLunitAppliedtoTimeFunction}
		= \frac{\ReciprocalLunitAppliedtoTimeFunction \Lunit + \ReciprocaluLunitAppliedtoTimeFunction \uLunit}{\ReciprocalLunitAppliedtoTimeFunction + \ReciprocaluLunitAppliedtoTimeFunction},
	&
	\modtophypnorm
	& = \frac{\newL + \newuL}{2}
		= \frac{(\ReciprocalLunitAppliedtoTimeFunction + \ReciprocaluLunitAppliedtoTimeFunction)\tophypnorm}{2}.
\end{align}

Moreover, straightforward calculations based on \eqref{E:DOUBLENULLALLTHENEWVECTORFIELDSARENULL} 
yield that 
$\tophypnorm$ 
and
$\modtophypnorm$
are $\gfour$-orthogonal to the vectorfield
$\spherenormal$ defined in \eqref{E:DOUBLENULLSPHERENORMAL}:
\begin{align} \label{E:DOUBLENULLTOPYHYPNORMANDMODTOPHYPNORMAREORTHGONALTOSPHERENORMAL}
	\gfour(\tophypnorm,\spherenormal)
	& 
	= \gfour(\modtophypnorm,\spherenormal)
	= 0.
\end{align}

In addition, using 
\eqref{E:DOUBLENULLALLTHENEWVECTORFIELDSARENULL},
\eqref{E:INNERPRODUCTOFEIKONALFUNCTIONNORMALIZEDNULLVECTORFIELDS},
and \eqref{E:DOUBLENULLTOPYHYPNORMANDMODTOPHYPNORM},
we compute that the scalar function $\lengthofmodtophypnorm := \sqrt{- \gfour(\modtophypnorm,\modtophypnorm)} > 0$
defined in \eqref{E:LENGTHOFTOPHYPNORMNORMALIZEDAGAINSTTIMEFUNCTION}
can be expressed as follows:
\begin{align} \label{E:DOUBLENULLLENTHOFMODTYPHYPNORM}
	\lengthofmodtophypnorm
	& = \sqrt{\frac{\MagnitueofinnerproductofnewLandnewuL}{2}}.
\end{align}
	
	Next, we define the scalar functions 
	$\uposinnerproduct$,
	$\posinnerproduct$,
	$\seconduposinnerproduct$,
	and
	$\secondposinnerproduct$
	as follows:
	\begin{subequations}
		\begin{align} \label{E:DOUBLENULLCASEINGOINGCONDITION}
			\uposinnerproduct
			& :=
			- \gfour(\spherenormal,\uLunit) 
			= \frac{-1}{\sqrt{2 \MagnitueofinnerproductofnewLandnewuL}} \gfour(\newL,\uLunit)
			= - \ReciprocaluLunitAppliedtoTimeFunction^{-1} \frac{1}{\sqrt{2 \MagnitueofinnerproductofnewLandnewuL}} \gfour(\newL,\newuL)
			= \ReciprocaluLunitAppliedtoTimeFunction^{-1} \sqrt{\frac{\MagnitueofinnerproductofnewLandnewuL}{2}}
			= \frac{\underline{\upmu}}{\sqrt{2 \MagnitueofinnerproductofnewLandnewuL}}
			> 0,
				\\
			\posinnerproduct
			& :=
			\gfour(\spherenormal,\Lunit) 
			= \frac{-1}{\sqrt{2 \MagnitueofinnerproductofnewLandnewuL}} \gfour(\newuL,\Lunit)
			= - \ReciprocalLunitAppliedtoTimeFunction^{-1} \frac{1}{\sqrt{2 \MagnitueofinnerproductofnewLandnewuL}} \gfour(\newL,\newuL)
			= \ReciprocalLunitAppliedtoTimeFunction^{-1} \sqrt{\frac{\MagnitueofinnerproductofnewLandnewuL}{2}}
			= \frac{\upmu}{\sqrt{2 \MagnitueofinnerproductofnewLandnewuL}}
			> 0,
				\label{E:DOUBLENULLCASEDUALINGOINGCONDITION} \\
			\seconduposinnerproduct 
			& := 
			- \gfour(\uLunit,\tophypnorm)
			= \frac{-\gfour(\uLunit,\newL)}{\ReciprocalLunitAppliedtoTimeFunction + \ReciprocaluLunitAppliedtoTimeFunction}
			= \frac{-\ReciprocalLunitAppliedtoTimeFunction \gfour(\uLunit,\Lunit)}{\ReciprocalLunitAppliedtoTimeFunction + \ReciprocaluLunitAppliedtoTimeFunction}
			= \frac{\MagnitueofinnerproductofnewLandnewuL}{\ReciprocaluLunitAppliedtoTimeFunction (\ReciprocalLunitAppliedtoTimeFunction + \ReciprocaluLunitAppliedtoTimeFunction)}
			> 0,
			\label{E:DOUBLENULLCASESECONDINGOINGCONDITION}
				\\
			\secondposinnerproduct 
			& := 
				- \gfour(\Lunit,\tophypnorm)
			= \frac{-\gfour(\Lunit,\newuL)}{\ReciprocalLunitAppliedtoTimeFunction + \ReciprocaluLunitAppliedtoTimeFunction}
			= \frac{-\ReciprocaluLunitAppliedtoTimeFunction \gfour(\Lunit,\uLunit)}{\ReciprocalLunitAppliedtoTimeFunction + \ReciprocaluLunitAppliedtoTimeFunction}
			= \frac{\MagnitueofinnerproductofnewLandnewuL}{\ReciprocalLunitAppliedtoTimeFunction (\ReciprocalLunitAppliedtoTimeFunction + \ReciprocaluLunitAppliedtoTimeFunction)}
			> 0,
			\label{E:DOUBLENULLCASESEDUALSECONDINGOINGCONDITION}
		\end{align}
		\end{subequations}
		where the further equalities in \eqref{E:DOUBLENULLCASEINGOINGCONDITION}-\eqref{E:DOUBLENULLCASESEDUALSECONDINGOINGCONDITION}
		follow from straightforward computations.
		
	\begin{remark}[On the signs of $\uposinnerproduct$,
	$\posinnerproduct$,
	$\seconduposinnerproduct$,
	and
	$\secondposinnerproduct$]
	\label{E:POSITIVITYOFKEYSCALARINNERPRODUCTFUNCTIONS}
	We have chosen the signs in \eqref{E:DOUBLENULLCASEINGOINGCONDITION}-\eqref{E:DOUBLENULLCASESEDUALSECONDINGOINGCONDITION}
	so that
	 $\uposinnerproduct$,
	$\posinnerproduct$,
	$\seconduposinnerproduct$,
	and
	$\secondposinnerproduct$
	are positive.
	We note that the functions
	``$\uposinnerproduct$''
	and
	``$\seconduposinnerproduct$''
	defined in \eqref{E:DOUBLENULLCASEINGOINGCONDITION} and \eqref{E:DOUBLENULLCASESECONDINGOINGCONDITION}
	have the same properties
	as the scalar functions denoted by the same symbols in Sects.\,\ref{S:SPACETIMEDOMAINS}-\ref{S:APRIORI} 
	(see \eqref{E:INGOINGCONDITION}-\eqref{E:SECONDINGOINGCONDITION}).
\end{remark}

\subsubsection{First fundamental forms and projections}
\label{SSS:DOUBLENULLFIRSTFUNDANDPROJECTIONS}
Let $\gsphere$, $\gsphere^{-1}$, and $\sphereproject$ be the tensorfields defined by the following
equations,
where we consider $\gfour$, $\MagnitueofinnerproductofnewLandnewuL$, $\newuL$, $\newL$ to have already been defined by 
\eqref{E:ACOUSTICALMETRIC}, \eqref{E:RECIPROCALOFNEGATIVENULLGEODESICVECTORFIELDINNERPRODUCT}, and \eqref{E:DOUBLENULLEIKONALFUNCTIONNORMALIZEDNULLVECTORFIELDS},
and $\updelta_{\ \beta}^{\alpha}$ is the Kronecker delta:
\begin{subequations}
\begin{align}
	\gfour_{\alpha \beta}
	& = 
		-
		\frac{1}{\MagnitueofinnerproductofnewLandnewuL}
		\newL_{\alpha} \newuL_{\beta}
		-
		\frac{1}{\MagnitueofinnerproductofnewLandnewuL}
		\newuL_{\alpha} \newL_{\beta}
		+
		\gsphere_{\alpha \beta},
			\label{E:DOUBLENULLSPHEREFIRSTFUNDDEFININGEQUATION} \\
	(\gfour^{-1})^{\alpha \beta}
	& = 
		-
		\frac{1}{\MagnitueofinnerproductofnewLandnewuL}
		\newL^{\alpha} \newuL^{\beta}
		-
		\frac{1}{\MagnitueofinnerproductofnewLandnewuL}
		\newuL^{\alpha} \newL^{\beta}
		+
		(\gsphere^{-1})^{\alpha \beta},
			\label{E:DOUBLENULLSPHEREINVERSEFIRSTFUNDDEFININGEQUATION} \\
		\sphereproject_{\ \beta}^{\alpha}
	& = \updelta_{\ \beta}^{\alpha}
			+
			\frac{1}{\MagnitueofinnerproductofnewLandnewuL}
			\newL^{\alpha} \newuL_{\beta}
			+
			\frac{1}{\MagnitueofinnerproductofnewLandnewuL}
			\newuL^{\alpha} \newL_{\beta}.
			\label{E:DOUBLENULLSPHEREPROJECTIONDEFININGEQUATION}
\end{align}
\end{subequations}
Next, we recall that $\mbox{\upshape span} \lbrace \newuL, \newL \rbrace$ is equal to the $\gfour$-orthogonal complement of $\mathcal{S}_{u,\underline{u}}$.
With the help of 
\eqref{E:INNERPRODUCTOFEIKONALFUNCTIONNORMALIZEDNULLVECTORFIELDS},
it is straightforward to check that 
$\gsphere$ is the first fundamental form of $\mathcal{S}_{u,\underline{u}}$,
that 
$\gsphere^{-1}$ is the inverse first fundamental form of $\mathcal{S}_{u,\underline{u}}$,
and that $\sphereproject$ $\gfour$-orthogonal projection onto $\mathcal{S}_{u,\underline{u}}$
in the sense that these three tensorfields have the properties
described in Lemma~\ref{L:BASICPROPSOFFUNDAMENTALFORMSANDPROJECTIONS}
(where ``$\mbox{\upshape span} \lbrace \tophypnorm, \spherenormal \rbrace$'' in 
\eqref{E:GSPHEREVANISHESONSPANOFTOPHYPNORMANDSPHERENORMAL} and \eqref{E:STPROJECTIONANNIHILATESNORMALS} 
is equal to ``$\mbox{\upshape span} \lbrace \newuL, \newL \rbrace$'').
In particular, when restricted to $\mathcal{S}_{u,\underline{u}}$, $\gsphere$ is the Riemannian metric
induced by $\gfour$.

\subsubsection{Geometric decompositions of various vectorfields}
\label{SSS:DOUBLENULLDECOMPOSITIONSOFVECTORFIELDS}

We start by defining the following two vectorfields:
\begin{align} \label{E:DOUBLENULLSPECIALGENERATOR}
		\uspecialgen
		& := 
				\Transport
				-
				\frac{1}{\uposinnerproduct}
				\spherenormal,
		&
		\specialgen
		& := 
				\Transport
				+
				\frac{1}{\posinnerproduct}
				\spherenormal.
	\end{align}
	Note that the vectorfield
	``$\uspecialgen$''
	defined in \eqref{E:DOUBLENULLSPECIALGENERATOR}
	has the same properties
	as the vectorfield denoted by the same symbol in Sects.\,\ref{S:SPACETIMEDOMAINS}-\ref{S:APRIORI} 
	(see \eqref{E:SPECIALGENERATOR}).

	Next, we define the vectorfields $\utang$ and $\tang$ by demanding that the following
	identities hold, where 
	$\seconduposinnerproduct$,
	$\secondposinnerproduct$,
	$\uspecialgen$, 
	and $\specialgen$
	are defined by 
	\eqref{E:DOUBLENULLCASESECONDINGOINGCONDITION},
	\eqref{E:DOUBLENULLCASESEDUALSECONDINGOINGCONDITION},
	and
	\eqref{E:DOUBLENULLSPECIALGENERATOR}:
	\begin{align} \label{E:DOUBLENULLSPECIALGENERATORTANGENTIDENTITY}
		\uspecialgen
		& = 
			\frac{1}{\seconduposinnerproduct}
			\uLunit 
			+ 
			\utang,
		&
		\specialgen
		& = 
			\frac{1}{\secondposinnerproduct}
			\Lunit 
			+ 
			\tang.
	\end{align}
	
From arguments nearly identical to those used in the proof of Lemma~\ref{L:KEYIDBETWEENVARIOUSVECTORFIELDS},
we find that $\utang$ and $\tang$ are both $\mathcal{S}_{u,\underline{u}}$-tangent and that
the following identities hold:
\begin{subequations}
\begin{align}
	\Transport
	& = \frac{1}{\seconduposinnerproduct}
			\uLunit 
			+ 
			\frac{1}{\uposinnerproduct}
			\spherenormal
			+ 
			\utang
			=
			\uspecialgen
			+
			\frac{1}{\uposinnerproduct}
			\spherenormal,
				\label{E:TRANSPORTINTERMSOFINGOINGNULL}
					\\
\Transport
	& = 
			\frac{1}{\secondposinnerproduct}
			\Lunit 
			- 
			\frac{1}{\posinnerproduct}
			\spherenormal
			+ 
			\tang
			=
			\specialgen
			-
			\frac{1}{\posinnerproduct}
			\spherenormal.
			\label{E:TRANSPORTINTERMSOFOUTGOINGNULL}
\end{align}
\end{subequations}
Note that the vectorfield
	``$\utang$''
 appearing in \eqref{E:DOUBLENULLSPECIALGENERATORTANGENTIDENTITY} and \eqref{E:TRANSPORTINTERMSOFINGOINGNULL}
	has the same properties
	as the vectorfield denoted by the same symbol in Sects.\,\ref{S:SPACETIMEDOMAINS}-\ref{S:APRIORI} 
	(see \eqref{E:STUTANGENTPARTOFTRANSPORT}).

Next, we define the following two vectorfields:
\begin{align} \label{E:DOUBLENULLREGULARFORMRESCALEDVERSIONOFHYPNORMMINUSNEWGEN}
		\urescalednewgenminushypnorm
		& := \frac{\uLunit - \tophypnorm}{\seconduposinnerproduct},
		&
		\rescalednewgenminushypnorm
		& := \frac{\Lunit - \tophypnorm}{\secondposinnerproduct}.
	\end{align}
Note that the vectorfield
	``$\urescalednewgenminushypnorm$''
defined in \eqref{E:DOUBLENULLREGULARFORMRESCALEDVERSIONOFHYPNORMMINUSNEWGEN}
	has the same properties
	as the vectorfield denoted by the same symbol in Sects.\,\ref{S:SPACETIMEDOMAINS}-\ref{S:APRIORI} 
	(see \eqref{E:REGULARFORMRESCALEDVERSIONOFHYPNORMMINUSNEWGEN}).
	Moreover,
	arguments nearly identical to those used in the proof of
	Lemma~\ref{L:PROPERTIESOFRESCALEDGENERATORMINUSSIDEHYPNORM}
	yield that $\urescalednewgenminushypnorm$ and $\rescalednewgenminushypnorm$
	are both $\Sigma_t$-tangent.

	Finally, we note that arguments nearly identical to those used in the proof of
	Lemma~\ref{L:NEWDECOMPOSITIONOFCOORDINATEPARTIALDERIVATIVEVECTORFIELDS}
	yield that there exist $\mathcal{S}_{u,\underline{u}}$-tangent vectorfields 
	$\utandecompvectorfielddownarg{\alpha}$
	and
	$\tandecompvectorfielddownarg{\alpha}$ 
	such that for $\alpha = 0,1,2,3$, the following identities hold:
	\begin{subequations}
	\begin{align} \label{E:DOUBLENULLHBARDECOMPOSITIONOFCOORDINATEPARTIALDERIVATIVEVECTORFIELDS}
			\partial_{\alpha}
		& = 
			-
			\uLunit_{\alpha} 
			\Transport
			+
			\urescalednewgenminushypnorm_{\alpha}
			\uLunit
			+
			\utandecompvectorfielddownarg{\alpha},
				\\
		\partial_{\alpha}
		& = 
			-
			\Lunit_{\alpha} 
			\Transport
			+
			\rescalednewgenminushypnorm_{\alpha}
			\Lunit
			+
			\tandecompvectorfielddownarg{\alpha}.
			\label{E:DOUBLENULLHDECOMPOSITIONOFCOORDINATEPARTIALDERIVATIVEVECTORFIELDS}
	\end{align}
\end{subequations}
	Note that the vectorfield
	``$\utandecompvectorfielddownarg{\alpha}$''
	in \eqref{E:DOUBLENULLHBARDECOMPOSITIONOFCOORDINATEPARTIALDERIVATIVEVECTORFIELDS}
	has the same properties
	as the vectorfield denoted by the same symbol in Sects.\,\ref{S:SPACETIMEDOMAINS}-\ref{S:APRIORI} 
	(see \eqref{E:NEWDECOMPOSITIONOFCOORDINATEPARTIALDERIVATIVEVECTORFIELDS}).

\subsubsection{Tensorfields with the same definitions as in Sects.\,\ref{S:SPACETIMEDOMAINS}-\ref{S:APRIORI}}
\label{SSS:TENSORFIELDSWITHSAMEDEFINITIONS}
In the rest of the paper, our convention is that if we refer to a
tensorfield that was not
explicitly defined or constructed in 
Subsubsects.\,\ref{SSS:DOUBLENULLGFOURNULLVECTORFIELDS}-\ref{SSS:DOUBLENULLDECOMPOSITIONSOFVECTORFIELDS},
then it has the same definition that it had in Sects.\,\ref{S:SPACETIMEDOMAINS}-\ref{S:APRIORI}
in terms of the tensorfields from Subsubsects.\,\ref{SSS:DOUBLENULLGFOURNULLVECTORFIELDS}-\ref{SSS:DOUBLENULLDECOMPOSITIONSOFVECTORFIELDS};
we refer to Appendix~\ref{A:APPENDIXFORBULK} as an aid for quickly referencing the relevant definitions.
As an example, we note that the scalar function
$\lengthoftophypnorm$
is defined by \eqref{E:LENGTHOFTOPHYPNORM},
where it is understood that the vectorfield $\tophypnorm$ on RHS~\eqref{E:LENGTHOFTOPHYPNORM}
is as in equation \eqref{E:DOUBLENULLTOPYHYPNORMANDMODTOPHYPNORM}.
Then the vectorfield $\hat{\tophypnorm}$ is understood to be the one defined in \eqref{E:TOPUNITHYPERSURFACENORMAL},
where the scalar function $\lengthoftophypnorm$ on RHS~\eqref{E:TOPUNITHYPERSURFACENORMAL} is as above and the
vectorfield $\tophypnorm$ on RHS~\eqref{E:TOPUNITHYPERSURFACENORMAL}
is as in equation \eqref{E:DOUBLENULLTOPYHYPNORMANDMODTOPHYPNORM}.
Similarly, the $\widetilde{\Sigma}_{\Timefunction}$ projection tensorfield $\topproject$
is defined by \eqref{E:TOPPROJECT},
where the vectorfield $\hat{\tophypnorm}$ on RHS~\eqref{E:TOPPROJECT}
is as above.
As a final example, we note that the vectorfield $\projectedtransport$
is defined in \eqref{E:PROJECTIONOFTRANSPORTONTTOTILDESIGMA},
where the scalar function $\lengthoftophypnorm$ on RHS~\eqref{E:PROJECTIONOFTRANSPORTONTTOTILDESIGMA} is as above,
the tensorfield $\topproject$ on RHS~\eqref{E:PROJECTIONOFTRANSPORTONTTOTILDESIGMA} is as above,
and the vectorfield $\Transport$ on RHS~\eqref{E:PROJECTIONOFTRANSPORTONTTOTILDESIGMA}
is the material derivative vectorfield defined in \eqref{E:MATERIALVECTORVIELDRELATIVECTORECTANGULAR}.

\subsubsection{Volume and area forms}
\label{SSS:DOUBLENULLVOLUMEANDAREAFORMS}
 In this subsubsection, we discuss the volume and area forms that play a role in our ensuing analysis.
\begin{itemize}
	\item As in Sects.\,\ref{S:SPACETIMEDOMAINS}-\ref{S:APRIORI}, 
	$d \varpi_{\gfour}$ denotes the canonical volume form on $\mathcal{M}_{u,\underline{u}}$ induced by $\gfour$,
		$d \varpi_{\topfirstfund}$ denotes the canonical volume form on $\widetilde{\Sigma}_{\Timefunction}$ induced by $\topfirstfund$,
		and $d \varpi_{\gsphere}$ denotes the canonical area form on $\mathcal{S}_{u,\underline{u}}$ induced by $\gsphere$.
	\item We endow $\mathcal{H}_u(\underline{u}_1,\underline{u}_2)$ with the volume form $d \varpi_{\gsphere} d \underline{u'}$,
		where for $\underline{u}' \in [\underline{u}_1,\underline{u}_2]$, 
		$d \varpi_{\gsphere}$ is the area form on $\mathcal{S}_{u,\underline{u}'}$.
	\item Similarly, we endow $\underline{\mathcal{H}}_{\underline{u}}(u_1,u_2)$ with the volume form $d \varpi_{\gsphere} d u'$,
	where for $u' \in [u_1,u_2]$, $d \varpi_{\gsphere}$ is the area form on $\mathcal{S}_{u',\underline{u}}$.
\end{itemize}
By \eqref{E:DOUBLENULLTIMEFUNCTION}, on $\mathcal{H}_u(\underline{u}_1,\underline{u}_2)$, since $u$ is fixed,
we have $d \varpi_{\gsphere} d \underline{u'} = d \varpi_{\gsphere} d \Timefunction'$,
where on the RHS, $d \varpi_{\gsphere}$ is the area form on $\mathcal{S}_{u,\Timefunction' - u}$
and $\Timefunction' = u + \underline{u'}$.
Similarly, on $\underline{\mathcal{H}}_{\underline{u}}(u_1,u_2)$, we have
$d \varpi_{\gsphere} d u' = d \varpi_{\gsphere} d \Timefunction'$,
where on the RHS, $d \varpi_{\gsphere}$ is the area form on $\mathcal{S}_{\Timefunction' - \underline{u},\underline{u}}$
and $\Timefunction' = u' + \underline{u}$.
We also note that 
the identity \eqref{E:SPACETIMEVOLUMEFORMEXPRESSIONWITHRESPECTTOTIMEFUNCTION} remains valid in the present context,
where in the present context, $\lengthofmodtophypnorm > 0$ verifies \eqref{E:DOUBLENULLLENTHOFMODTYPHYPNORM}.

\subsubsection{Integral identities involving $\mathcal{H}_u$, $\underline{\mathcal{H}}_{\underline{u}}$, and $\mathcal{S}_{u,\underline{u}}$}
\label{SSS:INTEGRALIDENTITIESSPHERES}
In this subsubsection, we provide an analog of Lemma~\ref{L:DIFFERENTIATIONANDINTEGRALIDENTITEISINVOLVINGST}
for our double-null foliations.

\begin{lemma}[Integral identities involving $\mathcal{H}_u$, $\underline{\mathcal{H}}_{\underline{u}}$, and $\mathcal{S}_{u,\underline{u}}$]
\label{L:DOUBLENULLINTEGRALIDENTITIESINVOLVINGSPHERES}
	Let $f$ be a smooth function defined on $\mathcal{M}_{U,\underline{U}}$.
	Let 
	$u$, $u_1$, $u_2$, $\underline{u}$, $\underline{u}_1$, and $\underline{u}_2$
	be real numbers satisfying
	$1 \leq u \leq U$,
	$1 \leq u_1 \leq u_2 \leq U$, 
	$0 \leq \underline{u} \leq \underline{U}$,
	and
	$0 \leq \underline{u}_1 \leq \underline{u}_2 \leq \underline{U}$.
	Let $\newL$ be the $\mathcal{H}_u$-tangent vectorfield defined in
	\eqref{E:DOUBLENULLEIKONALFUNCTIONNORMALIZEDNULLVECTORFIELDS},
	and let $\newuL$ be the $\underline{\mathcal{H}}_{\underline{u}}$-tangent vectorfield defined in
	\eqref{E:DOUBLENULLEIKONALFUNCTIONNORMALIZEDNULLVECTORFIELDS}.
	Let $\gsphere$ be the first fundamental form of $\mathcal{S}_{u,\underline{u}}$ 
	(see Subsubsect.\,\ref{SSS:DOUBLENULLFIRSTFUNDANDPROJECTIONS}).
	Then the following identities hold,
	where $\Lie_X$ denote Lie differentiation with respect to the vectorfield $X$,
	the definition of the area form $d \varpi_{\gsphere}$ is provided in Subsubsect.\,\ref{SSS:DOUBLENULLVOLUMEANDAREAFORMS},
	and we refer to Subsubsect.\,\ref{SSS:TENSORFIELDSWITHSAMEDEFINITIONS} regarding the notation:
	\begin{subequations}
	\begin{align}   \label{E:DOUBLENULLKEYHBARINTEGRALIDENTITY}
			\int_{\underline{\mathcal{H}}_{\underline{u}}(u_1,u_2)}
				(\newuL f)
			\, d \varpi_{\gsphere} 
			d u'
			& =
				-
				\frac{1}{2}
				\int_{\underline{\mathcal{H}}_{\underline{u}}(u_1,u_2)}
					f [(\gsphere^{-1})^{\alpha \beta} \Lie_{\newuL} \gsphere_{\alpha \beta}]
				\, d \varpi_{\gsphere} 
				d u'
				+
				\int_{\mathcal{S}_{u_2,\underline{u}}}
					f
				\, d \varpi_{\gsphere}
				-
				\int_{\mathcal{S}_{u_1,\underline{u}}}
					f
				\, d \varpi_{\gsphere},
				\\
		\int_{\mathcal{H}_u(\underline{u}_1,\underline{u}_2)}
				(\newL f)
			\, d \varpi_{\gsphere} 
			d \underline{u'}
			& =
				-
				\frac{1}{2}
				\int_{\mathcal{H}_u(\underline{u}_1,\underline{u}_2)}
					f [(\gsphere^{-1})^{\alpha \beta} \Lie_{\newL} \gsphere_{\alpha \beta}]
				\, d \varpi_{\gsphere} 
				d \underline{u'}
				+
				\int_{\mathcal{S}_{u,\underline{u}_2}}
					f
				\, d \varpi_{\gsphere}
				-
				\int_{\mathcal{S}_{u,\underline{u}_1}}
					f
				\, d \varpi_{\gsphere}.
					\label{E:DOUBLENULLKEYHINTEGRALIDENTITY}
		\end{align}
	\end{subequations}
	\end{lemma}
	
	\begin{proof}
		In view of the normalization conditions in \eqref{E:GEOMETRICCOORDINATENULLVECTORFIELDSAPPLIEDTOGEOMETRICCOORDINATES},
		we see that the two identities in the lemma can be derived by a straightforward modification 
		of the proof of Lemma~\ref{L:DIFFERENTIATIONANDINTEGRALIDENTITEISINVOLVINGST}.
	\end{proof}

\subsection{The main theorem for double-null foliations: Remarkable Hodge-transport integral identities for $\vortrenormalized$ and $\GradEnt$}
\label{SS:DOUBLENULLMAINTHEOREM}
We now state our main theorem, which is an analog of Theorem~\ref{T:MAINREMARKABLESPACETIMEINTEGRALIDENTITY}
for double-null foliated regions.
The proof is located in Subsect.\,\ref{SS:PROOFOFTHEOREMDOUBLENULLMAINTHEOREM}.
See Remark~\ref{R:DOUBLENULLNULLHYPERSURFACEERRORTERMSTANGENTIALDERIVATIVES} regarding the structure of the error terms on the 
$\gfour$-null hypersurfaces.

\begin{theorem}[Double-null foliations: Remarkable Hodge-transport integral identities for $\vortrenormalized$ and $\GradEnt$]
	\label{T:DOUBLENULLMAINTHEOREM}
	Let $1 < U$ and $0 < \underline{U}$, and let $\mathcal{M}_{U,\underline{U}}$ be a spacetime region that is
	double-null foliated by acoustical eikonal functions $u$ and $\underline{u}$,
	as is described in 
	Subsect.\,\ref{SS:DOUBLENULL}	
	(in particular, on $\mathcal{M}_{U,\underline{U}}$, 
	we have
	$1 \leq u \leq U$
	and $0 \leq \underline{u} \leq \underline{U}$).
	Let $\mathscr{Q}(\pmb{\partial} \mathbf{X},\pmb{\partial} \mathbf{X})$
	be the quadratic form defined by
	\eqref{E:MAINQUADRATICFORMFORCONTROLLINGFIRSTDERIVATIVESOFSPECIFICVORTICITYANDENTROPYGRADIENT},
	and recall that the positive definite nature of $\mathscr{Q}$ was revealed in
	Lemma~\ref{L:POSITIVITYPROPERTIESOFVARIOUSQUADRATICFORMS}.
	Let $\weight$ be an arbitrary scalar function.
	Let $\lengthofmodtophypnorm > 0$ be the scalar function defined in \eqref{E:LENGTHOFTOPHYPNORMNORMALIZEDAGAINSTTIMEFUNCTION}
	(recall also the identity \eqref{E:DOUBLENULLLENTHOFMODTYPHYPNORM}),
	let $\ReciprocalLunitAppliedtoTimeFunction > 0$ and $\ReciprocaluLunitAppliedtoTimeFunction > 0$ be the scalar functions defined in 
	\eqref{E:RATIOOFNULLGEOSICINNERPRODUCTANDFOLIATIONDENSITY}
	(see also \eqref{E:POSITIVITYOFFOLIATIONDENSITYANDNULLGEODESICINNERPRODUCTETC}),
	and let 
	$\uposinnerproduct > 0$, 
	$\seconduposinnerproduct > 0$,
	$\posinnerproduct > 0$, 
	and
	$\secondposinnerproduct > 0$
	be the scalar functions from \eqref{E:DOUBLENULLCASEINGOINGCONDITION}-\eqref{E:DOUBLENULLCASESEDUALSECONDINGOINGCONDITION}.
	For smooth solutions (see Remark~\ref{R:SMOOTHNESSNOTNEEDED})
	to the compressible Euler equations \eqref{E:TRANSPORTDENSRENORMALIZEDRELATIVECTORECTANGULAR}-\eqref{E:ENTROPYTRANSPORT},
	the following integral identities hold,
	where the definitions of the volume and area forms are provided in Subsubsect.\,\ref{SSS:DOUBLENULLVOLUMEANDAREAFORMS},
	and we refer to Subsubsect.\,\ref{SSS:TENSORFIELDSWITHSAMEDEFINITIONS} regarding the notation:
	\begin{subequations}
	\begin{align}
		&
		\int_{\mathcal{M}_{U,\underline{U}}}	
			\weight
			\lengthofmodtophypnorm^{-1}
			\mathscr{Q}(\pmb{\partial} \vortrenormalized,\pmb{\partial} \vortrenormalized)
		\, d \varpi_{\gfour}
		+
		\int_{\mathcal{S}_{U,\underline{U}}}
			\weight
			\left\lbrace
				\frac{\posinnerproduct}{\secondposinnerproduct \ReciprocalLunitAppliedtoTimeFunction}
				+
				\frac{\uposinnerproduct}{\seconduposinnerproduct \ReciprocaluLunitAppliedtoTimeFunction} 
			\right\rbrace
			|\angvortrenormalized|_{\gsphere}^2
		\, d \varpi_{\gsphere} 
			\label{E:DOUBLENULLSPACETIMEREMARKABLEIDENTITYSPECIFICVORTICITY} \\
		& = 
		\int_{\mathcal{S}_{U,0}}
			\weight
			\left\lbrace
				\frac{\posinnerproduct}{\secondposinnerproduct \ReciprocalLunitAppliedtoTimeFunction}
				+
				\frac{\uposinnerproduct}{\seconduposinnerproduct \ReciprocaluLunitAppliedtoTimeFunction} 
			\right\rbrace
			|\angvortrenormalized|_{\gsphere}^2
		\, d \varpi_{\gsphere}
		+
		\int_{\mathcal{S}_{1,\underline{U}}}
			\weight
			\left\lbrace
				\frac{\posinnerproduct}{\secondposinnerproduct \ReciprocalLunitAppliedtoTimeFunction}
				+
				\frac{\uposinnerproduct}{\seconduposinnerproduct \ReciprocaluLunitAppliedtoTimeFunction} 
			\right\rbrace
			|\angvortrenormalized|_{\gsphere}^2
		\, d \varpi_{\gsphere}
		-
		\int_{\mathcal{S}_{1,0}}
			\weight
			\left\lbrace
				\frac{\posinnerproduct}{\secondposinnerproduct \ReciprocalLunitAppliedtoTimeFunction}
				+
				\frac{\uposinnerproduct}{\seconduposinnerproduct \ReciprocaluLunitAppliedtoTimeFunction} 
			\right\rbrace
			|\angvortrenormalized|_{\gsphere}^2
		\, d \varpi_{\gsphere}
			\notag \\
	& \ \
	+	
	\int_{\mathcal{M}_{U,\underline{U}}}	
			\weight
			\lengthofmodtophypnorm^{-1}
			\left\lbrace
				\frac{1}{2}|\mathfrak{A}^{(\vortrenormalized)}|_{\topfirstfund}^2
				+
				|\mathfrak{B}_{(\vortrenormalized)}|_{\topfirstfund}^2
				+
				\mathfrak{C}^{(\vortrenormalized)}
				+
				\mathfrak{D}^{(\vortrenormalized)}
				+
				\mathfrak{J}_{(Coeff)}[\vortrenormalized,\pmb{\partial} \vortrenormalized]
			\right\rbrace
	\, d \varpi_{\gfour}
		\notag \\
	& \ \
	+	
	\int_{\mathcal{M}_{U,\underline{U}}}	
			\lengthofmodtophypnorm^{-1}
			\mathfrak{J}_{(\pmb{\partial} \weight)}[\vortrenormalized,\pmb{\partial} \vortrenormalized]
	\, d \varpi_{\gfour}
		\notag \\
	& \ \
		+
		\int_{\underline{\mathcal{H}}_{\underline{U}}(1,U)}
			\left\lbrace
				\underline{\mathfrak{H}}_{(\pmb{\partial} \weight)}[\vortrenormalized]
				+
				\weight 
				\underline{\mathfrak{H}}[\vortrenormalized]
				+
				\weight 
				\underline{\mathfrak{H}}_{(1)}[\vortrenormalized]
			\right\rbrace
		\, d \varpi_{\gsphere} d u'
			\notag \\
	& \ \
		+
		\int_{\mathcal{H}_U(0,\underline{U})}
			\left\lbrace
				\mathfrak{H}_{(\pmb{\partial} \weight)}[\vortrenormalized]
				+
				\weight
				\mathfrak{H}[\vortrenormalized]
				+
				\weight
				\mathfrak{H}_{(1)}[\vortrenormalized]
			\right\rbrace
		\, d \varpi_{\gsphere} d \underline{u}'
		\notag \\
	& \ \
		-
		\int_{\underline{\mathcal{H}}_{0}(1,U)}
			\left\lbrace
				\underline{\mathfrak{H}}_{(\pmb{\partial} \weight)}[\vortrenormalized]
				+
				\weight
				\underline{\mathfrak{H}}[\vortrenormalized]
				+
				\weight
				\underline{\mathfrak{H}}_{(1)}[\vortrenormalized]
			\right\rbrace
		\, d \varpi_{\gsphere} d u'
			\notag \\
	& \ \	
		-
		\int_{\mathcal{H}_1(0,\underline{U})}
			\left\lbrace
				\mathfrak{H}_{(\pmb{\partial} \weight)}[\vortrenormalized]
				+
				\weight
				\mathfrak{H}[\vortrenormalized]
				+
				\weight
				\mathfrak{H}_{(1)}[\vortrenormalized]
			\right\rbrace
		\, d \varpi_{\gsphere} d \underline{u}',
			\notag
	\end{align}
	
	\begin{align}
	&
		\int_{\mathcal{M}_{U,\underline{U}}}	
			\weight
			\lengthofmodtophypnorm^{-1}
			\mathscr{Q}(\pmb{\partial} \GradEnt,\pmb{\partial} \GradEnt)
		\, d \varpi_{\gfour}
		+
		\int_{\mathcal{S}_{U,\underline{U}}}
			\weight
			\left\lbrace
				\frac{\posinnerproduct}{\secondposinnerproduct \ReciprocalLunitAppliedtoTimeFunction}
				+
				\frac{\uposinnerproduct}{\seconduposinnerproduct \ReciprocaluLunitAppliedtoTimeFunction} 
			\right\rbrace
			|\angGradEnt|_{\gsphere}^2
		\, d \varpi_{\gsphere} 
			\label{E:DOUBLENULLSPACETIMEREMARKABLEIDENTITYENTROPYGRADIENT} \\
		& = 
		\int_{\mathcal{S}_{U,0}}
			\weight
			\left\lbrace
				\frac{\posinnerproduct}{\secondposinnerproduct \ReciprocalLunitAppliedtoTimeFunction}
				+
				\frac{\uposinnerproduct}{\seconduposinnerproduct \ReciprocaluLunitAppliedtoTimeFunction} 
			\right\rbrace
			|\angGradEnt|_{\gsphere}^2
		\, d \varpi_{\gsphere}
		+
		\int_{\mathcal{S}_{1,\underline{U}}}
			\weight
			\left\lbrace
				\frac{\posinnerproduct}{\secondposinnerproduct \ReciprocalLunitAppliedtoTimeFunction}
				+
				\frac{\uposinnerproduct}{\seconduposinnerproduct \ReciprocaluLunitAppliedtoTimeFunction} 
			\right\rbrace
			|\angGradEnt|_{\gsphere}^2
		\, d \varpi_{\gsphere}
		-
		\int_{\mathcal{S}_{1,0}}
			\weight
			\left\lbrace
				\frac{\posinnerproduct}{\secondposinnerproduct \ReciprocalLunitAppliedtoTimeFunction}
				+
				\frac{\uposinnerproduct}{\seconduposinnerproduct \ReciprocaluLunitAppliedtoTimeFunction} 
			\right\rbrace
			|\angGradEnt|_{\gsphere}^2
		\, d \varpi_{\gsphere}
			\notag \\
	& \ \
	+	
	\int_{\mathcal{M}_{U,\underline{U}}}	
			\weight
			\lengthofmodtophypnorm^{-1}
			\left\lbrace
				\frac{1}{2}|\mathfrak{A}^{(\GradEnt)}|_{\topfirstfund}^2
				+
				|\mathfrak{B}_{(\GradEnt)}|_{\topfirstfund}^2
				+
				\mathfrak{C}^{(\GradEnt)}
				+
				\mathfrak{D}^{(\GradEnt)}
				+
				\mathfrak{J}_{(Coeff)}[\GradEnt,\pmb{\partial} \GradEnt]
			\right\rbrace
	\, d \varpi_{\gfour}
		\notag \\
	& \ \
	+	
	\int_{\mathcal{M}_{U,\underline{U}}}	
			\lengthofmodtophypnorm^{-1}
			\mathfrak{J}_{(\pmb{\partial} \weight)}[\GradEnt,\pmb{\partial} \GradEnt]
	\, d \varpi_{\gfour}
		\notag \\
	& \ \
		+
		\int_{\underline{\mathcal{H}}_{\underline{U}}(1,U)}
			\left\lbrace
				\underline{\mathfrak{H}}_{(\pmb{\partial} \weight)}[\GradEnt]
				+
				\weight 
				\underline{\mathfrak{H}}[\GradEnt]
				+
				\weight
				\underline{\mathfrak{H}}_{(2)}[\GradEnt]
			\right\rbrace
		\, d \varpi_{\gsphere} d u'
			\notag \\
	& \ \
		+
		\int_{\mathcal{H}_U(0,\underline{U})}
			\left\lbrace
				\mathfrak{H}_{(\pmb{\partial} \weight)}[\GradEnt]
				+
				\weight
				\mathfrak{H}[\GradEnt]
				+
				\weight
				\mathfrak{H}_{(2)}[\GradEnt]
			\right\rbrace
		\, d \varpi_{\gsphere} d \underline{u}'
		\notag \\
	& \ \
		-
		\int_{\underline{\mathcal{H}}_{0}(1,U)}
			\left\lbrace
				\underline{\mathfrak{H}}_{(\pmb{\partial} \weight)}[\GradEnt]
				+
				\weight
				\underline{\mathfrak{H}}[\GradEnt]
				+
				\weight
				\underline{\mathfrak{H}}_{(2)}[\GradEnt]
			\right\rbrace
		\, d \varpi_{\gsphere} d u'
			\notag \\
	& \ \ 
		-
		\int_{\mathcal{H}_1(0,\underline{U})}
			\left\lbrace
				\mathfrak{H}_{(\pmb{\partial} \weight)}[\GradEnt]
				+
				\weight
				\mathfrak{H}[\GradEnt]
				+
				\weight
				\mathfrak{H}_{(2)}[\GradEnt]
			\right\rbrace
		\, d \varpi_{\gsphere} d \underline{u}'.
			\notag
	\end{align}
	\end{subequations}
	
	On RHSs~\eqref{E:DOUBLENULLSPACETIMEREMARKABLEIDENTITYSPECIFICVORTICITY}-\eqref{E:DOUBLENULLSPACETIMEREMARKABLEIDENTITYENTROPYGRADIENT},
	$\mathfrak{A}^{(\vortrenormalized)}$ 
	and
	$\mathfrak{A}^{(\GradEnt)}$
	are two-forms with the Cartesian components
	\begin{subequations}
	\begin{align}
		\mathfrak{A}_{\alpha \beta}^{(\vortrenormalized)}
		&  :=
				2 (\partial_{\beta} \ln \Speed) \vortrenormalized_{\alpha} 
				- 
				2 (\partial_{\alpha} \ln \Speed) \vortrenormalized_{\beta}
				+
				2 \updelta_{\alpha}^0 \vortrenormalized_a \partial_{\beta} v^a
				-
				2 \updelta_{\beta}^0 \vortrenormalized_a \partial_{\alpha} v^a
				\label{E:DOUBLENULLSPECIFICVORTICITYSPACETIMEERRORTERMANTISYMMETRIC} \\
			& \ \
				-
				\Speed^{-4} 
				\exp(-2 \LogDensity) 
				\frac{p_{;\Ent}}{\bar{\varrho}}
				\upepsilon_{\alpha \beta \gamma \delta}
				(\Transport v^{\gamma}) 
				\GradEnt^{\delta}
					\notag
					\\
			& \ \
				+
			\Speed^{-4} 
			\exp(-2 \LogDensity) 
			\frac{p_{;\Ent}}{\bar{\varrho}}
			\upepsilon_{\alpha \beta \gamma \delta}
			 \Transport^{\gamma}
			[\GradEnt^{\delta}
				(\partial_a v^a)
				-
				\GradEnt^a \partial_a v^{\delta}]
			\notag
				\\
		& \ \
			+
			\Speed^{-2}
			\exp(\LogDensity)
			\upepsilon_{\alpha \beta \gamma \delta}
			\Transport^{\gamma}
			\VortVort^{\delta},
			\notag
				\\
			\mathfrak{A}_{\alpha \beta}^{(\GradEnt)}
			& :
				= 
				2 (\partial_{\beta} \ln \Speed) \GradEnt_{\alpha} 
				- 
				2 (\partial_{\alpha} \ln \Speed) \GradEnt_{\beta},
	\end{align}
	\end{subequations}
	$\mathfrak{B}_{(\vortrenormalized)}$ and $\mathfrak{B}_{(\GradEnt)}$
	are $\Sigma_t$-tangent vectorfields with the Cartesian spatial components
	\begin{subequations}
	\begin{align}
			\mathfrak{B}_{(\vortrenormalized)}^i
			& 
			:=
			\vortrenormalized^a \partial_a v^i
			-
			\exp(-2 \LogDensity) \Speed^{-2} \frac{p_{;\Ent}}{\bar{\varrho}} \upepsilon_{iab} (\Transport v^a) \GradEnt^b,	
				\label{E:DOUBLENULLSPECIFICVORTICITYSPACETIMEERRORTERMTRANSPORTDERIVATIVES} \\
			\mathfrak{B}_{(\GradEnt)}^i
			& :=
				- 
				\GradEnt^a \partial_a v^i
				+ 
				\upepsilon_{iab} \exp(\LogDensity) \vortrenormalized^a \GradEnt^b,
				\label{E:DOUBLENULLENTROPYGRADIENTSPACETIMEERRORTERMTRANSPORTDERIVATIVES}
	\end{align}
	\end{subequations}
	$\mathfrak{C}^{(\vortrenormalized)}$ and $\mathfrak{C}^{(\GradEnt)}$ are scalar functions
	defined relative to the Cartesian coordinates by
	\begin{subequations}
	\begin{align}
		\mathfrak{C}^{(\vortrenormalized)}
		& :=
			-
			2 
			(\projectedtransport_a \projectedtransport \vortrenormalized^a)
			\vortrenormalized^b \partial_b \LogDensity
			-
			2
			\lengthoftophypnorm 
			(\projectedtransport_a \projectedtransport \vortrenormalized^a)
			\projectedtransport_b \mathfrak{B}_{(\vortrenormalized)}^b,
				\label{E:DOUBLENULLSPECIFICVORTICITYSPACETIMEERRORTERMNEEDSTOBEABSORBED} \\
	\mathfrak{C}^{(\GradEnt)}
		& :=
			2 
			(\projectedtransport_a \projectedtransport \GradEnt^a)
			\left\lbrace
			\exp(2 \LogDensity) \DivGradEnt 
				+ 
			\GradEnt^b \partial_b \LogDensity
			\right\rbrace
			-
			2
			\lengthoftophypnorm 
			(\projectedtransport_a \projectedtransport \GradEnt^a)
			\projectedtransport_b \mathfrak{B}_{(\GradEnt)}^b,
			\label{E:DOUBLENULLENTROPYGRADIENTSPACETIMEERRORTERMNEEDSTOBEABSORBED}
\end{align}
\end{subequations}
$\mathfrak{D}^{(\vortrenormalized)}$ and $\mathfrak{D}^{(\GradEnt)}$
are scalar functions defined relative to the Cartesian coordinates by
\begin{subequations}
\begin{align}
		\mathfrak{D}^{(\vortrenormalized)}
		&
		:=
			(\vortrenormalized^a \partial_a \LogDensity)^2
			+
			(
				\lengthoftophypnorm
				\projectedtransport_a \mathfrak{B}_{(\vortrenormalized)}^a
			)^2
			+
			2
			\lengthoftophypnorm 
			(\vortrenormalized^a \partial_a \LogDensity)
			\projectedtransport_b \mathfrak{B}_{(\vortrenormalized)}^b,
				\label{E:DOUBLENULLSPECIFICVORTICITYSPACETIMEERRORTERMDIVERGENCEERRORS} \\
		\mathfrak{D}^{(\GradEnt)}
		&
		:=
			\left\lbrace
			\exp(2 \LogDensity) \DivGradEnt 
				+ 
			\GradEnt^a \partial_a \LogDensity
			\right\rbrace^2
			+
			(
			\lengthoftophypnorm
			\projectedtransport_a \mathfrak{B}_{(\GradEnt)}^a
			)^2
			-
			2
			\lengthoftophypnorm 
			\left\lbrace
			\exp(2 \LogDensity) \DivGradEnt 
				+ 
			\GradEnt^a \partial_a \LogDensity
			\right\rbrace
			\projectedtransport_b
			\mathfrak{B}_{(\GradEnt)}^b,
				\label{E:DOUBLENULLENTROPYGRADIENTSPACETIMEERRORTERMDIVERGENCEERRORS} 
	\end{align}
	\end{subequations}
	for $\SigmatTan \in \lbrace \vortrenormalized, \GradEnt \rbrace$,
	the scalar function $\mathfrak{J}_{(Coeff)}[\SigmatTan,\pmb{\partial} \SigmatTan]$ 
	is defined relative to the Cartesian coordinates by
	\begin{align} \label{E:DOUBLENULLELLIPTICHYPERBOLICCURRENTCOEFFICIENTERRORTERM}
		\mathfrak{J}_{(Coeff)}[\SigmatTan,\pmb{\partial} \SigmatTan]
		& =
			\SigmatTan^{\alpha} 
			\gfour_{\beta \gamma}
			(\toppartialarg{\alpha} \hat{\tophypnorm}^{\beta})
			\hat{\tophypnorm} \SigmatTan^{\gamma}
			-
			\SigmatTan_{\alpha} 
			(\toppartialarg{\beta} \hat{\tophypnorm}^{\alpha})
			\hat{\tophypnorm} \SigmatTan^{\beta}
				\\
		& \ \
				+
		\SigmatTan^{\alpha} 
		\hat{\tophypnorm}_{\alpha} 
		(\toppartialarg{\beta} \hat{\tophypnorm}^{\beta}) 
		\toppartialarg{\gamma} \SigmatTan^{\gamma}
			-
			\SigmatTan^{\alpha} \hat{\tophypnorm}_{\alpha} 
			(\toppartialarg{\beta} \hat{\tophypnorm}^{\gamma}) 
			\toppartialarg{\gamma} 
			\SigmatTan^{\beta} 
				\notag \\
		& \ \
		+
		\SigmatTan^{\alpha} 
		\hat{\tophypnorm}_{\beta}
		(\toppartialarg{\alpha} \hat{\tophypnorm}^{\gamma})
		\toppartialarg{\gamma} \SigmatTan^{\beta}
			-
			\SigmatTan^{\alpha} 
			\hat{\tophypnorm}_{\beta} 
			(\toppartialarg{\gamma} \hat{\tophypnorm}^{\gamma})  
			\toppartialarg{\alpha} \SigmatTan^{\beta}
		\notag
			\\
		&  \ \
			+
			\SigmatTan^{\alpha} 
			\hat{\tophypnorm}^{\beta}
			(\toppartialarg{\alpha} \gfour_{\beta \gamma})
			\hat{\tophypnorm} \SigmatTan^{\gamma}
			-
			\SigmatTan^{\alpha} 
			\hat{\tophypnorm}^{\beta}
			(\toppartialarg{\gamma} \gfour_{\alpha \beta})
			\hat{\tophypnorm} \SigmatTan^{\gamma}
				\notag \\
		& \ \
			+
			\frac{1}{2}
			\SigmatTan^{\alpha} 
			\hat{\tophypnorm}_{\beta}
			\hat{\tophypnorm}^{\gamma}
			\hat{\tophypnorm}^{\delta}
			(\toppartialarg{\alpha} \gfour_{\gamma \delta})			
			\hat{\tophypnorm} \SigmatTan^{\beta}
			-
			\frac{1}{2}
			\SigmatTan^{\alpha} \hat{\tophypnorm}_{\alpha} 
			\hat{\tophypnorm}^{\beta}
			\hat{\tophypnorm}^{\gamma}
			(\toppartialarg{\delta} \gfour_{\beta \gamma}) 
			\hat{\tophypnorm} \SigmatTan^{\delta} 
				\notag \\
	& \ \
		+
		2
		\SigmatTan^{\alpha}
		(\toppartialarg{\beta} \gfour_{\alpha \gamma})
		\toppartialuparg{\gamma} \SigmatTan^{\beta}
		-
		2
		\topproject_{\ \beta}^{\alpha}
		\SigmatTan^{\gamma}
		(\toppartialuparg{\delta} \gfour_{\alpha \gamma}) 
		\toppartialarg{\delta} \SigmatTan^{\beta}
				\notag \\
	& \ \
		+
		\frac{1}{2}
		(\topfirstfund^{-1})^{\alpha \beta}
		\SigmatTan^{\gamma} 
		(\toppartialarg{\gamma} \gfour_{\alpha \beta})
		\toppartialarg{\delta} \SigmatTan^{\delta}
		-
		\frac{1}{2}
		(\topfirstfund^{-1})^{\alpha \beta}
		\SigmatTan^{\gamma} 
		(\toppartialarg{\delta} \gfour_{\alpha \beta})
		\toppartialarg{\gamma} \SigmatTan^{\delta}
		\notag
			\\
		& \ \
			+
			\SigmatTan^{\alpha} 
			\SigmatTan^{\beta}
			(\toppartialuparg{\gamma} \gfour_{\alpha \delta})
			(\toppartialuparg{\delta} \gfour_{\beta \gamma})
		-
		\SigmatTan^{\alpha} 
		\SigmatTan^{\beta}
		(\topfirstfund^{-1})^{\gamma \delta} 
		(\toppartialuparg{\kappa} \gfour_{\alpha \gamma}) 
		\toppartialarg{\kappa} \gfour_{\beta \delta},
				\notag
	\end{align}
	for $\SigmatTan \in \lbrace \vortrenormalized, \GradEnt \rbrace$,
	the scalar function $\mathfrak{J}_{(\pmb{\partial} \weight)}[\SigmatTan,\pmb{\partial} \SigmatTan]$ 
	is defined relative to the Cartesian coordinates by
	\begin{align} \label{E:DOUBLENULLELLIPTICHYPERBOLICCURRENTBULKERRORTERMWITHWEIGHTDERIVATIVES}
		\mathfrak{J}_{(\pmb{\partial} \weight)}[\SigmatTan,\pmb{\partial} \SigmatTan]
		:= - J[\SigmatTan] \weight 
		=
		\SigmatTan^{\kappa}  
		(\toppartialarg{\kappa} \weight)
		\toppartialarg{\lambda} \SigmatTan^{\lambda}
		-
		\SigmatTan^{\kappa} 
		(\toppartialarg{\lambda} \weight)
		\toppartialarg{\kappa} \SigmatTan^{\lambda},
	\end{align}
	for $\SigmatTan \in \lbrace \vortrenormalized, \GradEnt \rbrace$,
	the scalar functions 
	$\underline{\mathfrak{H}}_{(\pmb{\partial} \weight)}[\SigmatTan]$,
	$\underline{\mathfrak{H}}[\SigmatTan]$,
	$\mathfrak{H}_{(\pmb{\partial} \weight)}[\SigmatTan]$,
	and
	$\mathfrak{H}[\SigmatTan]$ 
	are defined relative to the Cartesian coordinates by
	\begin{subequations}
	\begin{align} 	\label{E:WEIGHTDERIVATIVESDOUBLENULLMAINTHMINGOINGLATERALBOUNDARYEASYERRORINTEGRANDTERMS}
		\underline{\mathfrak{H}}_{(\pmb{\partial} \weight)}[\SigmatTan]
		& :=
		\frac{\uposinnerproduct}{\seconduposinnerproduct \ReciprocaluLunitAppliedtoTimeFunction}
		|\angV|_{\gsphere}^2
		\newuL
		\weight
		+
		\uposinnerproduct
		|\angV|_{\gsphere}^2
		\utang
		\weight
		+
		\SigmatTan_{\alpha} \spherenormal^{\alpha}
		\angV \weight,
			 \\
		\underline{\mathfrak{H}}[\SigmatTan]
		& :=
			\frac{1}{2} 
			\frac{\uposinnerproduct}{\seconduposinnerproduct \ReciprocaluLunitAppliedtoTimeFunction}  
			|\angV|_{\gsphere}^2
			(\gsphere^{-1})^{\alpha \beta} 
			\Lie_{\newuL} \gsphere_{\alpha \beta}
			\label{E:DOUBLENULLMAINTHMINGOINGLATERALBOUNDARYEASYERRORINTEGRANDTERMS} 
			\\
		& \ \
		+
		\left\lbrace
		\newuL
			\left[\frac{\uposinnerproduct}{\seconduposinnerproduct \ReciprocaluLunitAppliedtoTimeFunction} (\gsphere^{-1})^{\alpha \beta} \right]
		\right\rbrace
		\SigmatTan_{\alpha} 
	  \SigmatTan_{\beta}
		+
		\left\lbrace
			\utang
			\left[\uposinnerproduct (\gsphere^{-1})^{\alpha \beta} \right]
		\right\rbrace
		\SigmatTan_{\alpha} 
	  \SigmatTan_{\beta}
		-
		2 
		\SigmatTan_{\alpha} \angV \Transport^{\alpha}
		\notag \\
		& \ \
				+
			\uposinnerproduct 
			|\angV|_{\gsphere}^2
			\angdiv \utang 
		+
		\spherenormal_{\alpha} \SigmatTan^{\alpha}  
		\SigmatTan^{\beta}
		\angdiv \angpartialarg{\beta}
		+
		\SigmatTan_{\alpha}
		\angV \spherenormal^{\alpha}	
			\notag \\
		& \ \
		+
		\SigmatTan^{\alpha}
		\spherenormal^{\beta}
		\angV \gfour_{\alpha \beta},
		\notag	
		\end{align}
	\end{subequations}
	
	\begin{subequations}
	\begin{align} \label{E:WEIGHTDERIVATIVESDOUBLENULLMAINTHMOUTGOINGLATERALBOUNDARYEASYERRORINTEGRANDTERMS}
		\mathfrak{H}_{(\pmb{\partial} \weight)}[\SigmatTan]
		& :=
		\frac{\posinnerproduct}{\secondposinnerproduct \ReciprocalLunitAppliedtoTimeFunction}
		|\angV|_{\gsphere}^2
		\newL
		\weight
		+
		\posinnerproduct
		|\angV|_{\gsphere}^2
		\tang
		\weight
		-
		\SigmatTan_{\alpha} \spherenormal^{\alpha}
		\angV \weight,
			\\
	\mathfrak{H}[\SigmatTan]
		& :=
			\frac{1}{2} 
			\frac{\posinnerproduct}{\secondposinnerproduct \ReciprocalLunitAppliedtoTimeFunction}  
			|\angV|_{\gsphere}^2
			(\gsphere^{-1})^{\alpha \beta} 
			\Lie_{\newL} 
			\gsphere_{\alpha \beta}
			\label{E:DOUBLENULLMAINTHMOUTGOINGLATERALBOUNDARYEASYERRORINTEGRANDTERMS} \\
		& \ \
		+
		\left\lbrace
		\newL
			\left[\frac{\posinnerproduct}{\secondposinnerproduct \ReciprocalLunitAppliedtoTimeFunction} (\gsphere^{-1})^{\alpha \beta} \right]
		\right\rbrace
		\SigmatTan_{\alpha} 
	  \SigmatTan_{\beta}
		+
		\left\lbrace
		\tang
			\left[\posinnerproduct (\gsphere^{-1})^{\alpha \beta} \right]
		\right\rbrace
		\SigmatTan_{\alpha} 
	  \SigmatTan_{\beta}
		-
		2 
		\SigmatTan_{\alpha} \angV \Transport^{\alpha}
		\notag \\
		& \ \
				+
			\posinnerproduct 
			|\angV|_{\gsphere}^2
			\angdiv \tang 
		-
		\spherenormal_{\alpha} \SigmatTan^{\alpha}  
		\SigmatTan^{\beta}
		\angdiv \angpartialarg{\beta}
		-
		\SigmatTan_{\alpha}
		\angV \spherenormal^{\alpha}	
			\notag \\
		& \ \
		-
		\SigmatTan^{\alpha}
		\spherenormal^{\beta}
		\angV \gfour_{\alpha \beta},
		\notag
	\end{align}
	\end{subequations}
	the scalar functions
	$\underline{\mathfrak{H}}_{(1)}[\vortrenormalized]$
	and
	$\mathfrak{H}_{(1)}[\vortrenormalized]$
	are defined relative to the Cartesian coordinates by
	\begin{subequations}
	\begin{align} \label{E:DOUBLENULLMAINTHMSPECIFICVORTITICYMAININGOINGLATERALERRORINTEGRAND}
		\underline{\mathfrak{H}}_{(1)}[\vortrenormalized]
		&  := 
				4 \uposinnerproduct 
				\vortrenormalized_{\alpha} \uspecialgen^{\alpha}
				\angvortrenormalized \ln \Speed 
				- 
				4 \uposinnerproduct 
				|\angvortrenormalized|_{\gsphere}^2
				\uspecialgen \ln \Speed 
				+
				4 \uposinnerproduct 
				\uspecialgen^0
				\vortrenormalized_a \angvortrenormalized v^a
				-
				4 \uposinnerproduct
				\angvortrenormalized^0
				\vortrenormalized_a \uspecialgen v^a
					\\
		& \ \
		+
		2 \uposinnerproduct
		\Speed^{-4} 
		\exp(-2 \LogDensity) 
		\frac{p_{;\Ent}}{\bar{\varrho}}
		\upepsilon_{\alpha \beta \gamma \delta}
		 \uspecialgen^{\alpha} 
		 \angvortrenormalized^{\beta}
		\urescalednewgenminushypnorm^{\gamma}
		\GradEnt^{\delta}
		\uLunit \LogDensity
		-
		2 \uposinnerproduct
		\Speed^{-4} 
		\exp(-2 \LogDensity) 
		\frac{p_{;\Ent}}{\bar{\varrho}}
		\GradEnt^a \uLunit_a
		\upepsilon_{\alpha \beta \gamma \delta}
		\uspecialgen^{\alpha} 
		\angvortrenormalized^{\beta} 
		\Transport^{\gamma}
		\urescalednewgenminushypnorm^{\delta} 
		\uLunit \LogDensity
			\notag 
			\\
	 & \ \
		+
		2 \uposinnerproduct
		\Speed^{-2} 
		\exp(-2 \LogDensity) 
		\frac{p_{;\Ent}}{\bar{\varrho}}
		\upepsilon_{\alpha \beta ab}
		\uspecialgen^{\alpha}  
		\angvortrenormalized^{\beta}
		(\utandecompvectorfielddownarg{a} \LogDensity)
		\GradEnt^b
		-
		2 \uposinnerproduct
		\Speed^{-2} 
		\exp(-2 \LogDensity) 
		\frac{p_{;\Ent}}{\bar{\varrho}}
		\GradEnt^a \uLunit_a
		\upepsilon_{\alpha \beta \gamma d}
		\uspecialgen^{\alpha} 
		\angvortrenormalized^{\beta} 
		\Transport^{\gamma}
		\utandecompvectorfielddownarg{d} \LogDensity
		\notag
			\\
	& \ \
		-
		2 \uposinnerproduct
		\Speed^{-4} 
				\exp(-2 \LogDensity) 
				\frac{p_{;\Ent}}{\bar{\varrho}}
		\GradEnt^a \urescalednewgenminushypnorm_a
		\upepsilon_{\alpha \beta \gamma \delta}
		\uspecialgen^{\alpha} 
		\angvortrenormalized^{\beta} 
		\Transport^{\gamma}
		\uLunit v^{\delta}
		-
		2 \uposinnerproduct
		\Speed^{-4} 
				\exp(-2 \LogDensity) 
				\frac{p_{;\Ent}}{\bar{\varrho}}
		\upepsilon_{\alpha \beta \gamma \delta}
		\uspecialgen^{\alpha} 
		\angvortrenormalized^{\beta} 
		\Transport^{\gamma}
		\GradEnt^a \utandecompvectorfielddownarg{a} v^{\delta}
		\notag
			\\
	& \ \
				+
			2 \uposinnerproduct
			\Speed^{-2}
			\exp(\LogDensity)
			\upepsilon_{\alpha \beta \gamma \delta}
			\uspecialgen^{\alpha}
			\angvortrenormalized^{\beta}
			\Transport^{\gamma}
			\VortVort^{\delta}
				\notag
				\\
	& \ \
		-
		2 \uposinnerproduct
		\Speed^{-4} 
		\exp(-3 \LogDensity) 
		\left\lbrace
			\frac{p_{;\Ent}}{\bar{\varrho}}
		\right\rbrace^2
		\GradEnt^a \uLunit_a
		\upepsilon_{\alpha \beta \gamma \delta}
		\uspecialgen^{\alpha} 
		\angvortrenormalized^{\beta} 
		\Transport^{\gamma}
		\GradEnt^{\delta},
				\notag
					\\
	\mathfrak{H}_{(1)}[\vortrenormalized]
		&  := 
				4 \posinnerproduct 
				\vortrenormalized_{\alpha} \specialgen^{\alpha}
				\angvortrenormalized \ln \Speed 
				- 
				4 \posinnerproduct 
				|\angvortrenormalized|_{\gsphere}^2
				\specialgen \ln \Speed 
				+
				4 \posinnerproduct 
				\specialgen^0
				\vortrenormalized_a \angvortrenormalized v^a
				-
				4 \posinnerproduct
				\angvortrenormalized^0
				\vortrenormalized_a \specialgen v^a
					\\
		& \ \
		+
		2 \posinnerproduct
		\Speed^{-4} 
		\exp(-2 \LogDensity) 
		\frac{p_{;\Ent}}{\bar{\varrho}}
		\upepsilon_{\alpha \beta \gamma \delta}
		 \specialgen^{\alpha} 
		 \angvortrenormalized^{\beta}
		\rescalednewgenminushypnorm^{\gamma}
		\GradEnt^{\delta}
		\Lunit \LogDensity
		-
		2 \posinnerproduct
		\Speed^{-4} 
		\exp(-2 \LogDensity) 
		\frac{p_{;\Ent}}{\bar{\varrho}}
		\GradEnt^a \Lunit_a
		\upepsilon_{\alpha \beta \gamma \delta}
		\specialgen^{\alpha} 
		\angvortrenormalized^{\beta} 
		\Transport^{\gamma}
		\rescalednewgenminushypnorm^{\delta} 
		\Lunit \LogDensity
			\notag 
			\\
	 & \ \
		+
		2 \posinnerproduct
		\Speed^{-2} 
		\exp(-2 \LogDensity) 
		\frac{p_{;\Ent}}{\bar{\varrho}}
		\upepsilon_{\alpha \beta ab}
		\specialgen^{\alpha}  
		\angvortrenormalized^{\beta}
		(\tandecompvectorfielddownarg{a} \LogDensity)
		\GradEnt^b
		-
		2 \posinnerproduct
		\Speed^{-2} 
		\exp(-2 \LogDensity) 
		\frac{p_{;\Ent}}{\bar{\varrho}}
		\GradEnt^a \Lunit_a
		\upepsilon_{\alpha \beta \gamma d}
		\specialgen^{\alpha} 
		\angvortrenormalized^{\beta} 
		\Transport^{\gamma}
		\tandecompvectorfielddownarg{d} \LogDensity
		\notag
			\\
	& \ \
		-
		2 \posinnerproduct
		\Speed^{-4} 
				\exp(-2 \LogDensity) 
				\frac{p_{;\Ent}}{\bar{\varrho}}
		\GradEnt^a \rescalednewgenminushypnorm_a
		\upepsilon_{\alpha \beta \gamma \delta}
		\specialgen^{\alpha} 
		\angvortrenormalized^{\beta} 
		\Transport^{\gamma}
		\Lunit v^{\delta}
		-
		2 \posinnerproduct
		\Speed^{-4} 
				\exp(-2 \LogDensity) 
				\frac{p_{;\Ent}}{\bar{\varrho}}
		\upepsilon_{\alpha \beta \gamma \delta}
		\specialgen^{\alpha} 
		\angvortrenormalized^{\beta} 
		\Transport^{\gamma}
		\GradEnt^a \tandecompvectorfielddownarg{a} v^{\delta}
		\notag
			\\
	& \ \
				+
			2 \posinnerproduct
			\Speed^{-2}
			\exp(\LogDensity)
			\upepsilon_{\alpha \beta \gamma \delta}
			\specialgen^{\alpha}
			\angvortrenormalized^{\beta}
			\Transport^{\gamma}
			\VortVort^{\delta}
				\notag
				\\
	& \ \
		-
		2 \posinnerproduct
		\Speed^{-4} 
		\exp(-3 \LogDensity) 
		\left\lbrace
			\frac{p_{;\Ent}}{\bar{\varrho}}
		\right\rbrace^2
		\GradEnt^a \Lunit_a
		\upepsilon_{\alpha \beta \gamma \delta}
		\specialgen^{\alpha} 
		\angvortrenormalized^{\beta} 
		\Transport^{\gamma}
		\GradEnt^{\delta},
				\notag
	\end{align}
\end{subequations}
and the scalar functions
$\underline{\mathfrak{H}}_{(2)}[\GradEnt]$
and
$\mathfrak{H}_{(2)}[\GradEnt]$
are defined relative to the Cartesian coordinates by
\begin{subequations}
\begin{align} \label{E:DOUBLENULLMAINTHMENTROPYGRADIENTMAININGOINGLATERALERRORINTEGRAND}
		\underline{\mathfrak{H}}_{(2)}[\GradEnt]
		& := 
			4 \uposinnerproduct \GradEnt_{\alpha} \uspecialgen^{\alpha} \angGradEnt \ln \Speed
			- 
			4 \uposinnerproduct |\angGradEnt|_{\gsphere}^2  \uspecialgen \ln \Speed,
				\\
		\mathfrak{H}_{(2)}[\GradEnt]
		& := 
			4 \posinnerproduct \GradEnt_{\alpha} \specialgen^{\alpha} \angGradEnt \ln \Speed
			- 
			4 \posinnerproduct |\angGradEnt|_{\gsphere}^2  \specialgen \ln \Speed.
			\label{E:DOUBLENULLMAINTHMENTROPYGRADIENTMAINOUTGOINGLATERALERRORINTEGRAND}
\end{align}
\end{subequations}
	
\end{theorem}

\begin{remark}[Extension of Theorem~\ref{T:STRUCTUREOFERRORTERMS} to double-null foliations]
	\label{R:DOUBLENULLNULLHYPERSURFACEERRORTERMSTANGENTIALDERIVATIVES}
	An analog of Theorem~\ref{T:STRUCTUREOFERRORTERMS} also holds in the present context of double-null foliations.
	More precisely, the error integrands on 
	RHSs~\eqref{E:DOUBLENULLSPACETIMEREMARKABLEIDENTITYSPECIFICVORTICITY}-\eqref{E:DOUBLENULLSPACETIMEREMARKABLEIDENTITYENTROPYGRADIENT}
	along the ingoing $\gfour$-null hypersurfaces
	$\underline{\mathcal{H}}_{\underline{U}}(1,U)$
	and
	$\underline{\mathcal{H}}_{0}(1,U)$,
	enjoy the properties revealed by Theorem~\ref{T:STRUCTUREOFERRORTERMS}, that is,
	they involve only tangential derivatives.
	Analogous results hold for the error integrands on 
	RHSs~\eqref{E:DOUBLENULLSPACETIMEREMARKABLEIDENTITYSPECIFICVORTICITY}-\eqref{E:DOUBLENULLSPACETIMEREMARKABLEIDENTITYENTROPYGRADIENT}
	along the outgoing $\gfour$-null hypersurfaces
	$\mathcal{H}_{0}(0,\underline{U})$
	and
	$\mathcal{H}_U(0,\underline{U})$.
	These properties could be shown by using essentially the same arguments that we used in the proof of
	Theorem~\ref{T:STRUCTUREOFERRORTERMS}; thus, for brevity, we do not provide details.
\end{remark}

\subsection{Preliminary identities for the proof of Theorem~\ref{T:DOUBLENULLMAINTHEOREM}}
\label{SS:DOUBLENULLPRELIMINARYIDENTITIES}
In this subsection, to facilitate the proof of Theorem~\ref{T:DOUBLENULLMAINTHEOREM},
we derive preliminary integral identities for the main null hypersurface error integrals
that arise in its proof.
The main result is Prop.\,\ref{P:DOUBLENULLSTRUCTUREOFERRORINTEGRALS},
which is a direct analog of Prop.\,\ref{P:STRUCTUREOFERRORINTEGRALS}.

\subsubsection{Preliminary analysis of the boundary integrands}
\label{SSS:DOUBLENULLPRELIMINARYANALYSISOFBOUNDARYINTEGRAND}
We start with the following lemma, which is a direct analog of Lemma~\ref{L:PRELIMINARYANALYSISOFBOUNDARYINTEGRAND}.
As before, the purpose of the lemma is to reveal preliminary good structures in the boundary terms that 
will appear when we apply the divergence theorem on $\widetilde{\Sigma}_{\Timefunction}$
via the divergence identity \eqref{E:NEWSTANDARDDIVERGENCEIDENTITYFORELLIPTICHYPERBOLICCURRENT}.
We highlight that in the present context of double-null foliations, 
the boundary of $\widetilde{\Sigma}_{\Timefunction}$ has two connected components:
an inner sphere and an outer sphere; see Remark~\ref{R:DOUBLENULLORIENTATIONOFSPHERENORMAL}. 
Moreover, to prove Theorem~\ref{T:DOUBLENULLMAINTHEOREM},
along each of the two spheres,
we need to identify good geo-analytic structures adapted to $\underline{\mathcal{H}}_{\underline{u}}$
and good geo-analytic structures adapted to $\mathcal{H}_u$.
This explains why
Lemma~\ref{L:DOUBLENULLPRELIMINARYANALYSISOFBOUNDARYINTEGRAND}
features four identities,
while Lemma~\ref{L:PRELIMINARYANALYSISOFBOUNDARYINTEGRAND} features only one.

\begin{lemma}[Double-null foliations: Preliminary analysis of the boundary integrands]
	\label{L:DOUBLENULLPRELIMINARYANALYSISOFBOUNDARYINTEGRAND}
	Let $\SigmatTan$ be a $\Sigma_t$-tangent vectorfield defined on $\mathcal{M}$,
	and let $\uspecialgen$, $\specialgen$, $\utang$, and $\tang$
	be the vectorfields from \eqref{E:DOUBLENULLSPECIALGENERATOR}-\eqref{E:DOUBLENULLSPECIALGENERATORTANGENTIDENTITY}.
	Let $\weight$ be an arbitrary scalar function.
	Under the assumptions of Theorem~\ref{T:DOUBLENULLMAINTHEOREM}, 
	the following identities hold, 
	where we refer to Subsubsect.\,\ref{SSS:TENSORFIELDSWITHSAMEDEFINITIONS} regarding the notation.
	
	\noindent \underline{\textbf{Identities in the ``dynamic region''}}.
	On $\mathcal{S}_{u,\underline{U}}$, for $u \in [1,U]$, we have:
	\begin{align} \label{E:DOUBLENULLPRELIMINARYDECOMPOFBOUNDARYINTEGRANDINTERESTINGCASE1}
		\weight 
		\spherenormal_{\alpha} J^{\alpha}[\SigmatTan]
		& = 
			-
		\newuL
		\left\lbrace
			\weight
			\frac{\uposinnerproduct}{\seconduposinnerproduct \ReciprocaluLunitAppliedtoTimeFunction} 
			|\SigmatTan|_{\gsphere}^2
		\right\rbrace
		-
		\angdiv
		\left\lbrace
			\weight
			\uposinnerproduct 
			|\angV|_{\gsphere}^2
			\utang 
		\right\rbrace
		-
		\angdiv
		\left\lbrace
			\weight
			\spherenormal_{\alpha} \SigmatTan^{\alpha}  
			\angV 
		\right\rbrace
			\\
		& \ \
		+
		2 
		\weight
		\uposinnerproduct
		\uspecialgen^{\alpha}
		\angV^{\beta}
		(\partial_{\alpha} \SigmatTan_{\beta} - \partial_{\beta} \SigmatTan_{\alpha})
		+
		\underline{\mathfrak{H}}_{(\pmb{\partial} \weight)}[\SigmatTan]
		+
		\weight
		\left\lbrace
			\underline{\mathfrak{H}}[\SigmatTan]
			-
			\frac{1}{2} 
			\frac{\uposinnerproduct}{\seconduposinnerproduct \ReciprocaluLunitAppliedtoTimeFunction}  
			|\angV|_{\gsphere}^2
			(\gsphere^{-1})^{\alpha \beta} 
			\Lie_{\newuL} \gsphere_{\alpha \beta}
		\right\rbrace,
		\notag
	\end{align}
	where $\underline{\mathfrak{H}}_{(\pmb{\partial} \weight)}[\SigmatTan]$ and $\underline{\mathfrak{H}}[\SigmatTan]$
	are defined in
	\eqref{E:WEIGHTDERIVATIVESDOUBLENULLMAINTHMINGOINGLATERALBOUNDARYEASYERRORINTEGRANDTERMS}-\eqref{E:DOUBLENULLMAINTHMINGOINGLATERALBOUNDARYEASYERRORINTEGRANDTERMS}.
	
	Moreover, on $\mathcal{S}_{U,\underline{u}}$, for $\underline{u} \in [0,\underline{U}]$, we have:
	\begin{align} \label{E:DOUBLENULLPRELIMINARYDECOMPOFBOUNDARYINTEGRANDINTERESTINGCASE2}
		- 
		\weight
		\spherenormal_{\alpha} J^{\alpha}[\SigmatTan]
		& = 
			-
		\newL
		\left\lbrace
			\weight
			\frac{\posinnerproduct}{\secondposinnerproduct \ReciprocalLunitAppliedtoTimeFunction} 
			|\SigmatTan|_{\gsphere}^2
		\right\rbrace
			-
		\angdiv
		\left\lbrace
			\weight
			\posinnerproduct 
			|\angV|_{\gsphere}^2
			\tang 
		\right\rbrace
		+
		\angdiv
		\left\lbrace
			\weight
			\spherenormal_{\alpha} \SigmatTan^{\alpha}  
			\angV 
		\right\rbrace
			\\
		& \ \
		+
		2 
		\weight
		\posinnerproduct
		\specialgen^{\alpha}
		\angV^{\beta}
		(\partial_{\alpha} \SigmatTan_{\beta} - \partial_{\beta} \SigmatTan_{\alpha})
		+
		\mathfrak{H}_{(\pmb{\partial} \weight)}[\SigmatTan]
		+
		\weight
		\left\lbrace
			\mathfrak{H}[\SigmatTan]
			-
			\frac{1}{2} 
			\frac{\posinnerproduct}{\secondposinnerproduct \ReciprocalLunitAppliedtoTimeFunction}  
			|\angV|_{\gsphere}^2
			(\gsphere^{-1})^{\alpha \beta} 
			\Lie_{\newL} 
			\gsphere_{\alpha \beta}
		\right\rbrace,
		\notag	
	\end{align}
	where $\mathfrak{H}_{(\pmb{\partial} \weight)}[\SigmatTan]$ and $\mathfrak{H}[\SigmatTan]$
	are defined in
	\eqref{E:WEIGHTDERIVATIVESDOUBLENULLMAINTHMOUTGOINGLATERALBOUNDARYEASYERRORINTEGRANDTERMS}-\eqref{E:DOUBLENULLMAINTHMOUTGOINGLATERALBOUNDARYEASYERRORINTEGRANDTERMS}.
		
		\medskip
	
		\noindent \underline{\textbf{Identities where the data are specified}}.
		On $\mathcal{S}_{u,0}$, for $u \in [1,U]$, we have:
		\begin{align} \label{E:DOUBLENULLPRELIMINARYDECOMPOFBOUNDARYINTEGRANDDATACASE1}
		-
		\weight
		\spherenormal_{\alpha} J^{\alpha}[\SigmatTan]
		& = 
		\newuL
		\left\lbrace
			\weight
			\frac{\uposinnerproduct}{\seconduposinnerproduct \ReciprocaluLunitAppliedtoTimeFunction} 
			|\SigmatTan|_{\gsphere}^2
		\right\rbrace
		+
		\angdiv
		\left\lbrace
			\weight
			\uposinnerproduct 
			|\angV|_{\gsphere}^2
			\utang 
		\right\rbrace
		+
		\angdiv
		\left\lbrace
			\weight
			\spherenormal_{\alpha} \SigmatTan^{\alpha}  
			\angV 
		\right\rbrace
			\\
		& \ \
		-
		2 
		\weight
		\uposinnerproduct
		\uspecialgen^{\alpha}
		\angV^{\beta}
		(\partial_{\alpha} \SigmatTan_{\beta} - \partial_{\beta} \SigmatTan_{\alpha})
		-
		\underline{\mathfrak{H}}_{(\pmb{\partial} \weight)}[\SigmatTan]
		-
		\weight
		\left\lbrace
			\underline{\mathfrak{H}}[\SigmatTan]
			-
			\frac{1}{2} 
			\frac{\uposinnerproduct}{\seconduposinnerproduct \ReciprocaluLunitAppliedtoTimeFunction}  
			|\angV|_{\gsphere}^2
			(\gsphere^{-1})^{\alpha \beta} 
			\Lie_{\newuL} \gsphere_{\alpha \beta}
		\right\rbrace.
		\notag
	\end{align}
		
	Finally, on $\mathcal{S}_{1,\underline{u}}$, for $\underline{u} \in [0,\underline{U}]$, we have:
	\begin{align} \label{E:DOUBLENULLPRELIMINARYDECOMPOFBOUNDARYINTEGRANDDATACASE2}
		\weight
		\spherenormal_{\alpha} J^{\alpha}[\SigmatTan]
		& = 
		\newL
		\left\lbrace
			\weight
			\frac{\posinnerproduct}{\secondposinnerproduct \ReciprocalLunitAppliedtoTimeFunction} 
			|\SigmatTan|_{\gsphere}^2
		\right\rbrace
		+
		\angdiv
		\left\lbrace
			\weight
			\posinnerproduct 
			|\angV|_{\gsphere}^2
			\tang 
		\right\rbrace
		-
		\angdiv
		\left\lbrace
			\weight
			\spherenormal_{\alpha} \SigmatTan^{\alpha}  
			\angV 
		\right\rbrace
			\\
		& \ \
		-
		2 
		\weight
		\posinnerproduct
		\specialgen^{\alpha}
		\angV^{\beta}
		(\partial_{\alpha} \SigmatTan_{\beta} - \partial_{\beta} \SigmatTan_{\alpha})
		-
		\mathfrak{H}_{(\pmb{\partial} \weight)}[\SigmatTan]
		-
		\weight
		\left\lbrace
			\mathfrak{H}[\SigmatTan]
			-
			\frac{1}{2} 
			\frac{\posinnerproduct}{\secondposinnerproduct \ReciprocalLunitAppliedtoTimeFunction}  
			|\angV|_{\gsphere}^2
			(\gsphere^{-1})^{\alpha \beta} 
			\Lie_{\newL} 
			\gsphere_{\alpha \beta}
		\right\rbrace.
		\notag	
	\end{align}
		
	\end{lemma}
	
	\begin{remark}[The signs in Lemma~\ref{L:DOUBLENULLPRELIMINARYANALYSISOFBOUNDARYINTEGRAND}]
	\label{R:SIGNSINLEMMADOUBLENULLPRELIMINARYANALYSISOFBOUNDARYINTEGRAND}
	Our sign choices in Lemma~\ref{L:DOUBLENULLPRELIMINARYANALYSISOFBOUNDARYINTEGRAND} are such
	the left-hand sides of the identities correspond to outward pointing normals.
	More precisely, the vectorfield $\spherenormal$ is outward pointing to $\widetilde{\Sigma}_{\Timefunction}$
	along $\mathcal{S}_{u,\underline{U}}$ and $\mathcal{S}_{1,\underline{u}}$,
	while its negation $-\spherenormal$ is outward pointing to $\widetilde{\Sigma}_{\Timefunction}$
	along $\mathcal{S}_{U,\underline{u}}$ and $\mathcal{S}_{u,0}$; see Fig.\,\ref{F:DOUBLENULL}.
	\end{remark}
	
	\begin{remark}[A cancellation in the proof of Prop.\,\ref{P:DOUBLENULLSTRUCTUREOFERRORINTEGRALS}]
		\label{R:TERMWILLBECANCELED}
		In the proof of Prop.\,\ref{P:DOUBLENULLSTRUCTUREOFERRORINTEGRALS},
		the product
		$
			-
			\frac{1}{2} 
			\weight
			\frac{\uposinnerproduct}{\seconduposinnerproduct \ReciprocaluLunitAppliedtoTimeFunction}  
			|\angV|_{\gsphere}^2
			(\gsphere^{-1})^{\alpha \beta} 
			\Lie_{\newuL} \gsphere_{\alpha \beta}
	$
	found in the last term on RHS~\eqref{E:DOUBLENULLPRELIMINARYDECOMPOFBOUNDARYINTEGRANDINTERESTINGCASE1}
	will be canceled by a term that arises from an application of Lemma~\ref{L:DOUBLENULLINTEGRALIDENTITIESINVOLVINGSPHERES}.
	This cancellation is the reason that we have expressed the terms in the 
	last braces on RHS~\eqref{E:DOUBLENULLPRELIMINARYDECOMPOFBOUNDARYINTEGRANDINTERESTINGCASE1} in their stated form.
	Similar remarks apply to the Lie derivative-involving products in the last braces on 
	RHSs~\eqref{E:DOUBLENULLPRELIMINARYDECOMPOFBOUNDARYINTEGRANDINTERESTINGCASE2}-\eqref{E:DOUBLENULLPRELIMINARYDECOMPOFBOUNDARYINTEGRANDDATACASE2}.
	\end{remark}
	
	\begin{proof}[Proof of Lemma~\ref{L:DOUBLENULLPRELIMINARYANALYSISOFBOUNDARYINTEGRAND}]
		The proof of \eqref{E:DOUBLENULLPRELIMINARYDECOMPOFBOUNDARYINTEGRANDINTERESTINGCASE1} mirrors of the proof of 
		\eqref{E:PRELIMINARYDECOMPOFBOUNDARYINTEGRAND},
		where we use \eqref{E:TRANSPORTINTERMSOFINGOINGNULL} in place of the identity 
		\eqref{E:TRANSPORTDECOMPOSITION} used in the proof of \eqref{E:PRELIMINARYDECOMPOFBOUNDARYINTEGRAND}.
		The proof of \eqref{E:DOUBLENULLPRELIMINARYDECOMPOFBOUNDARYINTEGRANDDATACASE1} is identical.
		
		Similarly, proof of \eqref{E:DOUBLENULLPRELIMINARYDECOMPOFBOUNDARYINTEGRANDINTERESTINGCASE2} mirrors of the proof of 
		\eqref{E:PRELIMINARYDECOMPOFBOUNDARYINTEGRAND},
		where we use \eqref{E:TRANSPORTINTERMSOFOUTGOINGNULL} in place of the identity 
		\eqref{E:TRANSPORTDECOMPOSITION} used in the proof of \eqref{E:PRELIMINARYDECOMPOFBOUNDARYINTEGRAND}.
		The proof of \eqref{E:DOUBLENULLPRELIMINARYDECOMPOFBOUNDARYINTEGRANDDATACASE2} is identical.
	
	\end{proof}

\subsubsection{Geometric decompositions and remarkable cancellations for the most subtle terms 
on RHS~\eqref{E:KEYIDENTITYANTISYMMETRICPARTOFSPECIFICVORTICITYDUALGRADIENT}}	
\label{SSS:DOUBLENULLPRELIMINARYDECOMPOSITIONOFSUBTLETERMS}
The next proposition is an analog of the combination of
Lemma~\ref{L:PRELIMINARYDECOMPOSITIONOFSUBTLETERMS}
and
Prop.\,\ref{P:KEYDETERMINANT}.
For the same reasons described in Subsects.\,\ref{SS:MOSTSUBTLEENTROPYTERM} and \ref{SS:REMARKABLESTRUCTUREOFMOSTSUBTLETERM},
the purpose of the proposition
is to reveal remarkable geo-analytic cancellations in the terms
$
2 	\weight
		\uposinnerproduct
		\uspecialgen^{\alpha}
		\angV^{\beta}
		(\partial_{\alpha} \vortrenormalized_{\beta} - \partial_{\beta} \vortrenormalized_{\alpha})
$
and
$
2 	\weight
		\posinnerproduct
		\specialgen^{\alpha}
		\angV^{\beta}
		(\partial_{\alpha} \vortrenormalized_{\beta} - \partial_{\beta} \vortrenormalized_{\alpha})
$
on RHSs~\eqref{E:DOUBLENULLPRELIMINARYDECOMPOFBOUNDARYINTEGRANDINTERESTINGCASE1}-\eqref{E:DOUBLENULLPRELIMINARYDECOMPOFBOUNDARYINTEGRANDDATACASE2}
in the case $\SigmatTan = \vortrenormalized$.

\begin{proposition}[Double-null foliations: Geometric decompositions and remarkable cancellations for the most subtle terms on RHS~\eqref{E:KEYIDENTITYANTISYMMETRICPARTOFSPECIFICVORTICITYDUALGRADIENT}]	
	\label{P:DOUBLENULLPRELIMINARYDECOMPOSITIONOFSUBTLETERMS}
	Let $\uspecialgen$ be the $\underline{\mathcal{H}}_{\underline{u}}$-tangent vectorfield
	from \eqref{E:DOUBLENULLSPECIALGENERATOR},
	let $\specialgen$ be the $\mathcal{H}_u$-tangent vectorfield
	from \eqref{E:DOUBLENULLSPECIALGENERATOR},
	let $\urescalednewgenminushypnorm$ and $\rescalednewgenminushypnorm$ be the $\Sigma_t$-tangent vectorfields 
	from \eqref{E:DOUBLENULLREGULARFORMRESCALEDVERSIONOFHYPNORMMINUSNEWGEN}, 
	let $\lbrace \utandecompvectorfielddownarg{\alpha} \rbrace_{\alpha=0,1,2,3}$ 
	be the $\mathcal{S}_{u,\underline{u}}$-tangent vectorfields from
	\eqref{E:DOUBLENULLHBARDECOMPOSITIONOFCOORDINATEPARTIALDERIVATIVEVECTORFIELDS},
	and let $\lbrace \tandecompvectorfielddownarg{\alpha} \rbrace_{\alpha=0,1,2,3}$
	be the $\mathcal{S}_{u,\underline{u}}$-tangent vectorfield from
	\eqref{E:DOUBLENULLHDECOMPOSITIONOFCOORDINATEPARTIALDERIVATIVEVECTORFIELDS}.
	For smooth solutions (see Remark~\ref{R:SMOOTHNESSNOTNEEDED})
	to the compressible Euler equations \eqref{E:TRANSPORTDENSRENORMALIZEDRELATIVECTORECTANGULAR}-\eqref{E:ENTROPYTRANSPORT} on $\mathcal{M}$,
	the following identities hold,
	where we refer to Subsubsect.\,\ref{SSS:TENSORFIELDSWITHSAMEDEFINITIONS} regarding the notation:
	\begin{subequations}
	\begin{align}
		&
		\upepsilon_{\alpha \beta \gamma \delta}
		\uspecialgen^{\alpha} 
		\angvortrenormalized^{\beta}
		\left\lbrace
			-
			(\Transport v^{\gamma}) 
			\GradEnt^{\delta}
			+
			\Transport^{\gamma}
			[\GradEnt^{\delta}
				(\partial_a v^a)
				-
				\GradEnt^a \partial_a v^{\delta}]
	\right\rbrace
					\label{E:DOUBLENULLINGOINGPRELIMINARYDECOMPOSITIONOFSUBTLETERMS} 
					\\
	& =
		\upepsilon_{\alpha \beta \gamma \delta}
		 \uspecialgen^{\alpha} 
		 \angvortrenormalized^{\beta}
		\urescalednewgenminushypnorm^{\gamma}
		\GradEnt^{\delta}
		\uLunit \LogDensity
		-
		\GradEnt^a \uLunit_a
		\upepsilon_{\alpha \beta \gamma \delta}
		\uspecialgen^{\alpha} 
		\angvortrenormalized^{\beta} 
		\Transport^{\gamma}
		\urescalednewgenminushypnorm^{\delta} 
		\uLunit \LogDensity
			\notag
			\\
	 & \ \
		+
		\Speed^2
		\upepsilon_{\alpha \beta ab}
		\uspecialgen^{\alpha}  
		\angvortrenormalized^{\beta}
		(\utandecompvectorfielddownarg{a} \LogDensity)
		\GradEnt^b
		-
		\Speed^2
		\GradEnt^a \uLunit_a
		\upepsilon_{\alpha \beta \gamma d}
		\uspecialgen^{\alpha} 
		\angvortrenormalized^{\beta} 
		\Transport^{\gamma}
		\utandecompvectorfielddownarg{d} \LogDensity
		\notag
			\\
	& \ \
		-
		\GradEnt^a \urescalednewgenminushypnorm_a
		\upepsilon_{\alpha \beta \gamma \delta}
		\uspecialgen^{\alpha} 
		\angvortrenormalized^{\beta} 
		\Transport^{\gamma}
		\uLunit v^{\delta}
		-
		\upepsilon_{\alpha \beta \gamma \delta}
		\uspecialgen^{\alpha} 
		\angvortrenormalized^{\beta} 
		\Transport^{\gamma}
		\GradEnt^a \utandecompvectorfielddownarg{a} v^{\delta}
		\notag
			\\
	& \ \
			-
		\exp(-\LogDensity) \frac{p_{;\Ent}}{\bar{\varrho}}
		\GradEnt^a \uLunit_a
		\upepsilon_{\alpha \beta \gamma \delta}
		\uspecialgen^{\alpha} 
		\angvortrenormalized^{\beta} 
		\Transport^{\gamma}
		\GradEnt^{\delta},
		\notag
			\\
	&
		\upepsilon_{\alpha \beta \gamma \delta}
		\specialgen^{\alpha} 
		\angvortrenormalized^{\beta}
		\left\lbrace
			-
			(\Transport v^{\gamma}) 
			\GradEnt^{\delta}
			+
			\Transport^{\gamma}
			[\GradEnt^{\delta}
				(\partial_a v^a)
				-
				\GradEnt^a \partial_a v^{\delta}]
		\right\rbrace
					\label{E:DOUBLENULLOUTGOINGPRELIMINARYDECOMPOSITIONOFSUBTLETERMS} 
					\\
	& =
		\upepsilon_{\alpha \beta \gamma \delta}
		 \specialgen^{\alpha} 
		 \angvortrenormalized^{\beta}
		\urescalednewgenminushypnorm^{\gamma}
		\GradEnt^{\delta}
		\Lunit \LogDensity
		-
		\GradEnt^a \Lunit_a
		\upepsilon_{\alpha \beta \gamma \delta}
		\specialgen^{\alpha} 
		\angvortrenormalized^{\beta} 
		\Transport^{\gamma}
		\urescalednewgenminushypnorm^{\delta} 
		\Lunit \LogDensity
			\notag
			\\
	 & \ \
		+
		\Speed^2
		\upepsilon_{\alpha \beta ab}
		\specialgen^{\alpha}  
		\angvortrenormalized^{\beta}
		(\tandecompvectorfielddownarg{a} \LogDensity)
		\GradEnt^b
		-
		\Speed^2
		\GradEnt^a \Lunit_a
		\upepsilon_{\alpha \beta \gamma d}
		\specialgen^{\alpha} 
		\angvortrenormalized^{\beta} 
		\Transport^{\gamma}
		\tandecompvectorfielddownarg{d} \LogDensity
		\notag
			\\
	& \ \
		-
		\GradEnt^a \urescalednewgenminushypnorm_a
		\upepsilon_{\alpha \beta \gamma \delta}
		\specialgen^{\alpha} 
		\angvortrenormalized^{\beta} 
		\Transport^{\gamma}
		\Lunit v^{\delta}
		-
		\upepsilon_{\alpha \beta \gamma \delta}
		\specialgen^{\alpha} 
		\angvortrenormalized^{\beta} 
		\Transport^{\gamma}
		\GradEnt^a \tandecompvectorfielddownarg{a} v^{\delta}
		\notag
			\\
	& \ \
			-
		\exp(-\LogDensity) \frac{p_{;\Ent}}{\bar{\varrho}}
		\GradEnt^a \Lunit_a
		\upepsilon_{\alpha \beta \gamma \delta}
		\specialgen^{\alpha} 
		\angvortrenormalized^{\beta} 
		\Transport^{\gamma}
		\GradEnt^{\delta}.
		\notag
	\end{align}
	\end{subequations}
\end{proposition}

\begin{proof}
	To prove \eqref{E:DOUBLENULLINGOINGPRELIMINARYDECOMPOSITIONOFSUBTLETERMS},
	we first note that the proof of \eqref{E:PRELIMINARYDECOMPOSITIONOFSUBTLETERMS} 
	goes through verbatim with $\uLunit$ in the role of both $\sidehypnorm$ and $\gen$
	(see Convention~\ref{C:NULLCASE} and \ref{E:NORMALISGENERATORINNULLCASE}),
	where we use the identity \eqref{E:DOUBLENULLHBARDECOMPOSITIONOFCOORDINATEPARTIALDERIVATIVEVECTORFIELDS}
	in the role that we used
	\eqref{E:NEWDECOMPOSITIONOFCOORDINATEPARTIALDERIVATIVEVECTORFIELDS}
	in the proof of \eqref{E:PRELIMINARYDECOMPOSITIONOFSUBTLETERMS}.
	Next, we consider the first term on RHS~\eqref{E:PRELIMINARYDECOMPOSITIONOFSUBTLETERMS}
	(again with $\uLunit$ in the role of $\sidehypnorm$),
	namely
	$
	-
	\upepsilon_{\alpha \beta \gamma \delta}
		 \uspecialgen^{\alpha} 
		 \angvortrenormalized^{\beta}
		\uLunit^{\gamma}
		(\GradEnt^{\delta} + \GradEnt^a \uLunit_a \Transport^{\delta})
		\Transport \LogDensity
	$.
	This term completely vanishes because the proof of the identity \eqref{E:NULLCASEKEYDETERMINANT}
	(again with $\uLunit$ in the role of $\sidehypnorm$)
	goes through in the present context. In total, we have proved 
	\eqref{E:DOUBLENULLINGOINGPRELIMINARYDECOMPOSITIONOFSUBTLETERMS}.
	
	Similarly, to prove \eqref{E:DOUBLENULLOUTGOINGPRELIMINARYDECOMPOSITIONOFSUBTLETERMS},
	we use the identity \eqref{E:DOUBLENULLHDECOMPOSITIONOFCOORDINATEPARTIALDERIVATIVEVECTORFIELDS}
	in the role that
	\eqref{E:NEWDECOMPOSITIONOFCOORDINATEPARTIALDERIVATIVEVECTORFIELDS}
	was used in the proof of \eqref{E:PRELIMINARYDECOMPOSITIONOFSUBTLETERMS},
	thus concluding that the identity \eqref{E:PRELIMINARYDECOMPOSITIONOFSUBTLETERMS} holds
	in the present context, but with $\Lunit$ in the role of both $\sidehypnorm$ and $\gen$,
	$\specialgen$ in the role of $\uspecialgen$,
	and 
	$\tandecompvectorfielddownarg{\alpha}$
	in the role of
	$\utandecompvectorfielddownarg{\alpha}$.
	Moreover, the analog of the first term on RHS~\eqref{E:PRELIMINARYDECOMPOSITIONOFSUBTLETERMS},
	namely
	$
	-
	\upepsilon_{\alpha \beta \gamma \delta}
		 \specialgen^{\alpha} 
		 \angvortrenormalized^{\beta}
		\Lunit^{\gamma}
		(\GradEnt^{\delta} + \GradEnt^a \Lunit_a \Transport^{\delta})
		\Transport \LogDensity
	$,
	again completely vanishes because 
	the identity \eqref{E:NULLCASEKEYDETERMINANT}
	(now with $\specialgen$ in the role of $\uspecialgen$
	and $\Lunit$ in the role of $\sidehypnorm$)
	holds in the present context. In total, we have proved 
	\eqref{E:DOUBLENULLOUTGOINGPRELIMINARYDECOMPOSITIONOFSUBTLETERMS},
	which completes the proof of the proposition.
	
\end{proof}

We now prove the main result of this subsection.	

\begin{proposition}[Double-null foliations: Key identity for the boundary error integrals]
\label{P:DOUBLENULLSTRUCTUREOFERRORINTEGRALS}
Let 
$u$, 
$u_1$, 
$u_2$, 
$\underline{u}$, 
$\underline{u}_1$, 
and $\underline{u}_2$
be real numbers satisfying\footnote{The four identities stated in the proposition 
in fact hold for $1 \leq u \leq U$ and $0 \leq \underline{u} \leq \underline{U}$.
We have stated the proposition only for the ``endpoint values'' $u \in \lbrace 1, U \rbrace$ and $\underline{u} \in \lbrace 0, \underline{U} \rbrace$
to help the reader navigate the proof of Theorem~\ref{T:DOUBLENULLMAINTHEOREM}
(see Subsect.\,\ref{SS:PROOFOFTHEOREMDOUBLENULLMAINTHEOREM});
the endpoint values are the only ones that we use in our proof of the theorem. 
\label{FN:UNNECESSARILYRESCTRICTEDEIKONALVALUES}}
$u \in \lbrace 1, U \rbrace$,
$1 \leq u_1 \leq u_2 \leq U$, 
$\underline{u} \in \lbrace 0, \underline{U} \rbrace$,
and
$0 \leq \underline{u}_1 \leq \underline{u}_2 \leq \underline{U}$.
Under the assumptions of Theorem~\ref{T:DOUBLENULLMAINTHEOREM},
for smooth solutions (see Remark~\ref{R:SMOOTHNESSNOTNEEDED})
to the compressible Euler equations \eqref{E:TRANSPORTDENSRENORMALIZEDRELATIVECTORECTANGULAR}-\eqref{E:ENTROPYTRANSPORT},
the following integral identities hold,
where we refer to Subsubsect.\,\ref{SSS:TENSORFIELDSWITHSAMEDEFINITIONS} regarding the notation:
\begin{subequations}
\begin{align} \label{E:DOUBLENULLMAINLATERALERRORINTEGRALSFORSPECIFICVORTICITYALONGHBAR}
		\int_{\underline{\mathcal{H}}_{\underline{u}}(u_1,u_2)}
			\weight 
			\spherenormal_{\alpha} J^{\alpha}[\vortrenormalized]
		\, d \varpi_{\gsphere} d u'
		& = 
		-
		\int_{\mathcal{S}_{u_2,\underline{u}}}
			\weight 
			\frac{\uposinnerproduct}{\seconduposinnerproduct \ReciprocaluLunitAppliedtoTimeFunction}
			|\angvortrenormalized|_{\gsphere}^2
		\, d \varpi_{\gsphere}
		+
		\int_{\mathcal{S}_{u_1,\underline{u}}}
			\weight 
			\frac{\uposinnerproduct}{\seconduposinnerproduct \ReciprocaluLunitAppliedtoTimeFunction}
			|\angvortrenormalized|_{\gsphere}^2
		\, d \varpi_{\gsphere}
			\\
	& \ \
		+
		\int_{\underline{\mathcal{H}}_{\underline{u}}(u_1,u_2)}
			\left\lbrace
				\underline{\mathfrak{H}}_{(\pmb{\partial} \weight)}[\vortrenormalized]
				+
				\weight
				\underline{\mathfrak{H}}[\vortrenormalized]
				+
				\weight
				\underline{\mathfrak{H}}_{(1)}[\vortrenormalized]
			\right\rbrace
		\, d \varpi_{\gsphere} d u',
		\notag
			\\
		\int_{\mathcal{H}_u(\underline{u}_1,\underline{u}_2)}
			\weight 
			\spherenormal_{\alpha} J^{\alpha}[\vortrenormalized]
		\, d \varpi_{\gsphere} d \underline{u}'
		& = 
		\int_{\mathcal{S}_{u,\underline{u}_2}}
			\weight 
			\frac{\posinnerproduct}{\secondposinnerproduct \ReciprocalLunitAppliedtoTimeFunction}
			|\angvortrenormalized|_{\gsphere}^2
		\, d \varpi_{\gsphere}
		-
		\int_{\mathcal{S}_{u,\underline{u}_1}}
			\weight 
			\frac{\posinnerproduct}{\secondposinnerproduct \ReciprocalLunitAppliedtoTimeFunction}
			|\angvortrenormalized|_{\gsphere}^2
		\, d \varpi_{\gsphere}
			\label{E:DOUBLENULLMAINLATERALERRORINTEGRALSFORSPECIFICVORTICITYALONGH} \\
	& \ \
		-
		\int_{\mathcal{H}_u(\underline{u}_1,\underline{u}_2)}
			\left\lbrace
				\mathfrak{H}_{(\pmb{\partial} \weight)}[\vortrenormalized]
				+
				\weight
				\mathfrak{H}[\vortrenormalized]
				+
				\weight
				\mathfrak{H}_{(1)}[\vortrenormalized]
			\right\rbrace
		\, d \varpi_{\gsphere} d \underline{u}',
		\notag
			\\
		\int_{\underline{\mathcal{H}}_{\underline{u}}(u_1,u_2)}
			\weight 
			\spherenormal_{\alpha} J^{\alpha}[\GradEnt]
		\, d \varpi_{\gsphere} d u'
		& = 
		-
		\int_{\mathcal{S}_{u_2,\underline{u}}}
			\weight 
			\frac{\uposinnerproduct}{\seconduposinnerproduct \ReciprocaluLunitAppliedtoTimeFunction}
			|\angGradEnt|_{\gsphere}^2
		\, d \varpi_{\gsphere}
		+
		\int_{\mathcal{S}_{u_1,\underline{u}}}
			\weight 
			\frac{\uposinnerproduct}{\seconduposinnerproduct \ReciprocaluLunitAppliedtoTimeFunction}
			|\angGradEnt|_{\gsphere}^2
		\, d \varpi_{\gsphere}
			 \label{E:DOUBLENULLMAINLATERALERRORINTEGRALSFORENTROPYGRADIENTALONGHBAR} \\
	& \ \
		+
		\int_{\underline{\mathcal{H}}_{\underline{u}}(u_1,u_2)}
			\left\lbrace
				\underline{\mathfrak{H}}_{(\pmb{\partial} \weight)}[\GradEnt]
				+
				\weight
				\underline{\mathfrak{H}}[\GradEnt]
				+
				\weight
				\underline{\mathfrak{H}}_{(2)}[\GradEnt]
			\right\rbrace
		\, d \varpi_{\gsphere} d u',
		\notag
			\\
		\int_{\mathcal{H}_u(\underline{u}_1,\underline{u}_2)}
			\weight 
			\spherenormal_{\alpha} J^{\alpha}[\GradEnt]
		\, d \varpi_{\gsphere} d \underline{u}'
		& = 
		\int_{\mathcal{S}_{u,\underline{u}_2}}
			\weight 
			\frac{\posinnerproduct}{\secondposinnerproduct \ReciprocalLunitAppliedtoTimeFunction}
			|\angGradEnt|_{\gsphere}^2
		\, d \varpi_{\gsphere}
		-
		\int_{\mathcal{S}_{u,\underline{u}_1}}
			\weight 
			\frac{\posinnerproduct}{\secondposinnerproduct \ReciprocalLunitAppliedtoTimeFunction}
			|\angGradEnt|_{\gsphere}^2
		\, d \varpi_{\gsphere}
			\label{E:DOUBLENULLMAINLATERALERRORINTEGRALSFORENTROPYGRADIENTALONGH} \\
	& \ \
		-
		\int_{\mathcal{H}_u(\underline{u}_1,\underline{u}_2)}
			\left\lbrace
				\mathfrak{H}_{(\pmb{\partial} \weight)}[\GradEnt]
				+
				\weight
				\mathfrak{H}[\GradEnt]
				+
				\weight
				\mathfrak{H}_{(2)}[\GradEnt]
			\right\rbrace
		\, d \varpi_{\gsphere} d \underline{u}'.
		\notag				
\end{align}
	\end{subequations}
	On RHSs~\eqref{E:DOUBLENULLMAINLATERALERRORINTEGRALSFORSPECIFICVORTICITYALONGHBAR}-\eqref{E:DOUBLENULLMAINLATERALERRORINTEGRALSFORENTROPYGRADIENTALONGH},
	for $\SigmatTan \in \lbrace \vortrenormalized, \GradEnt \rbrace$,
		the scalar functions
		$\underline{\mathfrak{H}}_{(\pmb{\partial} \weight)}[\SigmatTan]$,
		$\underline{\mathfrak{H}}[\SigmatTan]$,
			$\mathfrak{H}_{(\pmb{\partial} \weight)}[\SigmatTan]$,
		and
		$\mathfrak{H}[\SigmatTan]$
		are defined in 
		\eqref{E:WEIGHTDERIVATIVESDOUBLENULLMAINTHMINGOINGLATERALBOUNDARYEASYERRORINTEGRANDTERMS}-\eqref{E:DOUBLENULLMAINTHMOUTGOINGLATERALBOUNDARYEASYERRORINTEGRANDTERMS},
		and the scalar functions
		$\underline{\mathfrak{H}}_{(1)}[\vortrenormalized]$,
		$\mathfrak{H}_{(1)}[\vortrenormalized]$,
		$\underline{\mathfrak{H}}_{(2)}[\GradEnt]$,
		$\mathfrak{H}_{(2)}[\GradEnt]$
		are defined in
		\eqref{E:DOUBLENULLMAINTHMSPECIFICVORTITICYMAININGOINGLATERALERRORINTEGRAND}-\eqref{E:DOUBLENULLMAINTHMENTROPYGRADIENTMAINOUTGOINGLATERALERRORINTEGRAND}.
		
\end{proposition}

\begin{proof}
Prop.\,\ref{P:DOUBLENULLSTRUCTUREOFERRORINTEGRALS} follows from 
integrating the identities provided by Lemma~\ref{L:DOUBLENULLPRELIMINARYANALYSISOFBOUNDARYINTEGRAND}
in the same way that
Prop.\,\ref{P:STRUCTUREOFERRORINTEGRALS} followed from integrating the identity provided by Lemma~\ref{L:PRELIMINARYANALYSISOFBOUNDARYINTEGRAND},
where we use the identities
\eqref{E:DOUBLENULLKEYHBARINTEGRALIDENTITY}-\eqref{E:DOUBLENULLKEYHINTEGRALIDENTITY} 
(see also Remark~\ref{R:TERMWILLBECANCELED})
in place of the identity \eqref{E:KEYHYPERSURFACEINTEGRALIDENTITY}
used in the proof of Prop.\,\ref{P:STRUCTUREOFERRORINTEGRALS}.
For clarity, we will provide a few additional details.
When one proves \eqref{E:DOUBLENULLMAINLATERALERRORINTEGRALSFORSPECIFICVORTICITYALONGHBAR},
the argument described above, based on \eqref{E:DOUBLENULLKEYHBARINTEGRALIDENTITY}
as well as
\eqref{E:DOUBLENULLPRELIMINARYDECOMPOFBOUNDARYINTEGRANDINTERESTINGCASE1} in the case 
$\underline{u} = \underline{U}$
and
\eqref{E:DOUBLENULLPRELIMINARYDECOMPOFBOUNDARYINTEGRANDDATACASE1}
in the case $\underline{u} = 0$,
where $\vortrenormalized$ is in the role of $\SigmatTan$ in both
\eqref{E:DOUBLENULLPRELIMINARYDECOMPOFBOUNDARYINTEGRANDINTERESTINGCASE1}
and
\eqref{E:DOUBLENULLPRELIMINARYDECOMPOFBOUNDARYINTEGRANDDATACASE1},
leads to an integral identity analogous to \eqref{E:INTEGRATEDPRELIMINARYDECOMPOFBOUNDARYINTEGRAND},
which in particular features the following error integral on the right-hand side:
		$
		2 
		\int_{\underline{\mathcal{H}}_{\underline{u}}(u_1,u_2)}
			\weight
			\uposinnerproduct
			\uspecialgen^{\alpha}
			\angV^{\beta}
			(\partial_{\alpha} \vortrenormalized_{\beta} - \partial_{\beta} \vortrenormalized_{\alpha})
		\, d \varpi_{\gsphere} d u'
		$.
To handle this integral, we use the identity 
\eqref{E:KEYIDENTITYANTISYMMETRICPARTOFSPECIFICVORTICITYDUALGRADIENT} to substitute
for the integrand factor
$
\partial_{\alpha} \vortrenormalized_{\beta} - \partial_{\beta} \vortrenormalized_{\alpha}
$
in this integral.
The resulting identity features an integral that is analogous to \eqref{E:DIFFICULTINTEGRALPROOFOFSTRUCTUREOFERRORINTEGRALS},
namely 
\begin{align} \label{E:DOUBLENULLDIFFICULTINTEGRALPROOFOFSTRUCTUREOFERRORINTEGRALS}
2 
\int_{\underline{\mathcal{H}}_{\underline{u}}(u_1,u_2)}
			\weight
			\uposinnerproduct
			\Speed^{-4} 
			\exp(-2 \LogDensity) 
			\frac{p_{;\Ent}}{\bar{\varrho}}
			\upepsilon_{\alpha \beta \gamma \delta}
			\uspecialgen^{\alpha}
			\angV^{\beta}
			\left\lbrace
				-
				(\Transport v^{\gamma}) 
				\GradEnt^{\delta}
				+
				 \Transport^{\gamma}
				[\GradEnt^{\delta}
				(\partial_a v^a)
				-
				\GradEnt^a \partial_a v^{\delta}]
			\right\rbrace
		\, d \varpi_{\gsphere} d u',
\end{align}
which is generated by the fifth and sixth products on RHS~\eqref{E:KEYIDENTITYANTISYMMETRICPARTOFSPECIFICVORTICITYDUALGRADIENT}.
We then rewrite the integrand factors 
\[
\upepsilon_{\alpha \beta \gamma \delta}
			\uspecialgen^{\alpha}
			\angV^{\beta}
			\left\lbrace
				-
				(\Transport v^{\gamma}) 
				\GradEnt^{\delta}
				+
				 \Transport^{\gamma}
				[\GradEnt^{\delta}
				(\partial_a v^a)
				-
				\GradEnt^a \partial_a v^{\delta}]
			\right\rbrace
\]
from \eqref{E:DOUBLENULLDIFFICULTINTEGRALPROOFOFSTRUCTUREOFERRORINTEGRALS}
by using the remarkable identity \eqref{E:DOUBLENULLINGOINGPRELIMINARYDECOMPOSITIONOFSUBTLETERMS} 
for substitution. 
In total, these steps yield \eqref{E:DOUBLENULLMAINLATERALERRORINTEGRALSFORSPECIFICVORTICITYALONGHBAR}.

\eqref{E:DOUBLENULLMAINLATERALERRORINTEGRALSFORENTROPYGRADIENTALONGHBAR} can be proved using nearly
identical arguments, where we use the identity \eqref{E:KEYIDENTITYANTISYMMETRICPARTOFENTROPYGRADIENTDUALGRADIENT}
in place of the identity \eqref{E:KEYIDENTITYANTISYMMETRICPARTOFSPECIFICVORTICITYDUALGRADIENT} used in the previous paragraph.

The identity \eqref{E:DOUBLENULLMAINLATERALERRORINTEGRALSFORSPECIFICVORTICITYALONGH} 
can be proved using arguments similar to the ones we used to prove
\eqref{E:DOUBLENULLMAINLATERALERRORINTEGRALSFORSPECIFICVORTICITYALONGHBAR},
where we use 
\eqref{E:DOUBLENULLKEYHINTEGRALIDENTITY}
in the role of
\eqref{E:DOUBLENULLKEYHBARINTEGRALIDENTITY},
\eqref{E:DOUBLENULLPRELIMINARYDECOMPOFBOUNDARYINTEGRANDINTERESTINGCASE2}
in the role of
\eqref{E:DOUBLENULLPRELIMINARYDECOMPOFBOUNDARYINTEGRANDINTERESTINGCASE1},
and \eqref{E:DOUBLENULLPRELIMINARYDECOMPOFBOUNDARYINTEGRANDDATACASE2}
in the role of
\eqref{E:DOUBLENULLPRELIMINARYDECOMPOFBOUNDARYINTEGRANDDATACASE1}.

Finally, \eqref{E:DOUBLENULLMAINLATERALERRORINTEGRALSFORENTROPYGRADIENTALONGH} can be proved using arguments
nearly identical to the ones needed to prove \eqref{E:DOUBLENULLMAINLATERALERRORINTEGRALSFORSPECIFICVORTICITYALONGH},
where we use \eqref{E:KEYIDENTITYANTISYMMETRICPARTOFENTROPYGRADIENTDUALGRADIENT}
in the place of the identity \eqref{E:KEYIDENTITYANTISYMMETRICPARTOFSPECIFICVORTICITYDUALGRADIENT}
that is needed for the proof of \eqref{E:DOUBLENULLMAINLATERALERRORINTEGRALSFORSPECIFICVORTICITYALONGH}.
\end{proof}

\subsection{Proof of Theorem~\ref{T:DOUBLENULLMAINTHEOREM}}
\label{SS:PROOFOFTHEOREMDOUBLENULLMAINTHEOREM}
	In this subsection, we use the previously derived results to prove Theorem~\ref{T:DOUBLENULLMAINTHEOREM}.
	
	We first prove \eqref{E:DOUBLENULLSPACETIMEREMARKABLEIDENTITYSPECIFICVORTICITY}.
	For $\underline{u}''$ small and positive, we set
	$\mathcal{M}_1(\underline{u}'') := \mathcal{M}_{U,\underline{U}} \cap \lbrace 1 + \underline{u}'' \leq \Timefunction \leq 1 + \underline{U} \rbrace$.
	We set $\mathcal{M}_2 := \mathcal{M}_{U,\underline{U}} \cap \lbrace 1 + \underline{U} \leq \Timefunction \leq U \rbrace$.
	For $u''$ close to but less that $U$, we set
	$\mathcal{M}_3(u'') := \mathcal{M}_{U,\underline{U}} \cap \lbrace U \leq \Timefunction \leq u'' + \underline{U} \rbrace$.
	Note that $\mathcal{M}_{U,\underline{U}} = \mathcal{M}_1(0) \cup \mathcal{M}_2 \cup \mathcal{M}_3(U)$;
	see Fig.\,\ref{F:DOUBLENULLPROOF}, which can be viewed as a partitioned, ``spherically symmetric'' caricature 
	of Fig.\,\ref{F:DOUBLENULL}.
	
\begin{center}
\begin{overpic}[scale=.65,grid=false]{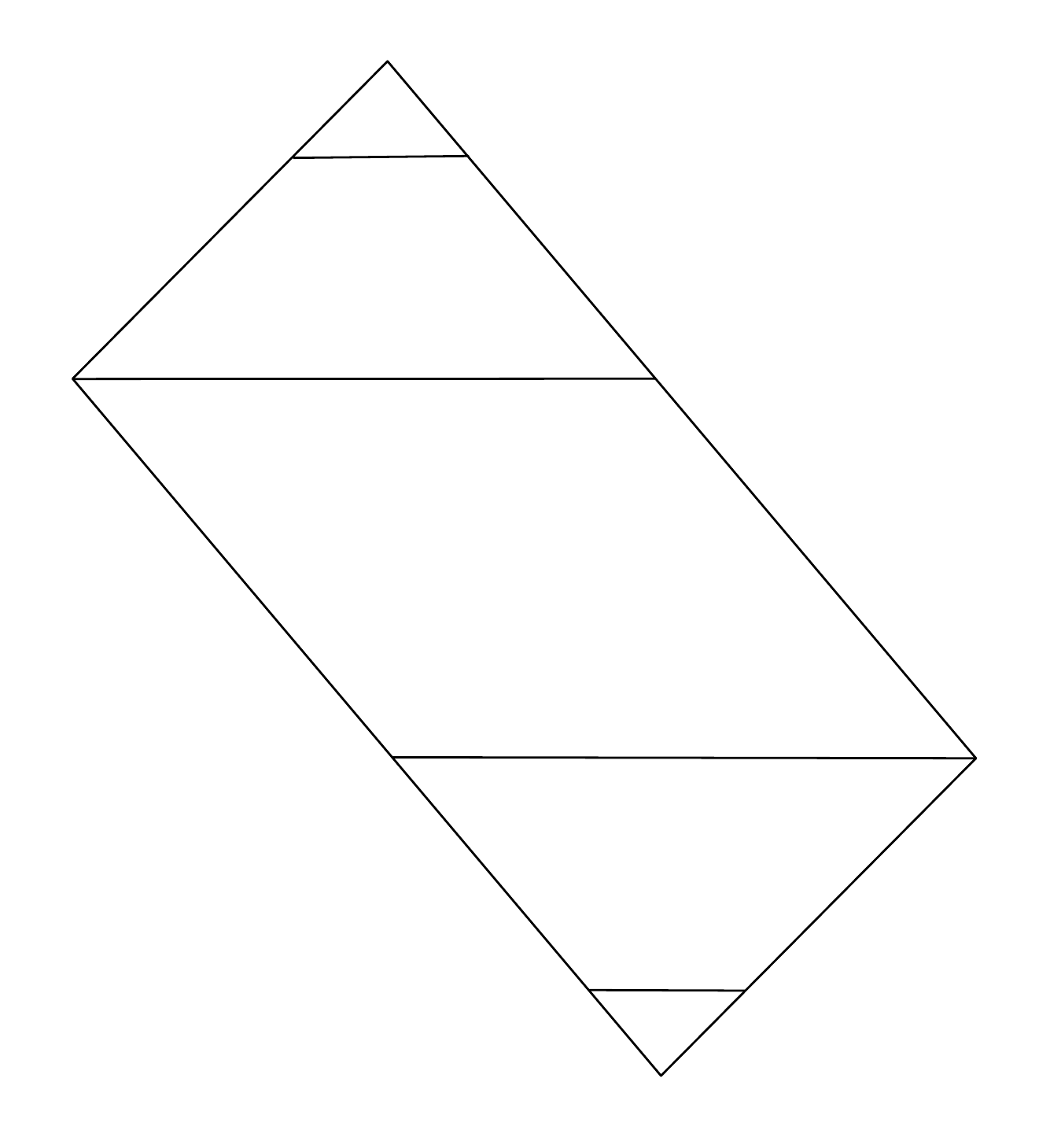} 
\put (52,23) {\tiny $\mathcal{M}_1(\underline{u}'')$}
\put (38,32.5) {\rotatebox{-50}{\tiny$\underline{\mathcal{H}}_{0}(1 + \underline{u}'',1 + \underline{U})$}}
\put (66,19) {\rotatebox{50}{\tiny$\mathcal{H}_1(\underline{u}'',\underline{U})$}}
\put (43,49) {\tiny$\mathcal{M}_2$}
\put (14,60) {\rotatebox{-50}{\tiny$\underline{\mathcal{H}}_{0}(1 + \underline{U},U)$}}
\put (56,62) {\rotatebox{-50}{\tiny$\underline{\mathcal{H}}_{\underline{U}}(1,U - \underline{U})$}}
\put (27,73) {\tiny$\mathcal{M}_3(u'')$}
\put (10,68) {\rotatebox{45}{\tiny$\mathcal{H}_U(0,\underline{U} + u'' - U)$}}
\put (30,96.5) {\tiny $\mathcal{S}_{U,\underline{U}}$}
\put (37,84) {\rotatebox{-50}{\tiny$\underline{\mathcal{H}}_{\underline{U}}(U - \underline{U},u'')$}}
\put (42,86.5) {\tiny $\mathcal{S}_{u'',\underline{U}}$}
\put (9,86.5) {\tiny $\mathcal{S}_{U,\underline{U} + \underline{u}'' - U}$}
\put (0,67) {\tiny $\mathcal{S}_{U,0}$}
\put (58,67) {\tiny $\mathcal{S}_{U - \underline{U},\underline{U}}$}
\put (67,13.5) {\tiny$\mathcal{S}_{1,\underline{u}''}$}
\put (39.5,13.5) {\tiny$\mathcal{S}_{1+\underline{u}'',0}$}
\put (23,34) {\tiny$\mathcal{S}_{1+\underline{U},0}$}
\put (86,34) {\tiny$\mathcal{S}_{1,\underline{U}}$}
\put (55,4) {\tiny$\mathcal{S}_{1,0}$}
\end{overpic}
\captionof{figure}{The spacetime regions in the proof of Theorem~\ref{T:DOUBLENULLMAINTHEOREM}}
\label{F:DOUBLENULLPROOF}
\end{center}

	The main step in proving \eqref{E:DOUBLENULLSPACETIMEREMARKABLEIDENTITYSPECIFICVORTICITY}
	is deriving the following three integral identities (where $u''$ and $\underline{u}''$ are fixed in the identities):
	\begin{align}
		&
		\int_{\mathcal{M}_1(\underline{u}'')}	
			\weight
			\lengthofmodtophypnorm^{-1}
			\mathscr{Q}(\pmb{\partial} \vortrenormalized,\pmb{\partial} \vortrenormalized)
		\, d \varpi_{\gfour}
			\label{E:FIRSTREGIONDOUBLENULLSPACETIMEREMARKABLEIDENTITYSPECIFICVORTICITY} 
			\\
		& = 
		\int_{\mathcal{S}_{1,\underline{U}}}
			\weight
			\frac{\posinnerproduct}{\secondposinnerproduct \ReciprocalLunitAppliedtoTimeFunction}
			|\angvortrenormalized|_{\gsphere}^2
		\, d \varpi_{\gsphere} 
		-
		\int_{\mathcal{S}_{1,\underline{u}''}}
			\weight
			\frac{\posinnerproduct}{\secondposinnerproduct \ReciprocalLunitAppliedtoTimeFunction}
			|\angvortrenormalized|_{\gsphere}^2
		\, d \varpi_{\gsphere}
			\notag \\
		& \ \
		+
		\int_{\mathcal{S}_{1 + \underline{U}},0}
			\weight
			\frac{\uposinnerproduct}{\seconduposinnerproduct \ReciprocaluLunitAppliedtoTimeFunction} 
			|\angvortrenormalized|_{\gsphere}^2
		\, d \varpi_{\gsphere}
		-
		\int_{\mathcal{S}_{1 + \underline{u}''},0}
			\weight
			\frac{\uposinnerproduct}{\seconduposinnerproduct \ReciprocaluLunitAppliedtoTimeFunction} 
			|\angvortrenormalized|_{\gsphere}^2
			|\angvortrenormalized|_{\gsphere}^2
		\, d \varpi_{\gsphere}
			\notag \\
	& \ \
	+	
	\int_{\mathcal{M}_1(\underline{u}'')}
			\weight
			\lengthofmodtophypnorm^{-1}
			\left\lbrace
				\frac{1}{2}|\mathfrak{A}^{(\vortrenormalized)}|_{\topfirstfund}^2
				+
				|\mathfrak{B}_{(\vortrenormalized)}|_{\topfirstfund}^2
				+
				\mathfrak{C}^{(\vortrenormalized)}
				+
				\mathfrak{D}^{(\vortrenormalized)}
				+
				\mathfrak{J}_{(Coeff)}[\vortrenormalized,\pmb{\partial} \vortrenormalized]
			\right\rbrace
	\, d \varpi_{\gfour}
		\notag \\
	& \ \
	+	
	\int_{\mathcal{M}_1(\underline{u}'')}	
			\lengthofmodtophypnorm^{-1}
			\mathfrak{J}_{(\pmb{\partial} \weight)}[\vortrenormalized,\pmb{\partial} \vortrenormalized]
	\, d \varpi_{\gfour}
		\notag \\
	& \ \
		-
		\int_{\underline{\mathcal{H}}_{0}(1 + \underline{u}'',1 + \underline{U})}
			\left\lbrace
				\underline{\mathfrak{H}}_{\pmb{\partial} \weight}[\vortrenormalized]
				+
				\weight
				\underline{\mathfrak{H}}[\vortrenormalized]
				+
				\weight
				\underline{\mathfrak{H}}_{(1)}[\vortrenormalized]
			\right\rbrace
		\, d \varpi_{\gsphere} d u'
			\notag \\
	& \ \
		-
		\int_{\mathcal{H}_1(\underline{u}'',\underline{U})}
			\left\lbrace
				\mathfrak{H}_{\pmb{\partial} \weight}[\vortrenormalized]
				+
				\weight
				\mathfrak{H}[\vortrenormalized]
				+
				\weight
				\mathfrak{H}_{(1)}[\vortrenormalized]
			\right\rbrace
		\, d \varpi_{\gsphere} d \underline{u}',
			\notag
	\end{align}
	
	\begin{align}
		&
		\int_{\mathcal{M}_2}	
			\weight
			\lengthofmodtophypnorm^{-1}
			\mathscr{Q}(\pmb{\partial} \vortrenormalized,\pmb{\partial} \vortrenormalized)
		\, d \varpi_{\gfour}
			\label{E:SECONDREGIONDOUBLENULLSPACETIMEREMARKABLEIDENTITYSPECIFICVORTICITY} 
			\\
		& = 
		\int_{\mathcal{S}_{U,0}}
			\weight
			\frac{\uposinnerproduct}{\seconduposinnerproduct \ReciprocaluLunitAppliedtoTimeFunction} 
			|\angvortrenormalized|_{\gsphere}^2
		\, d \varpi_{\gsphere} 
		-
		\int_{\mathcal{S}_{1 + \underline{U},0}}
			\weight
			\frac{\uposinnerproduct}{\seconduposinnerproduct \ReciprocaluLunitAppliedtoTimeFunction} 
			|\angvortrenormalized|_{\gsphere}^2
		\, d \varpi_{\gsphere} 
			\notag \\
		& \ \
		-
		\int_{\mathcal{S}_{U - \underline{U},\underline{U}}}
			\weight
			\frac{\uposinnerproduct}{\seconduposinnerproduct \ReciprocaluLunitAppliedtoTimeFunction} 
			|\angvortrenormalized|_{\gsphere}^2
		\, d \varpi_{\gsphere}
		+
		\int_{\mathcal{S}_{1,\underline{U}}}
			\weight
			\frac{\uposinnerproduct}{\seconduposinnerproduct \ReciprocaluLunitAppliedtoTimeFunction} 
			|\angvortrenormalized|_{\gsphere}^2
		\, d \varpi_{\gsphere}
			\notag \\
	& \ \
	+	
	\int_{\mathcal{M}_2}
			\weight
			\lengthofmodtophypnorm^{-1}
			\left\lbrace
				\frac{1}{2}|\mathfrak{A}^{(\vortrenormalized)}|_{\topfirstfund}^2
				+
				|\mathfrak{B}_{(\vortrenormalized)}|_{\topfirstfund}^2
				+
				\mathfrak{C}^{(\vortrenormalized)}
				+
				\mathfrak{D}^{(\vortrenormalized)}
				+
				\mathfrak{J}_{(Coeff)}[\vortrenormalized,\pmb{\partial} \vortrenormalized]
			\right\rbrace
	\, d \varpi_{\gfour}
		\notag \\
	& \ \
	+	
	\int_{\mathcal{M}_2}	
			\lengthofmodtophypnorm^{-1}
			\mathfrak{J}_{(\pmb{\partial} \weight)}[\vortrenormalized,\pmb{\partial} \vortrenormalized]
	\, d \varpi_{\gfour}
		\notag \\
	& \ \
		-
		\int_{\underline{\mathcal{H}}_{0}(1 + \underline{U},U)}
			\left\lbrace
				\underline{\mathfrak{H}}_{\pmb{\partial} \weight}[\vortrenormalized]
				+
				\weight
				\underline{\mathfrak{H}}[\vortrenormalized]
				+
				\weight
				\underline{\mathfrak{H}}_{(1)}[\vortrenormalized]
			\right\rbrace
		\, d \varpi_{\gsphere} d u'
			\notag \\
		& \ \
		+
		\int_{\underline{\mathcal{H}}_{\underline{U}}(1,U - \underline{U})}
			\left\lbrace
				\underline{\mathfrak{H}}_{\pmb{\partial} \weight}[\vortrenormalized]
				+
				\weight
				\underline{\mathfrak{H}}[\vortrenormalized]
				+
				\weight
				\underline{\mathfrak{H}}_{(1)}[\vortrenormalized]
			\right\rbrace
		\, d \varpi_{\gsphere} d \underline{u}',
			\notag
	\end{align}
	
	\begin{align}
		&
		\int_{\mathcal{M}_3(u'')}	
			\weight
			\lengthofmodtophypnorm^{-1}
			\mathscr{Q}(\pmb{\partial} \vortrenormalized,\pmb{\partial} \vortrenormalized)
		\, d \varpi_{\gfour}
			\label{E:THIRDREGIONDOUBLENULLSPACETIMEREMARKABLEIDENTITYSPECIFICVORTICITY} 
			\\
		& = 
		-
		\int_{\mathcal{S}_{U,\underline{U} + u'' - U}}
			\weight
			\frac{\posinnerproduct}{\secondposinnerproduct \ReciprocalLunitAppliedtoTimeFunction}
			|\angvortrenormalized|_{\gsphere}^2
		\, d \varpi_{\gsphere} 
		+
		\int_{\mathcal{S}_{U,0}}
			\weight
			\frac{\posinnerproduct}{\secondposinnerproduct \ReciprocalLunitAppliedtoTimeFunction}
			|\angvortrenormalized|_{\gsphere}^2
		\, d \varpi_{\gsphere}
			\notag \\
		& \ \
		-
		\int_{\mathcal{S}_{u'',\underline{U}}}
			\weight
			\frac{\uposinnerproduct}{\seconduposinnerproduct \ReciprocaluLunitAppliedtoTimeFunction} 
			|\angvortrenormalized|_{\gsphere}^2
		\, d \varpi_{\gsphere}
		+
		\int_{\mathcal{S}_{U - \underline{U},\underline{U}}}
			\weight
			\frac{\uposinnerproduct}{\seconduposinnerproduct \ReciprocaluLunitAppliedtoTimeFunction} 
			|\angvortrenormalized|_{\gsphere}^2
		\, d \varpi_{\gsphere}
			\notag \\
	& \ \
	+	
	\int_{\mathcal{M}_3(u'')}
			\weight
			\lengthofmodtophypnorm^{-1}
			\left\lbrace
				\frac{1}{2}|\mathfrak{A}^{(\vortrenormalized)}|_{\topfirstfund}^2
				+
				|\mathfrak{B}_{(\vortrenormalized)}|_{\topfirstfund}^2
				+
				\mathfrak{C}^{(\vortrenormalized)}
				+
				\mathfrak{D}^{(\vortrenormalized)}
				+
				\mathfrak{J}_{(Coeff)}[\vortrenormalized,\pmb{\partial} \vortrenormalized]
			\right\rbrace
	\, d \varpi_{\gfour}
		\notag \\
	& \ \
	+	
	\int_{\mathcal{M}_3(u'')}	
			\lengthofmodtophypnorm^{-1}
			\mathfrak{J}_{(\pmb{\partial} \weight)}[\vortrenormalized,\pmb{\partial} \vortrenormalized]
	\, d \varpi_{\gfour}
		\notag \\
	& \ \
		+
		\int_{\mathcal{H}_U(0,\underline{U} + u'' - U)}
			\left\lbrace
				\mathfrak{H}_{\pmb{\partial} \weight}[\vortrenormalized]
				+
				\weight
				\mathfrak{H}[\vortrenormalized]
				+
				\weight
				\mathfrak{H}_{(1)}[\vortrenormalized]
			\right\rbrace
		\, d \varpi_{\gsphere} d \underline{u}'
			\notag \\
		& \ \
		+
		\int_{\underline{\mathcal{H}}_{\underline{U}}(U - \underline{U},u'')}
			\left\lbrace
				\underline{\mathfrak{H}}_{\pmb{\partial} \weight}[\vortrenormalized]
				+
				\weight
				\underline{\mathfrak{H}}[\vortrenormalized]
				+
				\weight
				\underline{\mathfrak{H}}_{(1)}[\vortrenormalized]
			\right\rbrace
		\, d \varpi_{\gsphere} d u'.
			\notag
	\end{align}
Then by adding \eqref{E:FIRSTREGIONDOUBLENULLSPACETIMEREMARKABLEIDENTITYSPECIFICVORTICITY}-\eqref{E:THIRDREGIONDOUBLENULLSPACETIMEREMARKABLEIDENTITYSPECIFICVORTICITY},
noting the cancellation of the two integrals over $\mathcal{S}_{1 + \underline{U},0}$
and the two integrals over $\mathcal{S}_{U - \underline{U},\underline{U}}$ 
on RHSs~\eqref{E:FIRSTREGIONDOUBLENULLSPACETIMEREMARKABLEIDENTITYSPECIFICVORTICITY}-\eqref{E:THIRDREGIONDOUBLENULLSPACETIMEREMARKABLEIDENTITYSPECIFICVORTICITY},
taking the limit as $\underline{u}'' \downarrow 0$ in \eqref{E:FIRSTREGIONDOUBLENULLSPACETIMEREMARKABLEIDENTITYSPECIFICVORTICITY} 
and $u'' \uparrow U$ in \eqref{E:THIRDREGIONDOUBLENULLSPACETIMEREMARKABLEIDENTITYSPECIFICVORTICITY}, 
and recalling that $\mathcal{M}_{U,\underline{U}} = \mathcal{M}_1(0)  \cup \mathcal{M}_2 \cup \mathcal{M}_3(U)$,
we arrive at the desired identity \eqref{E:DOUBLENULLSPACETIMEREMARKABLEIDENTITYSPECIFICVORTICITY}.

It remains for us to prove 
\eqref{E:FIRSTREGIONDOUBLENULLSPACETIMEREMARKABLEIDENTITYSPECIFICVORTICITY}-\eqref{E:THIRDREGIONDOUBLENULLSPACETIMEREMARKABLEIDENTITYSPECIFICVORTICITY}.
We first prove \eqref{E:SECONDREGIONDOUBLENULLSPACETIMEREMARKABLEIDENTITYSPECIFICVORTICITY}.
The proof is similar to the proof of \eqref{E:SPACETIMEREMARKABLEIDENTITYSPECIFICVORTICITY},
so we only sketch the argument.
For each fixed $\Timefunction' \in [1 + \underline{U},U]$,
we consider the $\widetilde{\Sigma}_{\Timefunction'}$-divergence identity \eqref{E:NEWSTANDARDDIVERGENCEIDENTITYFORELLIPTICHYPERBOLICCURRENT}
with $\vortrenormalized$ in the role of $\SigmatTan$.
We integrate this identity over $\widetilde{\Sigma}_{\Timefunction'}$ with respect to the volume form
$d \varpi_{\topfirstfund}$ of $\topfirstfund$, 
and use the divergence theorem
to obtain an integral identity. 
This ``spatial'' integral identity features \emph{two boundary integrals}
coming from the term $\widetilde{\nabla}_{\alpha} \left(\weight J^{\alpha}[\SigmatTan] \right)$ 
on RHS~\eqref{E:NEWSTANDARDDIVERGENCEIDENTITYFORELLIPTICHYPERBOLICCURRENT}
(in the proof of \eqref{E:SPACETIMEREMARKABLEIDENTITYSPECIFICVORTICITY}, we encountered only one boundary integral):
	$
	-
	\int_{\mathcal{S}_{u''',0}}
		\weight \spherenormal_{\alpha} J^{\alpha}[\vortrenormalized]
	\,  d \varpi_{\gsphere}
	+
		\int_{\mathcal{S}_{\underline{u}''',\underline{U}}}
		\weight \spherenormal_{\alpha} J^{\alpha}[\vortrenormalized]
	\,  d \varpi_{\gsphere}
	$,
	where $u'''= \underline{u}''' + \underline{U} = \Timefunction'$.
	We clarify that by definition \eqref{E:DOUBLENULLSPHERENORMAL},
	$\spherenormal$ points inward to $\widetilde{\Sigma}_{\Timefunction'}$ at $\mathcal{S}_{u''',0}$,
	while $\spherenormal$ points outward to $\widetilde{\Sigma}_{\Timefunction'}$ at $\mathcal{S}_{\underline{u}''',\underline{U}}$;
	see also Figures~\ref{F:DOUBLENULL} and \ref{F:DOUBLENULLPROOF}.
	This explains the different signs of the two boundary integrals;
	see also Remark~\ref{R:DOUBLENULLORIENTATIONOFSPHERENORMAL}.
	To account for the last term on
	RHS~\eqref{E:IDENTITYMAINQUADRATICFORMFORCONTROLLINGFIRSTDERIVATIVESOFSPECIFICVORTICITYANDENTROPYGRADIENT}
	(with $\vortrenormalized$ in the role of $\SigmatTan$),
	we add the integral 
	$\int_{\widetilde{\Sigma}_{\Timefunction}'} 
		\weight 
		\topfirstfund_{\alpha \beta} (\Transport \vortrenormalized^{\alpha}) (\Transport \vortrenormalized^{\beta}) 
	d \varpi_{\topfirstfund}$ to each side of the integral identity.
	We then integrate the resulting integral identity with respect $\Timefunction'$ over the interval $[1 + \underline{U},U]$.
	Considering \eqref{E:DOUBLENULLTIMEFUNCTION}, we see that the two aforementioned boundary integrals lead to the
	following two $\gfour$-null hypersurface integrals:
	$
	-
	\int_{\underline{\mathcal{H}}_{0}(1 + \underline{U},U)}
		\weight \spherenormal_{\alpha} J^{\alpha}[\vortrenormalized]
	\,  d \varpi_{\gsphere} d u'
	+
	\int_{\underline{\mathcal{H}}_{\underline{U}}(1,U - \underline{U})}
		\weight \spherenormal_{\alpha} J^{\alpha}[\vortrenormalized]
	\,  d \varpi_{\gsphere} d \underline{u'}
	$.
	We then use two applications of
	\eqref{E:DOUBLENULLMAINLATERALERRORINTEGRALSFORSPECIFICVORTICITYALONGHBAR}
	(see Footnote~\ref{FN:UNNECESSARILYRESCTRICTEDEIKONALVALUES})
	to substitute for these two $\gfour$-null hypersurface integrals.
	Also using the arguments given in the discussion surrounding \eqref{E:FIRSTSTEPSPACETIMEREMARKABLEIDENTITYSPECIFICVORTICITY}
	and the identity 
	$d \varpi_{\topfirstfund} d \Timefunction' = \lengthofmodtophypnorm^{-1} d \varpi_{\gfour}$ 
	(see \eqref{E:SPACETIMEVOLUMEFORMEXPRESSIONWITHRESPECTTOTIMEFUNCTION}),
	we arrive at the desired identity \eqref{E:SECONDREGIONDOUBLENULLSPACETIMEREMARKABLEIDENTITYSPECIFICVORTICITY}.
	
	The identity \eqref{E:FIRSTREGIONDOUBLENULLSPACETIMEREMARKABLEIDENTITYSPECIFICVORTICITY} can be proved using arguments
	similar to the ones we used to prove \eqref{E:SECONDREGIONDOUBLENULLSPACETIMEREMARKABLEIDENTITYSPECIFICVORTICITY},
	as we now sketch.
	We argue as before, 
	this time obtaining a ``spatial'' integral identity (over $\widetilde{\Sigma}_{\Timefunction'}$) that involves
	the two boundary integrals
	$
	-
	\int_{\mathcal{S}_{u''',0}}
		\weight \spherenormal_{\alpha} J^{\alpha}[\vortrenormalized]
	\,  d \varpi_{\gsphere}
	+
		\int_{\mathcal{S}_{1,\underline{u}'''}}
		\weight \spherenormal_{\alpha} J^{\alpha}[\vortrenormalized]
	\,  d \varpi_{\gsphere}
	$,
	where $u''' = 1 + \underline{u}''' = \Timefunction'$.
	Again adding the integral 
	$\int_{\widetilde{\Sigma}_{\Timefunction}'} 
		\weight 
		\topfirstfund_{\alpha \beta} (\Transport \vortrenormalized^{\alpha}) (\Transport \vortrenormalized^{\beta}) 
	d \varpi_{\topfirstfund}$ to each side of the integral identity
	to account for the last term on
	RHS~\eqref{E:IDENTITYMAINQUADRATICFORMFORCONTROLLINGFIRSTDERIVATIVESOFSPECIFICVORTICITYANDENTROPYGRADIENT}
	and integrating with respect $\Timefunction'$ over $\Timefunction' \in [1 + \underline{u}'',1 + \underline{U}]$,
	we arrive at an integral identity that involves the following two $\gfour$-null hypersurface integrals:
	$
	-
	\int_{\underline{\mathcal{H}}_{0}(1 + \underline{u}'',1 + \underline{U})}
		\weight \spherenormal_{\alpha} J^{\alpha}[\vortrenormalized]
	\,  d \varpi_{\gsphere} d u'
	+
	\int_{\mathcal{H}_1(\underline{u}'',\underline{U})}
		\weight \spherenormal_{\alpha} J^{\alpha}[\vortrenormalized]
	\,  d \varpi_{\gsphere} d \underline{u'}
	$.
	Next, we respectively use
	\eqref{E:DOUBLENULLMAINLATERALERRORINTEGRALSFORSPECIFICVORTICITYALONGHBAR}-\eqref{E:DOUBLENULLMAINLATERALERRORINTEGRALSFORSPECIFICVORTICITYALONGH}
	to substitute for these $\gfour$-null hypersurface integrals
	(note that in proving \eqref{E:SECONDREGIONDOUBLENULLSPACETIMEREMARKABLEIDENTITYSPECIFICVORTICITY}, 
	we used only \eqref{E:DOUBLENULLMAINLATERALERRORINTEGRALSFORSPECIFICVORTICITYALONGHBAR}).
	Carefully noting the sign differences between the terms on RHS~\eqref{E:DOUBLENULLMAINLATERALERRORINTEGRALSFORSPECIFICVORTICITYALONGHBAR}
	and
	RHS~\eqref{E:DOUBLENULLMAINLATERALERRORINTEGRALSFORSPECIFICVORTICITYALONGH},
	and again using the arguments given in the discussion surrounding \eqref{E:FIRSTSTEPSPACETIMEREMARKABLEIDENTITYSPECIFICVORTICITY}
	as well as the identity 
	$d \varpi_{\topfirstfund} d \Timefunction' = \lengthofmodtophypnorm^{-1} d \varpi_{\gfour}$,
	we arrive at \eqref{E:FIRSTREGIONDOUBLENULLSPACETIMEREMARKABLEIDENTITYSPECIFICVORTICITY}.
	
	The identity \eqref{E:THIRDREGIONDOUBLENULLSPACETIMEREMARKABLEIDENTITYSPECIFICVORTICITY}
	can be proved using arguments similar to the ones we used to prove
	\eqref{E:FIRSTREGIONDOUBLENULLSPACETIMEREMARKABLEIDENTITYSPECIFICVORTICITY}.
	The two $\gfour$-null hypersurface integrals that one encounters are
	\[
	-
	\int_{\mathcal{H}_U(0,\underline{U} + u'' - U)}
		\weight \spherenormal_{\alpha} J^{\alpha}[\vortrenormalized]
	\,  d \varpi_{\gsphere} d \underline{u'}
	+
	\int_{\underline{\mathcal{H}}_{\underline{U}}(U - \underline{U},u'')}
		\weight \spherenormal_{\alpha} J^{\alpha}[\vortrenormalized]
	\,  d \varpi_{\gsphere} d u',
	\]
	and we again respectively use
	\eqref{E:DOUBLENULLMAINLATERALERRORINTEGRALSFORSPECIFICVORTICITYALONGHBAR}-\eqref{E:DOUBLENULLMAINLATERALERRORINTEGRALSFORSPECIFICVORTICITYALONGH}
	to substitute for them.
	This completes the proof of \eqref{E:DOUBLENULLSPACETIMEREMARKABLEIDENTITYSPECIFICVORTICITY}.
	
	The identity \eqref{E:DOUBLENULLSPACETIMEREMARKABLEIDENTITYENTROPYGRADIENT}
	can be proved using nearly identical arguments,
	where we use the identities
	\eqref{E:DOUBLENULLMAINLATERALERRORINTEGRALSFORENTROPYGRADIENTALONGHBAR}-\eqref{E:DOUBLENULLMAINLATERALERRORINTEGRALSFORENTROPYGRADIENTALONGH}
	in place of the identities
	\eqref{E:DOUBLENULLMAINLATERALERRORINTEGRALSFORSPECIFICVORTICITYALONGHBAR}-\eqref{E:DOUBLENULLMAINLATERALERRORINTEGRALSFORSPECIFICVORTICITYALONGH}
	that we used in proving \eqref{E:DOUBLENULLSPACETIMEREMARKABLEIDENTITYSPECIFICVORTICITY};
	we omit the details. We have therefore proved Theorem~\ref{T:DOUBLENULLMAINTHEOREM}.
	
	\hfill $\qed$

\appendix

\newcolumntype{L}[1]{>{\raggedright\arraybackslash}p{#1}}
\newcolumntype{C}[1]{>{\centering\arraybackslash}p{#1}}
\newcolumntype{R}[1]{>{\raggedleft\arraybackslash}p{#1}}

\section{Notation for Sections~\ref{S:INTRO}-\ref{S:APRIORI}}	
\label{A:APPENDIXFORBULK}
For the reader's convenience, in this appendix, we have gathered some of the notation
from Sects.\,\ref{S:INTRO}-\ref{S:APRIORI} into a table.
We caution that some of the symbols defined in Sects.\,\ref{S:INTRO}-\ref{S:APRIORI} 
have a slightly different -- although analogous -- definition in Sect.\,\ref{S:DOUBLENULL};
see Subsubsect.\,\ref{SSS:TENSORFIELDSWITHSAMEDEFINITIONS} for clarification of this point.
We also refer to Subsubsect.\,\ref{SSS:BASICNOTATION} for basic notation.

\begin{tabular}{ |L{3in} | L{3in} | }
	\hline
	Symbol & Description/ Reference
		\\
	\hline
	\hline
	$\weight$ & An arbitrary scalar function, fixed throughout 
		\\
		\hline
	$\Lie_{\mathbf{X}}$ & Lie differentiation with respect to $\mathbf{X}$
		\\
		\hline
	$\LogDensity$, $\vortrenormalized$, $\GradEnt$ & Def.\,\ref{D:ADDITIONALFLUIDVARIABLES}
		\\
	\hline
	$\vortrenormalized_{\flat}$, $\GradEnt_{\flat}$ & \eqref{E:DUALONEFORMINDICES}
	\\
	\hline
	$d \vortrenormalized_{\flat}$, $d \GradEnt_{\flat}$ & Cor.\,\ref{C:SHARPDECOMPOSITIONOFANTISYMMETRICGRADIENTS}
		\\
		\hline
	$\VortVort$, $\DivGradEnt$ & Def.\,\ref{D:RENORMALIZEDCURLOFSPECIFICVORTICITY}
		\\
	\hline
	$\gfour$, $\gfour^{-1}$ & Def.\,\ref{D:ACOUSTICALMETRIC}	
	\\ 
	\hline
	$\Transport$ & \eqref{E:MATERIALVECTORVIELDRELATIVECTORECTANGULAR}
	\\
	\hline
	$\square_{\gfour}$ & Def.\,\ref{D:COVWAVEOP}
		\\
		\hline
		$\Timefunction$ & Subsect.\,\ref{SS:DOMAINANDTIMEFUNCTIONETC}
	\\
	\hline
	$\mathcal{M}$, $\mathcal{M}_{\Timefunction}$, $\underline{\mathcal{H}}$, $\underline{\mathcal{H}}_{\Timefunction}$,
	$\mathcal{S}_{\Timefunction}$ & Subsect.\,\ref{SS:DOMAINANDTIMEFUNCTIONETC}
	\\
	\hline
	$\underline{\mathcal{N}}$, $\underline{\mathcal{N}}_{\Timefunction}$, $\uLunit$ & Convention~\ref{C:NULLCASE}
	\\
	\hline
	$\tophypnorm$, $\modtophypnorm$, $\sidehypnorm$, $\spherenormal$, $\gen$, $\modgen$ & Def.\,\ref{D:HYPNORMANDSPHEREFORMDEFS}
		\\
		\hline
	$\uposinnerproduct > 0$ & \eqref{E:INGOINGCONDITION}  
		\\
	\hline
	$\seconduposinnerproduct > 0$ & \eqref{E:SECONDINGOINGCONDITION}
		\\
	\hline
	$\lapsemodgen > 0$ & Def.\,\ref{D:HYPNORMANDSPHEREFORMDEFS}, \eqref{E:POSITIVITYOFLAPSEMODGEN}
	\\
	\hline
	$\lengthofgen$, $\lengthofmodgen$, $\lengthoftophypnorm$, $\lengthofmodtophypnorm > 0$, $\lengthofsidehypnorm$ & Def.\,\ref{D:LENGTHOFVARIOUSVECTORFIELDSETC}
		\\
		\hline
	$\hat{\tophypnorm}$, $\hat{\sidehypnorm}$ & Def.\,\ref{D:LENGTHOFVARIOUSVECTORFIELDSETC}
		\\
	\hline
	$g$, $g^{-1}$, $\topfirstfund$, $\topfirstfund^{-1}$, $\sidefirstfund$, $\sidefirstfund^{-1}$, $\gsphere$, $\gsphere^{-1}$
		& Def.\,\ref{D:FIRSTFUNDAMENTALFORMSANDPROJECTIONS}, Lemma~\ref{L:BASICPROPSOFFUNDAMENTALFORMSANDPROJECTIONS}
		\\
	\hline	
	$\Sigmatproject$, $\topproject$, $\sideproject$, $\sphereproject$, $\angV$, $\angvortrenormalized$, $\angGradEnt$
		& Def.\,\ref{D:FIRSTFUNDAMENTALFORMSANDPROJECTIONS}, 
		Lemma~\ref{L:BASICPROPSOFFUNDAMENTALFORMSANDPROJECTIONS},
		Def.\,\ref{D:TANGENTTENSORFIELDS}
			\\
		\hline
	$\toppartialarg{\alpha}$, $\sidepartialarg{\alpha}$, $\angpartialarg{\alpha}$,
	$\toppartialuparg{\alpha}$, $\sidepartialuparg{\alpha}$, $\angpartialuparg{\alpha}$
	&
	Def.\,\ref{D:PROJECTIONSOFCARTESIANCOORDINATEVECTORFIELDS}
		\\
		\hline
	$\Dfour$, $\nabla$, $\widetilde{\nabla}$, $\underline{\nabla}$, $\angD$, $\angdiv$ & Subsect.\,\ref{SS:LEVICIVITACONNECTIONS}
	\\
	\hline
	$\pmb{\partial} \pmb{\upxi}$,
	$\toppartial \pmb{\upxi}$,
	$\sidepartial \pmb{\upxi}$,
	$\angpartial \pmb{\upxi}$,
	$\partial \upxi$
	&
	Subsubsect.\,\ref{SSS:GRADIENTS}
		\\
	\hline
	$|\upxi|_{\euc}$, $|\upxi|_{\euct}$, $|\upxi|_g$, $|\upxi|_{\topfirstfund}$, $|\upxi|_{\gsphere}$
	& Subsubsect.\,\ref{SSS:GRADIENTSANDPOINTWISENORMS}
		\\
	\hline
	$\utang$, $\uspecialgen$ & Def.\,\ref{D:SPECIALGENERATOR}
		\\
	\hline
	$\projectedtransport$ & Def.\,\ref{D:PROJECTIONOFPONTTOTILDESIGMA}
	\\
	\hline
	$\mathscr{Q}$ & Def.\,\ref{D:QUADRATICFORMSFORCONTROLLINGFIRSTDERIVATIVESOFSPECIFICVORTICITYANDENTROPYGRADIENT}, 
		Lemma~\ref{L:POSITIVITYPROPERTIESOFVARIOUSQUADRATICFORMS}
	\\
	\hline
	$J[\SigmatTan]$ & Def.\,\ref{D:ELLIPTICHYPERBOLICCURRENT}
		\\
	\hline
	$\urescalednewgenminushypnorm$ & Def.\,\ref{D:RESCALEDVERSIONOFHYPNORMMINUSNEWGEN}
		\\
	\hline
	$\utandecompvectorfielddownarg{\alpha}$ & Lemma~\ref{L:NEWDECOMPOSITIONOFCOORDINATEPARTIALDERIVATIVEVECTORFIELDS}
		\\
		\hline
	$\LeftoverGradEnt$ & Lemma~\ref{L:ENTROPYVECTORFIELDKEYTENSORIALDECOMPOSITION}
		\\
		\hline
		$\keydetvectorfield$ & Def.\,\ref{D:VECTORFIELDINKEYDETERMINANT}
			\\
		\hline
		$\upsigma$ & \eqref{E:SIGNOFMAINDETERMINANTTERM}
			\\
		\hline
		$d \varpi_{\gfour}$, $d \varpi_g$, $d \varpi_{\topfirstfund}$, $d \varpi_{\sidefirstfund}$ $d \varpi_{\gsphere}$ & Def.\,\ref{D:VOLUMEFORMS}
		\\
		\hline
		$\enmomem$ & \eqref{E:ENMOMENTUMTENSOR}
			\\
		\hline
		$\mathbb{E}_{(Wave)}$, $\mathbb{E}_{(Transport)}$, $\mathbb{F}_{(Wave)}$, 
		$\mathbb{F}_{(Transport)}$ & Def.\,\ref{D:ENERGIESANDFLUXES}
			\\
		\hline
		$\Jenarg{\mathbf{X}}{\alpha}$
		& \eqref{E:MULTIPLIERVECTORFIELD}
			\\
		\hline
		$\deformarg{\mathbf{X}}{\alpha}{\beta}$
		&
		\eqref{E:DEFORMATIONTENSOR}
			\\
		\hline
		$\spacetimeen$, $\toten$ & Def.\,\ref{D:CONTROLLINGQUANTITIES}
		\\
		\hline
\end{tabular}	
\bibliographystyle{amsalpha}
\bibliography{JBib}

\end{document}